%% file: main.tex
\documentclass{article}
\pdfoutput=1

\input{_preamble}
\input{_commands}

\title{Enumerative invariants in self-dual categories.\\I. Motivic invariants}
\author{Chenjing Bu}
\date{January~2024}

\begin{document}

\initlengths

\maketitle

\vspace{12pt}

\begin{abstract}
    \input{abstract}
\end{abstract}

\vspace{24pt}

{ 
    \hypersetup{urlcolor=black,linkcolor=black}
    \begin{center}
        {\firamedium\large Note}

        \vspace{12pt}

        \begin{minipage}{.9\textwidth}
            \hspace*{1.5em}%
            This paper is superseded by a new paper,
            \emph{Orthosymplectic Donaldson--Thomas theory},
            arXiv:\href{https://arxiv.org/abs/2503.20667}{\texttt{2503.20667}}.
            The new paper significantly simplifies the theory,
            and contains most of the main ideas here.
        \end{minipage}
    \end{center}
}

{\hypersetup{linkcolor=black}
\tableofcontents}

\section{Introduction}

\input{intro}

\section{Exact categories and moduli stacks}

\input{excat}

\section{Self-dual categories and moduli stacks}

\input{sdcat}

\section{Motivic algebraic structures}

\input{alg}

\section{Motivic enumerative invariants}

\input{inv}

\section{Self-dual quivers}

\input{quiver}

\section{Orthogonal and symplectic sheaves}

\input{coh}

\clearpage
\appendix
\addtocontents{toc}{\protect\contentsline{section}{Appendices}{}{}}
\prepareappendices

\section{Rings of motives}

\input{motive}

\section{Stacks in categories}

\input{stacks-cat}

\section{2-vector bundles}

\input{two-vb}

\section{Proof of a combinatorial result}

\input{proof-comb}

\section{Proof of the no-pole theorem}

\input{no-pole}

\newpage
\phantomsection
\addcontentsline{toc}{section}{References}
\hyphenation{Grenz-geb}
\sloppy
\renewcommand*{\bibfont}{\normalfont\small}
\printbibliography

\nopagebreak
\par\noindent\rule{0.38\textwidth}{0.4pt}
{\par\noindent\small
\hspace*{2em}Mathematical Institute, University of Oxford, Oxford OX2 6GG, United Kingdom.\\[-2pt]
\hspace*{2em}Email: \texttt{bu@maths.ox.ac.uk}
}

\end{document}

%% file: _preamble.tex
\usepackage[utf8]{inputenc}
\usepackage[T1]{fontenc}
\usepackage[british]{babel}
\usepackage[style=british]{csquotes}
\usepackage[en-GB]{datetime2}
\DTMlangsetup[en-GB]{ord=omit}
\usepackage{xcolor}

\usepackage{amssymb}
\usepackage{amsmath}
\usepackage[amsmath,hyperref,thmmarks]{ntheorem}
\usepackage{mathpartir}
\usepackage{mathtools}
\usepackage{titlesec} 
\usepackage{hyperref}
\hypersetup{
    bookmarksnumbered,
    colorlinks,
    linkcolor={cyan!50!blue},
    citecolor={cyan!50!blue},
    urlcolor={cyan!50!blue}
}
\usepackage[capitalise,nameinlink]{cleveref}
\usepackage{multicol}

\usepackage{tikz}
\usepackage{tikz-cd}

\usepackage[lining,tabular,scale=.9]{FiraSans}
\usepackage[tt=false,semibold]{libertine}
\usepackage[libertine,smallerops,vvarbb]{newtxmath}
\usepackage[lining,scaled=.8]{FiraMono}
\usepackage[cal=euler,frak=euler,scr=boondox]{mathalpha}
\tikzcdset{
    arrow style=tikz,
    arrows={/tikz/line width=.5pt},
    diagrams={>={Straight Barb[scale=0.8]}}
}

\usepackage{bm}
\usepackage{cases}
\usepackage{accents}

\usepackage{setspace}
\usepackage{subdepth} 

\newcommand{\initlengths}{%
    \setlength{\abovedisplayshortskip}{3pt plus 9pt minus 3pt}%
    \setlength{\belowdisplayshortskip}{9pt plus 9pt minus 9pt}%
    \setlength{\abovedisplayskip}{9pt plus 9pt minus 9pt}%
    \setlength{\belowdisplayskip}{9pt plus 9pt minus 9pt}%
}

\usepackage{titling}

\numberwithin{paragraph}{subsection}
\newcommand{\parafont}{\bfseries}
\newcommand{\parasep}{9pt plus 3pt minus 3pt}

\titleformat{\section}[display]{\Large\firabook}{Section \thesection}{0pt}{\LARGE\mdseries}[\titlerule\vspace{24pt}]
\titleformat{\subsection}{\large\firamedium\boldmath}{\thesubsection}{1em}{}
\titleformat{\paragraph}[runin]{\parafont}{\theparagraph.}{.33em}{\normalfont\bfseries\boldmath}
\titlespacing{\paragraph}{0pt}{\parasep}{.5em}

\newcommand{\prepareappendices}{%
    \titleformat{\section}[display]{\Large\firabook}{Appendix \thesection}{0pt}{\LARGE\mdseries}[\titlerule\vspace{20pt}]
}

\usepackage[titles]{tocloft}

\renewenvironment{abstract}{%
    \centering\begin{minipage}{.9\textwidth}%
    \setlength{\parindent}{1.5em}%
    \centerline{\large\firamedium\abstractname}%
    \par\vspace{12pt}%
}{\end{minipage}\par}

\usepackage{rotating}

\makeatletter
\def\theorem@optheaderfont{\normalfont}
\renewtheoremstyle{plain}%
    {\item[\theorem@headerfont\hskip\labelsep {\parafont##2.}~##1\theorem@separator]}%
    {\item[\theorem@headerfont\hskip\labelsep {\parafont##2.}~##1\ {\theorem@optheaderfont(##3)}\theorem@separator]}
\newtheoremstyle{nonumber}%
    {\item[\theorem@headerfont\hskip\labelsep ##1\theorem@separator]}%
    {\item[\theorem@headerfont\hskip\labelsep ##1\ {\theorem@optheaderfont(##3)}\theorem@separator]}
\newtheoremstyle{tagged}%
    {\item[\theorem@headerfont\hskip\labelsep {\parafont##2.}~##1\theorem@separator]}%
    {\item[\theorem@headerfont\hskip\labelsep {\parafont##2.}~##1\ {\theorem@optheaderfont\formattag{##3}}\theorem@separator]\taglabel{##3}}
\newtheoremstyle{taggednonumber}%
    {\item[\theorem@headerfont\hskip\labelsep ##1\theorem@separator]}%
    {\item[\theorem@headerfont\hskip\labelsep ##1\ {\theorem@optheaderfont\formattag{##3}}\theorem@separator]\taglabel{##3}}
\newtheoremstyle{prooflike}%
    {\item[\theorem@headerfont\hskip\labelsep ##1\theorem@separator]}%
    {\item[\theorem@headerfont\hskip\labelsep ##3\theorem@separator]}
\newtheoremstyle{numproof}%
    {\item[\theorem@headerfont\hskip\labelsep {\parafont##2.}~##1\theorem@separator]}%
    {\item[\theorem@headerfont\hskip\labelsep {\parafont##2.}~##3\theorem@separator]}
\makeatother

\theoremseparator{.}
\theorempreskip{\parasep}
\theorempostskip{\parasep}

\theoremheaderfont{\normalfont\bfseries}
\newtheorem{theorem}[paragraph]{Theorem}
\newtheorem{lemma}[paragraph]{Lemma}

\newtheorem{proposition}[paragraph]{Proposition}
\newtheorem{conjecture}[paragraph]{Conjecture}

\theoremstyle{nonumber}
\theoremsymbol{}
\renewtheorem{theorem*}{Theorem}
\renewtheorem{lemma*}{Lemma}
\renewtheorem{corollary*}{Corollary}
\renewtheorem{proposition*}{Proposition}
\renewtheorem{conjecture*}{Conjecture}

\theoremstyle{plain}
\theorembodyfont{\normalfont}
\newtheorem{definition}[paragraph]{Definition}
\newtheorem{remark}[paragraph]{Remark}
\newtheorem{example}[paragraph]{Example}

\newtheorem{algorithm}[paragraph]{Algorithm}

\theoremstyle{nonumber}
\renewtheorem{definition*}{Definition}
\renewtheorem{remark*}{Remark}
\renewtheorem{example*}{Example}

\theoremstyle{tagged}

\theoremstyle{taggednonumber}
\renewtheorem{assumption*}{Assumption}
\renewtheorem{condition*}{Condition}

\theoremstyle{prooflike}
\theoremsymbol{\ensuremath{\Box}}
\newtheorem{proof}{Proof}
\theoremstyle{numproof}
\newtheorem{numproof}[paragraph]{Proof}

\numberwithin{equation}{section}

\makeatletter
\newcommand{\formattag}[1]{\mbox{\upshape\scshape(#1)}}
\newcommand{\taglabel}[1]{\def\@currentlabel{\formattag{#1}}\label{tag-#1}}
\newcommand{\taglabelvar}[2]{\def\@currentlabel{\formattag{#2}}\label{tag-#1}}
\newcommand{\tagref}[1]{\ref{tag-#1}}
\makeatother

\crefdefaultlabelformat{#2{\upshape#1}#3}

\crefformat{equation}{#2{\upshape(#1)}#3}
\crefrangeformat{equation}{{\upshape#3(#1)#4--#5(#2)#6}}
\crefmultiformat{equation}{#2{\upshape(#1)}#3}{ and #2{\upshape(#1)}#3}{, #2{\upshape(#1)}#3}{, and~#2{\upshape(#1)}#3}

\crefformat{enumi}{#2{\upshape#1}#3}
\crefrangeformat{enumi}{{\upshape#3#1#4--#5#2#6}}
\crefmultiformat{enumi}{#2{\upshape#1}#3}{ and #2{\upshape#1}#3}{, #2{\upshape#1}#3}{, and~#2{\upshape#1}#3}

\crefformat{section}{#2{\upshape\S#1}#3}
\crefrangeformat{section}{{\upshape#3\S\S#1#4--#5#2#6}}
\crefmultiformat{section}{{\upshape#2\S\S#1#3}}{ and {\upshape#2#1#3}}{, {\upshape#2#1#3}}{, and~{\upshape#2#1#3}}
\crefformat{subsection}{#2{\upshape\S#1}#3}
\crefrangeformat{subsection}{{\upshape#3\S\S#1#4--#5#2#6}}
\crefmultiformat{subsection}{{\upshape#2\S\S#1#3}}{ and {\upshape#2#1#3}}{, {\upshape#2#1#3}}{, and~{\upshape#2#1#3}}
\crefformat{paragraph}{#2{\upshape\S#1}#3}
\crefrangeformat{paragraph}{{\upshape#3\S\S#1#4--#5#2#6}}
\crefmultiformat{paragraph}{{\upshape#2\S\S#1#3}}{ and {\upshape#2#1#3}}{, {\upshape#2#1#3}}{, and~{\upshape#2#1#3}}
\crefformat{subappendix}{#2{\upshape\S#1}#3}
\crefrangeformat{subappendix}{{\upshape\S\S#3#1#4--#5#2#6}}

\crefname{algorithm}{Algorithm}{Algorithms}
\crefname{assumption}{Assumption}{Assumptions}
\crefname{construction}{Construction}{Constructions}
\crefname{situation}{Situation}{Situations}
\crefname{figure}{Figure}{Figures}

\usepackage{enumitem}
\setlist[enumerate]{label=\textnormal{(\roman*)}}

\usepackage[labelsep=period, labelfont={bf}]{caption}
\numberwithin{figure}{section}
\numberwithin{table}{section}

\usepackage[
    backend=biber,
    style=oxnum,
    giveninits,
    scnames,
    sorting=nyt
]{biblatex}

%% file: _commands.tex
\newcommand{\Aut}{\mathrm{Aut}}
\newcommand{\bbA}{\mathbb{A}}

\newcommand{\bbK}{\mathbb{K}}
\newcommand{\bbk}{\mathbb{k}}
\newcommand{\bbL}{\mathbb{L}}
\newcommand{\bbN}{\mathbb{N}}
\newcommand{\bbP}{\mathbb{P}}
\newcommand{\bbQ}{\mathbb{Q}}
\newcommand{\bbR}{\mathbb{R}}

\newcommand{\bbZ}{\mathbb{Z}}

\newcommand{\calA}{\mathcal{A}}
\newcommand{\calB}{\mathcal{B}}
\newcommand{\calC}{\mathcal{C}}
\newcommand{\calD}{\mathcal{D}}
\newcommand{\calE}{\mathcal{E}}
\newcommand{\calExt}{\mathcal{E}\mspace{-2mu}\mathit{xt}}
\newcommand{\calF}{\mathcal{F}}
\newcommand{\calG}{\mathcal{G}}

\newcommand{\calHom}{\mathcal{H}\mspace{-2mu}\mathit{om}}
\newcommand{\calI}{\mathcal{I}}

\newcommand{\calM}{\mathcal{M}}
\newcommand{\calN}{\mathcal{N}}
\newcommand{\calO}{\mathcal{O}}
\newcommand{\calP}{\mathcal{P}}
\newcommand{\calQ}{\mathcal{Q}}
\newcommand{\calR}{\mathcal{R}}
\newcommand{\calS}{\mathcal{S}}
\newcommand{\calT}{\mathcal{T}}
\newcommand{\calU}{\mathcal{U}}
\newcommand{\calV}{\mathcal{V}}
\newcommand{\calW}{\mathcal{W}}
\newcommand{\calX}{\mathcal{X}}
\newcommand{\calY}{\mathcal{Y}}
\newcommand{\calZ}{\mathcal{Z}}
\newcommand{\cat}[1]{\smash[b]{\textsf{\textup{#1}}}}
\newcommand{\ch}{\operatorname{ch}}
\newcommand{\chiJ}{\prescript{\chi}{}{\mathrm{J}}}
\newcommand{\citestacks}[1]{\cite[Tag~\href{https://stacks.math.columbia.edu/tag/#1}{\texttt{#1}}]{Stacks}}
\newcommand{\Coh}{\cat{Coh}}

\newcommand{\Db}{\cat{D}^{\smash{\mathrm{b}}}}

\newcommand{\DT}{\mathrm{DT}}
\newcommand{\End}{\mathrm{End}}
\newcommand{\et}{{\smash{\mathrm{\acute{e}t}}}}
\newcommand{\ex}{{\smash{\mathrm{ex}}}}
\newcommand{\Ext}{\mathrm{Ext}}

\newcommand{\frS}{\mathfrak{S}}

\newcommand{\GL}{\mathrm{GL}}
\newcommand{\Ga}{\mathbb{G}_{\mathrm{a}}}
\newcommand{\Gm}{\mathbb{G}_{\mathrm{m}}}
\newcommand{\heart}{\mathbin{\heartsuit}}
\newcommand{\Hom}{\mathrm{Hom}}

\newcommand{\id}{\mathrm{id}}

\newcommand{\iso}{{\smash{\mathrm{iso}}}}
\newcommand{\longsimto}{\mathrel{\overset{\smash{\raisebox{-.8ex}{$\sim$}}\mspace{3mu}}{\longrightarrow}}}
\newcommand{\LSF}{\mathrm{LSF}}
\newcommand{\LSFhat}{\widehat{\mathrm{LSF}}}

\newcommand{\mathhyphen}{\textnormal{-}}

\newcommand{\Mhat}{{\smash{\widehat{\mathrm{M}}}\vphantom{M}}}
\newcommand{\Mtilde}{{\smash{\widetilde{\mathrm{M}}}\vphantom{M}}}
\newcommand{\Mod}{\cat{Mod}}
\newcommand{\Msdss}{\mathcal{M}^\sdss}
\newcommand{\Mss}{\mathcal{M}^\ss}
\newcommand{\nai}{{\smash{\textnormal{na\"{\i}}}}}
\newcommand{\Num}{\mathrm{Num}}
\newcommand{\numberthis}{\addtocounter{equation}{1}\tag{\theequation}}
\newcommand{\oline}[1]{{\smash{\overline{#1}}}}
\newcommand{\op}{{\smash{\mathrm{op}}}}

\newcommand{\pn}[1]{{\smash{(#1)}}}
\newcommand{\poly}{{\smash{\mathrm{P}}}}
\newcommand{\pr}{\operatorname{\smash[b]{\mathrm{pr}}}}
\newcommand{\prestar}{\prescript{*}{}}
\newcommand{\QED}{\hfill\ensuremath{\square}}
\newcommand{\rank}{\operatorname{rank}}

\newcommand{\scalemath}[2]{\scalebox{#1}{\mbox{\ensuremath{\displaystyle #2}}}}
\newcommand{\sd}{{\smash{\mathrm{sd}}}}
\newcommand{\sda}[1]{{\smash{\mathrm{sd},\mspace{2mu}#1}}}
\newcommand{\sdnai}{{\smash{\mathrm{sd},\mspace{2mu}\nai}}}
\newcommand{\sdss}{{\smash{\mathrm{sd},\mspace{2mu}\mathrm{ss}}}}
\newcommand{\setquot}{\mathbin{/}}
\newcommand{\SF}{\mathrm{SF}}
\newcommand{\SFhat}{\smash{\widehat{\mathrm{SF}}}}
\newcommand{\simto}{\mathrel{\overset{\smash{\raisebox{-.8ex}{$\sim$}}\mspace{3mu}}{\to}}}
\newcommand{\simTo}{\mathrel{\overset{\smash{\raisebox{-.7ex}{$\sim$}}\mspace{3mu}}{\Rightarrow}}}
\newcommand{\SL}{\mathrm{SL}}

\newcommand{\Sp}{\mathrm{Sp}}
\newcommand{\Spec}{\operatorname{Spec}}

\renewcommand{\ss}{{\smash[b]{\mathrm{ss}}}}
\newcommand{\Stab}{\mathrm{Stab}}
\newcommand{\stirling}[2]{\genfrac{\{}{\}}{0pt}{}{#1}{#2}}
\newcommand{\stirlingmat}[4]{\left\{\begin{smallmatrix}#1&#2\\#3&#4\end{smallmatrix}\right\}}
\newcommand{\stirlingMat}[4]{\left\{\,\begin{matrix}#1&#2\\#3&#4\end{matrix}\,\right\}}

\newcommand{\Sym}{\operatorname{Sym}}
\newcommand{\td}{\operatorname{td}}
\newcommand{\tria}{{\mathbin{\vartriangle}}}
\newcommand{\upB}{\mathrm{B}}
\newcommand{\upC}{\mathrm{C}}
\newcommand{\upe}{\mathrm{e}}
\newcommand{\upi}{\mathrm{i}}
\newcommand{\upI}{\mathrm{I}}
\newcommand{\upJ}{\mspace{2mu}\mathrm{J}}

\newcommand{\upM}{\mathrm{M}}
\newcommand{\upN}{\mathrm{N}}
\newcommand{\upO}{\mathrm{O}}

\newcommand{\vrp}[1]{\Pi^{\mathrm{vi}}_{#1}}

\newcommand{\+}{\accentset{\mspace{2mu}\smash{\raisebox{-.3ex}{$\scriptscriptstyle+$}}}}

\renewcommand{\>}{\mspace{2mu}}
\renewcommand{\leq}{\leqslant}
\renewcommand{\geq}{\geqslant}
\renewcommand{\preceq}{\preccurlyeq}

\newcommand{\nicesubstack}[1]{\substack{\makebox{\scriptsize$\begin{gathered}#1\end{gathered}$}}}
\newcommand{\leftsubstack}[2][6em]{\substack{\makebox[#1][l]{\scriptsize$\begin{aligned}#2\end{aligned}$}}}

\newcommand{\circleddiamond}{\mathbin{\tikz[baseline=-.6ex]{\draw (0,0) circle (.65ex); \draw (.4ex,0)--(0,.4ex)--(-.4ex,0)--(0,-.4ex)--cycle;}}}

\newcommand{\fixlineheight}{\setlength\lineskiplimit{-\maxdimen}}
\newcommand{\unfixlineheight}{\setlength\lineskiplimit{0pt}}

\DeclareFontFamily{OMX}{MnSymbolE}{}
\DeclareSymbolFont{MnLargeSymbols}{OMX}{MnSymbolE}{m}{n}
\SetSymbolFont{MnLargeSymbols}{bold}{OMX}{MnSymbolE}{b}{n}
\DeclareFontShape{OMX}{MnSymbolE}{m}{n}{
    <-6>  MnSymbolE5
   <6-7>  MnSymbolE6
   <7-8>  MnSymbolE7
   <8-9>  MnSymbolE8
   <9-10> MnSymbolE9
  <10-12> MnSymbolE10
  <12->   MnSymbolE12
}{}
\DeclareFontShape{OMX}{MnSymbolE}{b}{n}{
    <-6>  MnSymbolE-Bold5
   <6-7>  MnSymbolE-Bold6
   <7-8>  MnSymbolE-Bold7
   <8-9>  MnSymbolE-Bold8
   <9-10> MnSymbolE-Bold9
  <10-12> MnSymbolE-Bold10
  <12->   MnSymbolE-Bold12
}{}
\DeclareMathDelimiter{[}    {\mathopen} {MnLargeSymbols}{'000}{MnLargeSymbols}{'000}
\DeclareMathDelimiter{]}    {\mathclose}{MnLargeSymbols}{'005}{MnLargeSymbols}{'005}
\DeclareMathDelimiter{\llbr}{\mathopen} {MnLargeSymbols}{'102}{MnLargeSymbols}{'102}
\DeclareMathDelimiter{\rrbr}{\mathclose}{MnLargeSymbols}{'107}{MnLargeSymbols}{'107}
                     
\makeatletter
\def\big#1{{\hbox{$\left#1\vbox to8.5\p@{}\right.\n@space$}}}
\makeatother

%% file: abstract.tex
This is the first paper in a series where
we propose a theory of enumerative invariants
counting self-dual objects in self-dual categories.
Ordinary enumerative invariants in abelian categories
can be seen as invariants for the structure group $\GL (n)$,
and our theory is an extension of this
to structure groups $\upO (n)$ and $\Sp (2n)$.
Examples of our invariants include
invariants counting principal orthogonal or symplectic bundles,
and invariants counting self-dual quiver representations.

In the present paper, we take the motivic approach,
and define our invariants as
elements in a ring of motives.
We also extract numerical invariants
by taking Euler characteristics of these elements.
We prove wall-crossing formulae relating our invariants
for different stability conditions.
We also provide an explicit algorithm computing invariants for quiver representations,
and we present some numerical results.

%% file: intro.tex
\addtocounter{subsection}{1}

\paragraph{}

A theory of \emph{enumerative invariants}
is a certain way to assign numbers, or other types of data,
to certain types of moduli stacks.
One major branch of the study of such invariants is the study of
moduli stacks of objects in an \emph{abelian category},
such as the category of coherent sheaves on a smooth projective variety,
or the category of representations of a quiver.
Plenty of work has been done in this area,
and different flavours of invariants have been developed.
Notable examples include:

\begin{enumerate}
    \item \sloppy
        Thomas'~\cite{Thomas2000} theory of Donaldson--Thomas invariants
        counting coherent sheaves on Calabi--Yau threefolds.
        
    \item 
        \label{itm-intro-motivic}
        Joyce's series of work
        \cite{Joyce2006I, Joyce2007II, Joyce2007III, Joyce2008IV, Joyce2007Stack}
        on motivic invariants in abelian categories.

    \item
        Mochizuki's~\cite{Mochizuki2009} theory of 
        algebraic Donaldson invariants
        counting coherent sheaves on algebraic surfaces.
        
    \item 
        The theory of Pandharipande--Thomas invariants~\cite{PandharipandeThomas2009}
        counting stable pairs of coherent sheaves on a Calabi--Yau $3$-fold.
        
    \item 
        \label{itm-intro-motivic-dt}
        The theory of generalized Donaldson--Thomas invariants
        in $3$-Calabi--Yau categories, using the motivic approach,
        developed independently by Joyce--Song~\cite{JoyceSong2012}
        and Kontsevich--Soibelman~\cite{KontsevichSoibelman2008}.
        
    \item 
        \label{itm-intro-coha}
        The theory of cohomological Hall algebras,
        initiated by the work of
        Kontsevich--Soibelman~\cite{KontsevichSoibelman2011},
        and the categorification of Donaldson--Thomas invariants,
        notable works on which include Efimov~\cite{Efimov2012},
        Meinhardt--Reineke~\cite{MeinhardtReineke2019},
        Davison--Meinhardt~\cite{DavisonMeinhardt2020},
        and others.

    \item
        The work of Tanaka--Thomas~\cite{TanakaThomasI, TanakaThomasII}
        and Thomas~\cite{Thomas2020}
        on the theory of Vafa--Witten invariants
        counting Higgs pairs on a surface.
        
    \item 
        \label{itm-intro-homological}
        Joyce's recent work~\cite{Joyce2021}
        on homological enumerative invariants,
        and the vertex algebra structure on the homology of moduli stacks.

    \item[]
        \ldots\ldots
\end{enumerate}

\paragraph{}

On the other hand, these are all part of a more general problem,
which is that of counting \emph{principal $G$-bundles} on a variety,
or some coherent sheaf analogue of them, for an algebraic group $G$.
One could also consider quiver analogues of this problem,
using a suitable notion of \emph{$G$-quivers},
such as the notion defined in Derksen--Weyman~\cite{DerksenWeyman2002},
or a modification of it.
We shall refer to such invariants as \emph{$G$-invariants}.

Theories of invariants in abelian categories, as discussed above,
address the cases when $G = \GL (n)$, or sometimes~$\SL (n)$,
which form the A family of algebraic groups.
One advantage in this case is that the objects in question
often form an \emph{abelian category},
which provides rich structures one could work with.

The present series of papers will address the cases when
$G = \upO (n)$~or $\Sp (2n)$,
which cover the B, C, and D families of algebraic groups.
In this case, the objects in question are often
\emph{self-dual objects in a self-dual category},
which is a useful structure.

In particular, the present paper aims to develop
a theory of $G$-invariants for $G = \upO (n)$ or $\Sp (2n)$,
in the flavour of \cref{itm-intro-motivic} above.

\paragraph{}

Previous work on this topic includes
Young~\cite{Young2015,Young2020},
who studied Donaldson--Thomas theory for self-dual quivers,
which are $G$-quivers for $G = \upO (n)$~or $\Sp (2n)$.
The present work is based on a general setting
that is similar to that of Young's work,
and a part of our constructions can also be seen as
generalizations of those of Young.

\paragraph{}

We summarize our main constructions and results as follows.
We start with the following data:

\begin{itemize}
    \item 
        A \emph{self-dual category} $\calA$,
        with moduli stack $\calM$.
        Here, \emph{self-dual} means that there is an equivalence
        $(-)^\vee \colon \calA \simto \calA^\op$,
        which squares to the identity of~$\calA$.
    \item 
        A \emph{groupoid of self-dual objects} $\calA^\sd$,
        with moduli stack $\calM^\sd$.
        Here, a \emph{self-dual object} is an object $E \in \calA$
        with an isomorphism $\phi \colon E \simto E^\vee$
        such that $\phi = \phi^\vee$.
\end{itemize}
For example, one could consider the following situations:

\begin{itemize}
    \item 
        $\calA$ is the category of vector bundles on an algebraic curve $X$,
        and $\calA^\sd$ is the groupoid of orthogonal or symplectic
        principal bundles on $X$.
    \item
        $\calA$ is the category of representations of a self-dual quiver $Q$,
        and $\calA^\sd$ is the groupoid of \emph{self-dual representations}.
\end{itemize}
Joyce~\cite{Joyce2007II} defined \emph{motivic Hall algebras}
associated to the moduli stack $\calM$.
These are algebraic structures
that describe \emph{extensions} in the category $\calA$.
Its elements, called \emph{stack functions},
can be thought of as \emph{motives} relative to $\calM$.
The Hall algebra operation is given by parametrizing
all possible middle terms in a short exact sequence
where the first and third terms are given.

Joyce~\cite{Joyce2008IV} then defined elements
$\delta_\alpha (\tau)$ and $\epsilon_\alpha (\tau)$
in the motivic Hall algebra,
where $\alpha$ is a $K$-theory class, and $\tau$ is a stability condition on $\calA$.
The element $\delta_\alpha (\tau)$
parametrizes the family of $\tau$-semistable objects of class $\alpha$;
the element $\epsilon_\alpha (\tau)$ is similar to $\delta_\alpha (\tau)$,
but assigns rational coefficients to strictly semistable objects,
so that these objects can be `counted with correct multiplicities'.
The relation between them can be written as
\begin{equation}
    \epsilon (\tau; t) = \log \delta (\tau; t) \ ,
\end{equation}
where $t$ is a fixed slope.
See \cref{eq-epsilon-delta-exp-1}
below for the precise meaning of this.

In our case of self-dual categories, we define \emph{motivic Hall modules},
which are modules for motivic Hall algebras,
describing \emph{self-dual extensions} of a self-dual object by an ordinary object.
We define elements $\delta^\sd_\theta (\tau)$ and $\epsilon^\sd_\theta (\tau)$
in the motivic Hall module,
where $\theta$ is a $K$-theory class for self-dual objects.
The element $\delta^\sd_\theta (\tau)$
parametrizes the family of $\tau$-semistable self-dual objects of class $\theta$;
the element $\epsilon^\sd_\theta (\tau)$
assigns rational coefficients to strictly semistable self-dual objects,
which are the correct multiplicities, given by
\begin{equation}
    \epsilon^\sd (\tau) = \delta (\tau; 0)^{-1/2} \diamond \delta^\sd (\tau) \ ,
\end{equation}
where $\diamond$ is the module operation,
as in~\cref{eq-def-epsilon-sd-compact} below.
This is one of the key constructions in this paper.

When we change the stability condition $\tau$,
the changes of the elements $\delta, \epsilon, \delta^\sd, \epsilon^\sd$
are described by \emph{wall-crossing formulae}.
For two stability conditions $\tau, \tilde{\tau}$, 
under certain assumptions,
we prove the wall-crossing formula
\begin{equation}
    \epsilon^\sd_\theta (\tilde{\tau}) =
    \sum_{ \leftsubstack[5em]{
        & n \geq 0; \, \alpha_1, \dotsc, \alpha_n \in C (\calA), \,
        \rho \in C^\sd (\calA) \colon \\[-1.5ex]
        & \theta = \bar{\alpha}_1 + \cdots + \bar{\alpha}_n + \rho 
    } } {}
    U^\sd (\alpha_1, \dotsc, \alpha_n; \tau, \tilde{\tau}) \cdot
    \epsilon_{\alpha_1} (\tau) \diamond \cdots \diamond
    \epsilon_{\alpha_n} (\tau) \diamond
    \epsilon^\sd_{\rho} (\tau) \ ,
\end{equation}
as in \cref{eq-wcf-epsilon-sd} below,
where $U^\sd ({\cdots}) \in \bbQ$ are combinatorial coefficients,
and other notations are explained there.
Note that this involves both self-dual and ordinary invariants.
This is the self-dual analogue of
the wall-crossing formula due to Joyce~\cite{Joyce2008IV},
and can also be used as a tool for computing invariants.

We also show that the wall-crossing formulae
can be rewritten in terms of Lie brackets in the motivic Hall algebra,
together with a Lie bracket-like action on the motivic Hall module,
denoted by $\heart$, as
\begin{multline}
    \epsilon^\sd_\theta (\tilde{\tau}) =
    \sum_{ \leftsubstack[10em]{
        \\[-3ex]
        & n \geq 0; \, m_1, \dotsc, m_n > 0; \\[-.5ex]
        & \alpha_{1,1}, \dotsc, \alpha_{1,m_1}; \dotsc;
        \alpha_{n,1}, \dotsc, \alpha_{n,m_n} \in C (\calA); \,
        \rho \in \smash{C^\sd (\calA)} \colon \\[-.5ex]
        & \theta = (\bar{\alpha}_{1,1} + \cdots + \bar{\alpha}_{1,m_1})
        + \cdots + (\bar{\alpha}_{n,1} + \cdots + \bar{\alpha}_{n,m_n}) + \rho
    } } {}
    \tilde{U}^\sd (\alpha_{1,1}, \dotsc, \alpha_{1,m_1}; \dotsc;
    \alpha_{n,1}, \dotsc, \alpha_{n,m_n}; \tau, \tilde{\tau}) \cdot {} \\[1ex]
    \bigl[ \bigl[ \epsilon_{\alpha_{1,1}} (\tau), \dotsc \bigr] ,
    \epsilon_{\alpha_{1,m_1}} (\tau) \bigr] \heart \cdots \heart
    \bigl[ \bigl[ \epsilon_{\alpha_{n,1}} (\tau), \dotsc \bigr] ,
    \epsilon_{\alpha_{n,m_n}} (\tau) \bigr] \heart 
    \epsilon^\sd_{\rho} (\tau) \ ,
\end{multline}
as in \cref{thm-comb},
where notations are explained there.
This gives hints on a possible
universal wall-crossing theory for type B/C/D invariants,
which can also apply to other situations where
we only have well-behaved Lie algebra and twisted module structures,
but not necessarily associative algebra and module structures.
In particular, this may apply to self-dual analogues
of~\cref{itm-intro-motivic-dt} and~\cref{itm-intro-homological} above.

One can extract invariants
$\upI_\alpha (\tau)$, $\upI^\sd_\theta (\tau)$,
$\upJ_\alpha (\tau)$, and $\upJ^\sd_\theta (\tau)$,
from the elements
$\delta_\alpha (\tau)$, $\delta^\sd_\theta (\tau)$,
$\epsilon_\alpha (\tau)$, and $\epsilon^\sd_\theta (\tau)$
by performing \emph{motivic integration},
where the ordinary case is due to Joyce~\cite{Joyce2008IV}.
These invariants live in a ring of motives,
and satisfy similar relations and wall-crossing formulae
as the ones mentioned above for the stack functions
$\delta_\alpha (\tau)$, $\delta^\sd_\theta (\tau)$,
$\epsilon_\alpha (\tau)$, and $\epsilon^\sd_\theta (\tau)$.

From these, one can then obtain numerical invariants
$\chiJ_\alpha (\tau)$, $\chiJ^\sd_\theta (\tau) \in \bbQ$,
by taking the \emph{Euler characteristics} of the invariants
$\upJ_\alpha (\tau)$, $\upJ^\sd_\theta (\tau)$.
These are well-defined due to the \emph{no-pole theorems},
\cref{thm-no-pole,thm-no-pole-sd},
which state that these Euler characteristics are finite,
which is not true for the invariants $\upI_\alpha (\tau)$ and $\upI^\sd_\theta (\tau)$.
The self-dual no-pole theorem, \cref{thm-no-pole-sd},
is one of the key results of this paper.
We also prove wall-crossing formulae
for these numerical invariants,
under a change of stability condition.

We also define \emph{Donaldson--Thomas invariants} in type~B/C/D,
for self-dual quivers with potentials,
as well as sheaves on a Calabi--Yau threefold,
in \cref{sect-quiv-dt,sect-threefolds},
as a self-dual analogue of the type~A case developed by
Joyce--Song~\cite{JoyceSong2012} and
Kontsevich--Soibelman~\cite{KontsevichSoibelman2008}.
We expect that these invariants satisfy wall-crossing formulae,
but we hope to study them in more detail in future work.

\paragraph{}

One limitation of the motivic method,
in its form in this paper,
is that wall-crossing formulae
are only true in one-dimensional categories,
or slightly more general situations,
such as the category of bundles over curves,
or representations of quivers without relations,
or sheaves on a surface of Kodaira dimension~$\leq 0$.
In more general situations, the invariants should still be well-defined,
but they may no longer satisfy the wall-crossing formulae.

However, in future papers of this series,
we hope to explore other approaches to defining self-dual invariants,
which would possibly be analogues
of~\cref{itm-intro-motivic-dt}, \cref{itm-intro-coha} or~\cref{itm-intro-homological} above.
For \cref{itm-intro-motivic-dt},
this paper defines \emph{Donaldson--Thomas invariants} in type~B/C/D,
for self-dual quivers with potentials,
as well as orthogonal or symplectic sheaves on a Calabi--Yau threefold,
and we hope to study them in more detail in future work.
Some work on~\cref{itm-intro-coha} has already been done by Young~\cites{Young2020}.
For~\cref{itm-intro-homological},
we also hope to construct an analogue in the self-dual setting.

\paragraph{}

This paper is organized as follows.
In \cref{sect-exact-cat},
we present background material on exact categories
and stability conditions,
and we formulate a notion of moduli stacks for exact categories.
These moduli stacks will be \emph{stacks in categories},
and carry extra structure that will be helpful for our theory.
In \cref{sect-sd-cat},
we define \emph{self-dual categories},
and study stability conditions on self-dual categories.
We also define moduli stacks
of self-dual objects in a self-dual category.
In \cref{sect-alg},
we define algebraic structures arising from moduli stacks
of exact categories and self-dual exact categories,
including \emph{motivic Hall algebras} and
\emph{motivic Hall modules}.
In \cref{sect-invariants},
we present most of our main results.
We define self-dual motivic enumerative invariants,
prove their wall-crossing formulae under a change of stability condition,
and we state the \emph{no-pole theorems},
which allows us to take the Euler characteristic of the invariants,
obtaining numerical invariants.

In the sections that follow, \crefrange{sect-quiver}{sect-coh},
we discuss two main applications of our theory,
namely, self-dual representations of self-dual quivers,
and orthogonal or symplectic coherent sheaves
on a smooth, projective variety.

In \cref{sect-motive},
we present background material on rings of motives,
and the theory of stack functions as in Joyce~\cite{Joyce2007Stack}.
In \cref{sect-stacks-in-cat,sect-two-vb},
we study theory of stacks in categories and
$2$-vector bundles,
and prove some basic results about them.
In \cref{sect-proof-comb,sect-proof-no-pole},
we present the proofs of two of our main results,
\cref{thm-comb,thm-no-pole-sd}.

\paragraph{Acknowledgements.}

The author is grateful to his supervisor, Dominic Joyce,
for his guidance throughout this project.
The author would also like to thank
Andr\'es Ib\'a\~nez N\'u\~nez, Tasuki Kinjo, and Henry Liu
for helpful discussions.

This work was done during the author's PhD programme supported by
the Mathematical Institute, University of Oxford.

\paragraph{Notations and conventions.}

Throughout this paper,

\begin{itemize}
    \item 
        $\bbK$ is an algebraically closed field of characteristic $0$.
        
    \item
        A \emph{$\bbK$-variety} is a $\bbK$-scheme
        that is separated and of finite type.
        We do not require a variety to be irreducible or reduced.
        However, a \emph{projective variety}
        is always assumed to be integral, i.e.~irreducible and reduced.
\end{itemize}

\paragraph{List of conditions.}

Some of the conditions that we use in this paper
are labelled with letters in brackets rather than numbers.
For the reader's convenience,
we list all of them here,
together with the section numbers
where they are defined.

\begin{multicols}{3}
    \newcommand{\tagrefbox}[1]{\makebox[3.8em][l]{\tagref{#1}}---}
    \begin{itemize}[nosep]
        \item 
            \tagrefbox{ExCat} \cref{para-excat}
        \item 
            \tagrefbox{Spl} \cref{para-spl}
        \item 
            \tagrefbox{Mod} \cref{para-exmod}
        \item 
            \tagrefbox{Stab} \cref{para-exstab}
        \item 
            \tagrefbox{Fin} \cref{para-exfin}
        \item 
            \tagrefbox{Ext} \cref{para-exext}
        \item 
            \tagrefbox{SdCat} \cref{para-sdcat}
        \item 
            \tagrefbox{SdMod} \cref{para-sdmod}
        \item 
            \tagrefbox{SdExt} \cref{para-sdext}
    \end{itemize}
\end{multicols}

%% file: excat.tex
\label{sect-exact-cat}

A major branch of enumerative geometry is concerned with
studying moduli spaces of objects in \emph{abelian categories}.
Here, we generalize this slightly to the setting of
\emph{exact categories} or \emph{quasi-abelian categories},
which are weaker than abelian categories,
but should still work reasonably well in the context of enumerative geometry.

The main reason for doing this is that
the abelian categories that we are interested in,
such as the category of coherent sheaves on a variety,
do not always have \emph{self-dual structures},
which we need in order to study enumerative invariants
for type B/C/D structure groups.
However, it is often the case that
we have a quasi-abelian subcategory of
the original abelian category that is self-dual,
so we will choose to work with this quasi-abelian subcategory instead.

\subsection{Exact categories}

\paragraph{}

We start by recalling the definition of an exact category.
The following list of axioms is taken from
Keller~\cite[Appendix~A]{Keller1990},
with some modifications.

\begin{definition*}
    An \emph{exact category} is a pair $(\calA, \cat{Ex})$,
    where $\calA$ is an additive category,
    and $\cat{Ex}$ is a class of sequences of maps
    \begin{equation}
        \label{eq-short-exact-seq}
        0 \to
        E \overset{i}{\longrightarrow}
        G \overset{p}{\longrightarrow}
        F \to 0
    \end{equation}
    in $\calA$, called \emph{short exact sequences},
    satisfying the following axioms.

    We call morphisms that appear as the first (resp.\ second)
    non-trivial arrow in a short exact sequence
    an \emph{inclusion} (resp.\ a \emph{projection}).
    Then,
    \begin{enumerate}
        \item 
            Sequences of the form
            $0 \to E \to E \oplus F \to F \to 0$,
            called \emph{split exact sequences},
            are short exact,
            where the two non-trivial maps are the canonical ones.
        \item 
            All short exact sequences are
            kernel--cokernel pairs.
            That is, in~\cref{eq-short-exact-seq},
            we always have $i = \ker (p)$ and
            $p = \operatorname{coker} (i)$.
        \item 
            Inclusions and projections
            are closed under composition.
        \item
            Pushouts along inclusions exist,
            and inclusions are closed under pushouts.
            Dually, pullbacks along projections exist,
            and projections are closed under pullbacks.
    \end{enumerate}
    In this case,
    we usually denote the exact category just by $\calA$.
\end{definition*}

\paragraph{Relation to abelian categories.}
\label{para-excat-abcat}

Every abelian category is an exact category,
by taking $\cat{Ex}$ to be the class of
short exact sequences in that abelian category.

Moreover, any additive full subcategory of an abelian category
is an exact category,
where the notion of exactness is inherited from the
original abelian category.

Conversely, Quillen~\cite[\S2]{Quillen1973} showed that
any essentially small exact category $\calA$
can be embedded in an abelian category $\hat{\calA}$,
as a full, additive subcategory that is closed under extensions,
such that exact sequences in $\calA$ are precisely
those sequences that are exact in $\hat{\calA}$.

For objects $E, F \in \calA$, one can define the \emph{Ext groups}
$\Ext^n (E, F)$ for $n \geq 0$,
as in Neeman--Retakh~\cite{NeemanRetakh1996},
as the group of isomorphism classes of $n$-extensions.
The embedding $\calA \hookrightarrow \hat{\calA}$ above
induces an isomorphism
$\Ext^n (E, F) \simeq \Ext^n_{\smash{\hat{\calA}}} (E, F)$
at least for $n = 0, 1$.

\paragraph{The Grothendieck monoid.}
\label{para-excat-groth}

Let $\calA$ be an exact category
that is \emph{essentially small},
that is, the class of isomorphism classes of objects of $\calA$ form a set,
then we define the \emph{Grothendieck monoid} of $\calA$
to be the abelian monoid $K_+ (\calA)$,
which is the quotient of the free abelian monoid generated by
isomorphism classes $[E]$ of objects $E \in \calA$,
by the relations
\[
    [0] \sim 0 \ , \qquad
    [G] \sim [E] + [F]
\]
for short exact sequences~\cref{eq-short-exact-seq}.

\paragraph{}
\label[paragraph]{para-excat}

We formulate the following set-up for exact categories,
which will be used in the rest of this paper.

\begin{condition*}[ExCat]
    Let $\calA$ be an essentially small $\bbK$-linear exact category,
    with the following properties and extra data:

    \begin{enumerate}
        \item \label{itm-excat-noetherian}
            We assume that $\calA$ is \emph{noetherian}.
            This means that for any object $E \in \calA$,
            any increasing chain of inclusions
            $E_0 \hookrightarrow E_1 \hookrightarrow \cdots \hookrightarrow E$
            must stabilize.

        \item \label{itm-excat-k-monoid}
            We are given
            a fixed quotient monoid $C^\circ (\calA)$ of $K_+ (\calA)$,
            called the \emph{class monoid} of $\calA$.
            For an object $E \in \calA$,
            write $\llbr E \rrbr \in C^\circ (\calA)$ for the class of $E$ in $C^\circ (\calA)$.
            We assume that $\llbr E \rrbr = 0$ implies $E \simeq 0$.
            Write $C (\calA) = C^\circ (\calA) \setminus \{ 0 \}$,
            which is a sub-semigroup of $C^\circ (\calA)$.
    \end{enumerate}
\end{condition*}

\begin{remark}
    \label{rem-sd-noeth-art}
    In the condition \tagref{Ext},
    if $\calA$ also has a \emph{self-dual structure}
    in the sense of \cref{sect-def-sd-cat} below,
    then the noetherian assumption on $\calA$
    also implies that $\calA$ is \emph{artinian},
    meaning that for any object $E \in \calA$,
    any decreasing chain of subobjects
    $E = E_0 \hookleftarrow E_1 \hookleftarrow \cdots$
    must stabilize.
\end{remark}

\subsection{Quasi-abelian categories}

\paragraph{}

Quasi-abelian categories are exact categories
satisfying certain extra properties.
A good reference for quasi-abelian categories
from the viewpoint of enumerative geometry and
stability conditions is Bridgeland~\cite[\S4]{Bridgeland2007},
who showed that quasi-abelian categories
can be viewed as slices in an abelian or triangulated category
with a stability condition,
consisting of objects with slopes in a certain interval.

Many constructions and results in this paper
work in the generality of exact categories.
However, the \emph{no-pole theorems},
\cref{thm-no-pole,thm-no-pole-sd} below,
only hold for quasi-abelian categories,
as their proofs rely on \cref{lem-qa-splitting} below.

\begin{definition}
    A \emph{quasi-abelian category}
    is an exact category $(\calA, \cat{Ex})$
    that has all kernels and cokernels,
    such that $\cat{Ex}$
    is precisely the class of all kernel--cokernel pairs,
    that is, sequences~\cref{eq-short-exact-seq}
    with $i = \ker (p)$ and $p = \operatorname{coker} (i)$.
\end{definition}

\paragraph{Relation to abelian categories.}

As in Schneiders~\cite[\S1.2.3]{Schneiders1998},
for each quasi-abelian category $\calA$,
there exist abelian categories $\calA^+$ and $\calA^-$,
and exact fully faithful embeddings $\calA \hookrightarrow \calA^+$ and $\calA \hookrightarrow \calA^-$,
such that $\calA$ is closed under extensions in both categories,
closed under taking subobjects in $\calA^+$,
and closed under taking quotient objects in $\calA^-$.

\begin{lemma}
    \label{lem-qa-splitting}
    Let $\calA$ be a $\bbK$-linear quasi-abelian category,
    $E \in \calA$ an object,
    $n \geq 1$ an integer, and
    $\iota \colon \bbK^n \hookrightarrow \End (E)$ an inclusion of\/ $\bbK$-algebras,
    where $\bbK^n$ has component-wise multiplication.
    Then, there exists a unique decomposition
    \begin{equation}
        \label{eq-qa-splitting}
        E \simeq E_1 \oplus \cdots \oplus E_n
    \end{equation}
    up to isomorphisms,
    such that for each $i$,
    the $i$-th basis element $e_i \in \bbK^n$
    acts on $E_i$ as the identity,
    and acts on $E_j$ as zero for $j \neq i$.
\end{lemma}

\begin{proof}
    Define $E_i = \ker ( \iota (1 - e_i) )$.
    The subfunctor
    $\Hom (-, E_i) \subset \Hom (-, E) \colon \calA^\op \to \bbK \mathhyphen \cat{Vect}$
    is the $1$-eigenspace of the action of $e_i$,
    and $\Hom (-, E)$ is the direct sum of such eigenspaces,
    giving the desired decomposition by the Yoneda lemma.
\end{proof}

\paragraph{}
\label{para-spl}

For convenience of referencing later,
we formulate a condition on exact categories,
which is stronger than \tagref{ExCat},
and requires the exact category to satisfy
the above splitting property,
which are satisfied by quasi-abelian categories.

\begin{condition*}[Spl]
    Let $\calA$ be a $\bbK$-linear exact category,    
    with properties and extra data as in~\tagref{ExCat},
    such that $\calA$ satisfies the property described in
    \cref{lem-qa-splitting} above,
    but $\calA$ is not necessarily quasi-abelian.
\end{condition*}

\subsection{Categories of filtrations}

\begin{definition}
    \label{def-cat-filt}
    Let $\calA$ be an exact category.
    
    For an integer $n \geq 0$,
    the \emph{category of $n$-step filtrations} of $\calA$
    is the exact category $\calA^\pn{n}$
    whose objects are sequences of inclusions
    \begin{equation}
        \label{eq-obj-filt}
        E_\bullet = \Bigl( \ 
            0 = E_0 \hookrightarrow E_1 \hookrightarrow E_2 \hookrightarrow \cdots \hookrightarrow E_n
        \ \Bigr)
    \end{equation}
    in $\calA$, and whose morphisms are commutative diagrams.
    Exact sequences in $\calA^\pn{n}$ 
    are defined degree-wise.
    
    Note that $\calA^\pn{0} \simeq \{ 0 \}$,
    $\calA^\pn{1} \simeq \calA$,
    and $\calA^\pn{2}$ can be identified with
    the category of short exact sequences in $\calA$.
\end{definition}

\paragraph{The class monoid.}
\label{para-filt-class-monoid}

If, moreover, $\calA$ is equipped with extra data as in~\tagref{ExCat},
including a class monoid $C^\circ (\calA)$,
then we can also define
$C^\circ (\calA^\pn{n}) = C^\circ (\calA)^n$,
and set
\begin{equation}
    \llbr E_\bullet \rrbr = (
        \llbr E_1 / E_0 \rrbr, \,
        \llbr E_2 / E_1 \rrbr, \,
        \dotsc, \,
        \llbr E_n / E_{n-1} \rrbr
    )
\end{equation}
for $E_\bullet \in \calA^\pn{n}$ as in \cref{eq-obj-filt},
where quotients denote cokernels of inclusions.
With this extra data, the $\bbK$-linear exact category
$\calA^\pn{n}$ now also satisfies~\tagref{ExCat}.

\paragraph{Projection functors.}
\label{para-filt-proj}

For an integer $n \geq 0$,
let $[n]$ denote the totally ordered set $\{ 0, 1, \dotsc, n \}$.
For an order-preserving map $f \colon [m] \to [n]$,
there is an exact functor
\begin{equation}
    p^f \colon \calA^\pn{n}
    \longrightarrow \calA^\pn{m} \ ,
\end{equation}
sending a sequence~\cref{eq-obj-filt} to the sequence
\begin{equation}
    p^f (E_\bullet) = \Bigl( \ 
        E_{f (1)} / E_{f (0)} \hookrightarrow \cdots \hookrightarrow E_{f (m)} / E_{f (0)}
    \ \Bigr) \ ,
\end{equation}
where quotients denote cokernels of inclusions.

For an integer $1 \leq i \leq n$, we introduce shorthands
\begin{alignat}{2}
    p^\pn{n} & = p^{( 0, \> n ) \colon [1] \hookrightarrow [n]}
    && \colon \calA^\pn{n} \longrightarrow \calA \ , \\
    p_i & = p^{( i - 1, \> i ) \colon [1] \hookrightarrow [n]}
    && \colon \calA^\pn{n} \longrightarrow \calA \ .
\end{alignat}

\subsection{Stability conditions}
\label{sect-exact-cat-stab}

\paragraph{}
\label[definition]{def-stability}

We define a notion of \emph{stability conditions}
on exact categories, based on ideas from
Rudakov~\cite{Rudakov1997},
Bridgeland~\cite{Bridgeland2007},
Joyce~\cite{Joyce2007III},
and many others.

\begin{definition*}
    Let $\calA$ be a $\bbK$-linear exact category, as in \tagref{ExCat}.
    Let $(T, \leq)$ be a totally ordered set, and let
    \[
        \tau \colon C (\calA) \longrightarrow T
    \]
    be a map of sets.
    Write $\tau (E)$ for $\tau ( \llbr E \rrbr )$ 
    for an object $E \in \calA$.
    
    An object $E \in \calA$ is \emph{$\tau$-semistable},
    if for any inclusion $i \colon F \hookrightarrow E$,
    with $0 \neq F \neq E$, we have $\tau (F) \leq \tau (E) \leq \tau (E / F)$,
    where $E / F$ denotes the cokernel of $i$.

    An object $E \in \calA$ is \emph{$\tau$-stable},
    if it is non-zero, and for any inclusion $i \colon F \hookrightarrow E$,
    with $0 \neq F \neq E$, we have $\tau (F) < \tau (E) < \tau (E / F)$.

    We say that $\tau$ is a \emph{weak stability condition} on $\calA$,
    if the following conditions are satisfied:

    \begin{enumerate}
        \item \label{itm-def-stability-1}
            For any $\alpha, \beta, \gamma \in C (\calA)$
            such that $\beta = \alpha + \gamma$, either
            \[
                \tau (\alpha) \leq \tau (\beta) \leq \tau (\gamma) \ ,
                \quad \text{or} \quad
                \tau (\alpha) \geq \tau (\beta) \geq \tau (\gamma) \ .
            \]
        \item \label{itm-def-stability-2}
            If $E, F \in \calA$ are non-zero and $\tau$-semistable,
            and $\tau (E) > \tau (F)$, then
            \[
                \mathrm{Hom}_{\calA} (E, F) = 0 \ .
            \]
        \item \label{itm-def-stability-3}
            Every object $E \in \calA$ has a
            \emph{Harder--Narasimhan filtration}, that is a filtration
            \[ \begin{tikzcd}[sep=small]
                \mathllap{0 = {}} E_0 \ar[r, hook] &
                E_1 \ar[r, hook] \ar[d, two heads] &
                E_2 \ar[r, hook] \ar[d, two heads] &
                \cdots \ar[r, hook] &
                E_k \mathrlap{{} = E} \ar[d, two heads] \\
                & F_1 & F_2 & & F_k \rlap{ ,}
            \end{tikzcd} \]
            where each $E_{i-1} \hookrightarrow E_i \twoheadrightarrow F_i$ is short exact,
            and each $F_i$ is non-zero and $\tau$-semistable, with
            \[
                \tau (F_1) > \tau (F_2) > \cdots > \tau (F_k) \ .
            \]
    \end{enumerate}
    We say that $\tau$ is a \emph{stability condition},
    if it is a weak stability condition,
    and moreover, the following conditions are satisfied:

    \begin{enumerate}[resume]
        \item[(i$'$)]
            For any $\alpha, \beta, \gamma \in C (\calA)$
            such that $\beta = \alpha + \gamma$, either
            \[
                \tau (\alpha) < \tau (\beta) < \tau (\gamma) \ ,
                \quad \text{or} \quad
                \tau (\alpha) = \tau (\beta) = \tau (\gamma) \ ,
                \quad \text{or} \quad
                \tau (\alpha) > \tau (\beta) > \tau (\gamma) \ .
            \]

        \item \label{itm-def-stability-4}
            For each $t \in T$, let $\calA (t) \subset \calA$
            be the full subcategory consisting of
            $\tau$-semistable objects $E \in \calA$
            with $E \simeq 0$ or $\tau (E) = t$.
            Then $\calA (t)$ is an abelian category,
            and the inclusion $\calA (t) \hookrightarrow \calA$ is exact.
            Moreover, $\calA (t)$ is \emph{artinian},
            in that any decreasing sequence of subobjects
            $E = E_0 \supset E_1 \supset \cdots $
            in $\calA (t)$ must stabilize.
    \end{enumerate}
\end{definition*}

\paragraph{Remark.}

\Cref{def-stability} can be simplified in the following two cases:

\begin{itemize}
    \item 
        If $\calA$ is an abelian category,
        then the condition~\cref{itm-def-stability-2} is automatic,
        since one can take images.
        If, moreover, $\calA$ is \emph{$\tau$-artinian}
        in the sense of Joyce~\cite[\S4]{Joyce2007III},
        then the conditions~\crefrange{itm-def-stability-3}{itm-def-stability-4}
        are also satisfied, as discussed there.

    \item 
        If $\calA$ is a \emph{self-dual exact category}
        in the sense of \cref{sect-def-sd-cat},
        then the artinian assumption in condition~\cref{itm-def-stability-4}
        is automatic, as discussed in \cref{rem-sd-noeth-art}.
\end{itemize}

\begin{theorem}
    \label{thm-hn-uniqueness}
    In \cref{def-stability},
    if $\tau$ is a weak stability condition,
    then the Harder--Narasimhan filtration of an object $E \in \calA$
    is unique up to a unique isomorphism.
\end{theorem}

\begin{proof}
    This follows from a standard argument.
    See, for example, Rudakov~\cite[Theorem~2]{Rudakov1997}
    and Joyce~\cite[Theorem~4.4]{Joyce2007III}.
\end{proof}

\begin{theorem}
    \label{thm-jh-filtration}
    In \cref{def-stability}, if $\tau$ is a stability condition,
    then every $\tau$-semistable object $E \in \calA$
    has a filtration
    \[ \begin{tikzcd}[sep=small]
        \mathllap{0 = {}} E_0 \ar[r, hook] &
        E_1 \ar[r, hook] \ar[d, two heads] &
        E_2 \ar[r, hook] \ar[d, two heads] &
        \cdots \ar[r, hook] &
        E_\ell \mathrlap{{} = E} \ar[d, two heads] \\
        & F_1 & F_2 & & F_\ell \rlap{ $,$}
    \end{tikzcd} \]
    called a \emph{$\tau$-Jordan--H\"older filtration},
    where each $E_{i-1} \hookrightarrow E_i \twoheadrightarrow F_i$ is short exact,
    and each $F_i$ is $\tau$-stable, with
    $\tau (F_1) = \cdots = \tau (F_\ell)$.
    Moreover, the objects $F_i$ are uniquely determined by $E$
    up to a permutation.
\end{theorem}

\begin{proof}
    This is a standard result in abelian categories,
    and follows from the ascending and descending chain conditions.
    See, for example, Rudakov~\cite[Theorem~3]{Rudakov1997}
    and Joyce~\cite[Theorem~4.5]{Joyce2007III}.
\end{proof}

\subsection{Moduli of objects}
\label{sect-exact-cat-moduli}

\paragraph{}
\label{para-exmod}

We define a notion of
\emph{moduli stacks of objects} in an exact category.
We take an axiomatic approach,
meaning that we do not attempt to construct such moduli stacks
from the category, but rather assume that they exist,
and we formulate the properties that we need from them.

We will use various notions of stacks in categories,
as in \cref{sect-stacks-in-cat}.

\begin{condition*}[Mod]
    Let $\calA$ be a $\bbK$-linear exact category, as in \tagref{ExCat}.

    We assume that there is a $\bbK$-stack in categories $\+{\calM}$,
    called a \emph{categorical moduli stack of objects} in $\calA$,
    satisfying the following properties.
    
    \begin{enumerate}
        \item \label{itm-asn-moduli-loc-ft}
            $\+{\calM}$ is an exact $\bbK$-linear algebraic stack,
            in the sense of \cref{def-linear-stack}.
    \end{enumerate}
    
    Write $\calM = \+{\calM}^\simeq$, and call it
    the \emph{moduli stack of objects} in $\calA$,
    which is an algebraic $\bbK$-stack locally of finite type.
    Write $\calM^I = (\+{\calM}^I)^\simeq$ and
    $\calM^\ex = (\+{\calM}^\ex)^\simeq$,
    which are also algebraic $\bbK$-stacks locally of finite type.
    We further assume the following.
    
    \begin{enumerate}[resume]
        \item \label{itm-asn-moduli-points}
            There is an equivalence of $\bbK$-linear exact categories
            \begin{equation}
                \calA \simeq \+{\calM} (\bbK) \ .
            \end{equation}
            In particular, every object $E \in \calA$
            defines a $\bbK$-point $E \in \calM (\bbK)$.

        \item \label{itm-asn-moduli-k-group}
            There is a decomposition
            \begin{equation}
                \calM = \coprod_{\alpha \in C^\circ (\calA)} \calM_\alpha \ ,
            \end{equation}
            where each $\calM_\alpha \subset \calM$
            is an open and closed substack,
            such that $\calM_\alpha (\bbK) \subset \calM (\bbK)$
            consists precisely of those $\bbK$-points $E$
            with $\llbr E \rrbr = \alpha$.
    \end{enumerate}
\end{condition*}

\paragraph{Geometric functors.}

Suppose we have a $\bbK$-linear exact functor
$F \colon \calA \to \calB$ between $\bbK$-linear exact categories,
and suppose that $\+{\calM}$ and $\+{\calN}$ are categorical moduli stacks
of objects in $\calA$ and $\calB$, respectively, as in \tagref{Mod}.
We say that $F$ is \emph{geometric},
if it extends to a morphism of exact linear algebraic stacks
\[
    F \colon \+{\calM} \longrightarrow \+{\calN}
\]
that is compatible with the identifications
$\calA \simeq \+{\calM} (\bbK)$ and
$\calB \simeq \+{\calN} (\bbK)$.

Similarly, we can define whether a natural isomorphism
between two geometric functors is geometric.

For example, the symmetric monoidal structure on $\calA$ given by $\oplus$
is geometric, since there is a canonical map
\begin{equation}
    \label{eq-def-oplus-moduli}
    \oplus \colon \+{\calM} \times \+{\calM} \longrightarrow \+{\calM} \ ,
\end{equation}
which, together with the unit $0 \colon {\Spec \bbK} \to \+{\calM}$
and various compatible $2$-morphisms,
exhibits $\+{\calM}$ as a commutative monoid
in the $2$-category of $\bbK$-stacks in categories.
In particular, this also exhibits $\calM$
as a commutative monoid in the $2$-category of algebraic $\bbK$-stacks.

\paragraph{Semistable loci.}
\label{para-exstab}

We now formulate assumptions on stability conditions on $\calA$,
such that the semistable locus in the moduli stack $\calM$
is well-behaved.

\begin{condition*}[Stab]
    \taglabelvar{Stab1}{Stab~\upshape i}
    \taglabelvar{Stab2}{Stab~\upshape ii}
    \taglabelvar{Stab3}{Stab~\upshape iii}
    \taglabelvar{Stab12}{Stab~\upshape i--ii}
    \taglabelvar{Stab13}{Stab~\upshape i, iii}
    Let $\calA$ be a $\bbK$-linear exact category,
    with a categorical moduli stack $\+{\calM}$, as in \tagref{Mod}.
    Let $\tau$ be a weak stability condition on $\calA$.
    We assume that the following conditions are satisfied:

    \begin{enumerate}
        \item
            For each $\alpha \in C^\circ (\calA)$,
            there is an open substack
            \[
                \Mss_\alpha (\tau) \subset \calM_\alpha \ ,
            \]
            such that for any $E \in \calA$ with $\llbr E \rrbr = \alpha$,
            we have $E \in \Mss_\alpha (\tau) (\bbK)$
            if and only if $E$ is $\tau$-semistable.

        \item
            For any $\alpha \in C^\circ (\calA)$,
            the substack $\Mss_\alpha (\tau) \subset \calM_\alpha$
            is of finite type.

        \item
            For any $\bbK$-scheme $U$ of finite type
            and any morphism $e \colon U \to \calM$,
            there exists a finite subset $S_e \subset C (\calA)$,
            such that for any $x \in U (\bbK)$,
            the $\tau$-Harder--Narasimhan filtration of the object
            $E_x \in \calA$ corresponding to $e (x)$
            has stepwise quotients whose classes all belong to $S_e$.
    \end{enumerate}
\end{condition*}
The condition~\tagref{Stab3} is very weak:
it is true in all reasonable situations,
and is only used to avoid pathological cases.

\subsection{Moduli of filtrations}

\paragraph{Definition.}
\label[definition]{def-mod-filt}

Let $\calA$ be a $\bbK$-linear exact category,
with a categorical moduli stack $\+{\calM}$, as in \tagref{Mod}.

For an integer $n \geq 0$,
we define a $\bbK$-stack in categories $\+{\calM}^\pn{n}$ by
\begin{equation}
    \label{eq-def-mnplus}
    \+{\calM}^\pn{n} (U) = \+{\calM} (U)^\pn{n}
\end{equation}
for any $\bbK$-scheme~$U$,
where the right-hand side denotes the category of $n$-step filtrations
in the exact category $\+{\calM} (U)$,
as in \cref{def-cat-filt}. 

The stack $\+{\calM}^\pn{n}$
can be seen as a categorical moduli stack of objects in $\calA^\pn{n}$.
It is an exact $\bbK$-linear algebraic stack,
and satisfies~\tagref{Mod} for $\calA^\pn{n}$,
as we will see below.

Define
\[
    \calM^\pn{n} = (\+{\calM}^\pn{n})^\simeq
\]
to be the underlying algebraic $\bbK$-stack of $\+{\calM}^\pn{n}$.

\paragraph{Projection morphisms.}
\label{para-mod-filt-proj}

Recall from~\cref{para-filt-proj} the projections
$p^f \colon \calA^\pn{n} \to \calA^\pn{m}$
for an order-preserving map $f \colon [m] \to [n]$,
and $p_i, p^\pn{n} \colon \calA^\pn{n} \to \calA$ for $1 \leq i \leq n$.
These extend naturally to
morphisms of $\bbK$-stacks in categories
\begin{align}
    \label{eq-pi-f-plus}
    \pi^f \colon & \+{\calM}^\pn{n} \longrightarrow \+{\calM}^\pn{m} \ , \\
    \label{eq-pi-i-plus}
    \pi_i, \pi^\pn{n} \colon & \+{\calM}^\pn{n} \longrightarrow \+{\calM} \ ,
\end{align}
which, in turn, induce morphisms of $\bbK$-stacks
\begin{align}
    \label{eq-pi-f}
    \pi^f \colon & \calM^\pn{n} \longrightarrow \calM^\pn{m} \ , \\
    \label{eq-pi-i}
    \pi_i, \pi^\pn{n} \colon & \calM^\pn{n} \longrightarrow \calM \ .
\end{align}

\paragraph{\tagref{Mod} for filtrations.}

Let us verify that $\+{\calM}^\pn{n}$ satisfies~\tagref{Mod}
for $\calA^\pn{n}$.

We have natural identifications
$\+{\calM}^\pn{0} \simeq \operatorname{Spec} \bbK$, \ 
$\+{\calM}^\pn{1} \simeq \+{\calM}$, \ 
$\+{\calM}^\pn{2} \simeq \+{\calM}^\ex$, 
with the notation $(-)^\ex$ as in \cref{para-exact-stack},
and
\begin{equation}
    \+{\calM}^\pn{n} \simeq
    \+{\calM}^\pn{2} \underset{\pi^\pn{2}, \, \+{\calM}, \, \pi_1 \mspace{18mu}}{\times}
    \cdots \underset{\pi^\pn{2}, \, \+{\calM}, \, \pi_1 \mspace{18mu}}{\times}
    \+{\calM}^\pn{2}
\end{equation}
for $n \geq 2$,
where $\+{\calM}^\pn{2}$ appears $(n-1)$ times.
In particular, this shows that each
$\+{\calM}^\pn{n}$ is an exact $\bbK$-linear algebraic stack.
Moreover, the projection morphisms $\pi^f$, $\pi_i$, and $\pi^\pn{n}$
in \crefrange{eq-pi-f-plus}{eq-pi-i-plus}
are morphisms of exact $\bbK$-linear algebraic stacks.

For $(\alpha_1, \dotsc, \alpha_n) \in C^\circ (\calA)^n \simeq C^\circ (\calA^\pn{n})$,
define an open and closed substack
\[
    \calM^\pn{n}_{\alpha_1 , \, \dotsc, \, \alpha_n} \subset \calM^\pn{n}
\]
to be the preimage of
$\calM_{\alpha_1} \times \cdots \times \calM_{\alpha_n}$
under the morphism
$\pi_1 \times \cdots \times \pi_n \colon \calM^\pn{n} \to \calM^n$.

This verifies that $\+{\calM}^\pn{n}$ satisfies~\tagref{Mod},
so that it qualifies as a categorical moduli stack for $\calA^\pn{n}$.

\paragraph{Semistable loci.}
\label{para-filt-ss}

In the situation of~\tagref{Mod},
suppose that $\tau$ is a weak stability condition on $\calA$,
satisfying~\tagref{Stab1}.
Then, for $\alpha_1, \dotsc, \alpha_n \in C^\circ (\calA)$, we define an open substack
\[
    \calM^{\pn{n}, \> \ss}_{\alpha_1, \, \dotsc, \, \alpha_n} (\tau)
    \subset \calM^\pn{n}_{\alpha_1, \, \dotsc, \, \alpha_n}
\]
to be the preimage of
$\Mss_{\alpha_1} (\tau) \times \cdots \times \Mss_{\alpha_n} (\tau)$
under the morphism
$\pi_1 \times \cdots \times \pi_n \colon \calM^\pn{n} \to \calM^n$.

\paragraph{Finiteness conditions.}
\label{para-exfin}

We also formulate certain finiteness conditions on $\calM^\pn{2}$,
which will be needed later.
Compare also Joyce~\cite[Assumption~8.1]{Joyce2006I}.

\begin{condition*}[Fin]
    \taglabelvar{Fin1}{Fin~\upshape i}
    \taglabelvar{Fin2}{Fin~\upshape ii}
    Suppose that we are in the situation of~\tagref{Mod}.
    We further assume the following:

    \begin{enumerate}
        \item 
            \label{itm-pi1-pi2-ft}
            For any $\alpha_1, \alpha_2 \in C^\circ (\calA)$, the morphism
            $\pi_1 \times \pi_2 \colon \calM^\pn{2}_{\alpha_1, \, \alpha_2} \to \calM_{\alpha_1} \times \calM_{\alpha_2}$
            is of finite type.
        \item 
            \label{itm-pi12-ft}
            For any $\alpha_1, \alpha_2 \in C^\circ (\calA)$, the morphism
            $\pi^\pn{2} \colon \calM^\pn{2}_{\alpha_1, \, \alpha_2} \to \calM_{\alpha_1 + \alpha_2}$
            is of finite type.
    \end{enumerate}
\end{condition*}

\begin{remark}
    The condition~\tagref{Fin1}
    is equivalent to the morphism
    $\pi_1 \times \pi_2 \colon \calM^\pn{2} \to \calM \times \calM$
    being of finite type.
    However, the condition~\tagref{Fin2}
    does not imply that the morphism
    $\pi^\pn{2} \colon \calM^\pn{2} \to \calM$
    is of finite type, since for a given $\alpha \in C^\circ (\calA)$,
    there can be infinitely many pairs $(\alpha_1, \alpha_2)$
    such that $\alpha_1 + \alpha_2 = \alpha$.
\end{remark}

\begin{lemma}
    \label{lem-filt-rep-ft}
    Suppose we are in the situation of \tagref{Mod}.
    
    \begin{enumerate}
        \item \label{itm-filt-rep}
            For any $n \geq 0,$ the morphism
            \begin{equation*}
                \pi^\pn{n} \colon \calM^\pn{n} \longrightarrow \calM
            \end{equation*}
            is representable.

        \item \label{itm-filt-ft}
            Assume that \tagref{Fin1} holds.
            Then, for any $n \geq 0,$ the morphism
            \begin{equation*}
                \pi_1 \times \cdots \times \pi_n \colon \calM^\pn{n} \longrightarrow \calM^n
            \end{equation*}
            is of finite type.

        \item \label{itm-filt-ft-2}
            Assume that \tagref{Fin2} holds.
            Then, for any $n \geq 0,$ the morphism
            \begin{equation*}
                \pi^\pn{n} \colon \calM^\pn{n} \longrightarrow \calM
            \end{equation*}
            is of finite type.
    \end{enumerate}
\end{lemma}

\begin{proof}
    For~\cref{itm-filt-rep}, by a criterion in
    Laumon--Moret-Bailly~\cite[Corollary~8.1.1]{LaumonMoret2000},
    it is enough to verify that for any $\bbK$-scheme $U$,
    the functor
    \[
        p^\pn{n} \colon \+{\calM} (U)^\pn{n} \longrightarrow \+{\calM} (U)
    \]
    induces injections of automorphism groups.
    This is true since inclusions in an exact category
    are monomorphisms.
    
    For~\cref{itm-filt-ft}, we use induction on $n$.
    If $n = 0, 1$, the result is trivial.
    If $n > 1$, consider the diagram
    \begin{equation}
    \begin{tikzcd}[column sep=3em]
        \calM^\pn{n} \ar[r, "\varphi_n \times \pi_n"'] \ar[d]
        \ar[dr, phantom, pos=.1, "\ulcorner", shift right=2]
        \ar[rrr, bend left=10, shift left, "\pi_1 \times \cdots \times \pi_n"]
        & \calM^\pn{n-1} \times \calM
        \ar[d, "\pi^\pn{n-1} \times \mathrm{id}"]
        \ar[rr, "(\pi_1 \times \cdots \times \pi_{n-1}) \times \mathrm{id}"']
        && \calM^n 
        \\ \calM^\pn{2} \ar[r, "\pi_1 \times \pi_2"]
        & \calM \times \calM \rlap{ ,}
    \end{tikzcd}
    \end{equation}
    where $\varphi_n$ is the morphism $\pi^f$ with
    $f = (0, \dotsc, n-1) \colon [n-1] \to [n]$.
    By~\tagref{Fin1},
    the bottom arrow is of finite type,
    and so is the map $\varphi_n \times \pi_n$.
    But $(\pi_1 \times \cdots \times \pi_{n-1}) \times \mathrm{id}$
    is of finite type by induction hypothesis.
    Since $\pi_1 \times \cdots \times \pi_n$ is the composition of these two arrows,
    it is of finite type.

    For~\cref{itm-filt-ft-2}, a similar argument works,
    using the diagram
    \begin{equation}
        \begin{tikzcd}[column sep=3em]
            \calM^\pn{n} \ar[r, "\psi_n"'] \ar[d]
            \ar[dr, phantom, pos=.1, "\ulcorner", shift right=2]
            \ar[rr, bend left=10, shift left, "\pi^\pn{n}"]
            & \calM^\pn{n-1}
            \ar[d, "\pi_{n-1}"]
            \ar[r, "\pi^\pn{n-1}\vphantom{^0}"']
            & \calM
            \\ \calM^\pn{2} \ar[r, "\pi^\pn{2}"]
            & \calM \rlap{ ,}
        \end{tikzcd}
    \end{equation}
    where $\psi_n$ is the morphism $\pi^f$ with
    $f = (0, \dotsc, n-2, n) \colon [n-1] \to [n]$.
\end{proof}

\paragraph{}

In particular, if~\tagref{Fin1} holds,
and $\tau$ is a weak stability condition on $\calA$
satisfying~\tagref{Stab12},
then the semistable loci
$\calM^{\pn{n}, \> \ss}_{\alpha_1, \, \dotsc, \, \alpha_n} (\tau)$
defined in \cref{para-filt-ss}
are of finite type.

\subsection{The extension bundle}

\paragraph{}

We formulate another set of conditions on the exact category $\calA$
and its moduli stack $\calM$,
which will be needed to prove wall-crossing formulae
for motivic enumerative invariants
in \crefrange{sect-motivic-inv}{sect-num-inv} below.
Namely, we assume the existence of \emph{extension bundles}.
These conditions are quite restrictive,
and require $\calA$ to have homological dimension at most $1$,
although we can weaken this assumption in some cases,
as discussed in \cref{rem-ext-weaken} below.

\paragraph{The Baer sum.}

In the situation of~\tagref{Mod},
for $\alpha_1, \alpha_2 \in C^\circ (\calA)$, consider the morphism
\begin{equation}
    \label{eq-pi12}
    \pi_1 \times \pi_2 \colon \calM^\pn{2}_{\alpha_1, \, \alpha_2} \longrightarrow
    \calM_{\alpha_1} \times \calM_{\alpha_2} \ .
\end{equation}
The left-hand side is thought of as the stack of
short exact sequences $E_1 \hookrightarrow E \twoheadrightarrow E_2$
with $E_1, E_2 \in \calA$ of classes $\alpha_1, \alpha_2 \in C^\circ (\calA)$.
Taking the Baer sum of two such sequences,
that is, taking the sum in $\Ext^1 (E_2, E_1)$,
as in Weibel~\cite[Definition~3.4.4]{Weibel1994},
gives a morphism
\begin{equation}
    \label{eq-pi12-plus}
    + \colon \calM^\pn{2}_{\alpha_1, \, \alpha_2}
    \underset{\calM_{\alpha_1} \times \calM_{\alpha_2}}{\times}
    \calM^\pn{2}_{\alpha_1, \, \alpha_2} \longrightarrow
    \calM^\pn{2}_{\alpha_1, \, \alpha_2} \ ,
\end{equation}
which is well-defined
since the limits and colimits involved in taking the Baer sum
are preserved by exact functors, as in \cref{para-exact-stack-pf-pb}.
This gives \cref{eq-pi12} the structure of an \emph{abelian stack},
in the sense of \cref{para-ab-stack}.

\paragraph{}
\label{para-exext}

We now formulate the notion of an \emph{extension bundle},
using the language of $2$-vector bundles
and affine $2$-vector bundles,
in the sense of \cref{sect-two-vb}.

\begin{condition*}[Ext]
    \taglabelvar{Ext1}{Ext~\upshape i}
    \taglabelvar{Ext2}{Ext~\upshape ii}
    Suppose that we are in the situation of~\tagref{Mod}.
    We assume that for each $\alpha_1, \alpha_2 \in C^\circ (\calA)$,
    we have a $2$-vector bundle
    \begin{equation}
        \calExt^{1/0}_{\smash{\alpha_1, \alpha_2}} \longrightarrow
        \calM_{\alpha_1} \times \calM_{\alpha_2}
    \end{equation}
    of constant rank, called the \emph{extension bundle}.
    We assume that the following conditions are satisfied:

    \begin{enumerate}
        \item \label{itm-exext-1}
            $\calA$ has homological dimension at most $1$,
            in that for any $E_1, E_2 \in \calA$,
            we have $\Ext^i (E_2, E_1) = 0$ for all $i > 1$.
        
        \item \label{itm-exext-2}
            We have an isomorphism
            $\calExt^{1/0}_{\smash{\alpha_1, \alpha_2}} \simeq \calM^\pn{2}_{\alpha_1, \alpha_2}$
            of abelian stacks over $\calM_{\alpha_1} \times \calM_{\alpha_2}$.
            The fibre of $\calExt^{\smash{1/0}}_{\smash{\alpha_1, \alpha_2}}$
            over $(E_1, E_2) \in (\smash{\calM_{\alpha_1}} \times \smash{\calM_{\alpha_2}}) (\bbK)$
            is isomorphic to the $2$-vector space
            \[
                \Ext^{1/0} (E_2, E_1) = 
                \Ext^1 (E_2, E_1) \times [*/\Ext^0 (E_2, E_1)] \ ,
            \]
            and each short exact sequence
            $E_1 \hookrightarrow E \twoheadrightarrow E_2$ as a $\bbK$-point of 
            $\smash{\calM^\pn{2}_{\alpha_1, \alpha_2}}$
            is identified with its class in $\Ext^{\smash{1/0}} (E_2, E_1)$.
    \end{enumerate}
    Note that by \cref{rem-2vb-ab-stack},
    the $2$-vector bundle structure on $\calM^\pn{2}_{\alpha_1, \alpha_2}$
    is not extra data, but rather a property.
\end{condition*}

\paragraph{Functoriality.}
\label{para-exext-func}

Assume that we are in the situation of~\tagref{Ext}.
Taking pushforwards and pullbacks
as in \cref{para-exact-stack-pf-pb},
we obtain the universal pushforward and pullback morphisms
\begin{align*}
    (s \times \id)^* \, \calExt^{1/0}_{\smash{\alpha_1, \beta}} & \longrightarrow
    (t \times \id)^* \, \calExt^{1/0}_{\smash{\alpha_2, \beta}} \ , \\
    (\id \times t)^* \, \calExt^{1/0}_{\smash{\alpha, \beta_2}} & \longrightarrow
    (\id \times s)^* \, \calExt^{1/0}_{\smash{\alpha, \beta_1}}
\end{align*}
of $2$-vector bundles on
$\calM^I_{\alpha_1, \alpha_2} \times \calM_{\smash{\beta}}$ and
$\calM_{\alpha} \times \calM^I_{\smash{\beta_1, \beta_2}}$, respectively,
where $\smash{\calM^I_{\alpha_1, \alpha_2}}$ is the stack of morphisms
$E_1 \to E_2$ with $E_1, E_2 \in \calA$ of classes
$\alpha_1, \alpha_2 \in C^\circ (\calA)$,
and $s, t$ are the source and target morphisms.
In plain words, this means that $\Ext^{1/0} (-, -)$
is contravariant in the first component,
and covariant in the second component.

These morphisms, in turn,
form short exact sequences of $2$-vector bundles
\begin{gather}
    \label{eq-ext-les-1}
    0 \to (\pi_1 \times \id)^* \, \calExt^{1/0}_{\smash{\alpha_1, \beta}} \longrightarrow
    (\pi^\pn{2} \times \id)^* \, \calExt^{1/0}_{\smash{\alpha_1 + \alpha_2, \beta}} \longrightarrow
    (\pi_2 \times \id)^* \, \calExt^{1/0}_{\smash{\alpha_2, \beta}} \to 0 \rlap{ ,} \\
    \label{eq-ext-les-2}
    0 \to (\id \times \pi_2)^* \, \calExt^{1/0}_{\smash{\alpha, \beta_2}} \longrightarrow
    (\id \times \pi^\pn{2})^* \, \calExt^{1/0}_{\smash{\alpha, \beta_1 + \beta_2}} \longrightarrow
    (\id \times \pi_1)^* \, \calExt^{1/0}_{\smash{\alpha, \beta_1}} \to 0
\end{gather}
over 
$\calM^\pn{2}_{\alpha_1, \alpha_2} \times \calM_\beta$ and
$\calM_\alpha \times \calM^\pn{2}_{\smash{\beta_1, \beta_2}}$, respectively.
In plain words, this means that there are six-term
long exact sequences of $\Ext$ groups in both components.

\paragraph{The Euler form.}
\label{para-euler}

In the situation of~\tagref{Ext},
for $\alpha_1, \alpha_2 \in C^\circ (\calA)$, define an integer
\begin{equation}
    \chi (\alpha_1, \alpha_2) =
    -{\rank \calExt^{1/0}_{\smash{\alpha_1, \alpha_2}}} \ .
\end{equation}
By \crefrange{eq-ext-les-1}{eq-ext-les-2},
this gives an $\bbN$-bilinear form
\begin{equation}
    \label{eq-def-chi}
    \chi \colon C^\circ (\calA) \times C^\circ (\calA) \longrightarrow \bbZ \ ,
\end{equation}
called the \emph{Euler form}.

\begin{remark}
    \label{rem-ext-weaken}
    The condition \tagref{Ext} can be weakened
    by only requiring $\calExt^{1/0}_{\smash{\alpha_1, \alpha_2}}$
    to be a $2$-vector bundle over the semistable locus
    $\Mss_{\alpha_1} (\tau) \times \Mss_{\alpha_2} (\tau)$,
    and only when $\tau (\alpha_1) \leq \tau (\alpha_2)$,
    for a weak stability condition $\tau$ on $\calA$,
    as we will do in \cref{sect-surfaces} below.
    This will allow us to slightly relax the
    homological dimension~$1$ condition,
    and most results that require \tagref{Ext} will still hold in this case.
\end{remark}

\begin{lemma}
    \label{lem-3step-affine-linear}
    In the situation of~\tagref{Ext},
    consider the canonical morphism
    \begin{equation}
        q \colon \calM^\pn{3}_{\alpha_1, \alpha_2, \alpha_3} \longrightarrow
        \calM^\pn{2}_{\alpha_1, \alpha_2}
        \underset{\calM_{\alpha_2}}{\times}
        \calM^\pn{2}_{\alpha_2, \alpha_3} \ .
    \end{equation}
    Then $q$ is an affine $2$-vector bundle,
    as in~\cref{def-affine-2vb},
    which is a torsor for
    the $2$-vector bundle
    $(\pi_1 \times \pi_2)^* \, \calExt^{\smash{1/0}}_{\alpha_1, \> \alpha_3}
    \to \calM^\pn{2}_{\alpha_1, \alpha_2} \times_{\calM_{\alpha_2}} \calM^\pn{2}_{\alpha_2, \alpha_3}$.
\end{lemma}

\begin{proof}
    We need to show that
    for objects $E_1, E_2, E_3 \in \calA$ of classes $\alpha_1, \alpha_2, \alpha_3$,
    the stack $\calF_{\smash{a, b}}$
    of $3$-step filtrations with stepwise quotients $E_1, E_2, E_3$,
    such that the induced elements
    $a \in \Ext^{1/0} (E_2, E_1) (\bbK)$ and 
    $b \in \Ext^{1/0} (E_3, E_2) (\bbK)$ are fixed,
    is a torsor for $\Ext^{1/0} (E_3, E_1)$.

    Indeed, consider the short exact sequence
    \begin{equation}
        0 \to \Ext^{1/0} (E_3, E_1) \longrightarrow
        \Ext^{1/0} (E_3, E_a) \longrightarrow
        \Ext^{1/0} (E_3, E_2) \to 0
    \end{equation}
    of $2$-vector spaces,
    where $E_a$ is the extension of $E_2$ by $E_1$ given by $a$.
    Then $\calF_{\smash{a, b}}$ is the fibre
    of the second non-trivial morphism at $b$,
    so the action of $\Ext^{1/0} (E_3, E_1)$ on $\calF_{\smash{a, b}}$
    by translation is a torsor.

    Repeating this argument in families
    proves the lemma.
\end{proof}

%% file: sdcat.tex
\label{sect-sd-cat}

We introduce the notion of \emph{self-dual categories},
which are categories that are equivalent to their own opposite categories.
We will be particularly interested in
self-dual exact or quasi-abelian categories,
as these will be the main setting for studying
enumerative geometry for type B/C/D structure groups.

Under this setting,
we will be interested in enumerative invariants
counting \emph{self-dual objects} in self-dual categories.
These are objects that are isomorphic to their own duals,
and can be seen as generalized orthogonal and symplectic bundles,
since they are the self-dual objects in the category of vector bundles,
with the usual notion of the dual bundle.

\subsection{Self-dual categories}
\label{sect-def-sd-cat}

\begin{definition}
    \label{def-sd-cat}
    A \emph{self-dual category} is a triple $(\calA, D, \eta)$, where

    \begin{itemize}
        \item
            $\calA$ is a category.

        \item
            $D \colon \calA \simto \calA^\op$
            is an equivalence,
            called the \emph{dual functor}.
            We often denote this functor by $(-)^\vee$.

        \item
            $\eta \colon D^\op \circ D \simTo \id_{\calA}$
            is a natural isomorphism,
            such that for any $E \in \calA$, we have
            \begin{equation}
                \eta_{E^\vee} = (\eta_E^\vee)^{-1} \colon
                E^{\vee\vee\vee} \longsimto E^\vee \ .
            \end{equation}
    \end{itemize}
    By abuse of language, we sometimes say that
    $\calA$ is a self-dual category
    when the extra data $D, \eta$ are clear from the context.
    
    If, in addition, $\calA$ is equipped with
    certain types of additional structure,
    such as the structure of a $\bbK$-linear category,
    an exact category, 
    or a quasi-abelian category, etc.,
    and if such structure is respected by $D$ and $\eta$,
    meaning that $D$ and $\eta$ are $1$- and $2$-morphisms
    in the $2$-category of such categories with additional structure,
    then we say that $(\calA, D, \eta)$
    is, for example, a \emph{self-dual $\bbK$-linear category}, etc.
\end{definition}

\begin{definition}
    \label{def-sd-obj}
    Let $(\calA, D, \eta)$ be a self-dual category.
    A \emph{self-dual object} in $\calA$
    is a pair $(E, \phi)$, where
    \begin{itemize}
        \item 
            $E \in \calA$ is an object.
        \item
            $\phi \colon E \simto E^\vee$ is an isomorphism in $\calA$,
            such that $\phi^\vee = \phi \circ \eta_E$\,:
            \begin{equation} \begin{tikzcd}[row sep={2.2em,between origins}, column sep=3em]
                E \ar[dr, "\phi" {pos=.47},
                    "\textstyle \sim" {pos=.52, anchor=center, rotate=-20, shift={(0, -1ex)}}] \\
                & E^\vee \rlap{ .} \\
                E^{\vee\vee} \ar[uu, "\eta_E", 
                    "\textstyle \sim" {anchor=center, rotate=90, shift={(0, -1ex)}}]
                \ar[ur, "\phi^\vee"' {pos=.42, inner sep=.05em},
                    "\textstyle \sim" {pos=.47, anchor=center, rotate=20, shift={(0, .6ex)}}]
            \end{tikzcd} \end{equation}
    \end{itemize}
    Let $\calA^\sd$ denote the groupoid
    whose objects are self-dual objects in $\calA$,
    and morphisms are isomorphisms of self-dual objects.
\end{definition}

\paragraph{Module structure.}
\label{para-sd-cat-module}

Let $\calA$ be a self-dual additive category.
Then we have a natural action of $\calA$ on $\calA^\sd$,
given by the functor
\begin{equation}
    \label{eq-oplus-sd}
    \oplus^\sd \colon \calA^\simeq \times \calA^\sd \longrightarrow \calA^\sd \ ,
\end{equation}
where $\calA^\simeq$ is the underlying groupoid of $\calA$,
and $\oplus^\sd$ sends $F \in \calA$ and $(E, \phi) \in \calA^\sd$
to the self-dual object $(F \oplus E \oplus F^\vee, \tilde{\phi})$,
with $\tilde{\phi} \colon F \oplus E \oplus F^\vee \simto
F^\vee \oplus E^\vee \oplus F^{\vee\vee}$ given by
\[
    \tilde{\phi} = \left( \, \begin{matrix}
        & & \id_F \\
        & \phi \\ 
        \eta_F^{-1}
    \end{matrix} \, \right) \ .
\]
This establishes $\calA^\sd$ as a module
for the monoidal groupoid $(\calA^\simeq, \oplus)$.

In fact, this construction will still work if we replace
the additive structure on $\calA$ by an arbitrary monoidal structure
compatible with the self-dual structure,
but this will not be needed for our purposes.

\subsection{Examples}

\paragraph{Example. Vector spaces.}
\label[example]{eg-vect-sp}

Let $\calA$ be the $\bbK$-linear abelian category of
$\bbK$-vector spaces of finite dimension,
and let $D = (-)^\vee \colon \calA \simto \calA^\op$
be the operation of taking the dual vector space.
Choose a sign $\varepsilon = \pm 1$,
and define a natural isomorphism
$\eta \colon (-)^{\vee\vee} \simTo \id_{\calA}$ by
\[
    \eta_V = \varepsilon \cdot \mathrm{ev}_V^{-1} \colon
    V^{\vee\vee} \longsimto V
\]
for $V \in \calA$,
where $\mathrm{ev}_V \colon V \simto V^{\vee\vee}$
is the evaluation map.
Then $(\calA, D, \eta)$ is a self-dual $\bbK$-linear abelian category.

By definition, a self-dual object in $\calA$
is a pair $(V, \phi)$, where $V$ is a vector space, and
$\phi \colon V \simto V^\vee$
is an isomorphism, such that $\phi^\vee = \phi \circ \eta_V$\,.
Equivalently, when $\varepsilon = 1$ (resp.~$-1$),
$\phi$ is a non-degenerate
symmetric (resp.~antisymmetric) bilinear form on~$V$.
Therefore, $\calA^\sd$ can be identified with the groupoid of
finite-dimensional orthogonal (resp.~symplectic) vector spaces,
where the morphisms are linear isomorphisms that preserve the bilinear forms.

\paragraph{Example. Self-dual quivers.}
\label[example]{eg-sym-quiver}

Let $Q$ be a \emph{self-dual quiver},
that is, a quiver with an involution $\sigma \colon Q \simto Q^\op$,
where $Q^\op$ is the opposite quiver of $Q$.
For the precise definition,
see \cref{sect-quiver} below.

Let $\calA = \Mod (\bbK Q)$ be the $\bbK$-linear abelian category of
finite-dimensional representations of $Q$ over $\bbK$.
One can define a self-dual structure $(-)^\vee \colon \calA \simto \calA^\op$
by sending each representation of $Q$
to the representation with the dual vector spaces and dual linear maps.
This also involves choosing signs when defining
$\eta \colon (-)^{\vee\vee} \simTo \id_{\calA}$\,, as in the previous example.

Self-dual objects in $\calA$ are called
\emph{self-dual representations} of $Q$,
and our enumerative invariants will be counting
these self-dual objects.

Enumerative invariant theories of quiver representations
can sometimes be seen as a model for
such theories for coherent sheaves on varieties,
which can be thought of as generalized $\GL (n)$-bundles.
We will argue in \cref{sect-quiv-coh-rel} that
self-dual quivers can also be thought of
as a model for enumerative invariant theories of $G$-bundles,
for $G = \upO (n)$ or $\Sp (n)$.

\paragraph{Example. Vector bundles.}
\label[example]{eg-vb}

Let $X$ be a $\bbK$-variety,
and let $\calA = \cat{Vect} (X)$
be the $\bbK$-linear exact category
of vector bundles on $X$ of finite rank,
or equivalently, locally free sheaves on $X$ of finite rank.

Define a self-dual structure $(-)^\vee \colon \calA \simto \calA^\op$
by sending a vector bundle to its dual vector bundle.
Choose a sign $\varepsilon = \pm 1$,
and define a natural isomorphism $\eta \colon (-)^{\vee\vee} \simTo \id_{\calA}$
to be $\varepsilon$ times the usual identification.

A self-dual object in $\calA$ is a pair $(E, \phi)$,
where $E$ is a vector bundle on $X$,
and $\phi \colon E \simto E^\vee$ is an isomorphism,
satisfying $\phi^\vee = \phi \circ \eta_E$\,.
Equivalently, $\phi$ is a symmetric (resp.~antisymmetric) bilinear form on $E$,
when $\varepsilon = +1$ (resp.~$-1$).
This is the same as an orthogonal (resp.~symplectic) vector bundle on $X$,
or a principal bundle on $X$ for the group $\upO (n)$ (resp.~$\Sp (n)$),
where $n$ is the rank of the vector bundle.

See \cref{sect-curves} below for further discussion
of this example in the case when $X$ is a curve.

\paragraph{Non-example. Coherent sheaves.}
\label{para-non-eg-coh}

Let $X$ be a smooth, projective $\bbK$-variety,
and let $\calA = \Coh (X)$ be the abelian category of coherent sheaves on $X$.
Then, when $\dim X > 0$,
$\calA$ is not a self-dual category,
since $\calA$ is noetherian in the sense of \cref{para-excat},
but not artinian, as there is an infinite decreasing chain
$\calO_X \supset \calO_X (-1) \supset \calO_X (-2) \supset \cdots$ in $\calA$
that does not stabilize.

However, the derived category of coherent sheaves on $X$
does have many self-dual structures, as in the next example.

\paragraph{Example. Derived category of coherent sheaves.}

Let $X$ be a smooth, projective $\bbK$-variety,
and let $\calC = \Db \Coh (X)$ be the derived category
of coherent sheaves on $X$.

Define $(-)^\vee \colon \calC \simto \calC^\op$
to be the derived dual functor
\[
    (-)^\vee = \bbR \calHom ( -, L ) \ ,
\]
where $L \to X$ is a line bundle, possibly with a shift.
For example, taking $L \simeq \calO_X$ gives the ordinary dual functor,
and taking $L \simeq \omega_X [n]$, 
where $\omega_X$ is the canonical bundle of $X$, and $n = \dim X$,
gives the Serre/Verdier duality functor.
As before, choose a sign $\varepsilon = \pm 1$
to identify the double dual of a sheaf to itself.

Using Bayer's \cite{Bayer2009} notion of
\emph{polynomial Bridgeland stability conditions},
it is possible to construct quasi-abelian subcategories
$\calA \subset \calC$ that are self-dual under the duality functor above.
Self-dual objects in $\calA$ 
can then be seen as generalized orthogonal or symplectic bundles.
We will discuss this example further in \cref{sect-coh-general} below.

\subsection{Self-dual filtrations}
\label{sect-sd-filt}

\paragraph{Definition.}
\label{para-def-filt-sd}
\label[definition]{def-filt-sd}

Let $\calA$ be a self-dual exact category,
and let $n \geq 0$ be an integer.
Then the category of filtrations $\calA^\pn{n}$,
as in \cref{def-cat-filt},
can be equipped with an induced self-dual structure.
Namely, for a filtration
$E_\bullet = (0 = E_0 \hookrightarrow E_1 \hookrightarrow \cdots \hookrightarrow E_n = E)$,
define its dual filtration $E_\bullet^\vee$ by
\begin{equation}
    \label{eq-dual-filt}
    0 = (E / E_n)^\vee
    \hookrightarrow (E / E_{n-1})^\vee
    \hookrightarrow \cdots
    \hookrightarrow (E / E_0)^\vee = E^\vee \ .
\end{equation}

A \emph{self-dual filtration} of length $n$ is
a self-dual object $(E_\bullet, \phi_\bullet) \in \calA^{\pn{n}, \> \sd}$.
We also say that $E_\bullet$ is a self-dual filtration of
the self-dual object $(E_n, \phi_n) \in \calA^\sd$.

In this case, if we write $F_i = E_i / E_{i-1}$ for $i = 1, \ldots, n$,
then there are induced isomorphisms $F_i \simeq F_{n+1-i}^\vee$,
defined via the isomorphism
$(E / E_{n-i})^\vee / (E / E_{n+1-i})^\vee \simeq (E_{n+1-i} / E_{n-i})^\vee$
given by the third isomorphism theorem,
which holds in an exact category by \cref{para-excat-abcat}.
In particular, if $n = 2 k + 1$ is odd,
then $F_{k+1}$ has an induced self-dual structure,
and becomes a self-dual object in $\calA$.

\paragraph{Projection functors.}
\label{para-filt-proj-sd}

Let $\calA$ be a self-dual exact category.
Recall from \cref{para-filt-proj} the projection functors
$\smash{p^f} \colon \calA^\pn{n} \to \calA^\pn{m}$ for order-preserving maps
$f \colon [m] \to [n]$.

Now, equip each set $[n]$ with the $\bbZ_2$-action given by
$(-)^\vee \colon [n] \simto [n]$, $i \mapsto n-i$.
Then, if an order-preserving map $f \colon [m] \to [n]$ 
is also $\bbZ_2$-equivariant,
then $\smash{p^f}$ is compatible with the self-dual structures
on $\calA^\pn{n}$ and $\calA^\pn{m}$, inducing a functor
\begin{equation}
    \label{eq-filt-proj-sd}
    p^{f, \> \sd} \colon \calA^{\pn{n}, \> \sd} \longrightarrow
    \calA^{\pn{m}, \> \sd} \ .
\end{equation}
In particular, taking $f$ to be the map
$(0, n) \colon [1] \to [n]$ or the map $(i-1, i) \colon [1] \to [2i-1]$,
similarly as in \cref{para-filt-proj},
we obtain the projection functors
\begin{alignat}{2}
    p^{\pn{n}, \> \sd} & \colon &
    \calA^{\pn{n}, \> \sd} & \longrightarrow \calA^\sd \ , \\
    p_i^\sd & \colon &
    \calA^{\pn{2i-1}, \> \sd} & \longrightarrow \calA^\sd \ .
\end{alignat}

\paragraph{Self-dual extensions.}
\label{para-def-sdext}

We are particularly interested in
self-dual filtrations of length $n = 3$, of the form
\begin{equation} \label{eq-def-sdext} \begin{tikzcd}[sep=small]
    \mathllap{0 = {}} E_0 \ar[r, hook] &
    E_1 \ar[r, hook] \ar[d, two heads] &
    E_2 \ar[r, hook] \ar[d, two heads] &
    E_3 \mathrlap{{} = E} \ar[d, two heads] \\
    & F_1 & F_2 & F_3 \rlap{ ,}
\end{tikzcd} \end{equation}
where the vertical arrows are cokernels of the horizontal arrows,
so that $F_i \simeq E_i / E_{i-1}$ for $i = 1, 2, 3$.
As in \cref{para-def-filt-sd},
$E$ has an induced self-dual structure $(E, \phi)$,
$F_2$ has an induced self-dual structure $(F_2, \psi)$,
and we have a natural isomorphism $F_3 \simeq F_1^\vee$.

In this case, we say that the self-dual object $(E, \phi)$
is a \emph{self-dual extension}
of the self-dual object $(F_2, \psi) \in \calA^\sd$
by the object $F_1 \in \calA$.

\paragraph{Isotropic subobjects.}
\label{para-isotropic}

Let $(E, \phi) \in \calA^\sd$, and let $F \in \calA$.
An inclusion $i \colon F \hookrightarrow E$ is said to be \emph{isotropic},
if the composition
\[
    F \overset{i}{\longrightarrow}
    E \overset{\phi}{\longrightarrow}
    E^\vee \overset{i^\vee}{\longrightarrow}
    F^\vee
\]
is zero.
We also say that $F$ is an \emph{isotropic subobject} of $(E, \phi)$.
In this case, writing
\begin{equation}
    \label{eq-def-perp}
    F^\perp = \ker \Bigl( \ 
        E \overset{\phi}{\longrightarrow}
        E^\vee \overset{i^\vee}{\longrightarrow} F^\vee
    \ \Bigr) \ ,
\end{equation}
we have a filtration of length $3$,
\begin{equation} \label{eq-isotr-ind-filt} \begin{tikzcd}[sep=small]
    0 \ar[r, hook] &
    F \ar[r, hook, "j"] \ar[d, two heads] &
    F^\perp \ar[r, hook] \ar[d, two heads] &
    E \ar[d, two heads] \\
    & F & A & F^\vee \rlap{ ,}
\end{tikzcd} \end{equation}
where $j$ is induced by $i$.
A standard argument shows that
$A$ has an induced self-dual structure $(A, \psi)$,
such that the above diagram is a self-dual extension
of $(A, \psi)$ by $F$.

Moreover, all self-dual extensions arise in this way.
In other words, given any self-dual extension~\cref{eq-def-sdext},
the inclusion $E_1 \hookrightarrow E$ is isotropic, $E_2 \simeq E_1^\perp$,
and the induced self-dual filtration~\cref{eq-isotr-ind-filt}
coincides with the original one.

\paragraph{The Grothendieck monoid module.}
\label{para-ksd}

Let $\calA$ be a self-dual exact category,
and recall from \cref{para-excat-groth} the
Grothendieck monoid $K_+ (\calA)$ of $\calA$,
which now has a $\bbZ_2$-action
$(-)^\vee \colon K_+ (\calA) \simto K_+ (\calA)$
given by the self-dual structure of $\calA$.

Define a $K_+ (\calA)$-module $K_+^\sd (\calA)$
to be the free $K_+ (\calA)$-module generated by
isomorphism classes of objects $(E, \phi) \in \calA^\sd$,
modulo the relations
\[
    [(E, \phi)] \sim [F_1] + [(F_2, \psi)]
\]
for self-dual extensions \cref{eq-def-sdext},
where `$+$' denotes the module action.

To avoid confusion, we will denote the action of
$K_+ (\calA)$ on $K_+^\sd (\calA)$ by
\[
    (\alpha, \theta) \longmapsto \bar{\alpha} + \theta \ ,
\]
instead of writing $\alpha + \theta$,
where $\alpha \in K_+ (\calA)$ and $\theta \in K_+^\sd (\calA)$.
This notation hints that the class $\bar{\alpha} + \theta$
should be thought of as a composition of
the classes $\alpha$, $\alpha^\vee$, and $\theta$.

For an object $E_\alpha \in \calA$ of class $\alpha$,
and a self-dual object $(E_\theta, \phi) \in \calA^\sd$ of class $\theta$,
the class $\bar{\alpha} + \theta$ is represented by the self-dual object
$E_\alpha \oplus^\sd (E_\theta, \phi)$ defined in \cref{para-sd-cat-module}.
In particular, we have
\begin{equation}
    \label{eq-ksd-involutive}
    \bar{\alpha} + \theta = \overline{\alpha^\vee} + \theta
\end{equation}
for any $\alpha \in K_+ (\calA)$ and $\theta \in K_+^\sd (\calA)$.

\paragraph{}
\label{para-sdcat}

We now formulate a set-up for
self-dual exact categories, similar to \tagref{ExCat},
which will be used in the rest of this paper.

\begin{condition*}[SdCat]
    Let $\calA$ be a self-dual $\bbK$-linear exact category,
    with properties and extra data as in \tagref{ExCat}.
    We also require the following:

    \begin{enumerate}
        \item 
            The $\bbZ_2$-action on $K_+ (\calA)$
            via the self-dual structure of $\calA$
            descends to a $\bbZ_2$-action
            \begin{equation} 
                (-)^\vee \colon C^\circ (\calA) \longsimto C^\circ (\calA) \ .
            \end{equation}

        \item
            We are given a $C^\circ (\calA)$-module $C^\sd (\calA)$,
            equipped with a surjection of $K_+ (\calA)$-modules
            $K_+^\sd (\calA) \to C^\sd (\calA)$.
            For $(E, \phi) \in \calA^\sd$, write $\llbr E, \phi \rrbr \in C^\sd (\calA)$
            for its class.
            We assume that there is a map of sets
            \begin{equation}
                \label{eq-j-csd}
                j \colon C^\sd (\calA) \longrightarrow C^\circ (\calA) \ ,
            \end{equation}
            such that $j (\llbr E, \phi \rrbr) = \llbr E \rrbr$
            for all $(E, \phi) \in \calA^\sd$.
    \end{enumerate}
    In this case, denote the action of $C^\circ (\calA)$ on $C^\sd (\calA)$ by
    \[
        (\alpha, \theta) \longmapsto \bar{\alpha} + \theta \ ,
    \]
    as in \cref{para-ksd}.
    For $E_\alpha \in \calA$ and $(E_\theta, \phi) \in \calA^\sd$
    with $\llbr E_\alpha \rrbr = \alpha$ and $\llbr E_\theta, \phi \rrbr = \theta$,
    the class $\bar{\alpha} + \theta$ is represented by the self-dual object
    $E_\alpha \oplus^\sd (E_\theta, \phi)$ defined in \cref{para-sd-cat-module}.
    In particular, \cref{eq-ksd-involutive}
    also holds for the $C^\circ (\calA)$-action on $C^\sd (\calA)$.
\end{condition*}

\paragraph{\tagref{SdCat} for filtrations.}
\label{para-filt-sd-class-monoid}

Let $\calA$ be a self-dual exact category, as in \tagref{SdCat}.
Then, the category of filtrations $\calA^\pn{n}$,
equipped with the induced self-dual structure as in \cref{para-def-filt-sd},
also satisfies \tagref{SdCat},
with the extra data $C^\sd (\calA^\pn{n})$ given by
\begin{equation}
    \label{eq-csd-filt}
    C^\sd (\calA^\pn{n}) = \begin{cases}
        C^\circ (\calA)^{n/2} & \text{if $n$ is even ,} \\
        C^\circ (\calA)^{(n-1)/2} \times C^\sd (\calA) & \text{if $n$ is odd ,}
    \end{cases}
\end{equation}
where a self-dual filtration $(E_\bullet, \phi_\bullet) \in \calA^{\pn{n}, \> \sd}$
is assigned the class
$(\llbr E_1 / E_0 \rrbr, \ldots, \break \llbr E_{n/2} / E_{n/2-1} \rrbr)$ if $n$ is even,
or $(\llbr E_1 / E_0 \rrbr, \ldots, \llbr E_{(n-1)/2} / E_{(n-3)/2} \rrbr, 
\llbr E_{(n+1)/2} / E_{(n-1)/2}, \psi \rrbr)$
if $n$ is odd, where $\psi$ is the induced self-dual structure.

\paragraph{}
\label[definition]{def-ext-sd}

We can also define a self-dual version of the Ext functors.
They will be related to self-dual extensions
in \cref{thm-sdext} below.

\begin{definition*}
    Let $(\calA, D, \eta)$ be a self-dual exact category.
    Let $(E, \phi) \in \calA^\sd$, and let $F \in \calA$.
    For every $i \geq 0$, define
    \begin{equation}
        \Ext^\sda{i} (E, F) =
        \Ext^i (E, F) \oplus
        \Ext^i (F^\vee, F)^{\bbZ_2},
    \end{equation}
    where $(-)^{\bbZ_2}$ denotes the $\bbZ_2$-invariant part,
    with the $\bbZ_2$-action on $\Ext^i (F^\vee, F)$ given by
    \begin{equation}
        \Ext^i (F^\vee, F)
        \overset{(-)^\vee}{\longrightarrow}
        \Ext^i (F^\vee, F^{\vee\vee})
        \overset{(\eta_F)_*}{\longrightarrow}
        \Ext^i (F^\vee, F).
    \end{equation}
\end{definition*}

\subsection{Self-dual stability conditions}

\paragraph{}

We define a notion of \emph{self-dual stability conditions} on
self-dual exact categories,
which are stability conditions in the sense of \cref{sect-exact-cat-stab}
that are compatible with the self-dual structure.

\begin{definition*}
    Let $\calA$ be a self-dual exact category,
    as in \tagref{SdCat}.
    Then, a \emph{self-dual \textnormal{(weak)} stability condition} on $\calA$
    consists of the following data:
    \begin{itemize}
        \item 
            A totally ordered set $T$,
            equipped with an order-reversing involution $t \mapsto -t$,
            fixing a unique element $0 \in T$.
        \item
            A (weak) stability condition $\tau \colon C (\calA) \to T$,
            as in \cref{def-stability},
            such that $\tau (\alpha) = -\tau (\alpha^\vee)$
            for all $\alpha \in C (\calA)$.
    \end{itemize}
\end{definition*}
In this case, we have the following basic observations:

\begin{itemize}
    \item 
        A non-zero object $E \in \calA$ is semistable with $\tau (E) = t$,
        if and only if $E^\vee$ is semistable with $\tau (E^\vee) = -t$.
    \item
        If $(E, \phi) \in \calA^\sd$ and $E$ is non-zero,
        then $\tau (E) = 0$.
\end{itemize}

\paragraph{Semistable and stable self-dual objects.}

Let $\tau$ be a self-dual weak stability condition on $\calA$, as above.
Recall from \cref{para-isotropic} the notion of isotropic subobjects.

\begin{itemize}
    \item 
        An object $(E, \phi) \in \calA^\sd$ is \emph{$\tau$-semistable},
        if for any isotropic subobject $F \subset E$
        with $0 \neq F \neq E$, we have $\tau (F) \leq 0 \leq \tau (E/F)$.

    \item
        An object $(E, \phi) \in \calA^\sd$ is \emph{$\tau$-stable},
        if for any isotropic subobject $F \subset E$
        with $0 \neq F \neq E$, we have $\tau (F) < 0 < \tau (E/F)$.
\end{itemize}
We will see in \cref{thm-sd-hn} that
a self-dual object is semistable if and only if
its underlying object is semistable.
However, this is not true for stable self-dual objects,
which are characterized by \cref{thm-st-sd-decomp} below.

Also, note that the object $(0, 0) \in \calA^\sd$ is considered
a stable self-dual object,
although the object $0 \in \calA$ is not stable as an ordinary object.
This might seem strange, but it is justified by
\cref{thm-st-sd-decomp} below.

\paragraph{Self-dual Harder--Narasimhan filtrations.}

Let $\tau$ be a self-dual weak stability condition on $\calA$,
and let $(E, \phi) \in \calA^\sd$.

A \emph{self-dual $\tau$-Harder--Narasimhan filtration} of $(E, \phi)$
is a self-dual filtration of $(E, \phi)$ of the form
\begin{equation}
    \label{eq-hn-sd}
    \hspace{1em} \begin{tikzcd}[sep=small]
        \mathllap{0 = {}} E_0 \ar[r, hook] &
        E_1 \ar[r, hook] \ar[d, two heads] &
        E_2 \ar[r, hook] \ar[d, two heads] &
        \cdots \ar[r, hook] &
        E_k \ar[r, hook] \ar[d, two heads] &
        E_k^\perp \ar[r, hook] \ar[d, two heads] &
        E_{k-1}^\perp \ar[r, hook] \ar[d, two heads] &
        \cdots \ar[r, hook] &
        E_1^\perp \ar[r, hook] \ar[d, two heads] &
        E_0^\perp \mathrlap{{} = E} \ar[d, two heads] \\
        & F_1 & F_2 & & F_k & A
        & F_k^\vee & & F_1^\vee & F_0^\vee \rlap{ ,}
    \end{tikzcd} \hspace{1em}
\end{equation}
where the horizontal maps are inclusions,
the vertical maps are cokernels of the horizontal maps,
each $F_i$ is non-zero and $\tau$-semistable,
$A$ is self-dual $\tau$-semistable with the induced self-dual structure,
and
\begin{equation}
    \tau (F_1) > \cdots > \tau (F_k) > 0 \ .
\end{equation}
Each inclusion $E_i \hookrightarrow E$ is then automatically isotropic,
with each $E_i^\perp$ as in \cref{para-isotropic}.

\begin{theorem}
    \label{thm-sd-hn}
    Let $\tau$ be a self-dual weak stability condition on $\calA$,
    and let $(E, \phi) \in \calA^\sd$.

    Then $(E, \phi)$ has a unique self-dual
    $\tau$-Harder--Narasimhan filtration, up to a unique isomorphism.
    It coincides with the $\tau$-Harder--Narasimhan filtration of $E$
    as an ordinary object,
    upon deleting the middle term $A$ if it is zero.

    In particular,
    a self-dual object $(E, \phi) \in \calA^\sd$ is $\tau$-semistable,
    if and only if the underlying object $E \in \calA$ is $\tau$-semistable.
\end{theorem}

\begin{proof}
    Let $E_\bullet$ be the Harder--Narasimhan filtration of $E$
    an ordinary object.
    Then its dual filtration $E^\vee_\bullet$,
    as in \cref{para-def-filt-sd},
    is a Harder--Narasimhan filtration of $E^\vee$.
    Identifying $E$ with $E^\vee$ using $\phi$,
    these two filtrations must coincide,
    giving $E_\bullet$ the structure of a self-dual filtration.
    Inserting an extra middle term $0$ if necessary,
    the self-dual filtration $E_\bullet$ has the form \cref{eq-hn-sd}.
    The middle term $A$ is semistable as an ordinary object,
    and hence, it is also semistable as a self-dual object by definition.
    This proves the existence part of the theorem.
    
    For uniqueness, it suffices to show that
    every self-dual Harder--Narasimhan filtration
    is also an ordinary Harder--Narasimhan filtration,
    upon deleting the middle term if it is zero.
    It is then enough to show that
    any semistable self-dual object
    is semistable as an ordinary object.
    This is because otherwise,
    its ordinary Harder--Narasimhan filtration
    would have the form described above,
    inducing a destabilizing isotropic subobject of $A$.
    This argument also proves the last statement of the theorem.
\end{proof}

\paragraph{Characterizing stable self-dual objects.}
\label[theorem]{thm-st-sd-decomp}

The following result gives a convenient description of
stable self-dual objects.
Note that this requires $\tau$ to be a stability condition,
rather than just a weak stability condition as above.

\begin{theorem*}
    Let $\tau$ be a self-dual stability condition on $\calA$,
    and $(E, \phi) \in \calA^\sd$ a $\tau$-stable self-dual object.
    Then there is a decomposition
    \begin{equation}
        (E, \phi) \simeq
        (E_1, \phi_1) \oplus \cdots \oplus (E_m, \phi_m) \ ,
    \end{equation}
    unique up to permuting the factors,
    where $(E_i, \phi_i) \in \calA^\sd$
    are pairwise non-isomorphic self-dual objects,
    such that each $E_i$ is a $\tau$-stable object in $\calA$.

    In particular, we have
    \begin{equation}
        \Aut (E, \phi) \simeq \bbZ_2^m .
    \end{equation}
\end{theorem*}

\begin{proof}
    Consider a $\tau$-Jordan--Hölder filtration of $E$,
    \[
        0 = F_0 \hookrightarrow F_1 \hookrightarrow \cdots \hookrightarrow F_\ell = E \ ,
    \]
    as in \cref{thm-jh-filtration}.
    If $\ell \leq 1$, then we are already done.
    Otherwise, let $G_1, \dotsc, G_\ell$ denote the stepwise quotients.
    The dual filtration
    \[
        0 = F_\ell^\perp \hookrightarrow F_{\ell-1}^\perp \hookrightarrow \cdots
        \hookrightarrow F_0^\perp = E
    \]
    is also a $\tau$-Jordan--Hölder filtration of $E$,
    where $F_i^\perp \simeq (E/F_i)^\vee$, 
    and we identified $E \simeq E^\vee$ via $\phi$.
    The stepwise quotients of this filtration are $G_\ell^\vee, \dotsc, G_1^\vee$.

    Since $(E, \phi)$ is $\tau$-stable,
    the map $F_1 \hookrightarrow E$ cannot factor through $F_1^\perp$.
    Thus, the composition $F_1 \hookrightarrow E \twoheadrightarrow E/F_1^\perp \simeq F_1^\vee$
    is non-zero, and hence is an isomorphism.
    This defines a self-dual structure on $F_1$\,.
    
    Now, the inclusion $F_1^\vee \simeq F_1 \hookrightarrow E$
    splits the short exact sequence
    $F_1^\perp \hookrightarrow E \twoheadrightarrow F_1^\vee$, so that
    $E \simeq F_1 \oplus F_1^\perp$.
    We can then identify $E^\vee \simeq F_1^\vee \oplus F_1^{\perp \vee}$.
    Since $\phi \colon E \simto E^\vee$ sends $F_1$ into $F_1^\vee$,
    and $F_1^\perp$ into $F_1^{\perp \vee}$, it follows that
    $\smash{\phi|_{F_1}}$ and $\smash{\phi|_{F_1^\perp}}$ must be isomorphisms onto
    $F_1^\vee$ and $F_1^{\perp \vee}$, respectively.
    They define self-dual structures on $F_1$ and $F_1^\perp$, and we have
    \[
        (E, \phi) \simeq (F_1, \phi|_{F_1}) \oplus (F_1^\perp, \phi|_{F_1^\perp}) \ ,
    \]
    with $F_1$ stable as an object of $\calA$.
    Repeating the process with $F_1^\perp$ in place of $E$,
    we will eventually obtain a desired decomposition.

    To see that the $(E_i, \phi_i)$ are pairwise non-isomorphic, 
    suppose the contrary,
    so $\psi \colon (E_i, \phi_i) \simto (E_j, \phi_j)$ for some $i \neq j$.
    Then the map $\id_{E_i} + \sqrt{-1} \cdot \psi \colon E_i \to E$
    gives an isotropic subobject, a contradiction.

    The uniqueness follows from the uniqueness part of
    \cref{thm-jh-filtration}, as any such decomposition
    gives rise to a $\tau$-Jordan--Hölder filtration of $E$.

    Finally, for the last statement,
    suppose that $\psi \colon E \simto E$ is an isomorphism.
    Then we must have $\psi (E_i) = E_i$ for all $i$.
    Since $\Aut (E_i) \simeq \Gm$\,,
    we have $\Aut (E_i, \phi_i) \simeq \bbZ_2$\,,
    as a scalar automorphism of $E_i$ preserves $\phi_i$
    if and only if it squares to the identity.
    This completes the proof.
\end{proof}

\subsection{Moduli of self-dual objects}

\paragraph{}
\label{para-sdmod}

We formulate a notion of the
\emph{moduli stack of self-dual objects}
in a $\bbK$-linear exact category $\calA$,
parallel to the approach of \cref{sect-exact-cat-moduli}.

\begin{condition*}[SdMod]
    Let $\calA$ be a self-dual $\bbK$-linear exact category,
    as in \tagref{SdCat},
    and let $\+{\calM}$ be a categorical moduli stack of $\calA$,
    as in \tagref{Mod}.
    We further assume that:

    \begin{enumerate}
        \item \label{itm-asn-sd-moduli-alg}
            There is a self-dual structure $(\+{\calM}, D, \eta)$ on $\+{\calM}$,
            in the sense of \cref{def-sd-stack},
            extending the self-dual structure of $\calA$.
    \end{enumerate}
    In this case, 
    we define the \emph{moduli stack of self-dual objects} in $\calA$
    to be the $\bbZ_2$-fixed locus
    \begin{equation}
        \calM^\sd = \calM^{\bbZ_2} \ ,
    \end{equation}
    as in \cref{para-sd-stack-z2-action}.
    It is an algebraic stack
    locally of finite type, and has affine stabilizers,
    as follows from \cref{lem-sd-stack-prop}~\cref{itm-sd-stack-lft}.

    We also assume that $\calM^\sd$ satisfies the following condition:

    \begin{enumerate}[resume]
        \item 
            There is a decomposition
            \begin{equation}
                \calM^\sd = \coprod_{\theta \in C^\sd (\calA) \vphantom{^0}}
                \calM^\sd_{\theta} \ ,
            \end{equation}
            where each $\calM^\sd_{\theta} \subset \calM^\sd$
            is an open and closed substack,
            such that $\calM^\sd_{\theta} (\bbK) \subset \calM^\sd (\bbK)$
            consists of $\bbK$-points $(E, \phi) \in \calA^\sd$
            with $\llbr (E, \phi) \rrbr = \theta$.
    \end{enumerate}
\end{condition*}

\paragraph{Module structure.}
\label{para-sd-moduli-module}

Under the condition~\tagref{SdMod},
the functor $\oplus^\sd \colon \calA^\simeq \times \calA^\sd \to \calA^\sd$
in \cref{para-sd-cat-module}
extends naturally to a map
\begin{equation}
    \label{eq-def-oplus-sd-moduli}
    \oplus^\sd \colon \calM \times \calM^\sd \longrightarrow \calM^\sd \ ,
\end{equation}
exhibiting $\calM^\sd$ as a module for the monoid $\calM$.

\paragraph{Semistable loci.}

Under the condition~\tagref{SdMod},
let $\tau$ be a self-dual weak stability condition on $\calA$,
satisfying \tagref{Stab1}.
For each $\theta \in C^\sd (\calA)$, define an open substack
\[
    \calM^\sdss_{\theta} (\tau) \subset \calM^\sd_{\theta}
\]
as the preimage of $\Mss_{j (\theta)} (\tau)$
under the forgetful morphism $\calM^\sd \to \calM$,
with $j$ as in~\cref{eq-j-csd}.
This can be thought of as the locus of
$\tau$-semistable self-dual objects of class~$\theta$,
by \cref{thm-sd-hn}.

If, moreover, $\tau$ satisfies \tagref{Stab2},
then $\calM^\sdss_{\theta} (\tau)$ is of finite type,
as follows from \cref{lem-sd-stack-prop}~\cref{itm-sd-stack-ft}.

\subsection{Moduli of self-dual filtrations}
\label{sect-mod-filt-sd}

\paragraph{Definition.}
\label[definition]{def-mod-filt-sd}

Let $\calA$ be a self-dual $\bbK$-linear exact category,
with a categorical moduli stack $\+{\calM}$, as in \tagref{SdMod}.
Let $n \geq 0$ be an integer.

Recall from \cref{def-mod-filt}
the moduli stack $\+{\calM}^\pn{n}$ of $n$-step filtrations in $\calA$.
It is now equipped with a self-dual structure,
in the sense of \cref{def-sd-stack},
given by the induced self-dual structures on $\+{\calM} (U)^\pn{n}$,
as in \cref{para-def-filt-sd}, for all $\bbK$-schemes $U$.
Therefore, $\+{\calM}^\pn{n}$ satisfies \tagref{SdMod}
as a moduli stack for the self-dual category $\calA^\pn{n}$.

In particular, the $\bbZ_2$-fixed locus
\[
    \calM^{\pn{n}, \> \sd} =
    ( \calM^\pn{n} )^{\bbZ_2}
\]
is an algebraic stack locally of finite type,
called the \emph{moduli stack of self-dual $n$-step filtrations} in $\calA$.

\paragraph{Projection morphisms.}

Recall from \cref{para-filt-proj-sd}
the projection functors
$p^{f, \> \sd} \colon \calA^{\pn{n}, \> \sd} \break \to \calA^{\pn{m}, \> \sd}$
for an order-preserving $\bbZ_2$-equivariant map $f \colon [m] \to [n]$,
and the total projection and middle projection functors
$p^{\pn{n}, \> \sd} \colon \calA^{\pn{n}, \> \sd} \to \calA^\sd$
and $p_i^\sd \colon \calA^{\pn{2i-1}, \> \sd} \to \calA^\sd$.
These extend naturally to
morphisms of algebraic $\bbK$-stacks
\begin{alignat}{2}
    \pi^{f, \> \sd} \colon && \calM^{\pn{n}, \> \sd} &\longrightarrow \calM^{\pn{m}, \> \sd} \ ,
    \\
    \pi^{\pn{n}, \> \sd} \colon && \calM^{\pn{n}, \> \sd} &\longrightarrow \calM^\sd \ ,
    \\
    \pi_i^\sd \colon && \calM^{\pn{2i-1}, \> \sd} &\longrightarrow \calM^\sd \ .
\end{alignat}

On the other hand, recall from \cref{para-mod-filt-proj}
the projection morphisms
$\pi^f \colon \calM^\pn{n} \to \calM^\pn{m}$ and
$\pi^{\pn{n}}, \pi_i \colon \calM^\pn{n} \to \calM$,
with notations explained there.
Composing with the forgetful morphism
$\calM^{\pn{n}, \> \sd} \to \calM^\pn{n}$,
we also obtain projection morphisms
\begin{align}
    \pi^f \colon \calM^{\pn{n}, \> \sd} &\longrightarrow \calM^{\pn{m}} \ ,
    \\
    \pi^{\pn{n}}, \pi_i \colon \calM^{\pn{n}, \> \sd} &\longrightarrow \calM \ .
\end{align}

\paragraph{Semistable loci.}
\label{para-sdss-filt}

In the situation of~\tagref{SdMod},
suppose that $\tau$ is a self-dual weak stability condition on $\calA$,
satisfying~\tagref{Stab1}.
Then, for $\alpha_1, \dotsc, \alpha_n \in C^\circ (\calA)$ and $\theta \in C^\sd (\calA)$,
we define an open substack
\[
    \calM^{\pn{2n+1}, \> \sdss}_{\alpha_1, \, \dotsc, \, \alpha_n, \> \theta} (\tau)
    \subset \calM^\pn{2n+1}_{\alpha_1, \, \dotsc, \, \alpha_n, \> \theta}
\]
to be the preimage of
$\Mss_{\alpha_1} (\tau) \times \cdots \times \Mss_{\alpha_n} (\tau) \times \Msdss_{\theta} (\tau)$
under the morphism
$\pi_1 \times \cdots \times \pi_n \times \pi_{n+1}^\sd \colon \calM^{\pn{2n+1}, \> \sd} \to \calM^n \times \calM^\sd$.

\begin{lemma}
    \label{lem-filt-sd-rep-ft}
    Suppose we are in the situation of \tagref{SdMod}.

    \begin{enumerate}
        \item \label{itm-filt-sd-rep}
            For any $n \geq 0,$ the morphism
            \begin{equation*}
                \pi^{\pn{n}, \> \sd} \colon \calM^{\pn{n}, \> \sd} \longrightarrow \calM^\sd
            \end{equation*}
            is representable.

        \item \label{itm-filt-sd-ft}
            Assume that \tagref{Fin1} holds.
            Then, for any $n \geq 0,$ the morphism
            \begin{equation}
                \label{eq-filt-sd-ft}
                \pi_1 \times \cdots \times \pi_n \times \pi_{n+1}^\sd
                \colon \calM^{\pn{2n+1}, \> \sd}
                \longrightarrow \calM^n \times \calM^\sd
            \end{equation}
            is of finite type.

        \item \label{itm-filt-sd-ft-2}
            Assume that \tagref{Fin2} holds.
            Then, for any $n \geq 0,$ the morphism
            \begin{equation}
                \label{eq-filt-sd-ft-2}
                \pi^{\pn{n}, \> \sd}
                \colon \calM^{\pn{n}, \> \sd}
                \longrightarrow \calM^\sd
            \end{equation}
            is of finite type.
    \end{enumerate}
\end{lemma}

\begin{proof}
    Part~\cref{itm-filt-sd-rep}
    follows from an analogous argument as in the proof of
    \cref{lem-filt-rep-ft}~\cref{itm-filt-rep}.
    
    For \cref{itm-filt-sd-ft},
    consider the $\bbZ_2$-action on $\calM^\pn{2n+1}$ as above,
    and the $\bbZ_2$-action on $\calM^{2n+1}$
    sending $(E_1, \dotsc, E_{2n+1})$ to
    $(E_{2n+1}^\vee, \dotsc, E_1^\vee)$.
    Then the morphism
    \begin{equation}
        \label{eq-pf-filt-sd-ft}
        \pi_1 \times \cdots \times \pi_{2n+1} \colon
        \calM^\pn{2n+1} \longrightarrow \calM^{2n+1}
    \end{equation}
    is $\bbZ_2$-equivariant,
    and the morphism~\cref{eq-filt-sd-ft}
    is the $\bbZ_2$-fixed locus in~\cref{eq-pf-filt-sd-ft}.
    But \cref{eq-pf-filt-sd-ft} is of finite type
    by \cref{lem-filt-rep-ft}~\cref{itm-filt-ft}.
    Therefore, by \cref{lem-sd-stack-prop}~\cref{itm-sd-stack-ft},
    the morphism~\cref{eq-filt-sd-ft}
    is of finite type.

    For \cref{itm-filt-sd-ft-2},
    the morphism $\pi^\pn{n} \colon \calM^\pn{n} \to \calM$
    is $\bbZ_2$-equivariant,
    and is of finite type by \cref{lem-filt-rep-ft}~\cref{itm-filt-ft-2},
    and the result follows from
    \cref{lem-sd-stack-prop}~\cref{itm-sd-stack-ft}.
\end{proof}

\paragraph{}

In particular, if~\tagref{Fin1} holds,
and $\tau$ is a self-dual weak stability condition on $\calA$
satisfying~\tagref{Stab12},
then the semistable loci
$\calM^{\pn{2n+1}, \> \sdss}_{\alpha_1, \, \dotsc, \, \alpha_n, \> \theta} (\tau)$
defined in \cref{para-sdss-filt}
are of finite type.

\subsection{The self-dual extension bundle}
\label{sect-sdext-sdext}

\paragraph{}

Let $\calA$ be a self-dual $\bbK$-linear exact category,
with a categorical moduli stack $\+{\calM}$, as in \tagref{SdMod}.
We also assume that $\calA$ has an extension bundle,
as in~\tagref{Ext}, parametrizing extensions of objects in $\calA$.
We aim to study the self-dual analogue of this problem,
and we study the bundle whose fibres
parametrize self-dual extensions of a self-dual object
by another object in $\calA$.

\begin{lemma}
    \label{lem-ext-invol}
    In the situation of \tagref{SdMod},
    suppose that \tagref{Ext} is satisfied.
    Then there is an isomorphism of $2$-vector bundles
    \begin{equation}
        \label{eq-ext-invol}
        \sigma^* \, \calExt^{1/0}_{\alpha_1, \alpha_2} \simeq
        \calExt^{1/0}_{\smash{\alpha_2^\vee, \alpha_1^\vee}} \ ,
    \end{equation}
    where $\sigma \colon \smash{\calM_{\alpha_1}} \times \smash{\calM_{\alpha_2}} \simto 
    \smash{\calM_{\alpha_2^\vee}} \times \smash{\calM_{\alpha_1^\vee}}$
    sends $(E_1, E_2)$ to $(E_2^\vee, E_1^\vee)$.

    In particular, we have isomorphisms of $2$-vector spaces
    $\Ext^{1/0} (E_2, E_1) \simeq \Ext^{1/0} (E_1^\vee, E_2^\vee)$
    for all $E_1, E_2 \in \calA$, and the Euler form satisfies
    \begin{equation}
        \label{eq-chi-invol}
        \chi (\alpha_1, \alpha_2) = \chi (\alpha_2^\vee, \alpha_1^\vee)
    \end{equation}
    for all $\alpha_1, \alpha_2 \in C^\circ (\calA)$.
\end{lemma}

\begin{proof}
    The natural $\bbZ_2$-action on $\calM^\pn{2}$
    in \cref{def-mod-filt-sd}
    is compatible with the automorphism $\sigma$.
    It is an automorphism of module stacks,
    because taking the opposite category of a $\bbK$-linear exact category
    is compatible with scaling morphisms
    and taking Baer sums of extensions.
\end{proof}

\paragraph{}
\label{para-sdext}

In particular, for any $\alpha \in C^\circ (\calA)$,
the isomorphism~\cref{eq-ext-invol} induces a $\bbZ_2$-action
on the $2$-vector bundle
$(\id \times D)^* \, \calExt^{\smash{1/0}}_{\smash{\alpha, \alpha^\vee}} \to \calM_\alpha$,
whose fibre at $E$ is the $2$-vector space $\Ext^{\smash{1/0}} (E^\vee, E)$.
Consider the fixed locus of this action,
\[
    (\calExt^{\smash{1/0}}_{\smash{\alpha, \alpha^\vee}})^{\bbZ_2}
    \longrightarrow \calM_\alpha \ ,
\]
which is a $2$-vector bundle over $\calM_\alpha$.

We will need the following compatibility condition
between the extension bundle and the self-dual structure.

\begin{condition*}[SdExt]
    In the situation of \tagref{SdMod},
    suppose that \tagref{Ext} is satisfied.
    We further assume that for all $\alpha \in C^\circ (\calA)$,
    the $2$-vector bundle
    $(\calExt^{\smash{1/0}}_{\smash{\alpha, \alpha^\vee}})^{\bbZ_2} \to \calM_\alpha$
    defined above has constant rank.
\end{condition*}

\paragraph{The self-dual Euler form.}
\label{para-euler-sd}

In the situation of \tagref{SdExt}, write
\begin{equation}
    \chi^\sd (\alpha, 0) =
    - {\rank (\calExt^{\smash{1/0}}_{\smash{\alpha, \alpha^\vee}})^{\bbZ_2}} \ .
\end{equation}
Define a map of sets 
$\chi^\sd \colon C^\circ (\calA) \times C^\sd (\calA) \to \bbZ$,
called the \emph{self-dual Euler form}, by
\begin{equation}
    \label{eq-def-chi-sd}
    \chi^\sd (\alpha, \theta) = 
    \chi (\alpha, j (\theta)) + \chi^\sd (\alpha, 0)
\end{equation}
for $\alpha \in C^\circ (\calA)$ and $\theta \in C^\sd (\calA)$,
where $j$ is the map in~\cref{eq-j-csd}.
One can deduce the relation
\begin{equation}
    \label{eq-chi-sd-assoc}
    \chi (\alpha_1, \alpha_2) + \chi^\sd (\alpha_1 + \alpha_2, \theta) =
    \chi^\sd (\alpha_2, \theta) + \chi^\sd (\alpha_1, \bar{\alpha}_2 + \theta)
\end{equation}
for all $\alpha_1, \alpha_2 \in C^\circ (\calA)$ and $\theta \in C^\sd (\calA)$.

\paragraph{}
\label[theorem]{thm-sdext}

We can now show that the family of self-dual extensions is,
roughly speaking,
parametrized by the self-dual Ext groups in \cref{def-ext-sd}.

\begin{theorem*}
    In the situation of \tagref{SdExt},
    let $(E, \phi) \in \calA^\sd$ be a self-dual object,
    and let $F \in \calA$ be an object.
    
    Then the stack of self-dual extensions of $(E, \phi)$ by $F$
    is an affine $2$-vector bundle over the $2$-vector space
    $\Ext^{1/0} (E, F)$, which is a torsor for the trivial $2$-vector bundle
    \begin{equation}
        \Ext^{1/0} (E, F) \times \Ext^{1/0} (F^\vee, F)^{\bbZ_2}
        \longrightarrow \Ext^{1/0} (E, F) \ .
    \end{equation}
\end{theorem*}

\begin{proof}
    Consider the $2$-vector bundle
    \[
        \calE \longrightarrow \Ext^{1/0} (E, F)
    \]
    whose fibre over $a \in \Ext^{1/0} (E, F) (\bbK)$
    is given by $\Ext^{1/0} (F^\vee, G_a)$,
    where $G_a$ is the extension of $E$ by $F$ given by $a$.
    Precisely, $\calE$ is the moduli stack of
    $3$-step filtrations with stepwise quotients $F, E, F^\vee$.

    Define
    \[
        \calF = \calE
        \underset{\Ext^{1/0} (E, F)^2}{\times}
        \Ext^{1/0} (E, F),
    \]
    where the morphism $\calE \to \Ext^{1/0} (E, F)^2$
    sends a filtration 
    $0 \hookrightarrow F \hookrightarrow G \hookrightarrow H$ to $(a, a')$,
    where $a$ is the class of $G$ as an extension of $E$ by $F$,
    and $a'$ is the image of the class of $H$ in $\Ext^{1/0} (F^\vee, G)$
    under the pushforward map to $\Ext^{1/0} (F^\vee, E)$,
    identified with $\Ext^{1/0} (E, F)$ using the self-dual structure of $E$.
    The morphism $\Ext^{1/0} (E, F) \to \Ext^{1/0} (E, F)^2$ is the diagonal morphism.

    Regard $\Ext^{1/0} (E, F)^2$ as a trivial $2$-vector bundle over $\Ext^{1/0} (E, F)$.
    It follows from \cref{lem-2vb-morphism-affine}
    that $\calF$ is an affine $2$-vector bundle over $\Ext^{1/0} (E, F)$,
    and is a torsor for the trivial $2$-vector bundle with fibre $\Ext^{1/0} (F^\vee, F)$.

    The $\bbZ_2$-action on $\calE$ by taking the dual filtration 
    induces a $\bbZ_2$-action on $\calF$,
    and the morphism $\calF \to \Ext^{1/0} (E, F)$ is $\bbZ_2$-invariant.
    The fixed locus $\calF^{\bbZ_2}$ is
    the moduli stack of self-dual extensions of $E$ by $F$.

    Moreover, this $\bbZ_2$-action on $\calF$ is affine linear.
    Indeed, for any $a \in \Ext^{1/0} (E, F) (\bbK)$,
    the fibre $\calF_a$ is $\Ext^{1/0} (F^\vee, G_a)_{a}$\,,
    where the last subscript $a$ denotes taking the fibre over
    $a \in \Ext^{1/0} (F^\vee, E) (\bbK)$.
    The $\bbZ_2$-action identifies this with
    $\Ext^{1/0} (G_a^\vee, F)_{a}$\,,
    and reinterprets its elements as $3$-step filtrations,
    giving elements of $\Ext^{1/0} (F^\vee, G_a)_{a}$ again,
    but note that now the roles of the two $a$'s have changed.
    These two steps are affine linear by
    \cref{lem-3step-affine-linear,lem-ext-invol}.
    
    Therefore, $\calF^{\bbZ_2}$ is an affine $2$-vector bundle 
    over $\Ext^{1/0} (E, F)$,
    and is a torsor for the trivial $2$-vector bundle
    with fibre $\Ext^{1/0} (F^\vee, F)^{\bbZ_2}$.
\end{proof}

\paragraph{}
\label[theorem]{thm-sdext-strong}

We also have a family version of \cref{thm-sdext}.

\begin{theorem*}
    In the situation of \tagref{SdExt},
    let $\alpha \in C^\circ (\calA)$ and $\theta \in C^\sd (\calA)$.
    Then the morphism
    \begin{equation}
        \pi_1 \times \pi^\sd_2 \colon
        \calM^{\smash{\pn{3}, \> \sd}}_{\alpha, \theta} \longrightarrow
        \calM_\alpha \times \calM^\sd_\theta
    \end{equation}
    factors as an affine $2$-vector bundle of rank $-\chi^\sd (\alpha, 0)$
    followed by a $2$-vector bundle of rank $-\chi (\alpha, j (\theta)),$
    so that the total rank is $-\chi^\sd (\alpha, \theta)$.
\end{theorem*}

\begin{proof}
    \fixlineheight
    Let $u \colon \smash{\calM^\sd_\theta} \to \calM_{\smash{j (\theta)}}$ be the forgetful morphism
    that forgets the self-dual structure.
    Consider the $2$-vector bundle
    $\calV = (\id \times u)^* \calExt^{1/0}_{\smash{\alpha, j (\theta)}}$
    on $\calM_\alpha \times \smash{\calM^\sd_\theta}$,
    whose fibre over $(F, (E, \phi))$ is $\Ext^{1/0} (E, F)$.
    Using $\calV$ in place of $\Ext^{1/0} (E, F)$ in the proof of \cref{thm-sdext},
    a similar argument shows that
    there is an affine $2$-vector bundle $\calF \to \calV$,
    together with an affine linear $\bbZ_2$-action,
    such that the fixed locus $\calF^{\bbZ_2}$
    can be identified with $\calM^{\smash{\pn{3}, \> \sd}}_{\alpha, \theta}$.
    Moreover, $\calF^{\bbZ_2}$ is a torsor for the $2$-vector bundle
    $p^* (\calExt^{1/0}_{\smash{\alpha, \alpha^\vee}})^{\bbZ_2} \to \calV$,
    where $p \colon \calV \to \calM_\alpha$ is the projection.
    This shows that the rank of $\smash{\calF^{\bbZ_2}}$ is as stated.
\end{proof}

%% file: alg.tex
\label{sect-alg}

In this \lcnamecref{sect-alg},
we introduce the algebraic structures
that arise in the enumerative geometry in self-dual categories,
including \emph{motivic Hall algebras},
\emph{motivic Hall modules},
and their corresponding Lie algebras and \emph{twisted modules}.
These structures will be closely related to
motivic enumerative invariants
that we will study in \cref{sect-invariants}.

\subsection{Involutive algebraic structures}

\paragraph{}

To fix notation and terminology,
we introduce several types of algebraic structures
that arise in the enumerative geometry in self-dual categories.

Throughout this section,
we assume that $\bbk$ is a field with $\operatorname{char} \bbk \neq 2$.

\paragraph{Involutive algebras.}
\label{para-inv-alg}

Let $A$ be an associative $\bbk$-algebra.
An \emph{involution} of $A$ is a $\bbk$-linear map
\[
    (-)^\vee \colon A \longsimto A \ ,
\]
such that $(-)^{\vee\vee} = \id_A$, and $(x y)^\vee = y^\vee x^\vee$
for all $x, y \in A$.
An algebra $A$ equipped with an involution is called
an \emph{involutive algebra} over $\bbk$.

In this case, define the subgroup of
\emph{isometric elements} in $A$ to be
\begin{equation}
    A^\iso = \{ x \in A \mid x x^\vee = 1 \} \ .
\end{equation}
When $A$ is finite-dimensional,
$A^\iso$ is a linear algebraic group over $\bbk$.

\paragraph{Involutive Lie algebras.}
\label{para-inv-lie-alg}

Let $L$ be a Lie algebra over $\bbk$.
An \emph{involution} of $L$ is a $\bbk$-linear map
\[
    (-)^\vee \colon L \longsimto L \ ,
\]
such that $(-)^{\vee\vee} = \id_L$, and $[x, y]^\vee = -[x^\vee, y^\vee]$
for all $x, y \in L$.
A Lie algebra $L$ equipped with an involution is called
an \emph{involutive Lie algebra} over $\bbk$.

In this case, define the subspace of
\emph{isometric elements} in $L$ to be
\begin{equation}
    L^+ = \{ x \in L \mid x^\vee = -x \} \ ,
\end{equation}
which is a Lie subalgebra of $L$.

\paragraph{Twisted modules.}
\label{para-tw-mod}

Let $L$ be an involutive Lie algebra over $\bbk$.
A \emph{twisted module} for $L$ is a $\bbk$-vector space $M$,
together with a bilinear map $\heart \colon L \otimes M \to M$,
such that for any $x, y \in L$ and $m \in M$, we have
\begin{align}
    x^\vee \heart m & = -x \heart m \ , \\
    \label{eq-tw-mod-jacobi}
    x \heart (y \heart m) - y \heart (x \heart m) & =
    [x, y] \heart m - [x^\vee, y] \heart m \ .
\end{align}
Here, the module structure is \emph{twisted} by the involution,
in the sense that we have an extra term
on the right-hand side of \cref{eq-tw-mod-jacobi},
compared to the usual Jacobi identity.

Equivalently, a twisted module for $L$ is just
a module $M$ for the Lie algebra $L^+$.
Namely, 
given a twisted module as above, for $x \in L^+$ and $m \in M$,
defining
\begin{equation}
    x \cdot m = \frac{1}{2} \, x \heart m
\end{equation}
equips $M$ with the structure of a Lie algebra module for $L^+$.
On the other hand, given an $L^+$-module $M$,
for any $x \in L$, defining
\begin{equation}
    x \heart m = (x - x^\vee) \cdot m
\end{equation}
recovers the twisted module structure.

\subsection{Motivic Hall algebras}

\paragraph{}

We summarize the construction of \emph{motivic Hall algebras}
on the space of stack functions, 
due to Joyce~\cite{Joyce2007II}.
See \cref{sect-motive} for an introduction to stack functions.

Throughout this section,
we assume that $\calA$ is a $\bbK$-linear exact category,
with a moduli stack $\calM$, satisfying
\tagref{Mod} and \tagref{Fin1}.

\begin{definition}
    \label{def-hall-alg}
    In the situation of \tagref{Mod} and \tagref{Fin1},
    define an operation
    \begin{equation}
        {*} =
        (\pi^\pn{2})_! \circ
        (\pi_1 \times \pi_2)^* \colon
        \SF ( \calM ) \otimes
        \SF ( \calM ) \longrightarrow
        \SF ( \calM ) \ ,
    \end{equation}
    where the composition is through $\SF (\calM^\pn{2})$.
    This is well-defined by 
    \cref{thm-sf-func}, \cref{lem-filt-rep-ft}, and~\tagref{Fin1}.
    
    Let $\delta_0 \in \SF ( \calM )$ be the characteristic function
    of the point $0 \in \calM (\bbK)$,
    represented by the morphism $0 \colon \Spec \bbK \hookrightarrow \calM$.

    This defines the \emph{motivic Hall algebra}
    $(\SF ( \calM ), *, \delta_0)$,
    which is an associative $\bbQ$-algebra, as stated below.
\end{definition}

\begin{theorem}
    \label{thm-hall-alg}
    In the situation of \tagref{Mod} and \tagref{Fin1},
    for any $f, g, h \in \SF ( \calM ),$ we have
    \begin{gather}
        \delta_0 * f = f = f * \delta_0 \ , \\
        (f * g) * h = f * (g * h) \ .
    \end{gather}
    Therefore, the operation $*$ defines
    an associative algebra structure on $\SF ( \calM )$.

    Moreover, for any $f_1, \dotsc, f_n \in \SF (\calM),$ we have
    \begin{equation}
        f_1 * \cdots * f_n =
        (\pi^\pn{n})_! \circ (\pi_1 \times \cdots \times \pi_n)^*
        (f_1 \boxtimes \cdots \boxtimes f_n) \ ,
    \end{equation}
    where $f_1 \otimes \cdots \otimes f_n \in \SF (\calM)^{\otimes n}
    \simeq \SF (\calM^n),$
    and the composition is through
    $\SF (\calM^\pn{n}).$
\end{theorem}

\begin{proof}
    This is essentially
    \cite[Theorem~5.2]{Joyce2007II}.
\end{proof}

\begin{remark}
    \label{rem-hall-alg-lsf}
    As in \cite[\S5.4]{Joyce2007II},
    if \tagref{Fin2} is also satisfied,
    then it is possible to define the motivic Hall algebra structure on the space
    $\dot{\LSF} (\calM) \subset \LSF (\calM)$
    of local stack functions supported on a finite union of the substacks $\calM_\alpha$\,.
\end{remark}

\paragraph{}
\label[theorem]{thm-hall-alg-inv}

Now, we also assume that $\calA$ has a self-dual structure,
with a self-dual moduli stack $\calM$ as in \tagref{SdMod}.
We show that the motivic Hall algebra $\SF ( \calM )$
is an involutive algebra in this case,
in the sense of \cref{para-inv-alg}.

\begin{theorem*}
    In the situation of \tagref{SdMod} and \tagref{Fin1},
    define a map
    \begin{equation}
        \label{eq-def-hall-alg-inv}
        (-)^\vee = D^* \colon \SF (\calM) \longsimto \SF (\calM) \ ,
    \end{equation}
    where $D \colon \calM \simto \calM$ is the self-dual structure.
    Then, for any $f, g \in \SF ( \calM )$, we have
    \begin{equation}
        (f * g)^\vee = g^\vee * f^\vee \ ,
    \end{equation}
    and $(-)^\vee$ defines an involution of
    the motivic Hall algebra $\SF ( \calM )$.
\end{theorem*}

\begin{proof}
    There is a commutative diagram
    \[ \begin{tikzcd}
        \calM \times \calM
        \ar[d, "D \times D"']
        & \calM^\pn{2} \ar[l, "\pi_2 \times \pi_1"']
        \ar[r, "\pi^\pn{2}"] \ar[d, "D"']
        & \calM \ar[d, "D"]
        \\ \calM \times \calM
        & \calM^\pn{2} \ar[l, "\pi_1 \times \pi_2"']
        \ar[r, "\pi^\pn{2}"]
        & \calM \rlap{ ,}
    \end{tikzcd} \]
    where the vertical morphisms are isomorphisms.
    Therefore, by \cref{thm-sf-func}~\cref{itm-thm-sf-func-bct},
    we have
    \begin{align*}
        (f * g)^\vee
        & = D^* \, (\pi^\pn{2})_! \, (\pi_1 \times \pi_2)^* \, (f \otimes g) \\
        & = (\pi^\pn{2})_! \, D^* \, (\pi_1 \times \pi_2)^* \, (f \otimes g) \\
        & = (\pi^\pn{2})_! \, (\pi_2 \times \pi_1)^* \, (D \times D)^* \, (f \otimes g) \\
        & = (\pi^\pn{2})_! \, (\pi_2 \times \pi_1)^* \, (f^\vee \otimes g^\vee) \\
        & = g^\vee * f^\vee.
    \end{align*}
\end{proof}

\paragraph{The Hall Lie algebra.}
\label{para-hall-lie-alg}

In particular,
in the situation of \tagref{Mod} and \tagref{Fin1},
the motivic Hall algebra $\SF ( \calM )$ is also a Lie algebra,
with the Lie bracket given by
\begin{equation}
    [f, g] = f * g - g * f
\end{equation}
for $f, g \in \SF ( \calM )$.

Moreover, in the situation of \tagref{SdMod} and \tagref{Fin1},
this Lie algebra is also an involutive Lie algebra,
with the involution given by \cref{eq-def-hall-alg-inv}.

\subsection{Motivic Hall modules}

\paragraph{}

Now, let $\calA$ be a self-dual $\bbK$-linear exact category,
with a self-dual moduli stack $\calM$ as in \tagref{SdMod},
so that we have a moduli stack $\calM^\sd$ of self-dual objects in $\calA$.
We construct on the space $\SF ( \calM^\sd )$
a structure of a module over the motivic Hall algebra $\SF ( \calM )$,
which we call the \emph{motivic Hall module}.

\begin{definition}
    \label{def-diamond}
    In the situation of \tagref{SdMod} and \tagref{Fin1},
    define an operation
    \begin{equation}
        {\diamond} =
        (\pi^{\pn{3}, \> \sd})_! \circ
        (\pi_1 \times \pi^\sd_2)^* \colon
        \SF ( \calM ) \otimes
        \SF ( \calM^\sd ) \longrightarrow
        \SF ( \calM^\sd ) \ ,
    \end{equation}
    where the composition is through $\SF (\calM^{\pn{3}, \> \sd})$.
    This is well-defined by
    \cref{thm-sf-func}, \cref{lem-filt-sd-rep-ft}, and~\tagref{Fin1}.

    This defines the \emph{motivic Hall module}
    structure on $\SF ( \calM^\sd )$, as stated below.
\end{definition}

\begin{theorem}
    \label{thm-hall-module}
    In the situation of \tagref{SdMod} and \tagref{Fin1},
    for any $f, g \in \SF ( \calM )$
    and $h \in \SF ( \calM^\sd )$, we have
    \begin{align}
        \label{eq-hall-module-1}
        \delta_0 \diamond h & = h \ , \\
        \label{eq-hall-module-2}
        (f * g) \diamond h & =
        f \diamond (g \diamond h) \ .
    \end{align}
    Therefore, the operation $\diamond$
    equips $\SF ( \calM^\sd )$ with the structure of a module
    over the motivic Hall algebra $\SF ( \calM )$.

    Moreover, for any $f_1, \dotsc, f_n \in \SF (\calM)$
    and $h \in \SF (\calM^\sd),$ writing
    \[
        f_1 \diamond \cdots \diamond f_n \diamond h =
        f_1 \diamond ( { \cdots \diamond (f_{n-1} \diamond (f_n \diamond h) ) \cdots } ) \ ,
    \]
    we have
    \begin{equation}
        f_1 \diamond \cdots \diamond f_n \diamond h =
        (\pi^{\pn{2n+1}, \> \sd})_! \circ
        (\pi_1 \times \cdots \times \pi_n \times \pi^\sd_{n+1})^*
        (f_1 \boxtimes \cdots \boxtimes f_n \boxtimes h) \ ,
    \end{equation}
    where $f_1 \boxtimes \cdots \boxtimes f_n \boxtimes h \in \SF (\calM^n \times \calM^\sd),$
    and the composition is through
    $\SF (\calM^{\pn{2n+1}, \> \sd}).$
\end{theorem}

\begin{proof}
    Equation \cref{eq-hall-module-1},
    follows from the fact that $\delta_0$
    is supported on $\{ 0 \} \subset \calM$,
    and that $\pi_1 \times \pi_2^\sd$
    restricts to an isomorphism
    $\pi_1^{-1} (\{ 0 \}) \simto \{ 0 \} \times \calM^\sd$.
    
    To prove \cref{eq-hall-module-2}, consider the diagram below,
    where subscripts $\alpha, \beta, \theta$ are added to help with understanding,
    indicating classes in $C^\circ (\calA)$ and $C^\sd (\calA)$,
    but one could also write down the diagram without them.
    \begin{equation} \begin{tikzcd}[background color=none, column sep=4.5em]
        \calM_\alpha \times \calM_\beta \times \calM^\sd_\theta &
        \calM^\pn{2}_{\alpha,\beta} \times \calM^\sd_\theta
        \ar[l, "(\pi_1 \times \pi_2) \times \id"', "\text{rep.}"]
        \ar[r, "\pi^\pn{2} \times \id", "\text{f.t.}"'] &
        \calM_{\alpha+\beta} \times \calM^\sd_\theta \\
        \calM_\alpha \times \calM^{\pn{3}, \> \sd}_{\beta, \theta}
        \ar[d, "\mathllap{\id \times \pi^{\pn{3}, \> \sd}} \enspace \mathrlap{\text{f.t.}}" description] 
        \ar[u, "\mathllap{\id \times (\pi_1 \times \pi^\sd_2)} \enspace \mathrlap{\text{rep.}}" description] &
        \calM^{\pn{5}, \> \sd}_{\alpha,\beta,\theta} 
        \ar[u, "\text{rep.}"] \ar[d, "\text{f.t.}"]
        \ar[l, "\text{rep.}"'] \ar[r, "\text{f.t.}"']
        \ar[ur, phantom, pos=.2, "\llcorner"]
        \ar[dl, phantom, pos=.2, "\urcorner"] &
        \calM^{\pn{3}, \> \sd}_{\alpha+\beta,\theta}
        \ar[d, "\mathllap{\text{f.t.}} \enspace \mathrlap{\pi^{\pn{3}, \> \sd}}" description]
        \ar[u, "\mathllap{\text{rep.}} \enspace \mathrlap{\pi_1 \times \pi^\sd_2}" description] \\ 
        \calM_\alpha \times \calM^\sd_{\bar{\beta}+\theta} &
        \calM^{\pn{3}, \> \sd}_{\alpha, \bar{\beta}+\theta}
        \ar[l, "\pi_1 \times \pi^\sd_2"', "\text{rep.}"]
        \ar[r, "\pi^{\pn{3}, \> \sd}", "\text{f.t.}"'] &
        \calM^\sd_{\bar{\alpha} + \bar{\beta} + \theta}
    \end{tikzcd} \end{equation}
    In the diagram, `rep.'\ means `representable' and `f.t.'\ means `finite type'.
    These properties follow from
    \cref{lem-filt-rep-ft,lem-filt-sd-rep-ft},
    applied to the outer eight morphisms;
    the other four morphisms are pullbacks of them.
    By \cref{thm-sf-func}, this gives rise to a commutative diagram
    \begin{equation} \begin{tikzcd}[column sep=1.5em, row sep={4em,between origins}]
        \begin{matrix}
            \SF ( \calM_\alpha ) \otimes \SF ( \calM_\beta ) \\
            {} \otimes \SF ( \calM^\sd_\theta )
        \end{matrix} \ar[r] \ar[d] &
        \SF ( \calM^\pn{2}_{\alpha,\beta} ) \otimes \SF ( \calM^\sd_\theta ) \ar[r] \ar[d] &
        \SF ( \calM_{\alpha+\beta} ) \otimes \SF ( \calM^\sd_\theta ) \ar[d] \\
        \SF ( \calM_\alpha ) \otimes \SF ( \calM^{\pn{3}, \> \sd}_{\beta, \theta} ) \ar[r] \ar[d] &
        \SF ( \calM^{\pn{5}, \> \sd}_{\alpha,\beta,\theta} ) \ar[r] \ar[d] &
        \SF ( \calM^{\pn{3}, \> \sd}_{\alpha+\beta,\theta} ) \ar[d] \\
        \SF ( \calM_\alpha ) \otimes \SF ( \calM^\sd_{\bar{\beta}+\theta} ) \ar[r] &
        \SF ( \calM^{\pn{3}, \> \sd}_{\alpha, \bar{\beta}+\theta} ) \ar[r] &
        \SF ( \calM^\sd_{\bar{\alpha} + \bar{\beta} + \theta} ) \rlap{ ,}
    \end{tikzcd} \end{equation}
    where the maps are pushforwards and pullbacks,
    depending on the direction of the arrows.
    The commutativity of the outer square is precisely
    the content of \cref{eq-hall-module-2}.
\end{proof}

\begin{remark}
    \label{rem-hall-mod-lsf}
    As in \cref{rem-hall-alg-lsf},
    if \tagref{Fin2} is also satisfied,
    then it is also possible to define
    the motivic Hall module structure on the space
    $\dot{\LSF} (\calM^\sd) \subset \LSF (\calM^\sd)$
    of local stack functions supported on a finite union of the substacks
    $\calM^\sd_{\smash{\theta}}$,
    which is a module over the algebra in \cref{rem-hall-alg-lsf}.
\end{remark}

\paragraph{The Hall twisted module.}
\label{para-hall-tw-mod}

In the situation of \tagref{SdMod} and \tagref{Fin1},
the motivic Hall module $\SF ( \calM^\sd )$
can also be equipped with the structure of a twisted module,
in the sense of \cref{para-tw-mod},
for the Hall Lie algebra $\SF ( \calM )$
defined in \cref{para-hall-lie-alg}, given by
\begin{equation}
    \label{eq-hall-tw-mod}
    f \heart h = f \diamond h - f^\vee \diamond h
\end{equation}
for $f \in \SF ( \calM )$ and $h \in \SF ( \calM^\sd )$.

Note that although the operation $\diamond$
already makes $\SF ( \calM^\sd )$ into an (untwisted) module
for the Hall Lie algebra $\SF ( \calM )$,
the twisted module structure is more analogous
to the Hall Lie bracket,
as one can think of \cref{eq-hall-tw-mod}
as a commutator of~$f$ and~$h$.
In our applications to enumerative invariants,
we will often find that properties enjoyed by the Hall Lie bracket,
but not the associative product,
will have analogues for the twisted module structure,
but not the usual module structure.

\subsection{Motivic integration}
\label{sect-motivic-int-1}

\paragraph{}
\label{para-sit-mot-int}

We study the \emph{motivic integration} of stack functions,
in the sense of \cref{para-mot-int},
and in particular, its relation to the motivic Hall algebra and module.
This will be helpful in studying motivic enumerative invariants,
which will be defined as motivic integrals of certain stack functions.

Throughout, we assume that 
$\calA$ is a self-dual $\bbK$-linear exact category,
with a self-dual moduli stack $\+{\calM}$ as in \tagref{SdMod}.
We also assume the condition \tagref{SdExt},
restricting the homological dimension of $\calA$.

Let $\Mhat_\bbK$ be the completed ring of motives
as in \cref{def-motive-varieties}.

\paragraph{The algebra $\Lambda (\calA)$.}
\label{para-lambda-a}

In the situation of \cref{para-sit-mot-int},
following a similar construction in
\cite[Definition~6.3]{Joyce2007II},
we define an associative $\Mhat_\bbK$-algebra
$\Lambda (\calA)$ as follows.

As an $\Mhat_\bbK$-module, set
\begin{equation}
    \Lambda (\calA) =
    \bigoplus_{\alpha \in C^\circ (\calA)} \Mhat_\bbK \cdot \lambda_\alpha \ ,
\end{equation}
where $\lambda_\alpha$ is a basis element.
Define the multiplication
$* \colon \Lambda (\calA) \otimes_{\smash{\Mhat_\bbK}} \Lambda (\calA) \to \Lambda (\calA)$ by 
\begin{equation}
    \label{eq-lambda-mult}
    \lambda_{\alpha_1} * \lambda_{\alpha_2} =
    \bbL^{-\chi (\alpha_1, \alpha_2)} \cdot \lambda_{\alpha_1 + \alpha_2}
\end{equation}
for all $\alpha_1, \alpha_2 \in C^\circ (\calA)$,
where $\chi (\alpha_1, \alpha_2)$ is the Euler form defined in \cref{para-euler}.
The associativity of this product follows from
the bilinearity of $\chi$.
The element $\lambda_0 \in \Lambda (\calA)$ is the multiplication identity.

The self-dual structure of $\calA$ equips
the algebra $\Lambda (\calA)$ with an involution,
in the sense of \cref{para-inv-alg},
given by $\lambda_\alpha^\vee = \lambda_{\smash{\alpha^\vee}}$ for all $\alpha \in C^\circ (\calA)$.

Note that we have an isomorphism
\begin{equation}
    \label{eq-lambda-iso-sf-disc}
    \Lambda (\calA) \simeq \SFhat (C^\circ (\calA))
\end{equation}
of $\Mhat_\bbK$-modules,
where $C^\circ (\calA)$ is regarded as a discrete space.
It is sometimes helpful to think of $\Lambda (\calA)$ in this way,
as some kind of discretized Hall algebra.

\paragraph{The module $\Lambda^\sd (\calA)$.}
\label{para-lambda-sd-a}

We define a $\Lambda (\calA)$-module $\Lambda^\sd (\calA)$ as follows.
As an $\Mhat_\bbK$-module, set
\begin{equation}
    \Lambda^\sd (\calA) =
    \bigoplus_{\theta \in C^\sd (\calA) \vphantom{^0}}
    \Mhat_\bbK \cdot \lambda^\sd_{\smash{\theta}} \, .
\end{equation}
Define an operation
$\diamond \colon \Lambda (\calA) \otimes_{\smash{\Mhat_\bbK}} \Lambda^\sd (\calA) \to \Lambda^\sd (\calA)$ by
\begin{equation}
    \label{eq-lambda-sd-mult}
    \lambda_\alpha \diamond \lambda^\sd_{\smash{\theta}} =
    \bbL^{-\chi^\sd (\alpha, \theta)} \cdot \lambda^\sd_{\smash{\bar{\alpha} + \theta}}
\end{equation}
for all $\alpha \in C^\circ (\calA)$ and $\theta \in C^\sd (\calA)$,
where $\chi^\sd (\alpha, \theta)$ is the self-dual Euler form
defined in~\cref{para-euler-sd}.
The associativity follows from the relation~\cref{eq-chi-sd-assoc}.

Similarly, we also have an isomorphism
\begin{equation}
    \label{eq-lambda-sd-iso-sf-disc}
    \Lambda^\sd (\calA) \simeq \SFhat (C^\sd (\calA))
\end{equation}
of $\Mhat_\bbK$-modules,
where $C^\sd (\calA)$ is regarded as a discrete space.

\paragraph{Explicit coefficients.}

We can describe the multiplication in
$\Lambda (\calA)$ and its action on $\Lambda^\sd (\calA)$ more explicitly.
For $\alpha_1, \dotsc, \alpha_n \in C^\circ (\calA)$ and $\rho \in C^\sd (\calA)$, write
\begin{align}
    \label{eq-def-multi-chi}
    \chi (\alpha_1, \dotsc, \alpha_n) & =
    \sum_{1 \leq i < j \leq n} \chi (\alpha_i, \alpha_j) \ , \\
    \label{eq-def-multi-chi-sd}
    \chi^\sd (\alpha_1, \dotsc, \alpha_n, \rho) & =
    \sum_{1 \leq i < j \leq n} {}
    \bigl( \chi (\alpha_i, \alpha_j) + \chi (\alpha_i, \alpha_j^\vee) \bigr) +
    \sum_{i=1}^n \chi^\sd (\alpha_i, \rho) \ .
\end{align}
Then we have
\begin{align}
    \label{eq-lem-multi-chi}
    \lambda_{\alpha_1} * \cdots * \lambda_{\alpha_n}
    & = \bbL^{-\chi (\alpha_1, \dotsc, \alpha_n)} \cdot
    \lambda_{\alpha_1 + \cdots + \alpha_n} \ , \\
    \label{eq-lem-multi-chi-sd}
    \lambda_{\alpha_1} \diamond \cdots \diamond \lambda_{\alpha_n} \diamond \lambda^\sd_\rho
    & = \bbL^{-\chi^\sd (\alpha_1, \dotsc, \alpha_n, \rho)} \cdot
    \lambda^\sd_{\bar{\alpha}_1 + \cdots + \bar{\alpha}_n + \rho} \ ,
\end{align}
which can be verified using the bilinearity of~$\chi$
and the relations~\cref{eq-def-chi-sd}--\cref{eq-chi-sd-assoc}
of $\chi^\sd$.

\paragraph{The integration maps.}
\label{para-int-psi}

We now define the motivic integration maps
\begin{alignat*}{2}
    \Psi \colon && \SFhat (\calM) 
    & \longrightarrow \Lambda (\calA) \ , \\
    \Psi^\sd \colon && \SFhat (\calM^\sd)
    & \longrightarrow \Lambda^\sd (\calA)
\end{alignat*}
by setting
\begin{align}
    \Psi (f) & =
    \sum_{\alpha \in C^\circ (\calA)}
    \lambda_\alpha \cdot \int_{\calM_\alpha} f |_{\calM_\alpha} \ , \\
    \Psi^\sd (h) & =
    \sum_{\theta \in C^\sd (\calA) \vphantom{^0}}
    \lambda^\sd_{\smash{\theta}} \cdot \int_{\calM^\sd_\theta} h |_{\calM^\sd_\theta} 
\end{align}
for $f \in \SFhat (\calM)$ and $h \in \SFhat (\calM^\sd)$,
where the integrals are defined in \cref{para-mot-int}.

Using the isomorphisms
\cref{eq-lambda-iso-sf-disc,eq-lambda-sd-iso-sf-disc},
these maps can also be interpreted as pushing forward
along the natural maps
$\calM \to C^\circ (\calA)$ and $\calM^\sd \to C^\sd (\calA)$.

\begin{theorem}
    \label{thm-mot-int-homo}
    In the setting of \cref{para-sit-mot-int},
    let $f, g \in \SFhat (\calM)$
    and $h \in \SFhat (\calM^\sd)$.
    Then
    \begin{align}
        \label{eq-psi-compat-star}
        \Psi (f * g) & =
        \Psi (f) *
        \Psi (g) \ , \\
        \label{eq-psi-compat-diamond}
        \Psi^\sd (f \diamond h) & =
        \Psi (f) \diamond
        \Psi^\sd (h) \ .
    \end{align}
\end{theorem}

\begin{proof}
    Equation~\cref{eq-psi-compat-star}
    is essentially~\cite[Theorem~6.4]{Joyce2007II},
    but we provide a more direct proof here,
    as our setting is slightly different,
    and the same strategy will help us prove
    \cref{eq-psi-compat-diamond} as well.
    
    To prove~\cref{eq-psi-compat-star},
    without loss of generality, we may assume that
    $f$ and $g$ are supported in $\calM_{\alpha_1}$ and $\calM_{\alpha_2}$\,,
    respectively,
    for some $\alpha_1, \alpha_2 \in C^\circ (\calA)$.
    By \tagref{Ext}, the morphism
    $\pi_1 \times \pi_2 \colon
    \calM^\pn{2}_{\alpha_1, \alpha_2} \to \calM_{\alpha_1} \times \calM_{\alpha_2}$
    is a $2$-vector bundle of rank $-\chi (\alpha_1, \alpha_2)$.
    Therefore, by \cref{thm-motive-2vb}, we have
    \begin{align*}
        \int_{\calM_{\alpha_1 + \alpha_2}} f * g
        & = \int_{\calM^\pn{2}_{\alpha_1, \alpha_2}}
        (\pi_1 \times \pi_2)^* (f \boxtimes g) \\
        & = \bbL^{-\chi (\alpha_1, \alpha_2)} \cdot
        \int_{\calM_{\alpha_1} \times \calM_{\alpha_2}} ( f \boxtimes g ) \\
        & = \bbL^{-\chi (\alpha_1, \alpha_2)} \cdot
        \int_{\calM_{\alpha_1}} f \cdot \int_{\calM_{\alpha_2}} g \ ,
    \end{align*}
    which implies~\cref{eq-psi-compat-star}.
    
    The proof of~\cref{eq-psi-compat-diamond} is analogous,
    where we use \cref{thm-sdext-strong} in place of \tagref{Ext}.
\end{proof}

\subsection{Motivic integration, Lie algebras and twisted modules}
\label{sect-motivic-int-2}

\paragraph{The setting.}
\label{para-sit-mot-int-lie}

We continue to develop other versions of the structures
presented in \cref{sect-motivic-int-1},
involving Lie algebras and twisted modules,
as in \cref{para-tw-mod}.

These versions will be helpful in defining
numerical enumerative invariants in \cref{sect-num-inv} below.
The idea is that we wish to define numerical enumerative invariants
as Euler characteristics of certain motivic integrals.
For the Euler characteristic to be well-defined, that is, not infinite,
these integrals should live in the subspace $\Mhat_\bbK^\circ \subset \Mhat_\bbK$.
However, the space of stack functions with this property
is not closed under the Hall algebra product or the Hall module action,
but only under the Hall Lie bracket and the twisted module action.
The integration map in this setting is described in 
\cref{thm-mot-int-lie-homo} below.

Throughout,
we assume that we are in the situation of \cref{para-sit-mot-int},
with the conditions \tagref{SdMod} and \tagref{SdExt}.

We also define enlarged versions of the rings of motives,
\begin{equation}
    \prestar \Mhat_\bbK^\circ = \Mhat_\bbK^\circ [\bbL^{1/2}] \ , \qquad
    \prestar \Mhat_\bbK = \Mhat_\bbK [\bbL^{1/2}] \ ,
\end{equation}
and we write
\begin{equation}
    \prestar \Lambda (\calA) =
    \Lambda (\calA) \underset{\phantom{^0} \Mhat_\bbK}{\otimes} 
    \prestar \Mhat_\bbK \ , \qquad
    \prestar \Lambda^\sd (\calA) =
    \Lambda^\sd (\calA) \underset{\phantom{^0} \Mhat_\bbK}{\otimes} 
    \prestar \Mhat_\bbK \ .
\end{equation}

\paragraph{The Lie algebra version.}

We describe a Lie algebra version
of the structures discussed in \cref{sect-motivic-int-1}.

For each $\alpha \in C^\circ (\calA)$, define an element
$\tilde{\lambda}_\alpha \in \prestar \Lambda (\calA)$ by
\begin{equation}
    \label{eq-def-lambda-tilde}
    \tilde{\lambda}_\alpha =
    \frac{\bbL^{-\chi (\alpha, \alpha) / 2}}
    {\bbL^{1/2} - \bbL^{-1/2}} \cdot
    \lambda_\alpha \ .
\end{equation}
Also, for $\alpha_1, \alpha_2 \in C^\circ (\calA)$, write
\begin{equation}
    \label{eq-def-chi-bar}
    \bar{\chi} (\alpha_1, \alpha_2) =
    \chi (\alpha_2, \alpha_1) - \chi (\alpha_1, \alpha_2) \ .
\end{equation}
One can verify using
\cref{eq-lambda-mult,eq-def-lambda-tilde} that
\begin{equation}
    \label{eq-lie-lambda}
    [\tilde{\lambda}_{\alpha_1} \, , \, \tilde{\lambda}_{\alpha_2}] =
    \frac{\bbL^{\bar{\chi} (\alpha_1, \alpha_2)/2} - \bbL^{-\bar{\chi} (\alpha_1, \alpha_2)/2}}
    {\bbL^{1/2} - \bbL^{-1/2}} \cdot 
    \tilde{\lambda}_{\alpha_1 + \alpha_2} \ .
\end{equation}

Note that the coefficient in~\cref{eq-lie-lambda}
has no poles when $\bbL = 1$. As a consequence,
the $\prestar \Mhat_\bbK^\circ$-submodule
\[
    \prestar \Lambda^\circ (\calA) \subset \prestar \Lambda (\calA) \ ,
\]
spanned over $\prestar \Mhat_\bbK^\circ$ by the elements $\tilde{\lambda}_\alpha$
for $\alpha \in C^\circ (\calA)$,
is a Lie subalgebra,
although it need not be an associative subalgebra.

The Lie algebras $\prestar \Lambda (\calA)$ and $\prestar \Lambda^\circ (\calA)$
are equipped with involutions,
in the sense of \cref{para-inv-lie-alg},
given by $\lambda_\alpha \mapsto \lambda_{\alpha^\vee}$ and $\tilde{\lambda}_\alpha \mapsto \tilde{\lambda}_{\alpha^\vee}$,
respectively. These two involutions are compatible,
since $\chi (\alpha, \alpha) = \chi (\alpha^\vee, \alpha^\vee)$
by~\cref{eq-chi-invol}.

\paragraph{The twisted module.}

We also consider the operation
\begin{align*}
    \numberthis
    {\heart} \colon \Lambda (\calA) \times \Lambda^\sd (\calA)
    & \longrightarrow \Lambda^\sd (\calA), \\
    (x, m) & \longmapsto x \diamond m - x^\vee \diamond m \ ,
\end{align*}
exhibiting $\Lambda^\sd (\calA)$
as a twisted module over the involutive Lie algebra $\Lambda (\calA)$,
in the sense of \cref{para-tw-mod}.
This also equips $\prestar \Lambda^\sd (\calA)$
with a twisted module structure over $\prestar \Lambda (\calA)$.

For each $\theta \in C^\sd (\calA)$, define an element
$\tilde{\lambda}^\sd_\theta \in \prestar \Lambda^\sd (\calA)$ by
\begin{equation}
    \label{eq-def-lambda-tilde-sd}
    \tilde{\lambda}^\sd_\theta =
    \bbL^{-\chi^\sd (j (\theta), 0) / 2} \cdot 
    \lambda^\sd_\theta \ ,
\end{equation}
with $j$ as in~\cref{eq-j-csd}.
For $\alpha \in C^\circ (\calA)$ and $\theta \in C^\sd (\calA)$, write
\begin{equation}
    \label{eq-def-chi-bar-sd}
    \bar{\chi}^\sd (\alpha, \theta) =
    \chi^\sd (\alpha^\vee, \theta) - \chi^\sd (\alpha, \theta) \ .
\end{equation}
One can verify using \cref{eq-lambda-sd-mult,eq-def-lambda-tilde-sd} that
\begin{equation}
    \label{eq-heart-lambda}
    \tilde{\lambda}_\alpha \heart
    \tilde{\lambda}^\sd_\theta =
    \frac{\bbL^{\bar{\chi}^\sd (\alpha, \theta)/2}
    - \bbL^{-\bar{\chi}^\sd (\alpha, \theta)/2}}
    {\bbL^{1/2} - \bbL^{-1/2}} \cdot 
    \tilde{\lambda}^\sd_{\bar{\alpha} + \theta} \ .
\end{equation}
Since the coefficient has no poles when $\bbL = 1$,
the $\prestar \Mhat_\bbK^\circ$-submodule
\[
    \prestar \Lambda^{\sd, \> \circ} (\calA) \subset \prestar \Lambda^\sd (\calA) \ ,
\]
spanned over $\prestar \Mhat_\bbK^\circ$ by the elements $\tilde{\lambda}^\sd_\theta$
for $\theta \in C^\sd (\calA)$,
is a twisted module for the Lie algebra $\prestar \Lambda^\circ (\calA)$
with the~$\heart$ action,
although it is not necessarily closed under
the~$\diamond$ action by~$\prestar \Lambda^\circ (\calA)$.

\paragraph{The numeric Lie algebra.}
\label{para-num-lie}

Let $\Mtilde_\bbK$ be the ring of motives as in \cref{def-motive-varieties}.
We define a Lie algebra $\Omega (\calA)$ over $\Mtilde_\bbK$,
as the free $\Mtilde_\bbK$-module
\begin{equation}
    \Omega (\calA) =
    \bigoplus_{\alpha \in C^\circ (\calA)}
    \Mtilde_\bbK \cdot \omega_\alpha \ ,
\end{equation}
equipped with the Lie bracket given by
\begin{equation}
    \label{eq-lie-omega}
    [ \omega_{\alpha_1} \, , \, \omega_{\alpha_2} ] =
    \bar{\chi} (\alpha_1, \alpha_2) \cdot \omega_{\alpha_1 + \alpha_2}
\end{equation}
for all $\alpha_1, \alpha_2 \in C^\circ (\calA)$,
where $\bar{\chi}$ is given by~\cref{eq-def-chi-bar}.
One can verify that this defines a Lie algebra structure on $\Omega (\calA)$.
The Lie algebra $\Omega (\calA)$
is equipped with an involution,
in the sense of \cref{para-inv-lie-alg},
given by $\omega_\alpha \mapsto \omega_{\alpha^\vee}$ for all $\alpha \in C^\circ (\calA)$.

Define a map of Lie algebras over $\Mhat_\bbK^\circ$,
\begin{align}
    \label{eq-def-big-xi}
    \Xi \colon \quad 
    \prestar \Lambda^\circ (\calA)
    & \longrightarrow \Omega (\calA) \ , \\
    \sum_{\alpha \in C^\circ (\calA)}
    a_\alpha \cdot \tilde{\lambda}_\alpha 
    & \longmapsto 
    \sum_{\alpha \in C^\circ (\calA)}
    a_\alpha \cdot \omega_\alpha \ , \notag
\end{align}
where each $a_\alpha \in \Mhat_\bbK^\circ$.
That $\Xi$ respects Lie brackets
can be seen by comparing~\cref{eq-lie-lambda} with~\cref{eq-lie-omega}.

\paragraph{The numeric twisted module.}
\label{para-num-tw-mod}

Define a free $\Mtilde_\bbK$-module
\begin{equation}
    \Omega^\sd (\calA) =
    \bigoplus_{\theta \in C^\sd (\calA) \vphantom{^0}}
    \Mtilde_\bbK \cdot \omega^\sd_\theta \ ,
\end{equation}
and define an action of $\Omega (\calA)$
on $\Omega^\sd (\calA)$ by
\begin{equation}
    \label{eq-heart-omega}
    \omega_\alpha \heart \omega^\sd_\theta =
    \bar{\chi}^\sd (\alpha, \theta) \cdot \omega^\sd_{\bar{\alpha} + \theta}
\end{equation}
for all $\alpha \in C^\circ (\calA)$ and $\theta \in C^\sd (\calA)$,
where $\bar{\chi}^\sd$ is given by~\cref{eq-def-chi-bar-sd}.
One can verify that this defines a twisted module
of the involutive Lie algebra $\Omega (\calA)$.

Define a map of $\Mhat_\bbK^\circ$-modules,
\begin{align}
    \label{eq-def-big-xi-sd}
    \Xi^\sd \colon \quad 
    \prestar \Lambda^{\sd, \> \circ} (\calA)
    & \longrightarrow \Omega^\sd (\calA) \ , \\
    \sum_{\theta \in C^\sd (\calA) \vphantom{^0}}
    a_\theta \cdot \tilde{\lambda}^\sd_\theta 
    & \longmapsto 
    \sum_{\theta \in C^\sd (\calA) \vphantom{^0}}
    a_\theta \cdot \omega^\sd_\theta \ , \notag
\end{align}
where each $a_\theta \in \Mhat_\bbK^\circ$.
The maps $\Xi$ and $\Xi^\sd$ respect the~$\heart$ action,
as can be seen by comparing~\cref{eq-heart-lambda} with~\cref{eq-heart-omega}.

\paragraph{Explicit coefficients.}

We also describe iterated Lie brackets in $\Omega (\calA)$
and their action on $\Omega^\sd (\calA)$ explicitly.
For classes $\alpha_1, \dotsc, \alpha_n \in C^\circ (\calA)$, write
\begin{equation}
    \label{eq-def-chi-tilde}
    \tilde{\chi} (\alpha_1, \dotsc, \alpha_n) =
    \prod_{j=2}^n \sum_{i=1}^{j-1} \bar{\chi} (\alpha_i, \alpha_j) \ .
\end{equation}
Then, in $\Omega (\calA)$, we have
\begin{equation}
    \label{eq-lem-chi-tilde}
    [ [ \dotsc [ \omega_{\alpha_1} \, , \,
    \omega_{\alpha_2} ], \dotsc], \, \omega_{\alpha_n} ]
    = \tilde{\chi} (\alpha_1, \dotsc, \alpha_n) \cdot
    \omega_{\alpha_1 + \cdots + \alpha_n} \ .
\end{equation}

Similarly, for classes $\alpha_{1,1}, \dotsc, \alpha_{1,m_1}; \dotsc;
    \alpha_{n,1}, \dotsc, \alpha_{n,m_n} \in C^\circ (\calA)$
and $\rho \in C^\sd (\calA)$, write
\begin{multline}
    \label{eq-def-chi-tilde-sd}
    \tilde{\chi}^\sd (\alpha_{1,1}, \dotsc, \alpha_{1,m_1}; \dotsc;
        \alpha_{n,1}, \dotsc, \alpha_{n,m_n}; \rho) = 
    \prod_{i=1}^n \tilde{\chi} 
        (\bar{\alpha}_{i,1} , \dotsc, \bar{\alpha}_{i,m_i}) \cdot {} \\*
    \prod_{i=1}^{n} \bar{\chi}^\sd \biggl(
        \alpha_{i,1} + \cdots + \alpha_{i,m_i} \, , \ 
        \sum_{j=i+1}^n {}
        (\bar{\alpha}_{j,1} + \cdots + \bar{\alpha}_{j,m_j})
        + \rho
    \biggr) \ .
\end{multline}
Then, in $\Omega^\sd (\calA)$, we have
\begin{multline}
    \label{eq-lem-chi-tilde-sd}
    \bigl[ \bigl[ \omega_{\alpha_{1,1}}, \dotsc \bigr] ,
    \omega_{\alpha_{1,m_1}} \bigr] \heart \cdots \heart
    \bigl[ \bigl[ \omega_{\alpha_{n,1}}, \dotsc \bigr] ,
    \omega_{\alpha_{n,m_n}} \bigr] \heart 
    \omega^\sd_{\rho} = \\
    \tilde{\chi}^\sd (\alpha_{1,1}, \dotsc, \alpha_{1,m_1}; \dotsc;
        \alpha_{n,1}, \dotsc, \alpha_{n,m_n}; \rho) \cdot
    \omega^\sd_\theta \ ,
\end{multline}
where $\theta = (\bar{\alpha}_{1,1} + \cdots + \bar{\alpha}_{1,m_1}) + \cdots +
(\bar{\alpha}_{n,1} + \cdots + \bar{\alpha}_{n,m_n}) + \rho$.

\paragraph{The integration maps.}
\label{para-int-psi-lie}

Recall the integration maps $\Psi, \Psi^\sd$
from \cref{para-int-psi}. Consider the subspaces
\begin{alignat}{3}
    \SFhat^\circ (\calM) & = {}
    & \Psi^{-1} & \bigl( \Lambda^\circ (\calA) \bigr)
    && \subset \SFhat (\calM) \ ,
    \\
    \SFhat^\circ (\calM^\sd) & = {}
    & (\Psi^\sd)^{-1} & \bigl( \Lambda^{\sd, \> \circ} (\calA) \bigr)
    && \subset \SFhat (\calM^\sd) \ ,
\end{alignat}
where $\Lambda^\circ (\calA) \subset \Lambda (\calA)$ is the
$\Mhat_\bbK^\circ$-subalgebra spanned by
the elements $\lambda_\alpha$ for $\alpha \in C^\circ (\calA)$,
and $\Lambda^{\sd, \> \circ} (\calA) \subset \Lambda^\sd (\calA)$ is the
$\Mhat_\bbK^\circ$-submodule spanned by
the elements $\lambda^\sd_\theta$ for $\theta \in C^\sd (\calA)$.

It then follows from \cref{thm-mot-int-homo}
that $\SFhat^\circ (\calM)$ is an involutive Lie subalgebra
of the Hall Lie algebra $\SFhat (\calM)$ as in \cref{para-hall-lie-alg},
and that $\SFhat^\circ (\calM^\sd)$ is a twisted module
for $\SFhat^\circ (\calM)$ under the action~$\heart$ in \cref{para-hall-tw-mod}.
Moreover, \cref{thm-mot-int-homo} implies the following.

\begin{theorem}
    \label{thm-mot-int-lie-homo}
    In the situation of \cref{para-sit-mot-int-lie},
    consider the integration maps
    \begin{alignat}{2}
        \label{eq-mot-int-lie-homo}
        \Xi \circ \Psi \colon &&
        \SFhat^\circ (\calM) & \longrightarrow
        \Omega (\calA) \ , \\
        \label{eq-mot-int-tw-mod-homo}
        \Xi^\sd \circ \Psi^\sd \colon &&
        \SFhat^\circ (\calM^\sd) & \longrightarrow
        \Omega^\sd (\calA) \ .
    \end{alignat}
    Then \cref{eq-mot-int-lie-homo} is a
    morphism of Lie algebras,
    and \cref{eq-mot-int-tw-mod-homo}
    is compatible with the twisted module structures on both sides,
    in the sense that the diagram
    \[
        \begin{tikzcd}[column sep=6em, row sep=1.5em]
            \SFhat^\circ (\calM) \underset{\bbQ}{\otimes}
            \SFhat^\circ (\calM^\sd)
            \ar[r, "(\Xi \circ \Psi) \otimes (\Xi^\sd \circ \Psi^\sd)"]
            \ar[d, "\heart"']
            & \Omega (\calA) \underset{\bbQ}{\otimes}
            \Omega^\sd (\calA)
            \ar[d, "\heart"]
            \\
            \SFhat^\circ (\calM^\sd)
            \ar[r, "\Xi^\sd \circ \Psi^\sd"]
            & \Omega^\sd (\calA)
        \end{tikzcd} \hspace{1.5em}
    \]
    commutes.
    \QED
\end{theorem}

%% file: inv.tex
\label{sect-invariants}

In this \lcnamecref{sect-invariants},
we state our main constructions of motivic enumerative invariants,
including the following different versions:

\begin{itemize}
    \item The \emph{stack function} version 
        $\delta_\alpha (\tau)$, $\epsilon_\alpha (\tau) \in \SF (\calM_\alpha)$
        and $\delta^\sd_\theta (\tau)$, $\epsilon^\sd_\theta (\tau) \in \SF (\calM^\sd_\theta)$.

    \item The \emph{motivic} version 
        $\upI_\alpha (\tau)$, $\upI^\sd_{\theta} (\tau)$, 
        $\upJ_\alpha (\tau)$, $\upJ^\sd_{\theta} (\tau) \in \Mhat_{\bbK}$, 
        obtained from the stack function invariants
        via motivic integration.

    \item The \emph{numeric} version 
        $\chiJ_\alpha (\tau)$, $\chiJ^\sd_{\theta} (\tau)$, 
        $\DT^\nai_\alpha (\tau)$, $\DT^\sdnai_{\theta} (\tau) \in \bbQ$, 
        obtained from the motivic invariants
        by taking the Euler characteristic.
\end{itemize}
We also prove \emph{wall-crossing formulae}
for each version of invariants,
under a change of the stability condition.

\subsection{Stack function invariants}
\label{sect-delta-epsilon}

In the following, we use the notion of \emph{stack functions}
due to Joyce~\cite{Joyce2007Stack}.
See \cref{sect-motive} for an introduction to stack functions.

\paragraph{The stack functions $\delta_\alpha (\tau)$, $\epsilon_\alpha (\tau)$.}

Let us recall Joyce's construction
\cite[Definition~8.1]{Joyce2007III}
of the stack functions
$\delta_\alpha (\tau)$, $\epsilon_\alpha (\tau)$,
representing motives that count
$\tau$-semi\-stable objects of a fixed class $\alpha$.

Let $\calA$ be $\bbK$-linear exact category,
with a categorical moduli stack $\+{\calM}$, 
satisfying \tagref{Mod} and~\tagref{Fin1}.
Let $\tau$ be a weak stability condition on $\calA$,
satisfying~\tagref{Stab}.

For each $\alpha \in C (\calA)$, define an element
\begin{equation}
    \delta_\alpha (\tau) \in \SF (\calM)
\end{equation}
to be the characteristic function of the open substack
$\Mss_{\alpha} (\tau) \subset \calM$,
that is, $\delta_\alpha (\tau) = [(\Mss_{\alpha} (\tau), \, i)]$,
where $i \colon \Mss_{\alpha} (\tau) \hookrightarrow \calM$ is the inclusion.

Define an element
$\epsilon_\alpha (\tau) \in \SF (\calM)$ by
\begin{align}
    \label{eq-def-epsilon}
    \epsilon_\alpha (\tau) & =
    \sum_{ \leftsubstack[6em]{
        \\[-1.5ex]
        & n > 0; \, \alpha_1, \dotsc, \alpha_n \in C (\calA) \colon \\[-.5ex]
        & \alpha = \alpha_1 + \cdots + \alpha_n \, , \\[-.5ex]
        & \tau (\alpha_1) = \cdots = \tau (\alpha_n)
    } } \hspace{.5em}
    \frac{(-1)^{n-1}}{n} \cdot
    \delta_{\alpha_1} (\tau) * \cdots *
    \delta_{\alpha_n} (\tau) \ , 
\end{align}
as in \cite[Definition~8.1]{Joyce2007III},
where $*$ is the multiplication in the motivic Hall algebra.
The sum in \cref{eq-def-epsilon} has only finitely many non-zero terms
by \tagref{Stab2} and \cite[Proposition~4.9]{Joyce2007III},
whose proof can be adapted to exact categories.
Formally inverting \cref{eq-def-epsilon},
as in \cite[Theorem~8.2]{Joyce2007III},
we have
\begin{align}
    \label{eq-def-epsilon-inverse}
    \delta_\alpha (\tau) & =
    \sum_{ \leftsubstack[6em]{
        \\[-1.5ex]
        & n > 0; \, \alpha_1, \dotsc, \alpha_n \in C (\calA) \colon \\[-.5ex]
        & \alpha = \alpha_1 + \cdots + \alpha_n \, , \\[-.5ex]
        & \tau (\alpha_1) = \cdots = \tau (\alpha_n)
    } } \hspace{.5em}
    \frac{1}{n!} \cdot
    \epsilon_{\alpha_1} (\tau) * \cdots *
    \epsilon_{\alpha_n} (\tau) \ , 
\end{align}
where the sum has only finitely many non-zero terms.

\paragraph{}

In more compact terms, for any $t \in T$,
if we define $\delta (\tau; t), \epsilon (\tau; t) \in \LSF (\calM)$ by
\begin{align*}
    \delta (\tau; t) =
    \delta_0 +
    \sum_{ \leftsubstack[4em]{
        & \alpha \in C (\calA) \colon \\[-.5ex]
        & \tau (\alpha) = t
    } }
    \delta_\alpha (\tau) \ , \hspace{2em}
    \epsilon (\tau; t) =
    \sum_{ \leftsubstack[4em]{
        & \alpha \in C (\calA) \colon \\[-.5ex]
        & \tau (\alpha) = t
    } }
    \epsilon_\alpha (\tau) \ ,
\end{align*}
where $\delta_0$ is as in \cref{def-hall-alg}, then
\begin{align}
    \label{eq-epsilon-delta-exp-1}
    \epsilon (\tau; t) & = \log \delta (\tau; t) \ , \\
    \label{eq-epsilon-delta-exp-2}
    \delta (\tau; t) & = \exp \epsilon (\tau; t) \ ,
\end{align}
where `$\log$' and `$\exp$' are interpreted as
formal power series using the multiplication~$*$.

\begin{remark}
    \label{rem-delta-lsf}
    In the situation above,
    if we drop the assumption \tagref{Stab2},
    then there are still well-defined elements
    $\delta_\alpha (\tau) \in \LSF (\calM)$,
    defined in the same way as above.
    However, the elements $\epsilon_\alpha (\tau)$
    may not be well-defined, even in $\LSF (\calM)$,
    as the sum in \cref{eq-def-epsilon} may not converge
    in the sense of \cref{def-lsf};
    a counterexample is given by trivial stability
    on the category of coherent sheaves on a curve.
\end{remark}

\paragraph{The stack functions $\delta^\sd_\theta (\tau)$, $\epsilon^\sd_\theta (\tau)$.}

We now define self-dual versions of the
$\delta$ and~$\epsilon$ stack functions,
which represent motives that count
$\tau$-semi\-stable self-dual objects of a fixed class $\theta$.

Let $\calA$ be a self-dual $\bbK$-linear exact category,
with a self-dual moduli stack $\+{\calM}$,
satisfying \tagref{SdMod} and~\tagref{Fin1}.
Let $\tau$ be a self-dual weak stability condition on $\calA$,
satisfying~\tagref{Stab}.

For each $\theta \in C^\sd (\calA)$,
define an element
\begin{equation}
    \delta^\sd_\theta (\tau) \in \SF (\calM^\sd)
\end{equation}
to be the characteristic function
of the open substack $\smash{\Msdss_{\theta} (\tau)} \subset \calM^\sd$.

Define an element
$\epsilon^\sd_\theta (\tau) \in \SF (\calM^\sd)$ by
\begin{align}
    \label{eq-def-epsilon-sd}
    \epsilon^\sd_\theta (\tau) & =
    \sum_{ \leftsubstack[6em]{
        \\[-1.5ex]
        & n \geq 0; \, \alpha_1, \dotsc, \alpha_n \in C (\calA), \,
        \rho \in C^\sd (\calA) \colon \\[-.5ex]
        & \theta = \bar{\alpha}_1 + \cdots + \bar{\alpha}_n + \rho \\[-.5ex]
        & \tau (\alpha_1) = \cdots = \tau (\alpha_n) = 0
    } } {}
    \binom{-1/2}{n} \cdot
    \delta_{\alpha_1} (\tau) \diamond \cdots \diamond
    \delta_{\alpha_n} (\tau) \diamond
    \delta^\sd_{\rho} (\tau) \ ,
\end{align}
where $\diamond$ is the multiplication in the motivic Hall module,
and there are only finitely many non-zero terms
by the proof of \cite[Proposition~4.9]{Joyce2007III}.
Formally inverting \cref{eq-def-epsilon-sd}, we have
\begin{align}
    \label{eq-def-epsilon-sd-inverse}
    \delta^\sd_\theta (\tau) & =
    \sum_{ \leftsubstack[6em]{
        \\[-1.5ex]
        & n \geq 0; \, \alpha_1, \dotsc, \alpha_n \in C (\calA), \,
        \rho \in C^\sd (\calA) \colon \\[-.5ex]
        & \theta = \bar{\alpha}_1 + \cdots + \bar{\alpha}_n + \rho \\[-.5ex]
        & \tau (\alpha_1) = \cdots = \tau (\alpha_n) = 0
    } } {}
    \frac{1}{2^n \, n!} \cdot
    \epsilon_{\alpha_1} (\tau) \diamond \cdots \diamond
    \epsilon_{\alpha_n} (\tau) \diamond
    \epsilon^\sd_{\rho} (\tau) \ ,
\end{align}
where only finitely many terms are non-zero.
Indeed, \cref{eq-def-epsilon-sd-inverse} follows from the following.

\paragraph{}

Define elements $\delta^\sd (\tau), \epsilon^\sd (\tau) \in \LSF (\calM^\sd)$ by
\begin{align*}
    \delta^\sd (\tau) =
    \sum_{ \leftsubstack[4em]{
        & \theta \in C^\sd (\calA) \vphantom{^0}
    } }
    \delta^\sd_\theta (\tau) \ , \hspace{2em}
    \epsilon^\sd (\tau) =
    \sum_{ \leftsubstack[4em]{
        & \theta \in C^\sd (\calA) \vphantom{^0}
    } }
    \epsilon^\sd_\theta (\tau) \ ,
\end{align*}
then \crefrange{eq-def-epsilon-sd}{eq-def-epsilon-sd-inverse}
translate to
\begin{align}
    \label{eq-def-epsilon-sd-compact}
    \epsilon^\sd (\tau) & =
    \delta (\tau; 0)^{-1/2} \diamond \delta^\sd (\tau) \ , \\
    \label{eq-def-epsilon-sd-inv-compact}
    \delta^\sd (\tau) & =
    \exp \Bigl( \frac{1}{2} \, \epsilon (\tau; 0) \Bigr) \diamond \epsilon^\sd (\tau) \ ,
\end{align}
where $(-)^{-1/2}$ and `$\exp$' are interpreted 
as formal power series using the multiplication~$*$.
Now, using \crefrange{eq-epsilon-delta-exp-1}{eq-epsilon-delta-exp-2},
it is straightforward to see that
\cref{eq-def-epsilon-sd-compact,eq-def-epsilon-sd-inv-compact}
are equivalent.

\begin{remark}
    The choice of the coefficients $\binom{-1/2}{n}$
    in the definition of
    $\epsilon^\sd_\theta (\tau)$ in~\cref{eq-def-epsilon-sd}
    is justified by the fact that
    it is the unique choice 
    only depending on~$n$, such that
    the \emph{no-pole theorem},
    \cref{thm-no-pole-sd} below, holds.
\end{remark}

\begin{remark}
    As in \cref{rem-delta-lsf},
    if we drop the assumption \tagref{Stab2},
    then there are still well-defined elements
    $\delta^\sd_\theta (\tau) \in \LSF (\calM^\sd)$,
    but the elements $\epsilon^\sd_\theta (\tau)$
    may not be well-defined in $\LSF (\calM^\sd)$.
\end{remark}

\subsection{Wall-crossing: Preparation}
\label{sect-wcf}

\paragraph{}
\label{para-wcf-intro}

We now study the change of the stack function invariants
$\delta_\alpha (\tau)$, $\epsilon_\alpha (\tau)$,
$\delta^\sd_\theta (\tau)$, $\epsilon^\sd_\theta (\tau)$
when we alter the weak stability condition $\tau$.
We express the change in terms of \emph{wall-crossing formulae}.
This was done for ordinary invariants in
Joyce~\cite[\S5]{Joyce2008IV},
and we now give a parallel theory for self-dual invariants.

Throughout this section,
$\calA$ will be a self-dual $\bbK$-linear exact category,
with a self-dual moduli stack $\+{\calM}$,
satisfying \tagref{SdMod} and~\tagref{Fin},
with the understanding that for results that
do not involve the self-dual structure of~$\calA$,
the category $\calA$ need not be self-dual,
and only needs to satisfy \tagref{Mod} and~\tagref{Fin}.

\begin{definition}
    \label{def-dominate}
    Let $\tau$ and $\tilde{\tau}$ be weak stability conditions on $\calA$.
    As in \cite[Definition~3.16]{Joyce2008IV},
    we say that $\tilde{\tau}$ \emph{dominates} $\tau$,
    if $\tau (\alpha) \leq \tau (\beta)$
    implies $\tilde{\tau} (\alpha) \leq \tilde{\tau} (\beta)$
    for any $\alpha, \beta \in C (\calA)$.
\end{definition}

\paragraph{}
\label[theorem]{thm-wcf-dom}

Here is the first version of wall-crossing formulae,
which is the special case when one stability condition dominates another.
We assume the setting of \cref{para-wcf-intro}.

\begin{theorem*}
    Let $\tau, \tilde{\tau}$ be self-dual weak stability conditions on $\calA$
    satisfying \tagref{Stab13},
    and suppose that $\tilde{\tau}$ dominates $\tau$.
    
    Then for any $\alpha \in C (\calA)$
    and $\theta \in C^\sd (\calA)$, we have
    \begin{align}
        \label{eq-wcf-dom}
        \delta_\alpha (\tilde{\tau}) & =
        \sum_{ \leftsubstack[7em]{
            & n > 0; \, \alpha_1, \dotsc, \alpha_n \in C (\calA) \colon \\[-.5ex]
            & \alpha = \alpha_1 + \cdots + \alpha_n \, , \\[-.5ex]
            & \tilde{\tau} (\alpha_1) = \cdots = \tilde{\tau} (\alpha_n), \\[-.5ex]
            & \tau (\alpha_1) > \cdots > \tau (\alpha_n)
        } }
        \delta_{\alpha_1} (\tau) * \cdots *
        \delta_{\alpha_n} (\tau) \ , \\[1ex]
        \label{eq-wcf-dom-sd}
        \delta^\sd_\theta (\tilde{\tau}) & =
        \sum_{ \leftsubstack[7em]{
            & n \geq 0; \, \alpha_1, \dotsc, \alpha_n \in C (\calA), \,
            \rho \in C^\sd (\calA) \vphantom{^0} \colon \\[-.5ex]
            & \theta = \bar{\alpha}_1 + \cdots + \bar{\alpha}_n + \rho, \\[-.5ex]
            & \tilde{\tau} (\alpha_1) = \cdots = \tilde{\tau} (\alpha_n) = 0, \\[-.5ex]
            & \tau (\alpha_1) > \cdots > \tau (\alpha_n) > 0
        } }
        \delta_{\alpha_1} (\tau) \diamond \cdots \diamond
        \delta_{\alpha_n} (\tau) \diamond
        \delta^\sd_{\rho} (\tau)
    \end{align}
    in $\LSF (\calM)$ and $\LSF (\calM^\sd)$, respectively,
    where the sums may be infinite,
    but they converge in the sense of \cref{def-lsf}.
\end{theorem*}

\begin{proof}
    Equation \cref{eq-wcf-dom}
    was proved in \cite[Theorem~5.11]{Joyce2008IV},
    but we repeat the argument here,
    fixing a gap in the original argument using \cref{thm-almost-bijection}.

    First, the $*$ and $\diamond$ products are well-defined by
    \cref{rem-hall-alg-lsf,rem-hall-mod-lsf}.
    Every finite type substack of $\calM_\alpha$ or $\calM^\sd_\theta$
    only intersects the support of finitely many terms in the sum,
    by \tagref{Stab3}.
    
    By \cref{thm-hn-uniqueness}, the morphism
    \begin{align}
        \coprod_{ \leftsubstack[6em]{
            & n > 0; \, \alpha_1, \dotsc, \alpha_n \in C (\calA) \colon \\[-.5ex]
            & \alpha = \alpha_1 + \cdots + \alpha_n \, , \\[-.5ex]
            & \tilde{\tau} (\alpha_1) = \cdots = \tilde{\tau} (\alpha_n), \\[-.5ex]
            & \tau (\alpha_1) > \cdots > \tau (\alpha_n)
        } }
        \calM^{\pn{n}, \ss}_{\alpha_1, \dotsc, \alpha_n} (\tau)
        \longrightarrow
        \Mss_\alpha (\tilde{\tau}) \ ,
    \end{align}
    given by $\pi^\pn{n}$ on each component,
    is bijective on $\bbK$-points,
    and induces isomorphisms of stabilizer groups at $\bbK$-points.
    Applying \cref{thm-almost-bijection} Zariski locally,
    we have
    \begin{align}
        \label{eq-wcf-dom-1}
        \sum_{ \leftsubstack[6em]{
            & n > 0; \, \alpha_1, \dotsc, \alpha_n \in C (\calA) \colon \\[-.5ex]
            & \alpha = \alpha_1 + \cdots + \alpha_n \, , \\[-.5ex]
            & \tilde{\tau} (\alpha_1) = \cdots = \tilde{\tau} (\alpha_n), \\[-.5ex]
            & \tau (\alpha_1) > \cdots > \tau (\alpha_n)
        } }
        (\pi^\pn{n})_! \, [ \calM^{\pn{n}, \ss}_{\alpha_1, \dotsc, \alpha_n} (\tau) ]
        =
        \delta_\alpha (\tilde{\tau})
    \end{align}
    in $\LSF (\calM_\alpha)$.
    Now, \cref{eq-wcf-dom} follows from
    \cref{eq-wcf-dom-1} and \cref{thm-hall-alg}.
    
    For \cref{eq-wcf-dom-sd},
    we use a similar strategy.
    By \cref{thm-sd-hn}, the morphism
    \begin{align}
        \coprod_{ \leftsubstack[6em]{
            & n \geq 0; \, \alpha_1, \dotsc, \alpha_n \in C (\calA), \,
            \rho \in C^\sd (\calA) \vphantom{^0} \colon \\[-.5ex]
            & \theta = \bar{\alpha}_1 + \cdots + \bar{\alpha}_n + \rho, \\[-.5ex]
            & \tilde{\tau} (\alpha_1) = \cdots = \tilde{\tau} (\alpha_n) = 0, \\[-.5ex]
            & \tau (\alpha_1) > \cdots > \tau (\alpha_n) > 0
        } }
        \calM^{\pn{2n+1}, \> \sdss}_{\alpha_1, \dotsc, \alpha_n, \rho} (\tau) \longrightarrow
        \calM^\sdss_\theta (\tilde{\tau}) \ ,
    \end{align}
    given by $\pi^{\pn{2n+1}, \> \sd}$ on each component,
    is bijective on $\bbK$-points,
    and induces isomorphisms of stabilizer groups at $\bbK$-points.
    Applying \cref{thm-almost-bijection} Zariski locally,
    we have
    \begin{align}
        \label{eq-wcf-dom-1-sd}
        \sum_{ \leftsubstack[6em]{
            & n \geq 0; \, \alpha_1, \dotsc, \alpha_n \in C (\calA), \,
            \rho \in C^\sd (\calA) \vphantom{^0} \colon \\[-.5ex]
            & \theta = \bar{\alpha}_1 + \cdots + \bar{\alpha}_n + \rho, \\[-.5ex]
            & \tilde{\tau} (\alpha_1) = \cdots = \tilde{\tau} (\alpha_n) = 0, \\[-.5ex]
            & \tau (\alpha_1) > \cdots > \tau (\alpha_n) > 0
        } }
        (\pi^{\pn{2n+1}, \> \sd})_! \, 
        [ \calM^{\pn{2n+1}, \> \sdss}_{\alpha_1, \dotsc, \alpha_n, \rho} (\tau) ]
        =
        \delta^\sd_\theta (\tilde{\tau})
    \end{align}
    in $\LSF (\calM^\sd_\theta)$.
    Now, \cref{eq-wcf-dom-sd} follows from
    \cref{eq-wcf-dom-1-sd} and \cref{thm-hall-module}.
\end{proof}

\paragraph{}
\label[theorem]{thm-wcf-anti-dom}

Formally inverting the dominant wall-crossing formulae
\crefrange{eq-wcf-dom}{eq-wcf-dom-sd},
we obtain the following `anti-dominant' wall-crossing formulae,
which is a step towards the general case.

\begin{theorem*}
    Let $\tau, \tilde{\tau}$ be self-dual weak stability conditions on $\calA$
    satisfying \tagref{Stab13},
    and suppose that $\tilde{\tau}$ dominates $\tau$.
    
    Then for any $\alpha \in C (\calA)$
    and $\theta \in C^\sd (\calA)$, the equations
    \upshape
    \begin{align}
        \label{eq-wcf-anti-dom}
        \delta_\alpha (\tau) & =
        \sum_{ \leftsubstack[6em]{
            & n > 0; \, \alpha_1, \dotsc, \alpha_n \in C (\calA) \colon \\[-.5ex]
            & \alpha = \alpha_1 + \cdots + \alpha_n \, , \\[-.5ex]
            & \tilde{\tau} (\alpha_1) = \cdots = \tilde{\tau} (\alpha_n), \\[-.5ex]
            & \tau (\alpha_1 + \cdots + \alpha_i) > \tau (\alpha_{i+1} + \cdots + \alpha_n)
            \text{ for all } i = 1, \dotsc, n - 1
        } } {}
        (-1)^{n-1} \,
        \delta_{\alpha_1} (\tilde{\tau}) * \cdots *
        \delta_{\alpha_n} (\tilde{\tau}) \ , \\[1ex]
        \label{eq-wcf-anti-dom-sd}
        \delta^\sd_\theta (\tau) & =
        \sum_{ \leftsubstack[6em]{
            & n \geq 0; \, \alpha_1, \dotsc, \alpha_n \in C (\calA), \,
            \rho \in C^\sd (\calA) \colon \\[-.5ex]
            & \theta = \bar{\alpha}_1 + \cdots + \bar{\alpha}_n + \rho, \\[-.5ex]
            & \tilde{\tau} (\alpha_1) = \cdots = \tilde{\tau} (\alpha_n) = 0, \\[-.5ex]
            & \tau (\alpha_1 + \cdots + \alpha_i) > 0
            \text{ for all } i = 1, \dotsc, n
        } } {}
        (-1)^n \,
        \delta_{\alpha_1} (\tilde{\tau}) \diamond \cdots \diamond
        \delta_{\alpha_n} (\tilde{\tau}) \diamond
        \delta^\sd_{\rho} (\tilde{\tau}) 
    \end{align}
    \itshape
    hold whenever the sums
    converge in the sense of \cref{def-lsf}.
\end{theorem*}

\begin{proof}
    Equation \cref{eq-wcf-anti-dom}
    was proved in~\cite[Theorem~5.12]{Joyce2008IV},
    by formally inverting \cref{eq-wcf-dom}.
    
    One can also prove \cref{eq-wcf-anti-dom-sd}
    using a similar strategy.
    Let $A_\theta$ denote the right-hand side of~\cref{eq-wcf-anti-dom-sd}.
    Expanding it using \cref{eq-wcf-dom} and~\cref{eq-wcf-dom-sd},
    and collecting like terms, we obtain
    \begin{align}
        \label{eq-pf-wcf-anti-dom-sd}
        A_\theta & =
        \sum_{ \leftsubstack[7em]{
            & n \geq 0; \, \alpha_1, \dotsc, \alpha_n \in C (\calA), \,
            \rho \in C^\sd (\calA) \colon \\[-.5ex]
            & \theta = \bar{\alpha}_1 + \cdots + \bar{\alpha}_n + \rho, \\[-.5ex]
            & \tilde{\tau} (\alpha_1) = \cdots = \tilde{\tau} (\alpha_n) = 0
        } } {}
        C (\alpha_1, \dotsc, \alpha_n) \cdot
        \delta_{\alpha_1} (\tau) \diamond \cdots \diamond
        \delta_{\alpha_n} (\tau) \diamond
        \delta^\sd_{\rho} (\tau) \ ,
    \end{align}
    where
    \begin{align*}
        C (\alpha_1, \dotsc, \alpha_n) & =
        \sum_{ \leftsubstack[10em]{
            & m \geq 0, \ 
            0 = a_0 < \cdots < a_m \leq n \colon \\[-.5ex]
            & \tau (\alpha_j) > \tau (\alpha_{j+1}) \text{ for all }
            j \in \{ 1, \dotsc, n \} \setminus \{ a_1, \dotsc, a_m \}, \\[-.5ex]
            & \tau (\alpha_1 + \cdots + \alpha_{a_i}) > 0
            \text{ for all } i = 1, \dotsc, m
        } } {}
        (-1)^m \ ,
    \end{align*}
    where we define $\tau (\alpha_{n+1}) = 0$ for convenience.

    To prove that this step is valid,
    we have to show that the sum~\cref{eq-pf-wcf-anti-dom-sd}
    is \emph{locally finite}, that is, after restricting to any finite type substack,
    only finitely many terms in the sum should be non-zero,
    even before collecting like terms.
    To prove this, it is enough to show that
    each term in~\cref{eq-wcf-anti-dom-sd}
    becomes a locally finite sum of terms in~\cref{eq-pf-wcf-anti-dom-sd}.
    By \tagref{Stab3} for $\tau$,
    each $\delta_{\alpha_i} (\tilde{\tau})$ and $\delta^\sd_\rho (\tilde{\tau})$
    in~\cref{eq-wcf-anti-dom-sd} split into locally finite sums.
    By \tagref{Fin2} and \cref{lem-filt-sd-rep-ft},
    the sum is locally finite
    after taking the $\diamond$ product.
    
    Now, the remaining task is to show that $C (\alpha_1, \dotsc, \alpha_n) = 0$
    unless $n = 0$.
    Indeed, let us fix the classes $\alpha_1, \dotsc, \alpha_n$\,,
    with $n > 0$. Define sets
    \begin{align*}
        I & = \{ 1, \dotsc, n \} \ , \\
        I_1 & =
        \{ j \in I \mid \tau (\alpha_j) > \tau (\alpha_{j+1}) \} \ , \\
        I_2 & =
        \{ j \in I \mid \tau (\alpha_1 + \cdots + \alpha_j) > 0 \} \ .
    \end{align*}
    Then
    \begin{align*}
        C (\alpha_1, \dotsc, \alpha_n) & =
        \sum_{ \nicesubstack{
            J \subset I \colon \\[-.5ex]
            I \setminus I_1 \subset J \subset I_2
        } } {}
        (-1)^{|J|} \ ,
    \end{align*}
    which is non-zero precisely when $I_2 = I \setminus I_1$\,.
    Assume that this is the case.
    Then if $1 \in I_1$\,, then $1 \notin I_2$\,,
    so $\tau (\alpha_1) \leq 0$, and so $\tau (\alpha_2) < 0$.
    Continuing, we see that $0 \geq \tau (\alpha_1) > \tau (\alpha_2) > \cdots$,
    contradicting with the definition that $\tau (\alpha_{n+1}) = 0$.
    Similarly, if $1 \in I_2$\,, then
    $0 < \tau (\alpha_1) \leq \tau (\alpha_2) \leq \cdots$,
    contradicting again with $\tau (\alpha_{n+1}) = 0$.
    
    Therefore, the only non-zero term in~\cref{eq-pf-wcf-anti-dom-sd}
    is the term with $n = 0$,
    which gives $\delta^\sd_\theta (\tau)$, as desired.
\end{proof}

\subsection{Wall-crossing: The general case}

\paragraph{}

We are now ready to prove the general wall-crossing formulae
for the stack function invariants
$\delta_\alpha (\tau)$, $\epsilon_\alpha (\tau)$,
$\delta^\sd_\theta (\tau)$, $\epsilon^\sd_\theta (\tau)$.
The strategy is to compose the
dominant wall-crossing formulae \crefrange{eq-wcf-dom}{eq-wcf-dom-sd}
with the anti-dominant wall-crossing formulae
\crefrange{eq-wcf-anti-dom}{eq-wcf-anti-dom-sd}.
The ordinary case was done by Joyce~\cite[\S5]{Joyce2008IV},
while the self-dual case is new.
The main result is \cref{thm-wcf-main}.

Throughout, we assume the same setting as in \cref{para-wcf-intro},
that is, we assume \tagref{SdMod} and~\tagref{Fin}.

\begin{definition}
    \allowdisplaybreaks
    In the setting of \cref{para-wcf-intro},
    let $\tau, \tilde{\tau}$ be weak stability conditions on $\calA$.
    For classes $\alpha_1, \dotsc, \alpha_n \in C (\calA)$, define
    \begin{align}
        \label{eq-def-s}
        \hspace{-1em}
        S (\alpha_1, \dotsc, \alpha_n; \tau, \tilde{\tau}) & =
        \prod_{i=1}^{n-1} {} \left\{
            \begin{array}{ll}
                1, & \tau (\alpha_i) > \tau (\alpha_{i+1}) \text{ and } \\
                & \hspace{2em} \tilde{\tau} (\alpha_1 + \cdots + \alpha_i)
                \leq \tilde{\tau} (\alpha_{i+1} + \cdots + \alpha_n) \\
                -1, & \tau (\alpha_i) \leq \tau (\alpha_{i+1}) \text{ and } \\
                & \hspace{2em} \tilde{\tau} (\alpha_1 + \cdots + \alpha_i)
                > \tilde{\tau} (\alpha_{i+1} + \cdots + \alpha_n) \\
                0, & \text{otherwise}
            \end{array}
        \right\}, 
        \hspace{-1em} \\
        \label{eq-def-ssd}
        \hspace{-1em}
        S^\sd (\alpha_1, \dotsc, \alpha_n; \tau, \tilde{\tau}) & =
        \prod_{i=1}^n {} \left\{
            \begin{array}{ll}
                1, & \tau (\alpha_i) > \tau (\alpha_{i+1}) \text{ and } 
                \tilde{\tau} (\alpha_1 + \cdots + \alpha_i) \leq 0 \\
                -1, & \tau (\alpha_i) \leq \tau (\alpha_{i+1}) \text{ and }
                \tilde{\tau} (\alpha_1 + \cdots + \alpha_i) > 0 \\
                0, & \text{otherwise}
            \end{array}
        \right\},
        \hspace{-1em}
    \end{align}
    where we set $\tau (\alpha_{n+1}) = 0$.
    Write $\alpha = \alpha_1 + \cdots + \alpha_n$\,.
    Define
    \begin{multline}
        \label{eq-def-u}
        U (\alpha_1, \dotsc, \alpha_n; \tau, \tilde{\tau}) = \\*
        \shoveright{
        \sum_{ \leftsubstack[7em]{
            \\[-.5ex]
            & 0 = a_0 < \cdots < a_m = n, \ 
            0 = b_0 < \cdots < b_l = m. \\[-.5ex]
            & \text{Define } \beta_1, \dotsc, \beta_m \text{ by }
            \beta_i = \alpha_{a_{i-1}+1} + \cdots + \alpha_{a_i} \, . \\[-.5ex]
            & \text{Define } \gamma_1, \dotsc, \gamma_l \text{ by }
            \gamma_i = \beta_{b_{i-1}+1} + \cdots + \beta_{b_i} \, . \\[-.5ex]
            & \text{We require } \tau (\beta_i) = \tau (\alpha_j)
            \text{ for all } a_{i-1} < j \leq a_i \, , \\[-.5ex]
            & \text{and } \tilde{\tau} (\gamma_i) = \tilde{\tau} (\alpha)
            \text{ for all } i = 1, \dotsc, l
        } } {}
        \frac{(-1)^{l-1}}{l} \cdot \biggl(
            \prod_{i=1}^{l}
            S (\beta_{b_{i-1}+1}, \dotsc, \beta_{b_i}; \tau, \tilde{\tau})
        \biggr) \cdot
        \biggl(
            \prod_{i=1}^m \frac{1}{(a_i - a_{i-1})!} 
        \biggr) \ , \hspace{-2em} } \\*[-5ex]
        \mathstrut
    \end{multline}
    \begin{multline}
        \label{eq-def-usd}
        U^\sd (\alpha_1, \dotsc, \alpha_n; \tau, \tilde{\tau}) = \\*
        \sum_{ \leftsubstack[7em]{
            \\[-.5ex]
            & 0 = a_0 < \cdots < a_m \leq n, \ 
            0 = b_0 < \cdots < b_l \leq m. \\[-.5ex]
            & \text{Define } \beta_1, \dotsc, \beta_m \text{ by }
            \beta_i = \alpha_{a_{i-1}+1} + \cdots + \alpha_{a_i} \, . \\[-.5ex]
            & \text{Define } \gamma_1, \dotsc, \gamma_l \text{ by }
            \gamma_i = \beta_{b_{i-1}+1} + \cdots + \beta_{b_i} \, . \\[-.5ex]
            & \text{We require } \tau (\beta_i) = \tau (\alpha_j)
            \text{ for all } a_{i-1} < j \leq a_i \, , \\[-.5ex]
            & \tau (\alpha_j) = 0
            \text{ for all } j > a_m \, , \\[-.5ex]
            & \text{and } \tilde{\tau} (\gamma_i) = 0
            \text{ for all } i = 1, \dotsc, l
        } } {}
        \binom{-1/2}{l} \cdot \biggl(
            \prod_{i=1}^{l}
            S (\beta_{b_{i-1}+1}, \dotsc, \beta_{b_i}; \tau, \tilde{\tau})
        \biggr) \cdot
        S^\sd (\beta_{b_l+1}, \dotsc, \beta_{m}; \tau, \tilde{\tau})
        \cdot {} \\*[-11ex]
        \shoveright{ \biggl(
            \prod_{i=1}^m \frac{1}{(a_i - a_{i-1})!} 
        \biggr) \cdot \frac{1}{2^{n-a_m} \, (n - a_m)!} \ . \hspace{-2em} } \\*
        \mathstrut
    \end{multline}
    Here, \cref{eq-def-s,eq-def-u} are as in Joyce \cite[\S4.1]{Joyce2008IV};
    \cref{eq-def-ssd,eq-def-usd} are new.
\end{definition}

\begin{definition}
    \label{def-finite-change}
    In the setting of \cref{para-wcf-intro},
    let $\tau, \tilde{\tau}$ be weak stability conditions on $\calA$,
    satisfying \tagref{Stab13}.
    
    As in \cite[Definition~5.1]{Joyce2008IV},
    we say that \emph{the change from $\tau$ to $\tilde{\tau}$ is locally finite},
    if for any $\alpha \in C^\circ (\calA)$,
    any finite type $\bbK$-scheme $U$, and any morphism $f \colon U \to \calM_\alpha$\,,
    there only exist finitely many choices of $\alpha_1, \dotsc, \alpha_n \in C (\calA)$,
    with $\alpha = \alpha_1 + \cdots + \alpha_n$\,,
    $S (\alpha_1, \dotsc, \alpha_n; \tau, \tilde{\tau}) \neq 0$,
    and $f^{-1} (\pi^{\pn{n}}
    (\calM^{\pn{n}, \> \ss}_{\alpha_1, \dotsc, \alpha_n} (\tau))) \neq \varnothing$.

    We say that \emph{the change from $\tau$ to $\tilde{\tau}$ is globally finite},
    if for any $\alpha \in C^\circ (\calA)$,
    there only exist finitely many choices of $\alpha_1, \dotsc, \alpha_n \in C (\calA)$,
    with $\alpha = \alpha_1 + \cdots + \alpha_n$\,,
    $S (\alpha_1, \dotsc, \alpha_n; \tau, \tilde{\tau}) \neq 0$,
    and $\calM^\ss_{\alpha_i} (\tau) \neq \varnothing$ for all $i$.

    In both cases, it is also true that for any $\theta \in C^\sd (\calA)$,
    there only exist finitely many choices of $\alpha_1, \dotsc, \alpha_n \in C (\calA)$
    and $\rho \in C^\sd (\calA)$,
    with $\theta = \bar{\alpha}_1 + \cdots + \bar{\alpha}_n + \rho$,
    $S^\sd (\alpha_1, \dotsc, \alpha_n; \tau, \tilde{\tau}) \neq 0$,
    such that the corresponding semistable moduli stacks
    (pulled back along $f$ in the first case) are non-empty.
    This is because
    $S^\sd (\alpha_1, \dotsc, \alpha_n; \tau, \tilde{\tau}) \neq 0$ if and only if
    $S (\alpha_1, \dotsc, \alpha_n, \rho, \alpha_n^\vee, \dotsc, \alpha_1^\vee; \tau, \tilde{\tau}) \neq 0$, where $\rho$ is omitted if it is zero.
\end{definition}

For example, when considering coherent sheaves on a higher dimensional variety,
the change from the trivial stability condition to Gieseker stability
may fail to be locally finite. But, as in Joyce~\cite[Proposition~5.7]{Joyce2008IV},
the change from \emph{purity} to Gieseker stability is locally finite.
Note, however, that Gieseker stability is not self-dual.

\paragraph{}
\label{para-sit-wcf}

From now on, we assume the following setting.
In the setting of \cref{para-wcf-intro},
including the conditions \tagref{SdMod} and \tagref{Fin},
let $\tau, \tilde{\tau}, \hat{\tau}$ be self-dual weak stability conditions on $\calA$
satisfying \tagref{Stab13}.
Suppose that $\tau$ and $\tilde{\tau}$ are dominated by $\hat{\tau}$,
and that the change from $\hat{\tau}$ to $\tilde{\tau}$ is locally finite,
in the sense of \cref{def-finite-change}.

We state the general wall-crossing formulae
for the stack function invariants as follows.

\begin{theorem}
    \label{thm-wcf-main}
    \allowdisplaybreaks
    In the setting of \cref{para-sit-wcf},
    for any $\alpha \in C (\calA)$ and $\theta \in C^\sd (\calA)$,
    \begin{align}
        \label{eq-wcf-delta}
        \delta_\alpha (\tilde{\tau})
        & =
        \sum_{ \leftsubstack[5em]{
            & n > 0; \, \alpha_1, \dotsc, \alpha_n \in C (\calA) \colon \\[-.5ex]
            & \alpha = \alpha_1 + \cdots + \alpha_n 
        } } {}
        S (\alpha_1, \dotsc, \alpha_n; \tau, \tilde{\tau}) \cdot
        \delta_{\alpha_1} (\tau) * \cdots *
        \delta_{\alpha_n} (\tau) \ , \\[2ex]
        \label{eq-wcf-delta-sd}
        \delta^\sd_\theta (\tilde{\tau})
        & =
        \sum_{ \leftsubstack[5em]{
            & n \geq 0; \, \alpha_1, \dotsc, \alpha_n \in C (\calA), \,
            \rho \in C^\sd (\calA) \colon \\[-.5ex]
            & \theta = \bar{\alpha}_1 + \cdots + \bar{\alpha}_n + \rho
        } } {}
        S^\sd (\alpha_1, \dotsc, \alpha_n; \tau, \tilde{\tau}) \cdot
        \delta_{\alpha_1} (\tau) \diamond \cdots \diamond
        \delta_{\alpha_n} (\tau) \diamond
        \delta^\sd_{\rho} (\tau) \ ,
    \end{align}
    where the sums converge in the sense of \cref{def-lsf}.
    If, moreover, $\tau$ and $\tilde{\tau}$ satisfy \tagref{Stab2}, then
    \begin{align}
        \label{eq-wcf-epsilon}
        \epsilon_\alpha (\tilde{\tau})
        & =
        \sum_{ \leftsubstack[5em]{
            & n > 0; \, \alpha_1, \dotsc, \alpha_n \in C (\calA) \colon \\[-.5ex]
            & \alpha = \alpha_1 + \cdots + \alpha_n 
        } } {}
        U (\alpha_1, \dotsc, \alpha_n; \tau, \tilde{\tau}) \cdot
        \epsilon_{\alpha_1} (\tau) * \cdots *
        \epsilon_{\alpha_n} (\tau) \ , \\[2ex]
        \label{eq-wcf-epsilon-sd}
        \epsilon^\sd_\theta (\tilde{\tau})
        & =
        \sum_{ \leftsubstack[5em]{
            & n \geq 0; \, \alpha_1, \dotsc, \alpha_n \in C (\calA), \,
            \rho \in C^\sd (\calA) \colon \\[-.5ex]
            & \theta = \bar{\alpha}_1 + \cdots + \bar{\alpha}_n + \rho 
        } } {}
        U^\sd (\alpha_1, \dotsc, \alpha_n; \tau, \tilde{\tau}) \cdot
        \epsilon_{\alpha_1} (\tau) \diamond \cdots \diamond
        \epsilon_{\alpha_n} (\tau) \diamond
        \epsilon^\sd_{\rho} (\tau) \ ,
    \end{align}
    where the sums converge.
\end{theorem}

\begin{proof}
    Equation \cref{eq-wcf-delta}
    was shown in \cite[Theorem~5.2]{Joyce2008IV},
    as a result of composing
    \cref{eq-wcf-anti-dom} with \cref{eq-wcf-dom}.
    
    For \cref{eq-wcf-delta-sd},
    we use a similar strategy by expanding
    the left-hand side using
    \cref{eq-wcf-anti-dom-sd},
    which expresses everything in terms of $\hat{\tau}$,
    and then expanding everything again using
    \cref{eq-wcf-dom} and \cref{eq-wcf-dom-sd}.
    Collecting like terms, this gives
    \begin{align}
        \label{eq-pf-wcf-delta-sd}
        \delta^\sd_\theta (\tilde{\tau})
        & =
        \sum_{ \leftsubstack[5em]{
            & n \geq 0; \, \alpha_1, \dotsc, \alpha_n \in C (\calA), \,
            \rho \in C^\sd (\calA) \colon \\[-.5ex]
            & \theta = \bar{\alpha}_1 + \cdots + \bar{\alpha}_n + \rho
        } } {}
        C (\alpha_1, \dotsc, \alpha_n) \cdot
        \delta_{\alpha_1} (\tau) \diamond \cdots \diamond
        \delta_{\alpha_n} (\tau) \diamond
        \delta^\sd_{\rho} (\tau) \ ,
    \end{align}
    where
    \begin{align*}
        C (\alpha_1, \dotsc, \alpha_n) =
        \sum_{ \leftsubstack[10em]{
            & m \geq 0, \ 
            0 = a_0 < \cdots < a_m \leq n \colon \\[-.5ex]
            & \tau (\alpha_j) > \tau (\alpha_{j+1}) \text{ for all }
            j \in \{ 1, \dotsc, n \} \setminus \{ a_1, \dotsc, a_m \}, \\[-.5ex]
            & \tilde{\tau} (\alpha_1 + \cdots + \alpha_{a_i}) > 0
            \text{ for all } i = 1, \dotsc, m
        } } {}
        (-1)^m \ ,
    \end{align*}
    where we write $\tau (\alpha_{n+1}) = 0$ as before.
    This step is valid by a similar argument
    as in the proof of \cref{thm-wcf-anti-dom},
    involving the local finiteness condition.
    
    Define sets
    \begin{align*}
        I & = \{ 1, \dotsc, n \} \ , \\
        I_1 & =
        \{ j \in I \mid \tau (\alpha_j) > \tau (\alpha_{j+1}) \} \ , \\
        I_2 & =
        \{ j \in I \mid \tilde{\tau} (\alpha_1 + \cdots + \alpha_j) > 0 \} \ .
    \end{align*}
    Then
    \begin{align*}
        C (\alpha_1, \dotsc, \alpha_n) & =
        \sum_{ \nicesubstack{
            J \subset I \colon \\[-.5ex]
            I \setminus I_1 \subset J \subset I_2
        } } {}
        (-1)^{|J|} \ ,
    \end{align*}
    which is non-zero precisely when $I_2 = I \setminus I_1$\,,
    in which case it is equal to $(-1)^{|I_2|}$.
    This agrees with the definition of\
    $S^\sd (\alpha_1, \dotsc, \alpha_n; \tau, \tilde{\tau})$.
    
    Equation \cref{eq-wcf-epsilon}
    was shown in \cite[Theorem~5.2]{Joyce2008IV},
    as a result of composing
    \cref{eq-wcf-delta} with
    \cref{eq-def-epsilon,eq-def-epsilon-inverse}.
    Similarly, \cref{eq-wcf-epsilon-sd}
    follows from \cref{eq-wcf-delta-sd}
    by composing it with
    \cref{eq-def-epsilon-sd,eq-def-epsilon-inverse,eq-def-epsilon-sd-inverse}.
\end{proof}

\begin{remark}
    In \cref{thm-wcf-main},
    although $\hat{\tau}$ does not appear anywhere
    in the wall-crossing formulae,
    it is important that it exists,
    satisfies \tagref{Stab13},
    and satisfies the condition
    that the change from $\hat{\tau}$ to $\tilde{\tau}$ is locally finite,
    as these conditions ensure that the wall-crossing formulae
    in \cref{thm-wcf-main} converge.
\end{remark}

\subsection{Wall-crossing: Lie algebras and twisted modules}

\paragraph{}
\label{para-sit-wcf-lie}

Having obtained the general wall-crossing formulae,
\cref{thm-wcf-main},
we now show that wall-crossing formulae for the 
$\epsilon$ and $\epsilon^\sd$ functions
can be expressed solely in terms of
the Lie algebra structure on $\SF (\calM)$,
defined in \cref{para-hall-lie-alg},
and the twisted module structure on $\SF (\calM^\sd)$,
defined in \cref{para-hall-tw-mod}.
The main result is \cref{thm-comb}.

Our running assumptions will now include
\tagref{SdMod}, \tagref{Fin},
and self-dual weak stability conditions $\tau, \tilde{\tau}, \hat{\tau}$
satisfying conditions in \cref{para-sit-wcf},
with $\hat{\tau}$ satisfying \tagref{Stab13},
and $\tau, \tilde{\tau}$ satisfying \tagref{Stab}.

\paragraph{}
\label[theorem]{thm-u-tilde}

The ordinary case of this result
was proved in~\cite[Theorem~5.4]{Joyce2008IV},
and we state it below.

\begin{theorem*}
    In the setting of \cref{para-sit-wcf-lie},
    for classes $\alpha_1, \dotsc, \alpha_n \in C (\calA),$
    consider the free Lie algebra $F_n$ over $\bbQ,$ generated by symbols
    $e_1, \dotsc, e_n$\,. Let $U (F_n)$ be its universal enveloping algebra,
    whose multiplication is denoted by $*$.
    Then the element
    \begin{equation}
        \label{eq-u-abstr}
        \sum_{\sigma \in \frS_n}
        U (\alpha_{\sigma(1)}, \dotsc, \alpha_{\sigma(n)}; \tau, \tilde{\tau}) \cdot
        e_{\sigma(1)} * \cdots * e_{\sigma(n)} \in U (F_n)
    \end{equation}
    lies in the subspace $F_n \subset U (F_n),$
    where $\frS_n$ is the symmetric group.
    Consequently, there exist coefficients
    \[
        \tilde{U} (\alpha_1, \dotsc, \alpha_n; \tau, \tilde{\tau}) \in \bbQ \ ,
    \]
    such that~\cref{eq-u-abstr} can be rewritten as
    \begin{equation}
        \sum_{\sigma \in \frS_n}
        \tilde{U} (\alpha_{\sigma(1)}, \dotsc, \alpha_{\sigma(n)}; \tau, \tilde{\tau}) \cdot
        \bigl[ \dotsc \bigl[ \bigl[ e_{\sigma(1)} \, , \, 
        e_{\sigma(2)} \bigr], \dotsc \bigr] ,
        e_{\sigma(n)} \bigr] \ .
    \end{equation}
    In particular, for any $\alpha \in C^\circ (\calA),$
    we have the wall-crossing formula
    \begin{align}
        \epsilon_\alpha (\tilde{\tau}) & =
        \sum_{ \leftsubstack[6em]{
            \\[-3ex]
            & n > 0; \, \alpha_1, \dotsc, \alpha_n \in C (\calA) \colon \\[-.5ex]
            & \alpha = \alpha_1 + \cdots + \alpha_n
        } } {}
        \tilde{U} (\alpha_1, \dotsc, \alpha_n; \tau, \tilde{\tau}) \cdot
        \bigl[ \dotsc \bigl[ \bigl[ \epsilon_{\alpha_1} (\tau), 
        \epsilon_{\alpha_2} (\tau) \bigr], \dotsc \bigr] ,
        \epsilon_{\alpha_n} (\tau) \bigr] \ ,
        \notag \\*[-4ex]
        \label{eq-u-tilde}
    \end{align}
    where the Lie brackets are as in \cref{para-hall-lie-alg},
    and the sum converges in the sense of \cref{def-lsf}.
\end{theorem*}

Note that the coefficients $\tilde{U} ({\cdots})$ above
are not necessarily unique.

\paragraph{}
\label[definition]{def-pre-thm-comb}

We now state the self-dual counterpart of \cref{thm-u-tilde},
as \cref{thm-comb} below,
expressing wall-crossing formulae for the $\epsilon^\sd$ functions
in terms of the Lie algebra structure on $\SF (\calM)$
and the twisted module structure on $\SF (\calM^\sd)$.

\begin{definition*}
    Let $n \geq 0$ be an integer,
    and let $L_n$ be the free Lie algebra over $\bbQ$, generated by symbols
    $e_1, e_1^\vee, \dotsc, e_n, e_n^\vee$.
    Define an involution $(-)^\vee$ on $L_n$
    by exchanging $e_i$ and $e_i^\vee$ for all $i,$
    and sending $[x, y]$ to $[y^\vee, x^\vee]$ for all $x, y \in L_n$\,.
    
    Let $L_n^+ \subset L_n$ be the subalgebra as in \cref{para-inv-lie-alg}.
    Define a map $(-)^+ \colon L_n \to L_n^+$ by $x \mapsto x - x^\vee$.

    Let $P_n$ be the set of sequences $x = (x_1, \dotsc, x_n)$
    of elements of $L_n$,
    such that for some $\sigma \in \frS_n$, we have
    $x_{\sigma (i)} \in \{ e_i, e_i^\vee \}$ for all $i$.

    Let $Q_n$ be the set of lists $y = (y_{i,j})_{1 \leq i \leq k, \, 1 \leq j \leq m_i}$\,,
    where $k \geq 0$ and $m_1, \dotsc, m_k \geq 1$ are integers,
    with $m_1 + \cdots + m_k = n$,
    such that each $y_{i,j}$ is equal to some $e_l$ or $e_l^\vee$,
    where $l \in \{ 1, \dotsc, n \}$,
    and for each such $l$, there is a unique $y_{i,j}$ that is equal to $e_l$ or $e_l^\vee$.
\end{definition*}

\begin{theorem}
    \label{thm-comb}
    In the setting of \cref{para-sit-wcf-lie},
    for any $\alpha_1, \dotsc, \alpha_n \in C (\calA),$
    the element
    \begin{equation}
        \label{eq-usd-abstr}
        \sum_{x \in P_n}
        U^\sd (\alpha_{\sigma(1)}, \dotsc, \alpha_{\sigma(n)}; \tau, \tilde{\tau}) \cdot
        x_1 * \cdots * x_n \in U (L_n)
    \end{equation}
    lies in $U (L_n^+),$
    where notations are as in \cref{def-pre-thm-comb}.
    Consequently, 
    there exist coefficients
    \[
        \tilde{U}^\sd (\alpha_{1,1}, \dotsc, \alpha_{1,m_1}; \dotsc;
        \alpha_{k,1}, \dotsc, \alpha_{k,m_k}; \tau, \tilde{\tau}) \in \bbQ \ ,
    \]
    for all lists of classes $\alpha_{1,1}, \dotsc, \alpha_{1,m_1}; \dotsc;
        \alpha_{n,1}, \dotsc, \alpha_{n,m_n} \in C (\calA),$
    such that~\cref{eq-usd-abstr} can be rewritten as
    \begin{equation}
        \sum_{y \in Q_n}
        \tilde{U}^\sd (\alpha_y; \tau, \tilde{\tau}) \cdot
        \bigl[ \bigl[ y_{1,1}, \dotsc \bigr] ,
        y_{1,m_1} \bigr]^+ * \cdots *
        \bigl[ \bigl[ y_{k,1}, \dotsc \bigr] ,
        y_{k,m_k} \bigr]^+ \ ,
    \end{equation}
    where $\alpha_y$ is the list obtained from $y$
    by replacing each $e_l, e^\vee_l$ with $\alpha_l, \alpha^\vee_l,$ respectively.
    
    In particular, for any $\theta \in C^\sd (\calA),$
    we have the wall-crossing formula
    \begin{multline}
        \label{eq-comb-main}
        \epsilon^\sd_\theta (\tilde{\tau}) =
        \sum_{ \leftsubstack[10em]{
            \\[-3ex]
            & n \geq 0; \, m_1, \dotsc, m_n > 0; \\[-.5ex]
            & \alpha_{1,1}, \dotsc, \alpha_{1,m_1}; \dotsc;
            \alpha_{n,1}, \dotsc, \alpha_{n,m_n} \in C (\calA); \,
            \rho \in \smash{C^\sd (\calA)} \colon \\[-.5ex]
            & \theta = (\bar{\alpha}_{1,1} + \cdots + \bar{\alpha}_{1,m_1})
            + \cdots + (\bar{\alpha}_{n,1} + \cdots + \bar{\alpha}_{n,m_n}) + \rho
        } } {}
        \tilde{U}^\sd (\alpha_{1,1}, \dotsc, \alpha_{1,m_1}; \dotsc;
        \alpha_{n,1}, \dotsc, \alpha_{n,m_n}; \tau, \tilde{\tau}) \cdot {} \\[1ex]
        \bigl[ \bigl[ \epsilon_{\alpha_{1,1}} (\tau), \dotsc \bigr] ,
        \epsilon_{\alpha_{1,m_1}} (\tau) \bigr] \heart \cdots \heart
        \bigl[ \bigl[ \epsilon_{\alpha_{n,1}} (\tau), \dotsc \bigr] ,
        \epsilon_{\alpha_{n,m_n}} (\tau) \bigr] \heart 
        \epsilon^\sd_{\rho} (\tau) \ ,
    \end{multline}
    where the Lie brackets and the operation $\heart$
    are defined in \cref{para-hall-lie-alg,para-hall-tw-mod},
    and the sum converges in the sense of \cref{def-lsf}.
\end{theorem}

The proof will be given in \cref{sect-proof-comb}.
Again, note that the coefficients $\tilde{U}^\sd ({\cdots})$ above
are not necessarily unique.

\subsection{Motivic invariants}
\label{sect-motivic-inv}

\paragraph{}
\label{para-sit-motivic-inv}

We now define \emph{motivic enumerative invariants}
in a self-dual category, which is one of the main goals of this paper.
This is done by applying \emph{motivic integration},
as in \cref{para-mot-int}, to the stack function invariants.
We will define the following types of motivic invariants:

\begin{itemize}
    \item 
        Motivic invariants $\upI_\alpha (\tau), \upI^\sd_{\theta} (\tau) \in \Mhat_\bbK$,
        defined as the integral of the $\delta$ and $\delta^\sd$ functions,
        representing motives of the semistable moduli stacks.
    \item 
        Motivic invariants $\upJ_\alpha (\tau), \upJ^\sd_{\theta} (\tau) \in \Mhat_\bbK$,
        defined as the integral of the $\epsilon$ and $\epsilon^\sd$ functions,
        representing weighted motives of the semistable moduli stacks.
        They have the additional property that they
        satisfy the \emph{no-pole theorems},
        \cref{thm-no-pole,thm-no-pole-sd},
        which allow us to define their Euler characteristics,
        as we will do in \cref{sect-num-inv} below.
\end{itemize}
The invariants $\upI_\alpha (\tau)$ and~$\upJ_\alpha (\tau)$ were defined in
Joyce~\cite[Definitions~6.1 and~6.7]{Joyce2008IV},
while the invariants $\upI^\sd_{\theta} (\tau)$ and~$\upJ^\sd_{\theta} (\tau)$ are new.

Throughout, we assume the conditions
\tagref{SdMod} and \tagref{Fin1}.

\paragraph{Definition.}

Let $\tau$ be a self-dual weak stability condition on $\calA$,
satisfying \tagref{Stab}.
For classes $\alpha \in C (\calA)$ and $\theta \in C^\sd (\calA)$,
define invariants 
\[
    \upI_\alpha (\tau) \ , \quad
    \upI^\sd_{\theta} (\tau) \ , \quad
    \upJ_\alpha (\tau) \ , \quad
    \upJ^\sd_{\theta} (\tau) \ \in \ \Mhat_\bbK
\]
by the motivic integrals
\begin{alignat}{2}
    \label{eq-def-i}
    \upI_\alpha (\tau) & =
    \int_{\calM_\alpha} \delta_\alpha (\tau) \ ,
    & \qquad
    \upJ_\alpha (\tau) & =
    \int_{\calM_\alpha} \epsilon_\alpha (\tau) \ , \\
    \label{eq-def-isd}
    \upI^\sd_\theta (\tau) & =
    \int_{\calM^\sd_\theta} \delta^\sd_\theta (\tau) \ ,
    & \qquad
    \upJ^\sd_\theta (\tau) & =
    \int_{\calM^\sd_\theta} \epsilon^\sd_\theta (\tau) \ ,
\end{alignat}
as in \cref{para-mot-int},
where the integrands were defined in \cref{sect-delta-epsilon}.

\paragraph{}

One can deduce the relations
\begin{align}
    \label{eq-inv-def-epsilon}
    \upJ_\alpha (\tau) & =
    \sum_{ \leftsubstack[6em]{
        \\[-1.5ex]
        & n > 0; \, \alpha_1, \dotsc, \alpha_n \in C (\calA) \colon \\[-.5ex]
        & \alpha = \alpha_1 + \cdots + \alpha_n \, , \\[-.5ex]
        & \tau (\alpha_1) = \cdots = \tau (\alpha_n)
    } }
    \frac{(-1)^{n-1}}{n} \cdot
    \bbL^{-\chi (\alpha_1, \dotsc, \alpha_n)} \cdot
    \upI_{\alpha_1} (\tau) \cdots
    \upI_{\alpha_n} (\tau) \ , \\
    \label{eq-inv-def-epsilon-inverse}
    \upI_\alpha (\tau) & =
    \sum_{ \leftsubstack[6em]{
        \\[-1.5ex]
        & n > 0; \, \alpha_1, \dotsc, \alpha_n \in C (\calA) \colon \\[-.5ex]
        & \alpha = \alpha_1 + \cdots + \alpha_n \, , \\[-.5ex]
        & \tau (\alpha_1) = \cdots = \tau (\alpha_n)
    } }
    \frac{1}{n!} \cdot
    \bbL^{-\chi (\alpha_1, \dotsc, \alpha_n)} \cdot
    \upJ_{\alpha_1} (\tau) \cdots
    \upJ_{\alpha_n} (\tau) \ , \\
    \label{eq-inv-def-epsilon-sd}
    \upJ^\sd_\theta (\tau) & =
    \sum_{ \leftsubstack[6em]{
        \\[-1.5ex]
        & n \geq 0; \, \alpha_1, \dotsc, \alpha_n \in C (\calA), \,
        \rho \in C^\sd (\calA) \colon \\[-.5ex]
        & \theta = \bar{\alpha}_1 + \cdots + \bar{\alpha}_n + \rho \\[-.5ex]
        & \tau (\alpha_1) = \cdots = \tau (\alpha_n) = 0
    } } {}
    \binom{-1/2}{n} \cdot
    \bbL^{-\chi^\sd (\alpha_1, \dotsc, \alpha_n, \rho)} \cdot
    \upI_{\alpha_1} (\tau) \cdots 
    \upI_{\alpha_n} (\tau) \cdot
    \upI^\sd_{\rho} (\tau) \ , \\
    \label{eq-inv-def-epsilon-sd-inverse}
    \upI^\sd_\theta (\tau) & =
    \sum_{ \leftsubstack[6em]{
        \\[-1.5ex]
        & n \geq 0; \, \alpha_1, \dotsc, \alpha_n \in C (\calA), \,
        \rho \in C^\sd (\calA) \colon \\[-.5ex]
        & \theta = \bar{\alpha}_1 + \cdots + \bar{\alpha}_n + \rho \\[-.5ex]
        & \tau (\alpha_1) = \cdots = \tau (\alpha_n) = 0
    } } {}
    \frac{1}{2^n \, n!} \cdot
    \bbL^{-\chi^\sd (\alpha_1, \dotsc, \alpha_n, \rho)} \cdot
    \upJ_{\alpha_1} (\tau) \cdots
    \upJ_{\alpha_n} (\tau) \cdot
    \upJ^\sd_{\rho} (\tau) \ ,
\end{align}
where $\chi (\alpha_1, \dotsc, \alpha_n)$ and
$\chi^\sd (\alpha_1, \dotsc, \alpha_n, \rho)$
are given by \cref{eq-def-multi-chi,eq-def-multi-chi-sd}.
These relations follow from~\cref{eq-def-epsilon},
\cref{eq-def-epsilon-inverse},
\cref{eq-def-epsilon-sd}, and
\cref{eq-def-epsilon-sd-inverse}, respectively,
using \cref{thm-mot-int-homo} and
\crefrange{eq-lem-multi-chi}{eq-lem-multi-chi-sd}.

\paragraph{Wall-crossing.}
\label{para-sit-mot-wcf}

We formulate \emph{wall-crossing formulae}
for the invariants defined above,
relating invariants of different stability conditions.

From now on, we assume the conditions \tagref{Ext} and \tagref{SdExt},
restricting the homological dimension of $\calA$,
so that results in \cref{sect-motivic-int-1}
can be applied.
We also assume \tagref{Fin}.

As in \cref{para-sit-wcf},
let $\tau, \tilde{\tau}, \hat{\tau}$ be self-dual weak stability condition on $\calA$,
such that $\tau$ and $\tilde{\tau}$ are dominated by $\hat{\tau}$
in the sense of \cref{def-dominate},
with $\hat{\tau}$ satisfying \tagref{Stab13},
and $\tau, \tilde{\tau}$ satisfying \tagref{Stab}.
We also assume that the change from $\hat{\tau}$ to $\tilde{\tau}$ is locally finite,
and that the change from $\tau$ to $\tilde{\tau}$
is globally finite, in the sense of \cref{def-finite-change}.

Applying the integration map to
\cref{eq-wcf-delta,eq-wcf-epsilon,eq-wcf-delta-sd,eq-wcf-epsilon-sd},
and using \cref{thm-mot-int-homo} and
\crefrange{eq-lem-multi-chi}{eq-lem-multi-chi-sd},
we obtain the following wall-crossing formulae
for the motivic invariants.

\begin{theorem}
    \label{thm-wcf-mot}
    \allowdisplaybreaks
    In the setting of \cref{para-sit-mot-wcf},
    for any $\alpha \in C^\circ (\calA)$ and $\theta \in C^\sd (\calA),$
    we have the following wall-crossing formulae:
    \begin{align}
        \label{eq-wcf-i}
        \upI_\alpha (\tilde{\tau}) & =
        \sum_{ \leftsubstack[5em]{
            & n > 0; \, \alpha_1, \dotsc, \alpha_n \in C (\calA) \colon \\[-.5ex]
            & \alpha = \alpha_1 + \cdots + \alpha_n
        } } {}
        S (\alpha_1, \dotsc, \alpha_n; \tau, \tilde{\tau}) \cdot
        \bbL^{-\chi (\alpha_1, \dotsc, \alpha_n)} \cdot
        \upI_{\alpha_1} (\tau) \cdots
        \upI_{\alpha_n} (\tau) \ , \\[1ex]
        \label{eq-wcf-j}
        \upJ_\alpha (\tilde{\tau}) & =
        \sum_{ \leftsubstack[5em]{
            & n > 0; \, \alpha_1, \dotsc, \alpha_n \in C (\calA) \colon \\[-.5ex]
            & \alpha = \alpha_1 + \cdots + \alpha_n
        } } {}
        U (\alpha_1, \dotsc, \alpha_n; \tau, \tilde{\tau}) \cdot
        \bbL^{-\chi (\alpha_1, \dotsc, \alpha_n)} \cdot
        \upJ_{\alpha_1} (\tau) \cdots
        \upJ_{\alpha_n} (\tau) \ , \\[1ex]
        \upI^\sd_\theta (\tilde{\tau}) & =
        \sum_{ \leftsubstack[7em]{
            & n \geq 0; \, \alpha_1, \dotsc, \alpha_n \in C (\calA), \,
            \rho \in C^\sd (\calA) \colon \\[-.5ex]
            & \theta = \bar{\alpha}_1 + \cdots + \bar{\alpha}_n + \rho 
        } } {}
        S^\sd (\alpha_1, \dotsc, \alpha_n; \tau, \tilde{\tau}) \cdot
        \bbL^{-\chi^\sd (\alpha_1, \dotsc, \alpha_n, \rho)} \cdot {} 
        \notag \\*[-4ex]
        \label{eq-wcf-isd}
        && \mathllap{ \upI_{\alpha_1} (\tau) \cdots
        \upI_{\alpha_n} (\tau) \cdot
        \upI^\sd_{\rho} (\tau) \ , } & \hspace{.1em} \\[2ex]
        \upJ^\sd_\theta (\tilde{\tau}) & =
        \sum_{ \leftsubstack[7em]{
            & n \geq 0; \, \alpha_1, \dotsc, \alpha_n \in C (\calA), \,
            \rho \in C^\sd (\calA) \colon \\[-.5ex]
            & \theta = \bar{\alpha}_1 + \cdots + \bar{\alpha}_n + \rho 
        } } {}
        U^\sd (\alpha_1, \dotsc, \alpha_n; \tau, \tilde{\tau}) \cdot
        \bbL^{-\chi^\sd (\alpha_1, \dotsc, \alpha_n, \rho)} \cdot {} 
        \notag \\*[-4ex]
        \label{eq-wcf-jsd}
        && \mathllap{ \upJ_{\alpha_1} (\tau) \cdots
        \upJ_{\alpha_n} (\tau) \cdot
        \upJ^\sd_{\rho} (\tau) \ , }
    \end{align}
    where $\chi (\alpha_1, \dotsc, \alpha_n)$ and
    $\chi^\sd (\alpha_1, \dotsc, \alpha_n, \rho)$
    are given by \crefrange{eq-def-multi-chi}{eq-def-multi-chi-sd}.
    \QED
\end{theorem}

The ordinary cases, \crefrange{eq-wcf-i}{eq-wcf-j},
were obtained in \cite[Theorem~6.8]{Joyce2008IV}.

\subsection{Numerical invariants}
\label{sect-num-inv}

\paragraph{}
\label{para-sit-num-inv}

Next, we define \emph{numerical enumerative invariants}
by taking the Euler characteristics 
of the $\upJ$ and $\upJ^\sd$ invariants in \cref{sect-motivic-inv},
coming from the $\epsilon$ and $\epsilon^\sd$ functions
defined in \cref{sect-delta-epsilon}.
The fact that this is well-defined relies on the
\emph{no-pole theorems}, \cref{thm-no-pole,thm-no-pole-sd},
which do not hold for the $\delta, \delta^\sd$ functions.
We will define the following invariants:

\begin{itemize}
    \item
        Numerical invariants
        $\chiJ_\alpha (\tau), \chiJ^\sd_{\theta} (\tau) \in \bbQ$,
        defined as the Euler characteristics of the
        $\upJ$ and $\upJ^\sd$ invariants.
    \item 
        Numerical invariants
        $\DT^\nai_\alpha (\tau), \DT^\sdnai_\theta (\tau) \in \bbQ$,
        which only differ from the $\chiJ$ and $\chiJ^\sd$ invariants
        by a sign.
\end{itemize}

Throughout, we assume the conditions
\tagref{SdMod} and \tagref{Fin1},
as in \cref{para-sit-motivic-inv}.
We also assume that $\calA$ satisfies \tagref{Spl},
which is important for the no-pole theorems.

\paragraph{The no-pole theorem, ordinary case.}
\label[theorem]{thm-no-pole}

Joyce \cite[Theorem~8.7]{Joyce2007III}
proved the following \emph{no-pole theorem} for the
stack function $\epsilon_\alpha (\tau)$ and the invariant $\upJ_\alpha (\tau)$,
stating that the motive $(\bbL - 1) \cdot \upJ_\alpha (\tau)$
has no poles at $\bbL = 1$.

\begin{theorem*}
    In the situation of \cref{para-sit-num-inv},
    let $\tau$ be a weak stability condition on $\calA$
    satisfying \tagref{Stab}.
    Then for any $\alpha \in C (\calA)$, the element
    \[ \epsilon_\alpha (\tau) \in \SF (\calM_\alpha) \]
    has pure virtual rank~$1$, in the sense of \cref{sect-vrp}.
    In particular, we have
    \begin{equation}
        (\bbL - 1) \cdot \upJ_\alpha (\tau) \in \Mhat_\bbK^\circ \ .
    \end{equation}
\end{theorem*}

\paragraph{The no-pole theorem, self-dual case.}
\label[theorem]{thm-no-pole-sd}

We state the following no-pole theorem
for the stack function $\epsilon^\sd_\theta (\tau)$
and the invariant $\upJ^\sd_\theta (\tau)$.
This is one of the main results in this paper,
and the proof will be given in \cref{sect-proof-no-pole}.

\begin{theorem*}
    In the situation of \cref{para-sit-num-inv},
    let $\tau$ be a self-dual weak stability condition on $\calA$
    satisfying \tagref{Stab}.
    Then for any $\theta \in C^\sd (\calA)$, the element
    \[ \epsilon^\sd_\theta (\tau) \in \SF (\calM^\sd_\theta) \]
    has pure virtual rank~$0$, in the sense of \cref{sect-vrp}.
    In particular, we have
    \begin{equation}
        \upJ^\sd_\theta (\tau) \in \Mhat_\bbK^\circ \ .
    \end{equation}
\end{theorem*}

\paragraph{The Euler characteristic.}
\label{para-euler-char}

Let $\chi \colon \Mhat_\bbK^\circ \to \bbQ$ be the Euler characteristic map,
which can be defined first on $K_0 (\mathrm{Var}_{\bbK})$,
then extended to $\Mhat_\bbK^\circ$ by setting $\chi (\bbL) = 1$.
The map $\chi$ factors through the ring $\Mtilde_\bbK$
in \cref{def-motive-varieties}.

The symbol $\chi$ here is not to be confused with
the Euler form $\chi (-, -)$ defined in \cref{para-euler},
which will also appear in the following.

Alternatively, $\chi$ can be taken to be any map
with $\chi (\bbL) = 1$,
as discussed in Joyce~\cite[Examples~6.2--6.4]{Joyce2007Stack},
and the theory should still work,
but we do not make this generalization here.

\paragraph{Numerical invariants.}
\label{para-num-inv}

In the situation of \cref{para-sit-num-inv},
let $\tau$ be a self-dual weak stability condition on $\calA$
satisfying \tagref{Stab}.

For classes $\alpha \in C (\calA)$ and $\theta \in C^\sd (\calA)$,
define the numerical enumerative invariants
$\chiJ_\alpha (\tau) \in \bbQ$ and
$\chiJ^\sd_\theta (\tau) \in \bbQ$
as the Euler characteristics,
\begin{align}
    \label{eq-def-j}
    \chiJ_\alpha (\tau) & = \chi ((\bbL - 1) \cdot \upJ_\alpha (\tau))
    = \int_{\calM_\alpha} (\bbL - 1) \, \epsilon_\alpha (\tau) \, d \chi \ ,
    \\
    \label{eq-def-jsd}
    \chiJ^\sd_\theta (\tau) & = \chi (\upJ^\sd_\theta (\tau))
    = \int_{\calM^\sd_\theta} \epsilon^\sd_\theta (\tau) \, d \chi \ .
\end{align}
These are well-defined by the no-pole theorems,
\cref{thm-no-pole,thm-no-pole-sd}.

Define the \emph{na\"{i}ve Donaldson--Thomas invariants}
$\DT^\nai_\alpha (\tau)$,
$\DT^\sdnai_\theta (\tau) \in \bbQ$ by
\begin{align}
    \label{eq-def-dt-nai}
    \DT^\nai_\alpha (\tau) & =
    -(-1)^{\chi (\alpha, \alpha)} \cdot \chiJ_\alpha (\tau) \ , \\
    \label{eq-def-dt-sd-nai}
    \DT^\sdnai_\theta (\tau) & =
    (-1)^{\chi^\sd (j (\theta), 0)} \cdot \chiJ^\sd_\theta (\tau) \ .
\end{align}

\begin{remark}
    \label{rem-dt-nai}
    We use the name \emph{na\"{i}ve Donaldson--Thomas invariants},
    because it is only conceptually correct for smooth moduli stacks.
    In general, one needs to introduce a weighting by the
    \emph{Behrend function},
    as in Joyce--Song~\cite{JoyceSong2012},
    to obtain the correct definition of the Donaldson--Thomas invariants.
    In the smooth case, the Behrend function is a constant function
    equal to $\pm 1$, hence the signs in 
    \crefrange{eq-def-dt-nai}{eq-def-dt-sd-nai}.
    We hope to explore the weighted version in the future.
\end{remark}

\paragraph{Wall-crossing.}
\label[theorem]{thm-wcf-num}

We also write down wall-crossing formulae
for the numeric invariants.

\begin{theorem*}
    \allowdisplaybreaks
    In the situation of\/ \cref{para-sit-num-inv},
    with the extra assumptions \tagref{Fin}, \tagref{Ext}, and \tagref{SdExt},
    let $\tau, \tilde{\tau}, \hat{\tau}$ be self-dual weak stability conditions
    satisfying conditions in \cref{para-sit-mot-wcf}.
    
    Then, for classes $\alpha \in C (\calA)$ and $\theta \in C^\sd (\calA),$
    we have the wall-crossing formulae
    \begin{align}
        \chiJ_\alpha (\tilde{\tau})
        & = 
        \sum_{ \leftsubstack[5em]{
            & n > 0; \, \alpha_1, \dotsc, \alpha_n \in C (\calA) \colon \\[-.5ex]
            & \alpha = \alpha_1 + \cdots + \alpha_n
        } } {}
        \tilde{U} (\alpha_1, \dotsc, \alpha_n; \tau, \tilde{\tau}) \cdot
        \tilde{\chi} (\alpha_1, \dotsc, \alpha_n) \cdot {}
        \hspace{-3em} \notag \\*[-4ex]
        \label{eq-wcf-j-hash}
        && \mathllap{
        \chiJ_{\alpha_1} (\tau) \cdots
        \chiJ_{\alpha_n} (\tau) \ ,}
        \\[2ex]
        \DT^\nai_\alpha (\tilde{\tau})
        & = 
        \sum_{ \leftsubstack[5em]{
            & n \geq 0; \, \alpha_1, \dotsc, \alpha_n \in C (\calA) \colon \\[-.5ex]
            & \alpha = \alpha_1 + \cdots + \alpha_n
        } } {}
        \tilde{U} (\alpha_1, \dotsc, \alpha_n; \tau, \tilde{\tau}) \cdot
        (-1)^{\bar{\chi} (\alpha_1, \dotsc, \alpha_n)} \cdot
        \tilde{\chi} (\alpha_1, \dotsc, \alpha_n) \cdot {}
        \hspace{-3em} \notag \\*[-4ex]
        \label{eq-wcf-dt}
        && \mathllap{
        \DT^\nai_{\alpha_1} (\tau) \cdots
        \DT^\nai_{\alpha_n} (\tau) \ , }
        & \hspace{.1em} \\[2ex]
        \chiJ^\sd_\theta (\tilde{\tau})
        & = 
        \sum_{ \leftsubstack[7.5em]{
            \\[-3ex]
            & n \geq 0; \, m_1, \dotsc, m_n > 0; \\[-.5ex]
            & \alpha_{1,1}, \dotsc, \alpha_{1,m_1}; \dotsc;
            \alpha_{n,1}, \dotsc, \alpha_{n,m_n} \in C (\calA); \,
            \rho \in \smash{C^\sd (\calA)} \colon \\[-.5ex]
            & \theta = (\bar{\alpha}_{1,1} + \cdots + \bar{\alpha}_{1,m_1})
            + \cdots + (\bar{\alpha}_{n,1} + \cdots + \bar{\alpha}_{n,m_n}) + \rho
        } } {}
        \tilde{U}^\sd (\alpha_{1,1}, \dotsc, \alpha_{1,m_1}; \dotsc;
        \alpha_{n,1}, \dotsc, \alpha_{n,m_n}; \tau, \tilde{\tau}) \cdot {}
        \hspace{-2em} \notag \\*[1ex]
        & \hspace{1em}
        \tilde{\chi}^\sd (\alpha_{1,1}, \dotsc, \alpha_{1,m_1}; \dotsc;
        \alpha_{n,1}, \dotsc, \alpha_{n,m_n}; \rho) \cdot {}
        \hspace{-3em} \notag \\*[1ex] 
        \label{eq-wcf-j-sd-hash}
        && \mathllap {
        \bigl( \chiJ_{\alpha_{1,1}} (\tau) \cdots
        \chiJ_{\alpha_{1,m_1}} (\tau) \bigr) \cdots
        \bigl( \chiJ_{\alpha_{n,1}} (\tau) \cdots
        \chiJ_{\alpha_{n,m_n}} (\tau) \bigr) \cdot
        \chiJ^\sd_\rho (\tau) \ , }
        & \hspace{.1em} \\[2ex]
        \DT^\sdnai_\theta (\tilde{\tau})
        & = 
        \sum_{ \leftsubstack[7.5em]{
            \\[-3ex]
            & n \geq 0; \, m_1, \dotsc, m_n > 0; \\[-.5ex]
            & \alpha_{1,1}, \dotsc, \alpha_{1,m_1}; \dotsc;
            \alpha_{n,1}, \dotsc, \alpha_{n,m_n} \in C (\calA); \,
            \rho \in \smash{C^\sd (\calA)} \colon \\[-.5ex]
            & \theta = (\bar{\alpha}_{1,1} + \cdots + \bar{\alpha}_{1,m_1})
            + \cdots + (\bar{\alpha}_{n,1} + \cdots + \bar{\alpha}_{n,m_n}) + \rho
        } } {}
        \tilde{U}^\sd (\alpha_{1,1}, \dotsc, \alpha_{1,m_1}; \dotsc;
        \alpha_{n,1}, \dotsc, \alpha_{n,m_n}; \tau, \tilde{\tau}) \cdot {}
        \hspace{-3em} \notag \\*[1ex]
        & 
        \mathrlap{ (-1)^{\bar{\chi}^\sd (\alpha_{1,1}, \dotsc, \alpha_{n,m_n}, \rho)} \cdot
        \tilde{\chi}^\sd (\alpha_{1,1}, \dotsc, \alpha_{1,m_1}; \dotsc;
        \alpha_{n,1}, \dotsc, \alpha_{n,m_n}; \rho) \cdot {} }
        \notag \\*[1ex] 
        \label{eq-wcf-dt-sd}
        && \mathllap {
        \bigl( \DT^\nai_{\alpha_{1,1}} (\tau) \cdots
        \DT^\nai_{\alpha_{1,m_1}} (\tau) \bigr) \cdots
        \bigl( \DT^\nai_{\alpha_{n,1}} (\tau) \cdots
        \DT^\nai_{\alpha_{n,m_n}} (\tau) \bigr) \cdot
        \DT^\sdnai_\rho (\tau) \ , } 
        & \hspace{.1em}
    \end{align}
    where $\bar{\chi} (\alpha_1, \dotsc, \alpha_n)$
    and $\bar{\chi}^\sd (\alpha_1, \dotsc, \alpha_n, \rho)$
    are as in \cref{eq-def-multi-chi,eq-def-multi-chi-sd},
    but with $\bar{\chi}$ in place of~$\chi$.
    The coefficients $\tilde{\chi} ({\cdots})$ and $\tilde{\chi}^\sd ({\cdots})$ 
    are given by \cref{eq-def-chi-tilde,eq-def-chi-tilde-sd}.
\end{theorem*}

\begin{proof}
    Note that we have the relations
    \begin{align}
        \label{eq-xi-psi-epsilon}
        ((\bbL - 1) \cdot \upJ_\alpha (\tau)) \cdot \omega_\alpha & \sim
        \Xi \circ \Psi (\epsilon_\alpha (\tau)) \ , \\
        \label{eq-xi-psi-epsilon-sd}
        \upJ^\sd_\theta (\tau) \cdot \omega^\sd_\theta & \sim
        \Xi^\sd \circ \Psi^\sd (\epsilon^\sd_\theta (\tau)) \ ,
    \end{align}
    in $\Omega (\calA)$ and $\Omega^\sd (\calA)$, respectively,
    where $\sim$ denotes equality up to multiplying by a power of $\bbL^{1/2}$,
    and $\Xi$, $\Psi$, $\Xi^\sd$, $\Psi^\sd$
    are maps defined in~\cref{sect-motivic-int-2}.
    The extra factor $(\bbL - 1)$ comes from the difference
    between the elements $\lambda_\alpha$ and $\tilde{\lambda}_\alpha$, as in \cref{eq-def-lambda-tilde}.
    The invariants $\chiJ_\alpha (\tau)$ and $\chiJ^\sd_\theta (\tau)$
    were defined by applying the Euler characteristic map
    to the left-hand sides of \cref{eq-xi-psi-epsilon,eq-xi-psi-epsilon-sd},
    respectively.

    Now, equations
    \cref{eq-wcf-j-hash,eq-wcf-j-sd-hash}
    follow from applying \cref{thm-mot-int-lie-homo}
    to~\cref{eq-u-tilde,eq-comb-main}, respectively,
    and then using \cref{eq-lem-chi-tilde,eq-lem-chi-tilde-sd}
    to compute the coefficients.
    Equations \cref{eq-wcf-dt,eq-wcf-dt-sd} then follow immediately.
\end{proof}

%% file: quiver.tex
\label{sect-quiver}

In this \lcnamecref{sect-quiver}, we study enumerative invariants counting
\emph{self-dual representations} of a \emph{self-dual quiver},
which is one of the main examples of our theory.

The notion of a self-dual quiver was introduced by
Derksen and Weyman~\cite{DerksenWeyman2002},
and studied by Young~\cite{Young2015, Young2020}
in the context of enumerative geometry.
They are a reasonable notion of a \emph{$G$-quiver},
where $G$ is a type B/C/D algebraic group.
Self-dual representations of a self-dual quiver
are thus analogous to principal $G$-bundles for such~$G$.

We start by giving some basic definitions,
and then we verify the various assumptions needed for our theory.
We also provide an algorithm to compute our invariants,
and we carry out the computation in some examples.

\subsection{Self-dual quivers}

\paragraph{Quivers.}

A \emph{quiver} is a quadruple $(Q_0, Q_1, s, t)$,
where $Q_0\,, Q_1$ are finite sets,
thought of as the sets of vertices and arrows,
and $s, t \colon Q_1 \to Q_0$ are maps,
sending each arrow to its source and target.

A \emph{representation} of $Q$ over $\bbK$ is the data
$E = ((E_i)_{i \in Q_0} \, , (e_a)_{a \in Q_1})$,
where each $E_i$ is a finite-dimensional $\bbK$-vector space,
and each $e_a \colon E_{s (a)} \to E_{t (a)}$ is a linear map.
A \emph{morphism} of representations $h \colon E \to F$
is a collection of maps $(h_i \colon E_i \to F_i)_{i \in Q_0}$\,,
such that $f_a \circ h_{s (a)} = h_{t (a)} \circ e_a$
for all $a \in Q_1$\,.

Let $\Mod (\bbK Q)$ denote the $\bbK$-linear
abelian category of $Q$-representations over $\bbK$.
Here, $\bbK Q$ is the \emph{path algebra} of $Q$,
whose finite-dimensional representations
are the same as $Q$-representations over $\bbK$.

For the extra data as in \tagref{ExCat}, define
\[
    C^\circ (Q) = \bbN^{Q_0} \ ,
\]
and for each representation $E$ of $Q$,
define $\llbr E \rrbr = \smash{(\dim E_i)_{i \in Q_0}} \in C^\circ (Q)$.
Let $C (Q) = C^\circ (Q) \setminus \{ 0 \}$.

The \emph{opposite quiver} of $Q$ is the quiver
$Q^\op = (Q_0, Q_1, t, s)$,
which is the quiver obtained from $Q$
by inverting the directions of all arrows.

\paragraph{Quivers with relations.}
\label{para-quiver-rel}

We take the following notion of a \emph{quiver with relations}
from Joyce~\cite[\S6]{Joyce2021}.

A \emph{quiver with relations} over $\bbK$ is a tuple
$\smash{\breve{Q}} = (Q_0, Q_1, Q_2, s, t, p, q, r)$,
where $Q = (Q_0, Q_1, s, t)$ is a quiver;
$Q_2$ is a finite set, interpreted as the set of relations;
$p, q \colon Q_2 \to Q_0$ are maps,
indicating the source and target of paths on which the relations are to be imposed;
$r \colon Q_2 \to \bbK Q$ is a map,
such that for any $b \in Q_2$\,, we have $r (b) \in \bbK Q_{p (b), q (b)}$,
the linear subspace in $\bbK Q$ spanned by paths
from $p (b)$ to $q (b)$.

Let $I \subset \bbK Q$ denote the two-sided ideal
generated by $r (b)$ for $b \in Q_2$\,.

A \emph{representation} of such a quiver with relations
is a finite-dimensional representation
of the $\bbK$-algebra $\bbK \breve{Q} = \bbK Q / I$.
Let $\Mod (\bbK \breve{Q})$ denote the
$\bbK$-linear abelian category of representations of $\breve{Q}$.
It can be seen as a full subcategory of $\Mod (\bbK Q)$.

The \emph{opposite quiver with relations} of $\breve{Q}$
is $\breve{Q}^\op = (Q_0, Q_1, Q_2, t, s, q, p, r^\op)$,
where $r^\op \colon Q_2 \to \bbK Q^\op$
is the composition of $r$ with the isomorphism of vector spaces $\bbK Q \simeq \bbK Q^\op$
sending a path to its reverse path.

\paragraph{Self-dual quivers.}
\label{para-sd-quiver}

A \emph{self-dual quiver}
is a quadruple $(Q, \sigma, u, v)$, where
\begin{itemize}
    \item 
        $Q = (Q_0, Q_1, s, t)$ is a quiver.
    \item 
        $\sigma = (\sigma_i \colon Q_i \to Q_i)_{i=0,1}$
        consists of two maps,
        such that $\sigma_i^2 = \id_{\smash{Q_i}}$ for $i = 0, 1$,
        \,$\sigma_0 \circ s = t \circ \sigma_1$\,,
        and $\sigma_0 \circ t = s \circ \sigma_1$\,.
        This is called a \emph{contravariant involution} of~$Q$,
        and can be seen as an isomorphism $\sigma \colon Q \simto Q^\op$.
        We also write these maps as $(-)^\vee \colon Q_i \to Q_i$ for $i = 0, 1$.
    \item
        $u \colon Q_0 \to \{ \pm1 \}$ and $v \colon Q_1 \to \{ \pm1 \}$ are functions,
        such that $u_i = u_{\smash{i^\vee}}$ for all $i \in Q_0$\,,
        and $v_a \, v_{\smash{a^\vee}} = u_{s (a)} \, u_{t (a)}$
        for all $a \in Q_1$\,.
\end{itemize}

\paragraph{Self-dual quivers with relations.}
\label{para-sd-quiver-rel}

A \emph{self-dual quiver with relations}
is a quintuple $(\breve{Q}, \sigma, u, v, w)$, where
\begin{itemize}
    \item 
        $\breve{Q} = (Q_0, Q_1, Q_2, s, t, p, q, r)$ is a quiver with relations.
    \item 
        $\sigma = (\sigma_i \colon Q_i \to Q_i)_{i=0,1,2}$
        consists of three maps,
        such that $\sigma_i^2 = \id_{\smash{Q_i}}$ for all $i$,
        \,$\sigma_0 \circ s = t \circ \sigma_1$\,,
        \,$\sigma_0 \circ t = s \circ \sigma_1$\,,
        \,$\sigma_0 \circ p = q \circ \sigma_2$\,,
        and $\sigma_0 \circ q = p \circ \sigma_2$\,.
        These are also written as $(-)^\vee \colon Q_i \to Q_i$ for $i = 0, 1, 2$.
    \item
        $u \colon Q_0 \to \{ \pm1 \}$, $v \colon Q_1 \to \{ \pm1 \}$,
        and $w \colon Q_2 \to \{ \pm1 \}$ are functions,
        such that $u, v$ satisfy the conditions above,
        and $w_b \, w_{b^\vee} = u_{p (b)} \, u_{q (b)}$
        for all $b \in Q_2$\,.
        Furthermore, if $(-)^\vee \colon \bbK Q \simto \bbK Q^\op$
        is the map that sends a path to its image under $\sigma$,
        multiplied by the product of $v_a$ for all edges $a$ in that path,
        then we require that $r (b)^\vee = w_b \cdot r^\op (b^\vee)$
        for all $b \in Q_2$\,.
\end{itemize}
In particular, a self-dual quiver is also
a self-dual quiver with relations with $Q_2 = \varnothing$.

\paragraph{The self-dual structure.}
\label{para-quiver-sd-str}

Let $(\breve{Q}, \sigma, u, v, w)$ be a self-dual quiver with relations.
Define a self-dual structure
on the $\bbK$-linear abelian category $\smash{\Mod (\bbK \breve{Q})}$ as follows.

For a $\breve{Q}$-representation
$E = (\smash{(E_i)_{i \in Q_0}}, \smash{(e_a)_{a \in Q_1}})$,
define its dual representation
$E^\vee = (\smash{(E'_i)_{i \in Q_0}}, \smash{(e'_a)_{a \in Q_1}})$
by
\begin{equation}
    E'_i = (E_{i^\vee})^\vee \ , \qquad
    e'_a = v_a \cdot (e_{a^\vee})^\vee \ ,
\end{equation}
where the duals $(-)^\vee$ outside the brackets
are the usual vector space duals.
Define the natural isomorphism $\eta_E \colon E^{\vee\vee} \simto E$ by
\begin{equation}
    (\eta_E)_i = u_i \cdot \mathrm{ev}_i^{-1}
\end{equation}
for all $i \in Q_0$\,,
where $\mathrm{ev}_i \colon E_i \simto E_i^{\vee\vee}$
is the evaluation isomorphism.
This defines a self-dual structure on $\Mod (\bbK \breve{Q})$.

A \emph{self-dual representation} of $\breve{Q}$
is a self-dual object in $\Mod (\bbK \breve{Q})$.
If $(E, \phi)$ is such a self-dual object,
then $E_{i^\vee} \simeq (E_i)^\vee$ for all $i \in Q_0$\,.
In particular, if $i = i^\vee$,
this isomorphism equips $E_i$ with
an orthogonal or symplectic structure,
depending on the sign $u_i$\,.

Define 
\begin{equation}
    C^\sd (Q) = \biggl\{ \,
        \theta \in C^\circ (Q) \biggm|
        \begin{aligned}[c]
            & \theta_i = \theta_{\smash{i^\vee}} \text{ for all } i, \\[-.5ex]
            & 2 \mid \theta_i \text{ if } i = i^\vee \text{ and } u_i = -1
        \end{aligned}
    \, \biggr\} \ .
\end{equation}
This is the set of all self-dual dimension vectors,
where the second line of the condition is due to the 
non-existence of odd-dimensional symplectic vector spaces.

\paragraph{Stability functions.}

We recall the notion of \emph{stability functions}
on a quiver, which is a class of stability conditions,
although not all stability conditions arise in this way.

Let $Q$ be a quiver, possibly with relations.
A \emph{stability function} on $Q$ is a function
\[
    \tau \colon Q_0 \longrightarrow \bbQ \ .
\]
Given such a stability function, we can define
a stability condition on $\cat{Mod} (\bbK Q)$,
denoted also by $\tau \colon C (Q) \to \bbQ$,
by setting
\begin{equation}
    \tau (\alpha) =
    \frac{\sum_{i \in Q_0} \tau (i) \, \alpha_i}{\sum_{i \in Q_0} \alpha_i}
\end{equation}
for all $\alpha \in C (Q)$.

Similarly, if $\breve{Q}$ is a quiver with relations,
a \emph{stability function} on $\breve{Q}$
is a stability function on its underlying quiver,
and this defines a stability condition on $\cat{Mod} (\bbK \breve{Q})$.

Now, suppose that $(Q, \sigma, u, v)$ is a self-dual quiver.
We say that a stability function $\tau$ on $Q$ is \emph{self-dual}, if
\begin{equation}
    \tau (i^\vee) = -\tau (i)
\end{equation}
for all $i \in Q_0$\,.
In this case, the corresponding stability condition on $\cat{Mod} (\bbK Q)$
is a self-dual stability condition.

\subsection{Verifying the assumptions}

\paragraph{}

Next, we verify the various assumptions
needed for our enumerative invariant theory.
These include the following:
\begin{itemize}
    \item 
        \tagref{Mod} and \tagref{Fin}
        for $\calA = \Mod (\bbK \breve{Q})$,
        where $\breve{Q}$ is a quiver with relations.
    \item 
        \tagref{SdMod} for $\calA = \Mod (\bbK \breve{Q})$,
        where $\breve{Q}$ is a self-dual quiver with relations.
    \item 
        \tagref{Ext} for $\calA = \Mod (\bbK Q)$,
        where $Q$ is a quiver, with no relations.
    \item 
        \tagref{SdExt} for $\calA = \Mod (\bbK Q)$,
        where $Q$ is a self-dual quiver, with no relations.
\end{itemize}
We also verify
\begin{itemize}
    \item 
        \tagref{Stab} for any weak stability condition $\tau$
        on $\calA = \Mod (\bbK \breve{Q})$, where $\breve{Q}$ is a quiver with relations.
\end{itemize}
In particular, in the case of self-dual quivers with no relations,
all assumptions mentioned above are satisfied,
and hence, all the results in \crefrange{sect-alg}{sect-invariants} apply.

When there are relations, the condition \tagref{Ext},
and hence \tagref{SdExt}, fails,
but all the invariants in \cref{sect-invariants} are still well-defined,
as their definitions do not depend on this.
However, in this case,
the wall-crossing formulae may no longer hold.

\paragraph{Defining the moduli stacks.}
\label{para-moduli-quiver}

Let $\breve{Q}$ be a quiver with relations over $\bbK$,
and let $\calA = \smash{\Mod (\bbK \breve{Q})}$.
We verify the conditions \tagref{Mod} and \tagref{Fin}.

For \tagref{Mod}, a categorical moduli stack $\+{\calM}$ of objects in $\calA$
was constructed in \cite[\S10.2]{Joyce2006I},
and can be described as follows.
For a $\bbK$-scheme $U$, the exact category $\+{\calM} (U)$
is the category of pairs $(E, \rho)$,
with $E$ a locally free sheaf on $U$ of finite rank,
and $\rho \colon \bbK \breve{Q} \to \mathrm{End} (E)$
is a map of $\bbK$-algebras.
It was shown in \cite[\S\S10.2--10.4]{Joyce2006I} that $\+{\calM}$ is
an exact $\bbK$-stack,
and that $\calM$ and $\calM^\ex$ 
are algebraic $\bbK$-stacks locally of finite type.
To see that $\calM^I$ is algebraic, note that it
can be identified with the moduli stack for $\Mod (\bbK \breve{Q}')$
for a certain quiver with relations $\breve{Q}'$.
This verifies \tagref{Mod}.

The condition \tagref{Fin} is automatically true,
since all moduli stacks in question are already of finite type.

Now, assume that $\breve{Q}$ is equipped with a self-dual quiver structure
$(\breve{Q}, \sigma, u, v, w)$,
and we verify \tagref{SdMod}.

We have constructed in \cref{para-quiver-sd-str}
an induced self-dual structure on $\calA$.
One can define a self-dual structure on $\+{\calM} (U)$
for any $\bbK$-scheme $U$ analogously,
using the description of $\+{\calM} (U)$ as above.
This defines a self-dual structure on $\+{\calM}$
and verifies \tagref{SdMod}.

Note that by the above construction, there are \emph{universal families}
of representations of $\breve{Q}$ over $\calM$,
consisting of universal vector bundles $\calU_i \to \calM$
for $i \in Q_0$\,,
and morphisms $\calU_{s(a)} \to \calU_{t(a)}$
for $a \in Q_1$\,,
satisfying the relations specified by $Q_2$\,.

\paragraph{Explicit description, ordinary case.}
    
The moduli stack $\calM$ of representations
of a quiver $Q$ without relations
has the following explicit description.
For each $\alpha \in C^\circ (Q)$, we have a canonical isomorphism
\begin{equation}
    \label{eq-moduli-quiver-explicit}
    \calM_\alpha \simeq [V_\alpha / G_\alpha] \ ,
\end{equation}
where
\begin{align}
    V_\alpha & = \prod_{a \in Q_1} \Hom (\bbK^{\alpha_{s (a)}}, \bbK^{\alpha_{t (a)}}) \ , \\
    G_\alpha & = \prod_{i \in Q_0} \GL (\bbK^{\alpha_i}) \ ,
\end{align}
and the (right) action of $G_\alpha$ on $V_\alpha$ is given as follows.
For $g = (g_i)_{i \in Q_0} \in G_\alpha$ and $e = (e_a)_{a \in Q_1} \in V_\alpha$\,, define
$g \cdot e = \smash{(e'_a)_{a \in Q_1}}$\,, where
$\smash{e'_a} = g_{\smash{t (a)}}^{\smash{-1}} \circ e_a \circ g_{s (a)}$\,.

\paragraph{Explicit description, self-dual case.}

Now, let $(Q, \sigma, u, v)$ be a self-dual quiver with no relations.
Similarly, there is an explicit description of $\calM^\sd$, as follows.

For convenience in notation,
we choose total orders $\leq$ on $Q_0$ and $Q_1$\,,
so that we can write
\begin{align}
    Q_0 & = Q_0^+ \sqcup Q_0^- \sqcup Q_0^\tria \sqcup Q_0^{\tria\vee} \ , \\
    Q_1 & = Q_1^+ \sqcup Q_1^- \sqcup Q_1^\tria \sqcup Q_1^{\tria\vee} \ , 
\end{align}
where
\begin{align}
    Q_0^\pm & = \{ i \in Q_0 \mid i = i^\vee, \, u_i = \pm 1 \} \ , \\
    Q_0^\tria & = \{ i \in Q_0 \mid i < i^\vee \} \ , \\
    Q_1^\pm & = \{ a \in Q_1 \mid a = a^\vee, \, u_{t (a)} \, v_a = \pm 1 \} \ , \\
    Q_1^\tria & = \{ a \in Q_1 \mid a < a^\vee \} \ .
\end{align}
This is purely for notational reasons,
and the underlying theory will not depend on this choice.

For each $\theta \in C^\sd (Q)$, we have a canonical isomorphism
\begin{equation}
    \label{eq-msd-quiver-explicit}
    \calM^\sd_\theta \simeq [V^\sd_\theta / G_\theta^\sd] \ ,
\end{equation}
where
\begin{align}
    \label{eq-quiv-vsd}
    V^\sd_\theta & =
    \prod_{a \in Q_1^\tria} {}
    \Hom (\bbK^{\theta_{s (a)}}, \bbK^{\theta_{t (a)}}) \times
    \prod_{a \in Q_1^+} {}
    \Sym^2 (\bbK^{\theta_{t (a)}}) \times
    \prod_{a \in Q_1^-} {}
    {\wedge}^2 (\bbK^{\theta_{t (a)}}), 
    \\
    \label{eq-quiv-gsd}
    G^\sd_\theta & =
    \prod_{i \in Q_0^\tria} {}
    \GL (\theta_i) \times
    \prod_{i \in Q_0^+} {}
    \upO (\theta_i) \times
    \prod_{i \in Q_0^-} {}
    \Sp (\theta_i),
\end{align}
and it is not difficult to write down
the action of $G^\sd_\theta$ on $V^\sd_\theta$,
so that the quotient produces the expected moduli stack.

\paragraph{Extension bundles.}
\label{cons-a0-a1}

We now verify \tagref{Ext}
for quivers without relations,
and \tagref{SdExt}
for self-dual quivers without relations.

Let $Q = (Q_0, Q_1, s, t)$ be a quiver,
and let $E, F \in \Mod (\bbK Q)$.
Define $\bbK$-vector spaces
\begin{align}
    \label{eq-def-a0}
    A^0 (E, F) & =
    \bigoplus_{i \in Q_0}
    \Hom (E_i \, , F_i) \ , \\
    \label{eq-def-a1}
    A^1 (E, F) & =
    \bigoplus_{a \in Q_1}
    \Hom (E_{s (a)} \, , F_{t (a)}) \ .
\end{align}
Define a map $d \colon A^0 (E, F) \to A^1 (E, F)$ as follows.
For each $\gamma = (\gamma_i)_{i \in Q_0} \in A^0 (E, F)$,
define $d \gamma \in A^1 (E, F)$ by
\begin{equation}
    \label{eq-def-d-gamma}
    (d \gamma)_a = \gamma_{t(a)} \circ e_a - f_a \circ \gamma_{s(a)}
\end{equation}
for all $a \in Q_1$\,,
where $e_a\,, f_a$ are structure maps of $E, F$.

As in, for example, Crawley-Boevey~\cite[p.\,8]{CrawleyBoevey1992},
the cohomology of the two-term complex
\begin{equation}
    \cdots \to 0 \to
    A^0 (E, F) \overset{d}{\longrightarrow}
    A^1 (E, F) \to 0 \to \cdots
\end{equation}
computes $\Ext^0 (E, F)$ and $\Ext^1 (E, F)$,
and this can be described as follows.

For each $\alpha = (\alpha_a)_{a \in Q_1} \in A^1 (E, F)$,
we define a representation $G_\alpha \in \Mod (\bbK Q)$ as follows.
For each $i \in Q_0$, let $(G_\alpha)_i = F_i \oplus E_i$\,.
For each $(i \xrightarrow{\smash{a}} j) \in Q_1$, define the structure map
$(g_\alpha)_a \colon F_i \oplus E_i \to F_j \oplus E_j$ by
\begin{equation}
    (g_\alpha)_a = \begin{pmatrix}
        f_a & \alpha_a \\
        0 & e_a
    \end{pmatrix} ,
\end{equation}
where $e_a \, , f_a$ are structure maps of $E$ and $F$.
We have a short exact sequence
\begin{equation}
    0 \to F \longrightarrow G_\alpha \longrightarrow E \to 0
\end{equation}
in $\Mod (\bbK Q)$,
where the maps are defined in the obvious way.
This defines an extension of $E$ by $F$.
Moreover, two elements $\alpha, \alpha' \in A^1 (E, F)$
give rise to isomorphic extensions,
if and only if $\alpha' - \alpha = d \gamma$ for some $\gamma \in A^0 (E, F)$.
Indeed, if this is the case, then the map
$\eta \colon G_\alpha \to G_{\alpha'}$ given by
\begin{equation}
    \eta_i = \begin{pmatrix}
        \id_{F_i} & \gamma_i \\
        0 & \id_{E_i}
    \end{pmatrix}
\end{equation}
is an isomorphism that is compatible with the inclusions from $F$
and the projections to $E$. The condition that $d \gamma = \alpha' - \alpha$
ensures that
\begin{equation}
    \eta_{t (a)} \circ (g_\alpha)_a =
    (g_{\alpha'})_a \circ \eta_{s (a)}
\end{equation}
for all $a \in Q_1$\,.
Conversely, all such isomorphisms must be of this form.
Therefore, the cokernel of $d$ parametrizes all extensions of $E$ by $F$,
and the kernel of $d$ parametrizes all automorphisms of such extensions
that induce identities on $E$ and $F$.

From the above argument, we conclude that the $2$-vector space
\begin{equation}
    [A^1 (E, F) / A^0 (E, F)] \simeq
    \Ext^1 (E, F) \times [*/\Ext^0 (E, F)]
\end{equation}
is the moduli stack of extensions of $E$ by $F$,
where we used the map $d \colon A^0 (E, F) \to A^1 (E, F)$.

All the above is true in families as well.
Indeed, we can define vector bundles $A^0, A^1 \to \calM \times \calM$
whose fibres at $(F, E)$ are $A^0 (E, F)$ and $A^1 (E, F)$,
which can be constructed from the universal families.
Then, define a morphism $d \colon A^0 \to A^1$
similarly as above.
The $2$-vector bundle
\begin{equation}
    \calExt^{1/0} = [A^1 / A^0]
    \longrightarrow \calM \times \calM
\end{equation}
is isomorphic to $\pi_1 \times \pi_2 \colon \calM^\pn{2} \to \calM \times \calM$,
the moduli stack of all extensions.
This verifies \tagref{Ext}.
The Euler form $\chi (-, -)$ is given
for dimension vectors $\alpha, \beta \in C^\circ (Q)$ by
\begin{align}
    \label{eq-chi-explicit}
    \chi (\alpha, \beta) & =
    \sum_{i \in Q_0} \alpha_i \, \beta_i -
    \sum_{a \in Q_1} \alpha_{s(a)} \, \beta_{t(a)} \ .
\end{align}

\paragraph{Self-dual extension bundles.}

Now, let $(Q, \sigma, u, v)$ be a self-dual quiver without relations.
The condition \tagref{SdExt} follows from the fact that
each $\calM_\alpha$ is connected.

The self-dual Euler form, as in \cref{para-euler-sd},
is given by
\begin{align*}
    \chi^\sd (\alpha, \theta) & = \chi (\alpha, j (\theta)) +
    \frac{1}{2} \sum_{i \in Q_0^\pm}
    \alpha_i \, ( \alpha_i - u_i ) +
    \sum_{i \in Q_0^\tria}
    \alpha_{i^\vee} \, \alpha_i \\
    & \hspace{2em} {} -
    \frac{1}{2} \sum_{a \in Q_1^\pm}
    \alpha_{t(a)} \, ( \alpha_{t(a)} + u_{t(a)} \, v_a ) -
    \frac{1}{2} \sum_{a \in Q_1^\tria}
    \alpha_{s(a)^\vee} \, \alpha_{t(a)}
    \numberthis
\end{align*}
for $\alpha \in C^\circ (Q)$ and $\theta \in C^\sd (Q)$,
with $j$ as in~\cref{eq-j-csd}.
See also Young~\cite[Eq.~(7)]{Young2020} for a similar expression.

\begin{proposition}
    Let $\breve{Q}$ be a quiver with relations,
    and let $\tau$ be any weak stability condition on $\Mod (\bbK \breve{Q})$.
    Then $\tau$ satisfies \tagref{Stab}.
\end{proposition}

\begin{proof}
    First, let us verify \tagref{Stab1}
    for a quiver $Q$ with no relations.
    It was shown in the proof of \cite[Lemma~3.2.2]{BelmansEtal2022}
    that for any $\alpha, \beta \in C (Q)$,
    there is a closed substack of $\calM_\alpha$
    consisting of representations of class $\alpha$
    that have sub-representations of class $\beta$.
    The semistable locus can be obtained from $\calM_\alpha$
    by removing these substacks for all $\beta$ satisfying
    $\alpha - \beta \in C (Q)$ and $\tau (\beta) > \tau (\alpha - \beta)$.
    Therefore, we see that the semistable locus is open.

    For a quiver with relations $\breve{Q}$,
    a representation of $\breve{Q}$ is $\tau$-semistable
    if and only if it is $\tau$-semistable as a representation of $Q$.
    This verifies \tagref{Stab1} in general.
    
    \tagref{Stab2} is automatically true,
    as each $\calM_\alpha$ is already of finite type.
    \tagref{Stab3}
    is also automatically true, because each element of $C (Q)$
    can only be written as a sum of elements of $C (Q)$ in finitely many ways.
\end{proof}

\paragraph{Conclusion.}
\label{para-quiv-inv}

By the above discussion,
for any self-dual quiver with relations $\breve{Q}$,
and any self-dual weak stability condition $\tau$ on $\calA = \smash{\Mod (\bbK \breve{Q})}$,
we have well-defined invariants
\begin{alignat*}{6}
    \delta_\alpha (\tau) , && \ 
    \epsilon_\alpha (\tau) & \in \SF (\calM_\alpha) \ , & \quad
    \upI_\alpha (\tau) , && \ 
    \upJ_\alpha (\tau) & \in \Mhat_\bbK \ , & \quad
    \chiJ_\alpha (\tau) , && \ 
    \DT^\nai_\alpha (\tau) & \in \bbQ \ ,
    \\
    \delta^\sd_\theta (\tau) , && \ 
    \epsilon^\sd_\theta (\tau) & \in \SF (\calM^\sd_\theta) \ , & \quad
    \upI^\sd_\theta (\tau) , && \ 
    \upJ^\sd_\theta (\tau) & \in \Mhat_\bbK \ , & \quad
    \chiJ^\sd_\theta (\tau) , && \ 
    \DT^\sdnai_\theta (\tau) & \in \bbQ \ ,
\end{alignat*}
for all $\alpha \in C^\circ (\breve{Q})$ and $\theta \in C^\sd (\breve{Q})$.

If, moreover, $\breve{Q}$ has no relations,
then the above invariants satisfy wall-crossing formulae
under a change of the self-dual weak stability condition $\tau$,
as in \cref{thm-wcf-main,thm-wcf-mot,thm-wcf-num}.

\subsection{An algorithm for computing invariants}

\paragraph{}

For a self-dual quiver $Q$ without relations,
we provide an algorithm to compute the invariants 
    $\upI_{\alpha} (\tau)$,
    $\upI^\sd_{\theta} (\tau)$,
    $\upJ_{\alpha} (\tau)$,
    $\upJ^\sd_{\theta} (\tau)$,
    $\chiJ_{\alpha} (\tau)$,
    $\chiJ^\sd_{\theta} (\tau)$,
    $\DT_{\alpha} (\tau)$, and
    $\DT^\sd_{\theta} (\tau)$,
defined above.
This will allow us to explicitly compute these invariants
for any given self-dual quiver,
and any given self-dual weak stability condition.

Note that we have dropped the superscript `$\nai$'
for the DT invariants, since when there are no relations,
the $\DT^\nai$ and $\DT^\sdnai$ invariants
coincide with the correct definition of DT invariants
in \cref{sect-quiv-dt} below.

\begin{algorithm}
    \label{alg-quiver}
    Let $Q$ be a self-dual quiver,
    and $\tau$ be a self-dual weak stability condition
    on $\calA = \Mod (\bbK Q)$.
    Let $\alpha \in C (Q)$ be a dimension vector,
    and $\theta \in C^\sd (Q)$ a self-dual dimension vector.
    
    One can compute the invariants
    $\upI_{\alpha} (\tau)$,
    $\upI^\sd_{\theta} (\tau)$,
    $\upJ_{\alpha} (\tau)$,
    $\upJ^\sd_{\theta} (\tau)$,
    and hence the invariants
    $\chiJ_{\alpha} (\tau)$,
    $\chiJ^\sd_{\theta} (\tau)$,
    $\DT_{\alpha} (\tau)$,
    $\DT^\sd_{\theta} (\tau)$, as follows.
    
    \begin{enumerate}
        \allowdisplaybreaks
        \item 
        \emph{Compute the invariants
        $\smash{\upI_{\alpha} (0)}$
        and $\smash{\upI^\sd_{\theta} (0)}$,
        for the trivial stability condition $0$.}
        
        This is done using the explicit descriptions
        \cref{eq-moduli-quiver-explicit,eq-msd-quiver-explicit},
        together with \cref{thm-motive-linear-action-quotient}.
        Explicitly, using the notations of~\cref{eq-moduli-quiver-explicit},
        and using~\cref{eq-motive-bgl}, we have
        \begin{align*}
            \upI_{\alpha} (0)
            & = \mu (V_\alpha) \cdot \mu ([*/G_\alpha]) \\*
            & = \prod_{a \in Q_1} \bbL^{\alpha_{s (a)} \, \alpha_{t (a)}} \cdot
            \prod_{i \in Q_0} \ 
            \prod_{j=0}^{\alpha_i-1} \frac{1}{\bbL^{\alpha_i} - \bbL^j} \ .
            \numberthis 
        \end{align*}
        Similarly, using the notations of~\cref{eq-msd-quiver-explicit},
        and using \cref{eq-motive-bgl,eq-motive-bo-even,eq-motive-bo-odd},
        we have
        \begin{align*}
            \upI^\sd_{\theta} (0)
            & = \mu (V^\sd_\theta) \cdot \mu ([*/G^\sd_\theta]) \\[1ex]
            & = \prod_{a \in Q_1^\tria}
            \bbL^{\theta_{s (a)} \, \theta_{t (a)}} \cdot
            \prod_{a \in Q_1^\pm}
            \bbL^{\theta_{t (a)} \, (\theta_{t (a)} \, + \, u_{t(a)} \, v_a) / 2} \cdot {} \\*
            & \hspace{2em} \biggl( {}
                \prod_{i \in Q_0^\tria} \ 
                \prod_{j=0}^{\theta_i-1} \frac{1}{\bbL^{\theta_i} - \bbL^j} 
            \biggr) \cdot \biggl( {}
                \prod_{\substack{ i \in Q_0^\pm : \xmathstrut[0]{.3} \\ \theta_i \text{ even} }} {}
                \bbL^{u_i \, \theta_i/2} \cdot
                \prod_{j=0}^{\theta_i/2-1} \frac{1}{\bbL^{\theta_i} - \bbL^{2j}}
            \biggr) \cdot {} \\*
            & \hspace{2em} \biggl( {}
                \prod_{\substack{ i \in Q_0^+ : \xmathstrut[0]{.3} \\ \theta_i \text{ odd} }} {}
                \bbL^{-(\theta_i - 1)/2} \cdot
                \prod_{j=0}^{(\theta_i-3)/2} \frac{1}{\bbL^{\theta_i-1} - \bbL^{2j}}
            \biggr) \ .
            \numberthis
        \end{align*}
    
        \item 
        \emph{Use wall-crossing to compute the invariants
        $\upI_{\alpha} (\tau)$
        and $\upI^\sd_{\theta} (\tau)$.}

        This can be done using~\cref{eq-wcf-i,eq-wcf-isd}.
        Since the trivial stability condition~$0$ dominates~$\tau$,
        these formulae can be simplified,
        similarly to \cref{thm-wcf-anti-dom}.
        Explicitly, for $\alpha \in C (\calA)$ and
        $\theta \in C^\sd (\calA)$, we have
        \begin{align}
            \upI_\alpha (\tau) & =
            \sum_{ \leftsubstack[6em]{
                & n > 0; \, \alpha_1, \dotsc, \alpha_n \in C (\calA) \colon \\[-.5ex]
                & \alpha = \alpha_1 + \cdots + \alpha_n \, , \\[-.5ex]
                & \tau (\alpha_1 + \cdots + \alpha_i) > \tau (\alpha_{i+1} + \cdots + \alpha_n)
                \text{ for all } i = 1, \dotsc, n - 1
            } } {}
            (-1)^{n-1} \cdot
            \bbL^{-\chi (\alpha_1, \dotsc, \alpha_n)} \cdot
            \upI_{\alpha_1} (0) \cdots
            \upI_{\alpha_n} (0) \ , \\[1ex]
            \upI^\sd_\theta (\tau) & =
            \sum_{ \leftsubstack[6em]{
                & n \geq 0; \, \alpha_1, \dotsc, \alpha_n \in C (\calA), \,
                \rho \in C^\sd (\calA) \colon \\[-.5ex]
                & \theta = \bar{\alpha}_1 + \cdots + \bar{\alpha}_n + \rho, \\[-.5ex]
                & \tau (\alpha_1 + \cdots + \alpha_i) > 0
                \text{ for all } i = 1, \dotsc, n
            } } {}
            (-1)^n \cdot
            \bbL^{-\chi (\alpha_1, \dotsc, \alpha_n, \rho)} \cdot
            \upI_{\alpha_1} (0) \cdots
            \upI_{\alpha_n} (0) \cdot
            \upI^\sd_{\rho} (0) \ .
            \raisetag{3ex}
        \end{align}
    
        \item 
        \emph{Compute the invariants
        $\smash{\upJ_{\alpha} (\tau)}$
        and $\smash{\upJ^\sd_{\theta} (\tau)}$.}

        This can be done using the relations
        \cref{eq-inv-def-epsilon,eq-inv-def-epsilon-sd},
        as we now have explicit expressions for the invariants
        $\upI_{\alpha} (\tau)$ and $\upI^\sd_{\theta} (\tau)$.
    
        \item 
        \emph{Compute the invariants}
        $\chiJ_{\alpha} (\tau)$,
        $\chiJ^\sd_{\theta} (\tau)$,
        $\DT_{\alpha} (\tau)$,
        $\DT^\sd_{\theta} (\tau)$.

        We have now obtained explicit expressions for
        $(\bbL - 1) \cdot \smash{\upJ_{\alpha} (\tau)}$
        and $\smash{\upJ^\sd_{\theta} (\tau)}$,
        which are rational functions in the variable~$\bbL$
        that have no poles at $\bbL = 1$.
        Setting $\bbL = 1$ gives the invariants
        $\chiJ_{\alpha} (\tau)$ and $\chiJ^\sd_{\theta} (\tau)$,
        and adding appropriate signs gives the corresponding
        na\"ive DT invariants.
    \end{enumerate}
\end{algorithm}

\subsection{Examples}
\label{sect-quiv-eg}

We now apply \cref{alg-quiver}
to several examples of self-dual quivers,
and compute the invariants explicitly.

\paragraph{The point quiver.}
\label{para-sit-point-quiver}

Let $Q = (\bullet)$ be the point quiver,
equipped with the unique contravariant involution.
Let $\varepsilon = \pm 1$ be a sign,
and let $u_1 = \varepsilon$,
where $1$ is the label of the unique vertex of $Q$.
Consider the self-dual quiver $(Q, \sigma, u, v)$,
where $v$ is trivial.

Then, self-dual representations of $Q$
can be seen as orthogonal (or symplectic) vector spaces,
when $\varepsilon = +1$ (or $-1$).

We use the trivial stability condition $0$,
and we will omit the stability condition in all notations.

\begin{remark}
    \label{rem-vect-equiv-p1}
    This setup in \cref{para-sit-point-quiver}
    is equivalent to the following situations.
    
    \begin{enumerate}
        \item
            As in \cref{eg-vect-sp},
            let $\calA$ be the self-dual abelian category of
            finite-dimensional $\bbK$-vector spaces,
            with a sign $\varepsilon = \pm 1$
            introduced into the natural isomorphism
            $\eta_V \colon V^{\vee\vee} \simto V$.
            Then $\calA \simeq \cat{Mod} (\bbK Q)$.
        
        \item 
            Let $Q' = (\bullet \to \bullet)$ be the $A_2$ quiver,
            equipped with a contravariant involution
            that exchanges the two vertices.
            We assign the sign $+1$ to both vertices,
            and the sign $\varepsilon = \pm 1$ to the edge.
            We use the stability function $(1, -1)$.
            Then, a self-dual representation of $Q$ is of the form
            \[
                V \overset{f}{\longrightarrow} V^\vee,
            \]
            where $V$ is a vector space, and $f = \varepsilon f^\vee$.
            Such a representation is semistable if and only if
            $f$ is an isomorphism, since otherwise,
            it would be destabilized by the sub-representation $(\ker f \to 0)$.
            Therefore, we see that semistable self-dual representations
            are equivalent to orthogonal or symplectic vector spaces.
            
        \item 
            Consider the $\bbK$-linear
            quasi-abelian category $\cat{Vect} (\bbP^1)$
            of vector bundles on $\bbP^1$,
            as in \cref{eg-vb}.
            Choose a sign $\varepsilon = \pm 1$
            as the way to identify a vector bundle and its double dual.
            Consider the usual slope stability on $\cat{Vect} (\bbP^1)$.
            Then, a self-dual vector bundle of rank $r$
            is the same as a vector bundle equipped with a
            non-degenerate, (anti-)symmetric bilinear form,
            which can then be identified with
            a principal $\upO (r)$- or $\Sp (r)$-bundle.
            
            It is well-known that every vector bundle over $\bbP^1$
            can be decomposed into a direct sum of line bundles
            of the form $\calO (d)$, for $d \in \bbZ$.
            Therefore, a semistable vector bundle of slope $0$
            must have the form $\calO \otimes V$,
            where $V$ is a finite-dimensional vector space.
            Since the underlying vector bundle of a
            semistable self-dual vector bundle is semistable of slope $0$
            (\cref{thm-sd-hn}),
            it must be of this form, and its self-dual structure
            gives rise to a self-dual structure on $V$.
            This shows that orthogonal or symplectic vector bundles on $\bbP^1$ 
            are equivalent to orthogonal or symplectic vector spaces.
    \end{enumerate}
\end{remark}

\paragraph{Orthogonal vector spaces.}

In the situation of \cref{para-sit-point-quiver},
set $\varepsilon = +1$.
Using \cref{alg-quiver}
to compute the self-dual invariants,
we obtain \cref{tab-vect}.
Here, the numerical invariants $\chiJ^\sd_d$
are obtained from $\upJ^\sd_d$ by setting $\bbL = 1$,
and the invariants $\DT^\sd_d$ are given by
$\smash{(-1)^{\chi^\sd (d, 0)}} \cdot \chiJ^\sd_d$\,,
where $\chi^\sd (d, 0) = d (d-1) / 2$.

\begin{table}[ht]
    \centering
    \begin{tabular}{c|c|c|c}
        $d$ & $\upJ^\sd_d$ &
        $\chiJ^\sd_d$ & $\DT^\sd_d$ \\[2pt] \hline
        $0 \xmathstrut[0]{.5}$ & $1$ & $1$ & $1$ \\
        $1 \xmathstrut[.5]{.5}$ & $1$ & $1$ & $1$ \\
        $2 \xmathstrut[1]{1}$ &
        $\displaystyle \frac{1}{2(\bbL + 1)}$ &
        $\dfrac{1}{4}$ & $-\dfrac{1}{4}$ \\
        $3 \xmathstrut[1]{1}$ &
        $\displaystyle -\frac{1}{2 \, \bbL (\bbL + 1)}$ &
        $-\dfrac{1}{4}$ & $\dfrac{1}{4}$ \\
        $4 \xmathstrut[1]{1.5}$ &
        $\displaystyle
        -\frac{\bbL^2 + 4 \, \bbL + 1}{8 \, \bbL^2 (\bbL + 1)^2 (\bbL^2 + 1)}$ &
        $-\dfrac{3}{32}$ & $-\dfrac{3}{32}$ \\
        $5 \xmathstrut[1]{1.5}$ &
        $\displaystyle
        \frac{3 \, \bbL^2 + 4 \, \bbL + 3}{8 \, \bbL^4 (\bbL + 1)^2 (\bbL^2 + 1)}$ &
        $\dfrac{5}{32}$ & $\dfrac{5}{32}$ \\
        $6 \xmathstrut[1]{1.5}$ &
        $\displaystyle
        \frac{\bbL^4 + 3 \, \bbL^3 + 6 \, \bbL^2 + 3 \, \bbL + 1}
        {16 \, \bbL^6 (\bbL + 1)^3 (\bbL^2 + 1) (\bbL^2 - \bbL + 1)}$ &
        $\dfrac{7}{128}$ & $-\dfrac{7}{128}$ \\
        $7 \xmathstrut[1]{1.5}$ &
        $\displaystyle
        -\frac{5 \, \bbL^4 + 7 \, \bbL^3 + 6 \, \bbL^2 + 7 \, \bbL + 5}
        {16 \, \bbL^9 (\bbL + 1)^3 (\bbL^2 + 1) (\bbL^2 - \bbL + 1)}$ &
        $-\dfrac{15}{128}$ & $\dfrac{15}{128}$ \\
        \dots & \dots & \dots & \dots
    \end{tabular}
    \caption{Invariants for orthogonal vector spaces}
    \label{tab-vect}
\end{table}

From the numerical evidence, we make the following conjecture.

\begin{conjecture}
    \label{conj-eg-vect}
    In the situation of \cref{para-sit-point-quiver}, with $\varepsilon = 1$,
    for any integer $n \geq 0$, we have
    \begin{equation} \begin{aligned}
        \chiJ^\sd_{2n} & = \binom{\,1/4\,}{n} \ , & 
        \DT^\sd_{2n} & = (-1)^n \, \binom{\,1/4\,}{n} \ , \\
        \chiJ^\sd_{2n+1} & = \binom{-1/4}{n} \ , & 
        \DT^\sd_{2n+1} & = (-1)^n \, \binom{-1/4}{n} \ .
    \end{aligned} \end{equation}
    Equivalently, we have the generating series
    \begin{equation} \begin{aligned}
        \sum_{n = 0}^\infty \chiJ^\sd_{\cat{D}_n} \, q^n & =
        (1+q)^{1/4} \ , &
        \sum_{n = 0}^\infty \DT^\sd_{\cat{D}_n} \, q^n & =
        (1-q)^{1/4} \ , \\
        \sum_{n = 0}^\infty \chiJ^\sd_{\cat{B}_n} \, q^n & =
        (1+q)^{-1/4} \ , &
        \sum_{n = 0}^\infty \DT^\sd_{\cat{B}_n} \, q^n & =
        (1-q)^{-1/4} \ ,
    \end{aligned} \end{equation}
    where we replaced the dimension vectors by their
    Dynkin types $\cat{B}_n$ and $\cat{D}_n$\,.
\end{conjecture}

\paragraph{Symplectic vector spaces.}

In the situation of \cref{para-sit-point-quiver},
now set $\varepsilon = -1$.
By \cref{eq-motive-bo-odd}, for any $n \geq 0$, we have
\[
    \mu ([*/\Sp (2n)])
    = \mu ([*/\upO (2n+1)]) \ .
\]
Using \cref{alg-quiver}, in each step,
we obtain identical expressions as in the odd orthogonal case.
Therefore, if we write $\smash{\upJ_{\cat{C}_n}}$\,, etc.,
for the symplectic invariants for $\Sp (2n)$, then
\[
    \upJ^\sd_{\cat{C}_n} = \upJ^\sd_{\cat{B}_n} \ , \qquad
    \chiJ^\sd_{\cat{C}_n} = \chiJ^\sd_{\cat{B}_n} \ , \qquad
    \DT^\sd_{\cat{C}_n} = \DT^\sd_{\cat{B}_n}
\]
for all $n \geq 0$.
In particular, if \cref{conj-eg-vect} is true,
then we have the generating series
\begin{equation}
    \sum_{n = 0}^\infty \chiJ^\sd_{\cat{C}_n} \, q^n =
    (1+q)^{-1/4} \ , \qquad
    \sum_{n = 0}^\infty \DT^\sd_{\cat{C}_n} \, q^n =
    (1-q)^{-1/4} \ .
\end{equation}
We make the remark that the types~B and~C
invariants being equal is probably related to the fact that
these types are Langlands dual to each other.

\paragraph{Loop quivers.}
\label{para-sit-loop-quiver}

We now consider the more general situation of loop quivers.

Let $m \geq 0$ be an integer,
and let $Q$ be the \emph{$m$-loop quiver},
the quiver with a unique vertex and $m$ edges.
Equip $Q$ with the contravariant involution $\sigma$
that fixes all the edges.

For convenience, we denote the self-dual quiver~$(Q, \sigma, u, v)$
using the notation~$m^u_v$\,,
where $m$ is the number of loops,
$u$ is a sign, and $v$ is a list of $m$ signs.
For example, $1^{+}_{-}$~is the $1$-loop quiver
whose vertex sign is $+1$ and edge sign is $-1$.

Consider the trivial stability condition $0$,
and omit it from all notations.

\begin{sidewaystable}
    \centering
    \scalebox{.9}{
    \begin{tabular}{c||c|c|c||c|c|c|c||c|c|c|c|c}
        $d$ & $1^{+}_{+}$ & $1^{+}_{-}$\,, $1^{-}_{-}$ & $1^{-}_{+}$
        & $2^{+}_{++}$ & $2^{+}_{+-}$\,, $2^{-}_{--}$
        & $2^{+}_{--}$\,, $2^{-}_{+-}$ & $2^{-}_{++}$
        & $3^{+}_{+++}$ & $3^{+}_{++-}$\,, $3^{-}_{---}$
        & $3^{+}_{+--}$\,, $3^{-}_{+--}$ & $3^{+}_{---}$\,, $3^{-}_{++-}$
        & $3^{-}_{+++}$
        \\[2pt] \hline
        $0 \xmathstrut[0]{.5}$
        & $1$ & $1$ & $1$
        & $1$ & $1$ & $1$ & $1$
        & $1$ & $1$ & $1$ & $1$ & $1$
        \\
        $1 \xmathstrut{.5}$
        & $1$ & $1$ &
        & $1$ & $1$ & $1$ &
        & $1$ & $1$ & $1$ & $1$ &
        \\
        $2 \xmathstrut{1.1}$
        & $\dfrac{3}{4}$
        & $\dfrac{1}{4}$
        & $-\dfrac{1}{4}$
        & $\dfrac{5}{4}$
        & $\dfrac{3}{4}$
        & $\dfrac{1}{4}$
        & $-\dfrac{1}{4}$
        & $\dfrac{7}{4}$
        & $\dfrac{5}{4}$
        & $\dfrac{3}{4}$
        & $\dfrac{1}{4}$
        & $-\dfrac{1}{4}$
        \\
        $3 \xmathstrut{1.1}$
        & $\dfrac{3}{4}$
        & $\dfrac{1}{4}$ &
        & $\dfrac{7}{4}$
        & $\dfrac{5}{4}$
        & $\dfrac{3}{4}$ &
        & $\dfrac{11}{4}$
        & $\dfrac{9}{4}$
        & $\dfrac{7}{4}$
        & $\dfrac{5}{4}$ &
        \\
        $4 \xmathstrut{1.1}$
        & $\dfrac{21}{32}$
        & $\dfrac{5}{32}$
        & $-\dfrac{3}{32}$
        & $\dfrac{93}{32}$
        & $\dfrac{53}{32}$
        & $\dfrac{21}{32}$
        & $-\dfrac{3}{32}$
        & $\dfrac{213}{32}$
        & $\dfrac{149}{32}$
        & $\dfrac{93}{32}$
        & $\dfrac{45}{32}$
        & $\dfrac{5}{32}$
        \\
        $5 \xmathstrut{1.1}$
        & $\dfrac{21}{32}$
        & $\dfrac{5}{32}$ &
        & $\dfrac{141}{32}$
        & $\dfrac{93}{32}$
        & $\dfrac{53}{32}$ &
        & $\dfrac{365}{32}$
        & $\dfrac{285}{32}$
        & $\dfrac{213}{32}$
        & $\dfrac{149}{32}$ &
        \\
        $6 \xmathstrut{1.1}$
        & $\dfrac{77}{128}$
        & $\dfrac{15}{128}$
        & $-\dfrac{7}{128}$
        & $\dfrac{1043}{128}$
        & $\dfrac{589}{128}$
        & $\dfrac{255}{128}$
        & $\dfrac{25}{128}$
        & $\dfrac{3993}{128}$
        & $\dfrac{2795}{128}$
        & $\dfrac{1797}{128}$
        & $\dfrac{983}{128}$
        & $\dfrac{337}{128}$
        \\
        $7 \xmathstrut{1.1}$
        & $\dfrac{77}{128}$
        & $\dfrac{15}{128}$ &
        & $\dfrac{1633}{128}$
        & $\dfrac{1043}{128}$
        & $\dfrac{589}{128}$ &
        & $\dfrac{7053}{128}$
        & $\dfrac{5407}{128}$
        & $\dfrac{3993}{128}$
        & $\dfrac{2795}{128}$ &
        \\
        $8 \xmathstrut{1.1}$
        & $\dfrac{1155}{2048}$
        & $\dfrac{195}{2048}$
        & $-\dfrac{77}{2048}$
        & $\dfrac{51283}{2048}$
        & $\dfrac{28995}{2048}$
        & $\dfrac{13283}{2048}$
        & $\dfrac{2867}{2048}$
        & $\dfrac{330275}{2048}$
        & $\dfrac{232259}{2048}$
        & $\dfrac{152595}{2048}$
        & $\dfrac{89107}{2048}$
        & $\dfrac{39747}{2048}$
        \\
        $9 \xmathstrut{1.1}$
        & $\dfrac{1155}{2048}$
        & $\dfrac{195}{2048}$ &
        & $\dfrac{81555}{2048}$
        & $\dfrac{51283}{2048}$
        & $\dfrac{28995}{2048}$ &
        & $\dfrac{590707}{2048}$
        & $\dfrac{448947}{2048}$
        & $\dfrac{330275}{2048}$
        & $\dfrac{232259}{2048}$ &
        \\
        $10 \xmathstrut{1.1}$
        & $\dfrac{4389}{8192}$
        & $\dfrac{663}{8192}$
        & $-\dfrac{231}{8192}$
        & $\dfrac{665667}{8192}$
        & $\dfrac{378021}{8192}$
        & $\dfrac{180279}{8192}$
        & $\dfrac{51801}{8192}$
        & $\dfrac{7229697}{8192}$
        & $\dfrac{5111379}{8192}$
        & $\dfrac{3414261}{8192}$
        & $\dfrac{2078919}{8192}$
        & $\dfrac{1051433}{8192}$
        \\
        $11 \xmathstrut{1.1}$
        & $\dfrac{4389}{8192}$
        & $\dfrac{663}{8192}$ &
        & $\dfrac{1067313}{8192}$
        & $\dfrac{665667}{8192}$
        & $\dfrac{378021}{8192}$ &
        & $\dfrac{12996685}{8192}$
        & $\dfrac{9834399}{8192}$
        & $\dfrac{7229697}{8192}$
        & $\dfrac{5111379}{8192}$ &
        \\
        \dots & \dots & \dots & \dots 
        & \dots & \dots & \dots & \dots
        & \dots & \dots & \dots & \dots & \dots
    \end{tabular}
    }
    \caption{The invariants $\chiJ^\sd_d$ for loop quivers}
    \label{tab-loop}
\end{sidewaystable}

\paragraph{}

Applying \cref{alg-quiver},
we obtain numerical results listed in \cref{tab-loop}.
The table only shows the invariants $\chiJ^\sd_d$\,,
as they only differ from $\DT^\sd_d$ by a sign.
Also, we have placed some of the orthogonal and symplectic invariants
in the same column when they coincide,
but it is understood that
the symplectic invariants are undefined when $d$ is odd.

\begin{remark}
    In \cref{tab-loop},
    one can see clear patterns in the first few rows,
    but the numbers become more chaotic further down.
    For one-loop quivers, these numbers are,
    up to signs, binomial coefficients
    with $-3/4$, $-1/4$ or $1/4$ on the top.
    However, in general,
    we have not been able to find a pattern
    when there are at least two loops.
    
    For stability conditions satisfying some \emph{genericity} conditions,
    we expect that there are integral invariants
    lying behind these rational invariants,
    which are integers and are sometimes called \emph{BPS invariants},
    which encode the dimension of the space of BPS states.
    In the ordinary (non-self-dual) case,
    one has a \emph{multiple cover formula}
    \begin{equation}
        \DT_d = \sum_{n \mid d} \frac{1}{n^2} \, \mathrm{BPS}_{d/n} \ ,
    \end{equation}
    as in Joyce--Song~\cite[\S6.2]{JoyceSong2012},
    where $\mathrm{BPS}_{d/n}$ denotes the integral invariant.
    We expect a similar relation to be true for self-dual invariants as well,
    where the right-hand side should involve
    both self-dual and ordinary versions of integral invariants.
    Such self-dual integral invariants are
    possibly related to discussions in Young~\cite{Young2020}.
    The zero-loop and one-loop invariants having a simple pattern
    is possibly due to the vanishing of these integral invariants in those cases.
\end{remark}

\paragraph{The $\tilde{A}_1$ quiver.}
\label{para-sit-a1-quiver}

Finally, we consider the $\tilde{A}_1$ quiver,
which is related to coherent sheaves on \smash{$\bbP^1$},
as will be discussed in \cref{sect-quiv-coh-rel}.

Let $Q = (\bullet \rightrightarrows \bullet)$
be the quiver with two vertices and
two arrows pointing in the same direction,
called the \emph{$\tilde{A}_1$ quiver}.
Let $\sigma$ be the contravariant involution of $Q$
that exchanges the two vertices but fixes the edges.

We use the simplified notation $\tilde{A}_1^{u,v}$
to denote the self-dual quiver $(Q, \sigma, u, v)$.
For example, $\tilde{A}_1^{+,++}$ means that we assign the sign $+1$
to all vertices and edges.
Note that both vertices must have the same sign.

We use the self-dual stability function $\tau = (1, -1)$,
and exclude it from the notations.

\paragraph{}

Applying \cref{alg-quiver},
we obtain numerical results that suggest the following.

\begin{conjecture*}
    In the situation of \cref{para-sit-a1-quiver},
    we should have the following.

    \begin{enumerate}
        \item 
            For $\tilde{A}_1^{+,++}$ and $\tilde{A}_1^{-,--}$,
            we have
            \begin{align}
                \sum_{n = 0}^\infty \upJ^\sd_{{(n,n)}} \, q^{n/2}
                & = \frac{(1 - q \, \bbL)^{1/2}}{(1 - q^{1/2}) (1 - q^{1/2} \, \bbL)} \ , \\
                \sum_{n = 0}^\infty \chiJ^\sd_{(n,n)} \, q^{n/2}
                & = \frac{(1+q^{1/2})^{1/2}} {(1-q^{1/2})^{3/2}} \ , \\
                \sum_{n = 0}^\infty \DT^\sd_{(n,n)} \, q^{n/2}
                & = \frac{(1-q^{1/2})^{1/2}} {(1+q^{1/2})^{3/2}} \ .
            \end{align}
            
        \item 
            For $\tilde{A}_1^{+,+-}$ and $\tilde{A}_1^{-,+-}$,
            we have
            \begin{align*}
                \sum_{n = 0}^\infty \upJ^\sd_{(n,n)} \, q^{n/2} =
                \sum_{n = 0}^\infty \chiJ^\sd_{(n,n)} \, q^{n/2}
                = \sum_{n = 0}^\infty \DT^\sd_{(n,n)} \, q^{n/2}
                = \frac{(1+q^{1/2})^{1/2}}{(1-q^{1/2})^{1/2}} \ .
                \numberthis
            \end{align*}
            
        \item 
            For $\tilde{A}_1^{+,--}$ and $\tilde{A}_1^{-,++}$,
            we have
            \begin{align}
                \sum_{n = 0}^\infty \upJ^\sd_{(n,n)} \, q^{n/2}
                & = (1-q \, \bbL^{-1})^{1/2} \ , 
                \hspace{3em} \\
                \sum_{n = 0}^\infty \chiJ^\sd_{(n,n)} \, q^{n/2} =
                \sum_{n = 0}^\infty \DT^\sd_{(n,n)} \, q^{n/2}
                & = (1-q)^{1/2} \ .
            \end{align}
    \end{enumerate}
\end{conjecture*}

\subsection{Donaldson--Thomas invariants}
\label{sect-quiv-dt}

\paragraph{}

We discuss a self-dual version of the theory of
\emph{Donaldson--Thomas invariants}
for \emph{quivers with potentials},
as a self-dual analogue of the ordinary case studied in
Joyce--Song~\cite[Chapter~7]{JoyceSong2012}.

\paragraph{Quivers with potentials.}

For a quiver~$Q$, a \emph{potential} for~$Q$
is a $\bbK$-linear combination of oriented cycles in~$Q$.
Given such a potential~$W$,
we define a set $Q_2$ of relations on~$Q$
(see \cref{para-quiver-rel}),
by setting $Q_2 = Q_1$, and for each $b \in Q_2$,
we set $p (b) = t (b)$, $q (b) = s (b)$,
and $r (b) = \partial_b W$,
where for an oriented cycle~$C$ in~$Q$,
$\partial_b C$ is the sum over all occurrences of~$b$ in~$C$
of the path obtained by removing that occurrence of~$b$ from~$C$.
Such a pair $(Q, W)$ is called a \emph{quiver with potential},
and $\breve{Q} = (Q, Q_2, p, q, r)$ is its
associated quiver with relations.
We also refer to representations of $\breve{Q}$
as \emph{representations of $(Q, W)$}.

If, moreover, $Q$ is equipped with a self-dual structure
(\cref{para-sd-quiver}),
then a potential~$W$ for~$Q$ is called \emph{self-dual}
if $W = W^\vee$, where $W^\vee$ is the potential obtained
from $W$ by replacing each cycle~$C$ with~$C^\vee$.
In this case, we have a self-dual quiver with relations~$\breve{Q}$
(\cref{para-sd-quiver-rel}),
with the data $Q_2, p, q, r$ as above,
the involution on~$Q_2$ given by that on~$Q_1$,
and the sign function $w \colon Q_2 \to \{ \pm 1 \}$ given by $w = v$.
Such a pair $(Q, W)$ is called a \emph{self-dual quiver with potential},
and $\breve{Q}$ is its associated self-dual quiver with relations.
In particular, this defines a self-dual structure
on the $\bbK$-linear abelian category of representations of $(Q, W)$.

\paragraph{Definition.}

Let $(Q, W)$ be a self-dual quiver with potential,
and write $\calA$ for the self-dual $\bbK$-linear abelian category
of representations of its associated
self-dual quiver with relations~$\breve{Q}$.

For a self-dual weak stability condition $\tau$ on $\calA$,
and classes $\alpha \in C (Q)$ and $\theta \in C^\sd (Q)$,
define \emph{Donaldson--Thomas invariants}
$\DT_\alpha (\tau)$, $\DT^\sd_\theta (\tau) \in \bbQ$ by
\begin{align}
    \label{eq-dt-quiv-pot}
    \DT_\alpha (\tau) & =
    - \int_{\calM} {} (\bbL - 1) \cdot \epsilon_\alpha (\tau) \cdot \nu_{\calM} \, d \chi \ ,
    \\
    \label{eq-dt-quiv-pot-sd}
    \DT^\sd_\theta (\tau) & =
    \int_{\calM^\sd} {} \epsilon^\sd_\theta (\tau) \cdot \nu_{\calM^\sd} \, d \chi \ ,
\end{align}
where $\nu_{\calM}$ and $\nu_{\calM^\sd}$ are the \emph{Behrend functions}
for $\calM$ and $\calM^\sd$, as in Behrend~\cite{Behrend2009}
or Joyce--Song~\cite[Chapter~4]{JoyceSong2012}.
The ordinary case~\cref{eq-dt-quiv-pot}
is due to \cite[Definition~7.15]{JoyceSong2012}.
These invariants are well-defined by the no-pole theorems,
\cref{thm-no-pole,thm-no-pole-sd}.

Note that when $W = 0$,
the theory is reduced to that of quivers with no relations,
and these invariants coincide with the
na\"ive Donaldson--Thomas invariants
defined in \cref{para-quiv-inv},
since the Behrend functions are constant in this case,
due to the stacks $\calM, \calM^\sd$ being smooth.

\paragraph{Wall-crossing.}

Let $(Q, W)$ be a self-dual quiver with potential,
and let $\tau, \tilde{\tau}$ be self-dual weak stability conditions
on the category $\calA$ above.
Joyce--Song~\cite[Theorem~7.17]{JoyceSong2012}
showed that the invariants $\DT_\alpha (\tau)$ and $\DT_\alpha (\tilde{\tau})$
satisfy the wall-crossing formula~\cref{eq-wcf-dt}
with $\DT$ in place of $\DT^\nai$,
with the coefficients $\bar{\chi} ({\cdots})$, $\tilde{\chi} ({\cdots})$
given by those of $Q$ as a quiver without relations.
A key ingredient in their proof is a set of identities
\cite[Theorem~7.11]{JoyceSong2012}
involving the Behrend function $\nu_\calM$.

We expect that similar identities and wall-crossing formulae
should hold for the self-dual invariants $\DT^\sd_\theta (\tau)$,
and we state the expected properties as follows.

\begin{conjecture*}
    Let $(Q, W)$ be a self-dual quiver with potential.
    Then the Donaldson--Thomas invariants defined above satisfy
    the wall-crossing formula~\cref{eq-wcf-dt-sd}
    for any two self-dual weak stability conditions,
    with the coefficients
    $\bar{\chi} ({\cdots})$, $\tilde{\chi} ({\cdots})$,
    $\bar{\chi}^\sd ({\cdots})$, $\tilde{\chi}^\sd ({\cdots})$
    given by those of $Q$ as a self-dual quiver without relations.
\end{conjecture*}

Results in \cref{para-quiv-inv} imply that
the conjecture is true when there are no relations,
that is, when $W = 0$.

\subsection{Relation to coherent sheaves}
\label{sect-quiv-coh-rel}

\paragraph{}

The theory of enumerative invariants of coherent sheaves
is closely related to that of quiver representations.
For a smooth, projective variety $X$ over $\bbK$,
consider the bounded derived category $\Db \Coh (X)$
of coherent sheaves on $X$.
A result of Bondal~\cite{Bondal1990} states that
for some special choices of $X$,
there is an equivalence of triangulated categories
\begin{equation}
    \label{eq-dbcoh-equiv-dbmod}
    \Phi \colon \Db \Coh (X) \longsimto \Db \Mod (\bbK \breve{Q}) \ ,
\end{equation}
where $\breve{Q}$ is a quiver with relations.
When $X = \bbP^n$, one may take
\begin{equation} 
\label{eq-quiver-pn}
\begin{tikzpicture}[baseline={(0,-.5ex)}, line width=.5, >={Straight Barb[scale=0.8]}]
    \node [anchor=east] at (-.25, 0) {$Q = \Biggl($};
    \node [anchor=west] at (5.25, 0) {$\Biggr) \ ,$};
    \fill (0, 0) circle (.05) node [anchor=north] {$\scriptstyle 0$};
    \fill (1.5, 0) circle (.05) node [anchor=north] {$\scriptstyle 1$};
    \fill (3.5, 0) circle (.05) node [anchor=north] {$\scriptstyle n-1$};
    \fill (5, 0) circle (.05) node [anchor=north] {$\scriptstyle n$};
    \draw [->, shorten <=10, shorten >=10] (0, .45) -- (1.5, .45); 
    \draw [->, shorten <=10, shorten >=10] (0, .1) -- (1.5, .1); 
    \draw [->, shorten <=10, shorten >=10] (0, -.5) -- (1.5, -.5); 
    \draw [->, shorten <=10, shorten >=10] (3.5, .45) -- (5, .45); 
    \draw [->, shorten <=10, shorten >=10] (3.5, .1) -- (5, .1); 
    \draw [->, shorten <=10, shorten >=10] (3.5, -.5) -- (5, -.5); 
    \node [inner sep=0] at (.75, .575) {$\scriptstyle \varphi_{1,0}$};
    \node [inner sep=0] at (.75, .225) {$\scriptstyle \varphi_{1,1}$};
    \node [inner sep=0] at (.75, -.1) {$\scriptstyle \vdots$};
    \node [inner sep=0] at (.75, -.675) {$\scriptstyle \varphi_{1,n}$};
    \node at (2.5, 0) {$\cdots$};
    \node [inner sep=0] at (4.25, .575) {$\scriptstyle \varphi_{n,0}$};
    \node [inner sep=0] at (4.25, .225) {$\scriptstyle \varphi_{n,1}$};
    \node [inner sep=0] at (4.25, -.1) {$\scriptstyle \vdots$};
    \node [inner sep=0] at (4.25, -.675) {$\scriptstyle \varphi_{n,n}$};
\end{tikzpicture} \end{equation}
with $(n + 1)$ vertices and $n (n+1)$ arrows,
with the relations
$\smash{\varphi_{i+1, j} \, \varphi_{i, k}} = \smash{\varphi_{i+1, k} \, \varphi_{i, j}}$
for $i \in \{ 1, \dotsc, n-1 \}$ and $j, k \in \{ 0, \dotsc, n \}$.
In particular, when $X = \bbP^1$, we have
\[ \begin{tikzpicture}[line width=.5, >={Straight Barb[scale=0.8]}]
    \node [anchor=east] at (-.25, 0) {$Q = \Bigl($};
    \node [anchor=west] at (1.75, 0) {$\Bigr) \ ,$};
    \fill (0, 0) circle (.05) node [anchor=north] {$\scriptstyle 0$};
    \fill (1.5, 0) circle (.05) node [anchor=north] {$\scriptstyle 1$};
    \draw [->, shorten < = 10, shorten > = 10] (0, .1) -- (1.5, .1); 
    \draw [->, shorten < = 10, shorten > = 10] (0, -.1) -- (1.5, -.1); 
\end{tikzpicture} \]
with no relations.
This is the $\tilde{A}_1$ quiver discussed in \cref{para-sit-a1-quiver}.

\paragraph{}

The functor $\Phi$ in~\cref{eq-dbcoh-equiv-dbmod}
is constructed using a \emph{strong exceptional collection}
in the triangulated category $\Db \Coh (X)$, as in~\cite{Bondal1990}.
When $X = \bbP^n$, such a collection may be taken to be
\[
    \bigl( \calO, \ \calO (1), \ \dotsc, \ \calO (n) \bigr) \ ,
\]
and for a complex $E \in \Db \Coh (X)$,
its image under $\Phi$ is given by
\begin{equation}
    \Phi (E)_i =
    \bbR \Hom_X (\calO (n-i), E)
\end{equation}
for $i = 0, \dotsc, n$,
where for an object $V \in \Db \Mod (\bbK \breve{Q})$,
we denote by $V_i$ the complex assigned to the vertex $i$.

For example, when $X = \bbP^1$,
using Serre duality to compute the $\bbR \Hom$ functor,
one obtains
\begin{equation} \begin{aligned}
    & \quad \mathclap{\, \dotsc \dotsc \mathrlap{\ ,}} \\
    \Phi (\calO (-2)) & \quad \mathclap{=} &
    \bigl( \, \bbK^2 [-1] & \longrightarrow \bbK [-1] \, \bigr) \ , \\
    \Phi (\calO (-1)) & \quad \mathclap{=} &
    \bigl( \, \bbK [-1] & \longrightarrow 0 \, \bigr) \ , \\
    \Phi (\calO (0)) & \quad \mathclap{=} &
    \bigl( \, 0 & \longrightarrow \bbK \, \bigr) \ , \\
    \Phi (\calO (1)) & \quad \mathclap{=} &
    \bigl( \, \bbK & \longrightarrow \bbK^2 \, \bigr) \ , \\
    \Phi (\calO (2)) & \quad \mathclap{=} &
    \bigl( \, \bbK^2 & \longrightarrow \bbK^3 \, \bigr) \ , \\
    & \quad \mathclap{\, \dotsc \dotsc }
\end{aligned} \end{equation}

\paragraph{Self-dual structures.}

We now define and compare
self-dual structures on the categories
$\smash{\Db \Coh (X)}$ and $\smash{\Db \Mod (\bbK \breve{Q})}$.

For coherent sheaves, we take an object $L \in \Db \Coh (X)$
such that 
\[
    \bbR \calHom _X (L, L) \simeq \calO_X \ .
\]
For example, $L$ may be a line bundle, or a shift of a line bundle.
We define the dual operation on the triangulated category $\Db \Coh (X)$ by
\begin{equation}
    E^\vee = \bbR \calHom_X (E, L)
\end{equation}
for a complex of sheaves $E$.
Then $E^{\vee\vee}$ is canonically isomorphic to $E$.
We may insert a sign $\pm 1$ into this identification,
which will correspond to the orthogonal and symplectic cases.

On the other hand, when $\breve{Q}$ is the quiver with relations
given by~\cref{eq-quiver-pn},
we define a contravariant involution of $\breve{Q}$
by sending the vertex $i$ to the vertex $n - i$ for $i = 0, \dotsc, n$,
and sending the arrow $\smash{\varphi_{i,j}}$
to the arrow $\smash{\varphi_{n+1-i,j}}$
for $i = 1, \dotsc, n$ and $j = 0, \dotsc, n$.
We define a dual operation on the triangulated category $\Db \Mod (\bbK \breve{Q})$
to be induced by the dual operation on the abelian category $\Mod (\bbK \breve{Q})$
induced by the contravariant involution of $\breve{Q}$ described above.

Using the functor $\Phi$,
the quiver-induced dual operation on $\Db \Mod (\bbK \breve{Q})$
corresponds to the choice
\[
    L = \calO (-1) [1]
\]
for the dual operation on $\Db \Coh (\bbP^n)$.

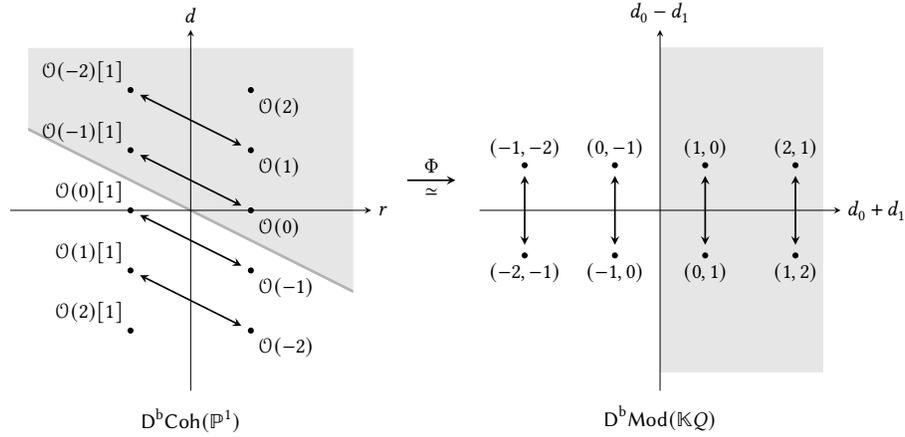
\begin{figure}[ht]
    \centering
    \scalebox{.8}{%
    \begin{tikzpicture}
        \begin{scope}[shift={(-4, 0)}]
            \fill [black!10] (-2.7, 1.35) -- (2.7, -1.35) -- (2.7, 2.7) -- (-2.7, 2.7);
            \fill [black!30] (-2.7, 1.325) -- (2.7, -1.375) -- (2.7, -1.325) -- (-2.7, 1.375);
            \draw [-stealth] (0, -3) -- (0, 3) node [anchor=south] {$d$};
            \draw [-stealth] (-3, 0) -- (3, 0) node [anchor=west] {$r$};
            \foreach \i in {-2, ..., 2} {
                \fill (1, \i) circle (.05)
                    node [anchor=north west] {$\calO (\i)$};
            }
            \foreach \i in {2, 1, 0, -1, -2} {
                \fill (-1, -\i) circle (.05)
                    node [anchor=south east] {$\calO (\i) [1]$};
            }
            \foreach \i in {-1, ..., 2} {
                \draw [thick, stealth-stealth, shorten > = 5, shorten < = 5] (-1, \i) -- (1, \i-1);
            }
            \node at (0, -3.5) {$\Db \Coh (\bbP^1)$};
        \end{scope}
        \begin{scope}[shift={(0, .5)}]
            \node [anchor=south] at (0, 0) {$\Phi$};
            \node [anchor=north] at (0, 0) {$\simeq$};
            \draw [thick, -stealth] (-.4, 0) -- (.4, 0);
        \end{scope}
        \begin{scope}[shift={(3.8, 0)}]
            \fill [black!10] (0, -2.7) -- (2.7, -2.7) -- (2.7, 2.7) -- (0, 2.7);
            \draw [-stealth] (0, -3) -- (0, 3) node [anchor=south] {$d_0 - d_1$};
            \draw [-stealth] (-3, 0) -- (3, 0) node [anchor=west] {$d_0 + d_1$};
            \foreach \i in {-2, ..., 1} {
                \fill (1.5 * \i + .75, .75) circle (.05);
                \fill (1.5 * \i + .75, -.75) circle (.05);
            }
            \foreach \i in {-2, ..., 1} {
                \draw [thick, stealth-stealth, shorten > = 5, shorten < = 5] (1.5 * \i + .75, -.75) -- (1.5 * \i + .75, .75);
            }
            \node [anchor=south] at (-2.25, .75) {$(-1, -2)$};
            \node [anchor=south] at (-.75, .75) {$(0, -1)$};
            \node [anchor=south] at (.75, .75) {$(1, 0)$};
            \node [anchor=south] at (2.25, .75) {$(2, 1)$};
            \node [anchor=north] at (-2.25, -.75) {$(-2, -1)$};
            \node [anchor=north] at (-.75, -.75) {$(-1, 0)$};
            \node [anchor=north] at (.75, -.75) {$(0, 1)$};
            \node [anchor=north] at (2.25, -.75) {$(1, 2)$};
            \node at (0, -3.5) {$\Db \Mod (\bbK Q)$};
        \end{scope}
    \end{tikzpicture}
    }
    \caption{The quiver-induced dual operation}
    \label{fig-coh-p1-duality}
\end{figure}

When $X = \bbP^1$,
the quiver-induced dual operation is shown in \cref{fig-coh-p1-duality}.
The left-hand side plots complexes of sheaves on $\bbP^1$
using their rank $r$ and degree $d$,
while the right-hand side plots complexes of representations of $Q = \tilde{A}_1$
using the dimension vector $(d_0, d_1)$.
The two-way arrows indicate points that are exchanged by the dual operation.
The shaded region indicates the usual heart of $\Db \Mod (\bbK Q)$,
given by $\Mod (\bbK Q)$, and the corresponding region in $\Db \Coh (\bbP^1)$.
The coordinate systems are arranged so that the phase angles
give rise to Bridgeland stability conditions in both categories,
namely, slope stability for coherent sheaves,
and the stability function $(1, -1)$ for quiver representations.

%% file: coh.tex
\label{sect-coh}

We now apply our theory of enumerative invariants
to the following two cases.

\begin{itemize}
    \item 
        We consider the self-dual category $\cat{Vect} (X)$
        of vector bundles on a smooth algebraic curve $X$,
        whose self-dual objects will be
        principal bundles on the curve
        for the orthogonal or symplectic groups.

    \item
        More generally, we also study self-dual
        quasi-abelian subcategories in the derived category
        $\Db \Coh (X)$ of coherent sheaves on
        a smooth, projective variety $X$,
        where the moduli of self-dual objects
        can be seen as a compactification
        of the moduli of orthogonal or symplectic bundles on the variety.
\end{itemize}
These two cases will be discussed in \cref{sect-curves,sect-coh-general},
respectively.
Then, in \cref{sect-surfaces},
we discuss a modified version of our theory
that works for surfaces of Kodaira dimension~$\leq 0$,
and in \cref{sect-threefolds},
we discuss a possible generalization to
Calabi--Yau threefolds,
for which we can define self-dual Donaldson--Thomas invariants.

\subsection{Orthogonal and symplectic bundles on curves}
\label{sect-curves}

\paragraph{Principal bundles.}
\label{para-osp-bundles}

Recall that for a $\bbK$-variety $X$
and a linear algebraic group~$G$ over~$\bbK$,
a \emph{principal $G$-bundle} on $X$
is a morphism $\pi \colon P \to X$ of $\bbK$-schemes
with a right $G$-action on $P$,
such that $\pi$ is $G$-invariant and \'etale locally isomorphic to
the trivial $G$-bundle.

As in Serre~\cite[\S5.6]{Serre1958},
a principal $G$-bundle on $X$,
where $G = \upO (n)$ with $n \geq 0$,
or $\Sp (n)$ with $n \geq 0$ even,
is equivalent to a vector bundle $E \to X$ of rank $n$,
equipped with a non-degenerate bilinear form
$\phi \in H^0 (X; \Sym^2 E^\vee)$ when $G = \upO (n)$,
or $\phi \in H^0 (X; {\wedge^2} \, E^\vee)$ when $G = \Sp (n)$.

In other words, orthogonal or symplectic bundles
are self-dual objects in the self-dual category of vector bundles on $X$,
as in \cref{eg-vb}.

\paragraph{The setting.}
\label{para-sit-curves}

We now specialize ourselves to the following situation.

Let $X$ be a smooth, projective curve over $\bbK$,
and let $\calA = \cat{Vect} (X)$
be the $\bbK$-linear quasi-abelian category of
vector bundles on $X$ of finite rank.

Fix a line bundle $L \to X$.
Define a self-dual structure $(-)^\vee \colon \calA \simeq \calA^\op$
by sending a vector bundle $E$ to the vector bundle
\[
    E^\vee = \calHom (E, L) \ ,
\]
and choosing a sign $\varepsilon = \pm 1$,
so that the natural isomorphism $\eta \colon (-)^{\vee\vee} \simeq \id$
is given by $\varepsilon$ times the usual identification.
In particular, when $L = \calO_X$, the self-dual objects of $\calA$
can be identified with principal orthogonal or symplectic bundles,
depending on the sign $\varepsilon$, as in \cref{para-osp-bundles}.

For the extra data in \tagref{ExCat}, we take
\begin{equation*}
    C (X) = \bbZ_{>0} \times \bbZ \ , \qquad
    C^\circ (X) = C (X) \cup \{ (0, 0) \} \ , 
\end{equation*}
and for $E \in \calA$, define $\llbr E \rrbr = (r, d)$,
where $r = \rank (E)$ and $d = \deg (E)$.

For the extra data in \tagref{SdCat}, we may take
\begin{equation}
    \label{eq-curves-csd}
    C^\sd (X) = \begin{cases}
        \bbN & \text{if } \varepsilon = +1 \ , \\
        2 \bbN & \text{if } \varepsilon = -1 \ ,
    \end{cases}
\end{equation}
recording the rank of a self-dual vector bundle.
The degree of a self-dual vector bundle $(E, \phi)$
is determined by its rank, $\deg E = \rank E \cdot \deg L / 2$.

\paragraph{Remark.}

In the special case when $L = \calO_X$ and $\varepsilon = +1$,
one can use a more refined version,
\begin{equation}
    \begin{aligned}
        C'^\sd (X) & = \bigl\{
            (\rank (E), w_1 (E), w_2 (E)) \bigm|
            (E, \phi) \in \calA^\sd
        \bigr\} \\
        & \subset \bbZ \oplus H^1 (X; \bbZ_2) \oplus H^2 (X; \bbZ_2) \ ,
    \end{aligned}
\end{equation}
where $H^i (X; \bbZ_2)$ denotes \'etale cohomology,
isomorphic to $\bbZ_2^{\smash{2g}}$ when $i = 1$ and $\bbZ_2$ when $i = 2$,
where $g$ is the genus of $X$,
and $w_i (E)$ denotes the $i$-th Stiefel--Whitney class of $(E, \phi)$,
as in Laborde~\cite{Laborde1976}.
However, we will not use this in the following,
as the simpler and coarser version \cref{eq-curves-csd}
is sufficient for our theory to work.

\paragraph{Verifying the assumptions.}
\label{para-curves-assumptions}

In the setting of \cref{para-sit-curves},
we verify the conditions
\tagref{Mod}, \tagref{SdMod}, \tagref{Fin},
\tagref{Ext}, and \tagref{SdExt}.

Define a categorical moduli stack $\+{\calM}$ of objects in $\calA$ as follows.
For a $\bbK$-scheme~$U$,
let $\+{\calM} (U)$ be the exact category of locally free sheaves on $X \times U$
of finite rank.
By faithfully flat descent of locally free sheaves
\cite[Proposition~1.10]{SGA1VIII},
$\+{\calM}$ satisfies the descent property,
and is an exact stack.
To show that $\+{\calM}$ is algebraic,
we notice that $\calM$ and $\calM^\ex$
are open substacks of the moduli stacks of coherent sheaves on $X$
and of exact sequences of coherent sheaves on $X$, respectively,
as in Joyce~\cite[\S9]{Joyce2006I}.
It was shown there that the latter two moduli stacks are algebraic.
That $\calM^I$ is algebraic follows from
García-Prada~\emph{et al.}~\cite[\S4.1]{GarciaPradaEtal2014}.
This verifies \tagref{Mod}.

For \tagref{SdMod},
we need to define a self-dual structure on the exact category $\+{\calM} (U)$
for any $\bbK$-scheme $U$.
Define $(-)^\vee \colon \+{\calM} (U) \simto \+{\calM} (U)^\op$
sending a locally free sheaf $E \to X \times U$
to $\calHom (E, L \boxtimes \calO_U)$,
and identify $E^{\vee\vee}$ with $E$ using the sign $\varepsilon$.

The condition \tagref{Fin} follows from Joyce~\cite[Theorem~9.7]{Joyce2006I}.

For \tagref{Ext},
the condition \tagref{Ext1} holds since
$X$ is a smooth, projective curve.
To construct the extension bundle,
let $\calU \to X \times \calM$ be the universal locally free sheaf,
and define a perfect complex $\calE^\bullet \to \calM \times \calM$ by
\begin{equation}
    \calE^\bullet = (\pr_{2,3})_* \, \calHom 
    \bigl( \pr_{1,3}^* (\calU), \pr_{1,2}^* (\calU) \bigr) \ ,
\end{equation}
where the $\pr_{i,j}$ are projections from $X \times \calM \times \calM$
to the product of its $i$-th and $j$-th factors,
and we are using the derived pushforward functor.
Then $\calE^\bullet$ is concentrated in degrees $0$ and $1$.
By \cref{thm-2vb-morphism},
we obtain a $2$-vector bundle
\[
    \calExt^{1/0} \longrightarrow \calM \times \calM \ ,
\]
as the realization of $\calE^\bullet [1]$,
which can be identified with $\pi_1 \times \pi_2 \colon \calM^\pn{2} \to \calM \times \calM$.
Indeed, for any affine $\bbK$-scheme $U$,
and any morphism $U \to \calM \times \calM$ classifying locally free sheaves
$E, F \to X \times U$,
the groupoid $\calExt^{1/0} (U)$
is the groupoid of sections of
$(\pr_2)_* \, \calHom_{X \times U} (F, E) [1]$ over $U$,
which is the groupoid of objects of the $2$-vector space
$\pi_* \, \calHom_{X \times U} (F, E) [1]$,
where $\pi \colon X \times U \to {*}$ is the projection,
since the derived pushforward respects composition.
On the other hand, $\calM^\pn{2} (U)$ is the groupoid of extensions
of $F$ by $E$ on $X \times U$,
which is the groupoid of objects of the same $2$-vector space.
This shows that $\calExt^{1/0} \simeq \calM^\pn{2}$
as $2$-vector bundles over $\calM \times \calM$.
We have the Euler pairing
\begin{equation}
    \chi (\alpha_1, \alpha_2) =
    (1 - g) \, r_1 r_2 - r_1 d_2 + r_1 d_2
\end{equation}
for classes $\alpha_i = (r_i, d_i) \in C^\circ (X)$ for $i = 1, 2$,
which follows from the Riemann--Roch formula
\begin{equation}
    \chi (\alpha_1, \alpha_2) = \int_X \ch (E_2^\vee) \ch (E_1) \td (X),
\end{equation}
for objects $E_i \in \calA$ of classes $\alpha_i$ for $i = 1, 2$,
where the integrand is an element of the Chow ring of $X$;
$\td (X) = 1 + (1/2) \, c_1 (T_X)$ is the Todd class of $X$,
where $T_X$ is the tangent bundle of $X$;
$\ch (E_1) = r_1 + c_1 (E_1)$, and
$\ch (E_2^\vee) = r_2 - c_1 (E_2)$.

For the condition \tagref{SdExt},
one computes the self-dual Euler pairing
\begin{equation} 
    \chi^\sd (\alpha, \theta) = \begin{cases}
        \displaystyle \chi (\alpha, j (\theta)) +
        (1 - g) \, \binom{r + 1}{2} + (r + 1) \, d
        & \text{if } \varepsilon = +1 \ , \\[1ex]
        \displaystyle \chi (\alpha, j (\theta)) +
        (1 - g) \, \binom{r}{\, 2 \,} + (r - 1) \, d
        & \text{if } \varepsilon = -1 \ ,
    \end{cases}
\end{equation}
for classes $\alpha = (r, d) \in C^\circ (X)$
and $\theta \in C^\sd (X)$.
Indeed, when $\varepsilon = +1$, we have
\begin{equation}
    \chi^\sd (\alpha, \theta) =
    \chi (\alpha, j (\theta)) +
    \int_X {} \ch (\mathrm{Sym}^2 (E)) \td (X) \ ,
\end{equation}
for $E \in \calA$ of class $\alpha$.
By the splitting principle,
we may assume that $E \simeq L_1 \oplus \cdots \oplus L_r$
is a sum of line bundles, so that
\begin{align}
    \mathrm{rank} (\mathrm{Sym}^2 (E)) & =
    \binom{r + 1}{2} \ , \\
    c_1 (\mathrm{Sym}^2 (E)) & =
    \sum_{1 \leq i \leq j \leq r} {}
    (c_1 (L_i) + c_1 (L_j))
    = (r + 1) {} \sum_{i=1}^n c_1 (L_i) \ .
\end{align}
These give the desired expression.
A similar technique works when $\varepsilon = -1$,
where we use $\wedge^2 (E)$ instead of $\mathrm{Sym}^2 (E)$.

\paragraph{Stability conditions.}

In the setting of \cref{para-sit-curves},
define a self-dual stability condition $\tau \colon C (X) \to \bbQ$,
called the \emph{slope stability condition}, by
\begin{equation}
    \tau ((r, d)) = d/r - \deg L / 2
\end{equation}
for $(r, d) \in C (X)$.
We verify the condition \tagref{Stab} for this stability condition.

Indeed, we may assume that $\deg L = 0$,
since shifting $\tau$ by a constant does not affect \tagref{Stab}.
Now, \tagref{Stab12} follow from Huybrechts--Lehn
\cite[Proposition~2.3.1 and Theorem~3.3.7]{HuybrechtsLehn2010}.
For \tagref{Stab3}, let $E \in \calM (U)$ be such a family.
We may assume that $U$ is connected,
so there exists $(r, d) \in C (\calA)$ such that
$\llbr E_x \rrbr = (r, d)$ for all $x$.
By \cite[Lemma~1.7.6]{HuybrechtsLehn2010},
there exists $m \in \bbN$ such that
$E_x$ is $m$-regular for all $x \in U$.
By \cite[Lemma~1.7.2~(i)]{HuybrechtsLehn2010},
each $E_x$ admits a surjection from $\calO (-m)^{\oplus n}$ for some $n$.
This means that the slopes of the Harder--Narasimhan factors of $E_x$
are bounded below by $t = \tau ( \llbr \calO (-m) \rrbr )$.
Hence, they are also bounded above by $d - (r - 1) t$.
Therefore, we can take $S_E$ to be the set of classes with rank $\leq r$
and with slope between $t$ and $d - (r - 1) t$.

\paragraph{Conclusion.}

By the above arguments,
all results in \crefrange{sect-alg}{sect-invariants}
apply to the situation of \cref{para-sit-curves},
with the stability condition $\tau$ defined above.
However, there are no meaningful wall-crossing formulae in this case,
as we do not have other interesting stability conditions
to compare with.

\subsection{Orthogonal and symplectic sheaves}
\label{sect-coh-general}

\paragraph{}

Let $X$ be a smooth, projective variety over $\bbK$.
As was mentioned in \cref{para-non-eg-coh},
the category $\Coh (X)$ of coherent sheaves on $X$
does not admit self-dual structures in general.

On the other hand, the derived category $\Db \Coh (X)$
does admit self-dual quasi-abelian subcategories
for which our theory can apply.
We will call the self-dual objects in these self-dual categories
\emph{orthogonal sheaves} or \emph{symplectic sheaves}.
For curves, one example of such a quasi-abelian subcategory
is the category of vector bundles on the curve,
and we recover the situation of \cref{sect-curves}.
However, in general, such quasi-abelian subcategories
are not contained in the usual heart $\Coh (X)$,
and self-dual objects will be complexes of coherent sheaves in general.

Note that this notion of orthogonal or symplectic sheaves
is different from the approach of
Fernandez Herrero \emph{et al.}~\cite{FHGZ2021},
who defined a different compactification
of the moduli of principal $G$-bundles for
connected reductive groups~$G$.

We discuss two versions of this construction,
using \emph{polynomial Bridgeland stability}
and \emph{Bridgeland stability}, respectively, in
\cref{para-polynomial-bridgeland,para-bridgeland-stability} below.
We will show that the assumptions \tagref{SdMod} and \tagref{Fin1}
are satisfied in both cases,
so that our invariants in \cref{sect-invariants} are defined
for self-dual weak stability conditions satisfying \tagref{Stab}.
However, since \tagref{Ext} fails when $\dim X > 1$,
we do not expect the wall-crossing formulae to hold for these invariants,
except in the case of certain surfaces,
which will be discussed in \cref{sect-surfaces} below.

\paragraph{Polynomial Bridgeland stability conditions.}
\label{para-polynomial-bridgeland}

Let $X$ be a smooth, projective variety over $\bbK$, of dimension $n$.
Let $A^* (X)$ be the Chow ring of $X$,
and $\Num (X)$ the group of cycles in $A^* (X)$ modulo numerical equivalence.

Let $\calC = \Db \Coh (X)$ be the bounded derived category
of coherent sheaves on $X$.
Choose an object $L \in \calC$ such that
$L [-\ell]$ is a line bundle for some $\ell \in \bbZ$.
Define a dual operation $(-)^\vee \colon \calC \simto \calC^\op$ by
\begin{equation}
    \label{eq-dbcoh-dual}
    (-)^\vee = \bbR \calHom ( -, L ) \ .
\end{equation}
Typical choices of $L$ include $\calO_X$ and $\omega_X [n]$,
giving the usual duality and Serre/Verdier duality, respectively.
Choose a sign $\varepsilon = \pm 1$, and use it to identify
$E^{\vee\vee}$ with $E$ for all $E \in \calC$.

Write $\beta = c_1 (L [-\ell]) / 2 \in A^1 (X)_{\bbR}$,
and let $\omega \in A^1 (X)_{\bbR}$ be an ample class.
Consider the \emph{polynomial Bridgeland stability condition}
on $\calC$, in the sense of Bayer~\cite{Bayer2009},
given by the polynomial central charge
\begin{equation}
    \label{eq-polynomial-bridgeland}
    Z^\poly_{L, \omega} (E, m) = \upi^{n - \ell} \cdot \int_X
        \exp (-\beta - \upi m \omega) \cdot \ch (E) \ ,
\end{equation}
where $E \in \calC$, and $m$ is a real parameter.
We will consider the limit $m \gg 0$.
Here, we introduced the factor $\upi^n$
so that the central charge can be self-dual.
Indeed, we have
\begin{equation}
    \label{eq-polynomial-bridgeland-sd}
    Z^\poly_{L, \omega} (E^\vee, m) =
    \overline{Z^\poly_{L, \omega} (E, m) \vphantom{^0}}
\end{equation}
for all $E \in \calC$, where $(-)^\vee$ is as in \cref{eq-dbcoh-dual},
and the bar denotes complex conjugation.
As in Bayer~\cite[Theorem~3.2.2]{Bayer2009},
choose the \emph{perversity function}
$p \colon \{ 0, \dotsc, n \} \to \bbZ$ to be given by
$p (d) = \lfloor (n - \ell - d - 1) / 2 \rfloor$ for all $d$.
Then we have a phase function
$\phi^\poly_{L, \omega} (E, m)$ for $Z^\poly_{L, \omega} (E^\vee, m)$
in the sense of \cite[\S\S2.2--2.3]{Bayer2009},
defined for $m \gg 0$ and all $E \in \calC$ with $\ch E \neq 0$,
such that $Z^\poly_{L, \omega} (E, m) \in \bbR_{>0} \cdot 
\smash{\upe^{\upi \phi^\poly_{L, \omega} (E, m)}}$
for $m \gg 0$, and
\begin{equation}
    \label{eq-polynomial-bridgeland-phase-sd}
    \phi^\poly_{L, \omega} (E^\vee, m) = - \phi^\poly_{L, \omega} (E, m)
\end{equation}
for all such $E$. 
We have quasi-abelian full subcategories $\calP^\poly_{L, \omega} (I) \subset \calC$,
closed under extensions,
for all intervals $I$ in the set of phase functions
such that $I \cap (I + 1) = \varnothing$,
and we have $\calP^\poly_{L, \omega} (I)^\vee = \calP^\poly_{L, \omega} (-I)$ for all such intervals~$I$.
In particular, the quasi-abelian subcategory
\begin{equation}
    \calA_L = \calP^\poly_{L, \omega} \bigl(
        \mathopen{]} -1/2, 1/2 \mathclose{[}
    \bigr) \subset \calC
\end{equation}
is self-dual, and contains the self-dual abelian category $\calP^\poly_{L, \omega} (0)$.

Alternatively, the category $\calA_L$ can be described as follows.
Let $\smash{\calA^{p'}}$ be the abelian category of
\emph{perverse coherent sheaves},
as in Bayer~\cite[\S3.1]{Bayer2009},
with the perversity function $p' \colon \{ 0, \dotsc, n \} \to \bbZ$ given by
$p (d) = \lfloor (n - \ell - d) / 2 \rfloor$ for all $d$.
Then $\smash{\calA^{p'}} = \calP^\poly_{L, \omega} (\mathopen{]} -1/2, 1/2 ]) \subset \calC$,
and we have
\begin{equation}
    \label{eq-perverse-cap-dual}
    \calA_L = \smash{\calA^{p'}} \cap (\smash{\calA^{p'}})^\vee
\end{equation}
in $\calC$. In particular, this shows that
$\calA_L$ only depends on $L$, and justifies its notation.

There is a self-dual stability condition 
\[
    \tau^\poly_{L, \omega} \colon C (\calA_L) \longrightarrow T
\]
induced by the polynomial Bridgeland stability condition $Z^\poly_{L, \omega}$\,,
where $T$ is the totally ordered set of phase functions.

\paragraph{Verifying the assumptions.}
\label{para-polynomial-bridgeland-moduli}

We sketch how to verify the conditions \tagref{SdMod} and \tagref{Fin1}
for the self-dual $\bbK$-linear quasi-abelian category
$\calA_L$ defined above.

For the extra data in \tagref{ExCat} and \tagref{SdCat}, take
\begin{alignat}{2}
    C^\circ (\calA_L) & = \{
        \ch (E) \mid E \in \calA_L
    \}
    && \subset \Num (X)_{\bbQ} \ ,
    \\
    C^\sd (\calA_L) & = \{
        \ch (E) \mid (E, \phi) \in (\calA_L)^\sd
    \}
    && \subset \Num (X)_{\bbQ} \ ,
\end{alignat}
with $\llbr E \rrbr = \ch (E) \in C^\circ (\calA_L)$
for $E \in \calA_L$,
and $\llbr E, \phi \rrbr = \ch (E) \in \smash{C^\sd (\calA_L)}$
for $(E, \phi) \in \smash{(\calA_L)^\sd}$.

For \tagref{Mod}, we define
a categorical moduli stack $\+{\calM}$ of objects in $\calA_L$ as follows.
By Lieblich~\cite{Lieblich2006}
or To\"en--Vaqui\'e \cite[Corollary~3.21]{ToenVaquie2007},
there is an algebraic stack $\oline{\calM}$, locally of finite type,
parametrizing objects $E \in \calC$ with $\Ext^{<0} (E, E) = 0$.
We have an open substack $\calM \subset \oline{\calM}$
of objects in $\calA_L$\,,
which can be defined via~\cref{eq-perverse-cap-dual},
since being perverse coherent is an open condition.
For any $\bbK$-scheme $U$, now define $\+{\calM} (U)$
to be the full subcategory of $\cat{Perf} (X \times U)$
spanned by objects in $\calM (U)$,
with the induced exact structure.
Then $\calM^I$ and $\calM^\ex$
are also algebraic stacks locally of finite type,
since the former admits a representable, finite type morphism
to $\calM \times \calM$, and the latter an open immersion into $\calM^I$.

For \tagref{SdMod},
we need to define a self-dual structure on $\+{\calM} (U)$
for any $\bbK$-scheme~$U$.
Define $(-)^\vee \colon \+{\calM} (U) \simto \+{\calM} (U)^\op$
sending an object $E \in \+{\calM} (U) \subset \cat{Perf} (X \times U)$ to
$\bbR \calHom (E, L \boxtimes \calO_U)$,
and identify $E^{\vee\vee}$ with $E$ using the sign $\varepsilon$.

For \tagref{Fin1},
the morphism $\pi_1 \times \pi_2 \colon \calM^\pn{2} \to \calM \times \calM$
is the classical truncation of a \emph{derived vector bundle},
by a similar argument as in \cref{para-curves-assumptions},
and hence is of finite type.

\paragraph{Invariants.}
\label{para-polynomial-bridgeland-inv}

In the situation of \cref{para-polynomial-bridgeland},
for any self-dual weak stability condition $\tau$
on $\calA_L$ satisfying \tagref{Stab},
by results in
\cref{sect-delta-epsilon,sect-motivic-inv,sect-num-inv},
we have well-defined invariants
\begin{alignat*}{5}
    \delta_\alpha (\tau) , && \ 
    \epsilon_\alpha (\tau) & \in \SF (\calM_\alpha) \ , & \quad
    \upI_\alpha (\tau) , && \ 
    \upJ_\alpha (\tau) & \in \Mhat_\bbK \ , & \quad
    \chiJ_\alpha (\tau) & \in \bbQ \ ,
    \\
    \delta^\sd_\theta (\tau) , && \ 
    \epsilon^\sd_\theta (\tau) & \in \SF (\calM^\sd_\theta) \ , & \quad
    \upI^\sd_\theta (\tau) , && \ 
    \upJ^\sd_\theta (\tau) & \in \Mhat_\bbK \ , & \quad
    \chiJ^\sd_\theta (\tau) & \in \bbQ \ ,
\end{alignat*}
for all $\alpha \in C^\circ (\calA_L)$
and $\theta \in C^\sd (\calA_L)$.
Also, when changing between weak stability conditions on $\calA_L$,
the $\delta, \epsilon, \delta^\sd, \epsilon^\sd$ invariants
satisfy wall-crossing formulae in \cref{thm-wcf-main},
as long as the involved weak stability conditions
satisfy the assumptions given there.
However, the other invariants are not expected to satisfy
wall-crossing formulae discussed in
\crefrange{sect-motivic-inv}{sect-num-inv},
due to the failure of \tagref{Ext} and \tagref{SdExt}.

In particular, if one can verify \tagref{Stab}
for the self-dual stability condition
$\tau^\poly_{L, \omega}$ in \cref{para-polynomial-bridgeland},
then these invariants will be defined for $\tau^\poly_{L, \omega}$.
This is known when $\dim X \leq 2$,
where the case of surfaces follows from
Bayer~\cite[Lemma~4.2]{Bayer2009}.

\paragraph{Bridgeland stability conditions.}
\label{para-bridgeland-stability}

We continue to use the notations
$X$, $\calC$, $L$, $\beta$, $\omega$, $\varepsilon$
from \cref{para-polynomial-bridgeland},
and we mention \emph{Bridgeland stability conditions}
on $\calC$, as in Bridgeland~\cite{Bridgeland2007},
that are self-dual with respect to the dual operation $(-)^\vee$
in \cref{eq-dbcoh-dual}, depending on $L$ and~$\varepsilon$.
Bridgeland stability conditions are
typically difficult to construct,
and are not known to exist for general $X$ with $\dim X > 2$.

We follow Bayer--Macr\`i--Toda~\cite{BayerMacriToda2014},
who constructed Bridgeland stability conditions
for $X$ a curve, a surface, or a threefold
satisfying the conjectural Bogomolov--Gieseker inequality,
\cite[Conjecture~3.2.7]{BayerMacriToda2014}.

Consider the central charge
\begin{equation}
    Z_{L, \omega} (E) = \upi^{n - \ell} \cdot \int_X
        \exp (-\beta - \upi \omega) \cdot \ch (E) \ ,
\end{equation}
where $E \in \calC$.
For $X$ as above,
Bayer--Macr\`i--Toda~\cite{BayerMacriToda2014}
showed that $Z_{L, \omega}$ defines a Bridgeland stability condition on $\calC$.
Arguing as in \cref{para-polynomial-bridgeland},
with a suitable choice of phase function,
this Bridgeland stability condition is self-dual, in that
\begin{equation}
    \calP_{L, \omega} (t)^\vee = \calP_{L, \omega} (-t)
\end{equation}
for all $t \in \bbR$,
where $\calP_{L, \omega} (t) \subset \calC$ is the full subcategory
of semistable objects of phase~$t$.
In particular, the quasi-abelian subcategory
\begin{equation}
    \calA_{L, \omega} = \calP_{L, \omega} \bigl(
        \mathopen{]} -1/2, 1/2 \mathclose{[}
    \bigr) \subset \calC
\end{equation}
is self-dual, and contains the self-dual abelian category $\calP_{L, \omega} (0)$.
Moreover, $Z_{L, \omega}$ induces a self-dual stability condition
$\tau_{L, \omega}$ on $\calA_{L, \omega}$.

\paragraph{Verifying the assumptions.}

We now sketch the proof of \tagref{SdMod} and \tagref{Fin1}
for $\calA_{L, \omega}$, as well as \tagref{Stab} for $\tau_{L, \omega}$,
in the setting of \cref{para-bridgeland-stability},
where $X$ is a curve, a surface, or a threefold
satisfying \cite[Conjecture~3.2.7]{BayerMacriToda2014}.

Consider the stack $\oline{\calM}$ in
\cref{para-polynomial-bridgeland-moduli}.
There is an open substack $\calM \subset \oline{\calM}$
of objects of $\calA_{L, \omega}$,
and open substacks
$\calM^\ss_\alpha (\tau_{L, \omega}) \subset \calM$ of finite type,
consisting of $\tau_{L, \omega}$-semistable objects
with Chern character $\alpha$.
The case of surfaces is by
Toda~\cite[\S4]{Toda2008} and Piyaratne--Toda~\cite[\S1.1]{PiyaratneToda2019},
and the case of threefolds
by Piyaratne--Toda~\cite[Corollary~4.21]{PiyaratneToda2019};
see also Alper--Halpern-Leistner--Heinloth
\cite[Examples~7.22 and~7.29]{AHLH2023},
Halpern-Leistner \cite[Example~6.2.8]{HalpernLeistner2022}.
Then, arguing as in \cref{para-polynomial-bridgeland-moduli}
proves \tagref{SdMod} and \tagref{Fin1},
and \tagref{Stab} for $\tau_{L, \omega}$
also follows from this description.

Therefore, similarly to \cref{para-polynomial-bridgeland-inv},
we have well-defined enumerative invariants for
any self-dual weak stability condition $\tau$
on $\calA_{L, \omega}$ satisfying \tagref{Stab},
and in particular, for the self-dual stability condition
$\tau_{L, \omega}$.

\subsection{Surfaces of Kodaira dimension \texorpdfstring{$\leq$}{≤} 0}
\label{sect-surfaces}

\paragraph{}

We now consider a special case of
orthogonal and symplectic sheaves
discussed in \cref{sect-coh-general} above,
where $X$ is an algebraic surface with Kodaira dimension~$\leq 0$.
In this case, the invariants defined in \cref{sect-coh-general}
above will satisfy wall-crossing formulae,
similarly to the case of curves in \cref{sect-curves},
since we have a weaker version of the assumptions
\tagref{Ext} and \tagref{SdExt}.
This is a self-dual analogue of results in
Joyce~\cite[\S6.4]{Joyce2008IV}.

By standard results on classification of algebraic surfaces,
as described in, for example,
Iskovskikh--Shafarevich~\cite{IskovskikhShafarevich},
such a surface $X$ either has Kodaira dimension~$-\infty$,
so it is a rational surface or a ruled surface,
or it has Kodaira dimension~$0$,
so it is a K3 surface, an abelian surface, an Enriques surface,
or a hyperelliptic surface.

\paragraph{The setting.}
\label{para-surf-setting}

Let $X$ be an algebraic surface
of Kodaira dimension~$\leq 0$.
Choose a line bundle $L$ on $X$,
a sign $\varepsilon = \pm 1$,
and an ample class $\omega \in \smash{A^1} (X)_{\bbR}$
as in \cref{para-polynomial-bridgeland}.
Consider the self-dual full subcategory
\begin{equation}
    \label{eq-def-al-plus}
    \calA_L^+ = \calP^\poly_{L, \omega} \bigl(
        \mathopen{]} -1/4, 1/4 \mathclose{[}
    \bigr) \subset \calA_L \ ,
\end{equation}
which consists of all objects $E \in \calA_L$
with $\rank E > 0$,
so in particular, it is independent of $\omega$.
Choose the data $C^\circ (\calA_L^+)$ and $C^\sd (\calA_L^+)$
as the corresponding subsets of $C (\calA_L)$ and $C^\sd (\calA_L)$,
respectively, as in \cref{para-polynomial-bridgeland-moduli},
and we use the moduli stacks constructed there.

\paragraph{}
\label{para-surf-ext}
\label[proposition]{prop-surf-ext}

The main property of this class of surfaces that we will use
is the following weaker version of \tagref{Ext}.

\begin{proposition*}
    In the setting of \cref{para-surf-setting},
    let $\alpha_1, \alpha_2 \in C (\calA_L^+)$,
    and let $\calM^>_{\alpha_1, \alpha_2} (\tau^\poly_{L, \omega}) \subset \calM_{\alpha_1} \times \calM_{\alpha_2}$
    be the open substack of pairs $(E_1, E_2)$ such that
    the minimal slope of~$E_1$ with respect to $\tau^\poly_{L, \omega}$
    is greater than the maximal slope of~$E_2$.

    Then, for any $(E_1, E_2) \in \calM^>_{\alpha_1, \alpha_2} (\tau^\poly_{L, \omega})$,
    we have $\Ext^2 (E_2, E_1) = 0$.
    Moreover, the morphism
    \begin{equation}
        \pi_1 \times \pi_2 \colon \calM^\pn{2}_{\alpha_1, \alpha_2}
        \longrightarrow
        \calM_{\alpha_1} \times \calM_{\alpha_2} \ ,
    \end{equation}
    restricted to the preimage of
    $\calM^>_{\alpha_1, \alpha_2} (\tau^\poly_{L, \omega})$
    is a $2$-vector bundle of rank $-\chi (\alpha_1, \alpha_2)$,
    where
    \begin{equation}
        \label{eq-surf-chi}
        \chi (\alpha_1, \alpha_2) =
        \int_X \alpha_1 \, \alpha_2^\vee \, \td (X) \ .
    \end{equation}
\end{proposition*}

\begin{proof}
    As in Joyce~\cite[Lemma~6.18]{Joyce2008IV},
    we have either $c_1 (K_X) \cdot \omega < 0$ or $c_1 (K_X) = 0$,
    where $K_X$ is the canonical bundle of~$X$.
    One can use this to verify that
    the operation $- \otimes K_X$ on $\calA_L^+$
    does not increase the $\tau^\poly_{L, \omega}$-slope,
    and an argument as in
    \cite[Propositions~6.19 and~6.20]{Joyce2008IV}
    shows that $\Ext^2 (E_2, E_1) = 0$
    for such $(E_1, E_2)$.
    As in \cref{para-polynomial-bridgeland-moduli},
    the morphism $\pi_1 \times \pi_2$ is the classical truncation
    of a derived vector bundle,
    and the vanishing of $\Ext^2$ implies that
    it is a $2$-vector bundle over $\calM^>_{\alpha_1, \alpha_2} (\tau^\poly_{L, \omega})$.
\end{proof}

\paragraph{}
\label{para-surf-sdext}

Similarly, we have a weaker version of \tagref{SdExt}.
For $\alpha \in C (\calA_L^+)$,
let $\calM^>_\alpha (\tau^\poly_{L, \omega}) \subset \calM_\alpha$ be the open substack of objects $E$
such that the minimal slope of $E$ is greater than $0$.
Then $\Ext^2 (E^\vee, E) = 0$ for such $E$,
and the morphism $\pi_1 \colon (\calM^\pn{2}_{\smash{\alpha, \alpha^\vee}})^{\bbZ_2} \to \calM_\alpha$
restricted to the preimage of $\calM^>_\alpha (\tau^\poly_{L, \omega})$
is a $2$-vector bundle of rank $-\chi^\sd (\alpha, 0)$, 
with the notation~$\chi^\sd$ as in \cref{para-euler-sd}, where
\begin{equation}
    \label{eq-surf-chi-sd}
    \chi^\sd (\alpha, \theta) = \chi (\alpha, j (\theta)) +
    \int_X \Sym^2 (\alpha) \, \td (X) \ ,
\end{equation}
where $\Sym^2 (\alpha)$ is the Chern character of either
$\Sym^2 (E)$ or ${\wedge^2} (E)$, depending on the sign $\varepsilon$,
for any object $E$ with Chern character $\alpha$.

\paragraph{Wall-crossing for $\upI$ and $\upI^\sd$.}
\label{para-surf-wcf-i}

Using the properties in \crefrange{para-surf-ext}{para-surf-sdext}
in place of the conditions \tagref{Ext} and \tagref{SdExt},
the relation \cref{eq-psi-compat-star}
still holds if $f \boxtimes g$ is supported on 
$\calM^>_{\alpha_1, \alpha_2} (\tau^\poly_{L, \omega})$ for some $\alpha_1, \alpha_2 \in C (\calA_L^+)$,
and \cref{eq-psi-compat-diamond} holds if
$f \boxtimes h$ is supported on the preimage of
$\calM^>_{\smash{\alpha, j (\theta)}} (\tau^\poly_{L, \omega})$ in $\calM_\alpha \times \calM^\sd_\theta$
for some $\alpha \in C (\calA_L^+)$
and $\theta \in C^\sd (\calA_L^+) \setminus \{ 0 \}$,
or on $\calM^>_\alpha (\tau^\poly_{L, \omega}) \times \calM^\sd_0 \subset \calM_\alpha \times \calM^\sd_0$.

We argue that these properties are enough for the wall-crossing formulae
in \cref{thm-wcf-main} to hold for the invariants
$\upI_\alpha (\tau^\poly_{L, \omega})$ and $\upI^\sd_\theta (\tau^\poly_{L, \omega})$,
similarly to the arguments in Joyce~\cite[\S6.4]{Joyce2008IV}.

Let $\omega_1, \omega_2 \in A^1 (X)_{\bbR}$ be ample classes.
We claim that the change from
$\tau^\poly_{L, \omega_1}$ to $\tau^\poly_{L, \omega_2}$
is globally finite, in the sense of \cref{def-finite-change}.
Indeed, this was shown in Joyce~\cite[Theorem~5.16]{Joyce2008IV}
for slope stability and Gieseker stability,
a key ingredient being the \emph{Bogomolov inequality}
$\Delta (\alpha) \geq 0$ for classes $\alpha$ admitting semistable objects,
where $\Delta (\alpha) = \int_X {} \bigl( \ch_1 (\alpha)^2 - 2 \ch_0 (\alpha) \ch_2 (\alpha) \bigr)$.
The same inequality holds for polynomial Bridgeland stability
by Li--Qin~\cite[Theorem~1.1]{LiQin2011},
and a similar argument works in our case as well.

Joyce~\cite[Theorem~6.21]{Joyce2008IV} then showed that
the wall-crossing formulae for the $\upI$ and $\upI^\sd$ invariants
hold for Gieseker stability, when changing
between ample classes $\omega_1, \omega_2 \in A^1 (X)_{\bbR}$.
The strategy of the proof is to consider the family
$\omega (t) = (1 - t) \omega_1 + t \omega_2$ for $t \in [0, 1]$,
and then identify a finite set of walls $W \subset [0, 1]$,
so that all the involved invariants
are unchanged when no wall is crossed.
When crossing each wall $t \in W$,
the $\upI$ invariant at $t$ can be expressed in terms of
the $\upI$ invariants at $t - \epsilon$ and $t + \epsilon$ for small $\epsilon > 0$,
as in \cref{eq-wcf-i}, despite the failure of \tagref{Ext},
since all the Hall products involved
are covered by \cref{para-surf-ext},
so that \cref{eq-psi-compat-star} holds for them.
Consequently, the wall-crossing formula \cref{eq-wcf-i}
is valid when crossing each wall, and hence globally.

The same strategy also works for
polynomial Bridgeland stability,
and for the self-dual invariants as well.
Note that there is no extra global finiteness requirements
for the self-dual invariants.
We have thus obtained the following result.

\begin{theorem*}
    In the situation of \cref{para-surf-setting},
    let $\omega_1, \omega_2 \in A^1 (X)_{\bbR}$ be ample classes.
    Then the wall-crossing formulae
    \cref{eq-wcf-i,eq-wcf-isd} hold
    with $\tau^\poly_{\smash{L, \omega_1}}, \tau^\poly_{\smash{L, \omega_2}}$
    in place of $\tau, \tilde{\tau}$.
    \QED
\end{theorem*}

\paragraph{Wall-crossing for $\upJ$ and $\upJ^\sd$.}
\label{para-surf-wcf-j}

The methods in \cref{para-surf-wcf-i}
do not work for the $\upJ$ and $\upJ^\sd$ invariants directly,
since the Hall algebra and module operations
involved in defining the $\epsilon, \epsilon^\sd$ functions
are not covered by \crefrange{para-surf-ext}{para-surf-sdext}.

We can work around this problem following
Joyce~\cite[Definition~6.22]{Joyce2008IV}.
Namely, we redefine the invariants
$\upJ_\alpha (\tau^\poly_{L, \omega})$ and $\upJ^\sd_\theta (\tau^\poly_{L, \omega})$
by the relations \cref{eq-inv-def-epsilon,eq-inv-def-epsilon-sd},
instead of integrating the $\epsilon$ and $\epsilon^\sd$ functions.
Using this alternative definition, it immediately follows that
the wall-crossing formulae hold for these invariants.

\begin{theorem*}
    In the situation above, 
    let $\omega_1, \omega_2 \in A^1 (X)_{\bbR}$ be ample classes.
    Then the wall-crossing formulae
    \cref{eq-wcf-j,eq-wcf-jsd} hold
    for the $\upJ$ and $\upJ^\sd$ invariants redefined above,
    with $\tau^\poly_{\smash{L, \omega_1}}, \tau^\poly_{\smash{L, \omega_2}}$
    in place of $\tau, \tilde{\tau}$.
    \QED
\end{theorem*}

However, since we have changed the definition of
$\upJ_\alpha (\tau^\poly_{L, \omega})$ and $\upJ^\sd_\theta (\tau^\poly_{L, \omega})$,
the no-pole theorems may no longer hold,
except when $X$ is a del~Pezzo surface, as discussed below.

\paragraph{Del~Pezzo surfaces.}

Now, assume that $X$ is a \emph{del~Pezzo surface},
also known as a $2$-dimensional Fano variety,
which means that the anti-canonical bundle $\smash{K_X^{-1}}$ is ample.
By standard results on classification of algebraic surfaces,
$X$ is either $\bbP^1 \times \bbP^1$,
or a blow-up of $\bbP^2$ at $d$ points, where $0 \leq d \leq 9$.

In this case, results in \crefrange{para-surf-ext}{para-surf-sdext}
can be strengthened to hold for
$\calM^{\smash{\geq}}_{\alpha_1, \alpha_2} (\tau^\poly_{L, \omega})$ in place of
$\calM^{\smash{>}}_{\alpha_1, \alpha_2} (\tau^\poly_{L, \omega})$,
where $\calM^{\smash{\geq}}_{\alpha_1, \alpha_2} (\tau^\poly_{L, \omega})$
consists of pairs $(E_1, E_2)$ such that
the minimal slope of~$E_1$
is greater than or equal to the maximal slope of~$E_2$.
This is because the operation $- \otimes K_X$ on $\calC$
strictly increases the (polynomial) slope,
by the proof of \cref{prop-surf-ext}.
Consequently, the relation \cref{eq-psi-compat-star}
holds if $f \boxtimes g$ is supported on
$\calM^{\smash{\geq}}_{\alpha_1, \alpha_2} (\tau^\poly_{L, \omega})$,
and \cref{eq-psi-compat-diamond} holds if
$f \boxtimes h$ is supported on the preimage of
$\calM^{\smash{\geq}}_{\smash{\alpha, j (\theta)}} (\tau^\poly_{L, \omega})$ in $\calM_\alpha \times \calM^\sd_\theta$
or on $\calM^{\smash{\geq}}_\alpha (\tau^\poly_{L, \omega}) \times \calM^\sd_0 \subset \calM_\alpha \times \calM^\sd_0$.
These relations cover all the Hall algebra and Hall module operations
involved in defining the $\epsilon, \epsilon^\sd$ functions,
and as a result, the modified invariants
$\upJ_\alpha (\tau^\poly_{L, \omega})$ and $\upJ^\sd_\theta (\tau^\poly_{L, \omega})$
coincide with the unmodified version,
defined via integrating the $\epsilon, \epsilon^\sd$ functions.
See also Joyce~\cite[Proposition~6.23]{Joyce2008IV}
for the case of Gieseker stability.
We have thus obtained the following:

\begin{theorem*}
    For $X$ a del~Pezzo surface,
    the no-pole theorems, \cref{thm-no-pole,thm-no-pole-sd},
    hold for the $\upJ$ and $\upJ^\sd$ invariants
    defined above, and they
    satisfy the wall-crossing formulae \cref{eq-wcf-j,eq-wcf-jsd}
    with $\tau^\poly_{\smash{L, \omega_1}}, \tau^\poly_{\smash{L, \omega_2}}$
    in place of $\tau, \tilde{\tau}$, for ample classes $\omega_1, \omega_2 \in A^1 (X)_{\bbR}$.

    The numerical invariants
    $\chiJ_\alpha (\tau^\poly_{L, \omega})$ and $\chiJ^\sd_\theta (\tau^\poly_{L, \omega})$
    are well-defined, and satisfy the wall-crossing formulae
    \cref{eq-wcf-j-hash,eq-wcf-j-sd-hash}
    with $\tau^\poly_{\smash{L, \omega_1}}, \tau^\poly_{\smash{L, \omega_2}}$
    in place of $\tau, \tilde{\tau}$.
    \QED
\end{theorem*}

\paragraph{Surfaces of Kodaira dimension 0.}

On the other hand, if $X$ is a surface of Kodaira dimension~$0$,
that is, a K3 surface, an abelian surface, an Enriques surface,
or a hyperelliptic surface,
we have $c_1 (K_X) = 0$,
so that $\td_1 (X) = -c_1 (K_X) / 2 = 0$.
It follows from \cref{eq-surf-chi,eq-surf-chi-sd}
that we have the relations
\begin{equation}
    \bar{\chi} (\alpha_1, \alpha_2) = 0 \ , \qquad
    \bar{\chi}^\sd (\alpha, \theta) = 0
\end{equation}
for all $\alpha_1, \alpha_2, \alpha \in C^\circ (\calA_L^+)$
and $\theta \in C^\sd (\calA_L^+)$,
where $\bar{\chi}$ and $\bar{\chi}^\sd$ are as in
\cref{eq-def-chi-bar,eq-def-chi-bar-sd}.
This means that the algebra $\Lambda (\calA_L^+)$ in \cref{para-lambda-a}
is commutative, and the module $\Lambda^\sd (\calA_L^+)$ in \cref{para-lambda-sd-a}
satisfies $x \cdot m = x^\vee \cdot m$
for all $x \in \Lambda (\calA_L^+)$ and $m \in \Lambda^\sd (\calA_L^+)$.
Consequently, the Lie bracket on $\Lambda (\calA_L^+)$
and the $\heart$ action on $\Lambda^\sd (\calA_L^+)$ vanish.
The wall-crossing formulae~\cref{eq-wcf-j,eq-wcf-jsd}
together with \cref{thm-u-tilde,thm-comb}
imply that the invariants
$\upJ_\alpha (\tau^\poly_{L, \omega})$ and $\upJ^\sd_\theta (\tau^\poly_{L, \omega})$
are independent of $\omega$. 
See also Joyce~\cite[Theorem~6.24]{Joyce2008IV}
for the case of Gieseker stability.
We state this as the following:

\begin{theorem*}
    For $X$ a surface of Kodaira dimension~$0$,
    the invariants $\upJ_\alpha (\tau^\poly_{L, \omega})$ and $\upJ^\sd_\theta (\tau^\poly_{L, \omega})$
    defined in \cref{para-surf-wcf-j} are independent of
    the ample class~$\omega$.
    \QED
\end{theorem*}

\subsection{Donaldson--Thomas invariants}
\label{sect-threefolds}

\paragraph{}

Let $X$ be a \emph{Calabi--Yau threefold},
that is, a smooth, projective threefold over $\bbK$
with trivial canonical bundle, $K_X \simeq \calO_X$.
Joyce--Song \cite{JoyceSong2012} and
Kontsevich--Soibelman \cite{KontsevichSoibelman2008}
defined \emph{Donaldson--Thomas invariants}
$\DT_\alpha (\tau) \in \bbQ$, where $\tau$ is a Gieseker stability condition
or a slope stability condition on
$\Coh (X)$, and $\alpha$ is a Chern character.
These invariants are defined by
\begin{equation}
    \label{eq-def-dt-js}
    \DT_\alpha (\tau) = - \int_{\calM} {}
    (\bbL - 1) \cdot \epsilon_\alpha (\tau) \cdot \nu_{\calM} \, d \chi \ ,
\end{equation}
where $\nu_{\calM}$ is the \emph{Behrend function},
as in Behrend~\cite{Behrend2009},
and $\int (-) \, d \chi$ denotes applying the Euler characteristic map,
as in \cref{para-euler-char}, to the motivic integral
defined in \cref{para-mot-int}.
Joyce--Song showed that these invariants satisfy nice properties,
including wall-crossing formulae analogous to~\cref{eq-wcf-dt},
and \emph{deformation invariance},
meaning the invariants are unchanged under deformations of~$X$.

In this setting,
the Euler form $\bar{\chi} (-, -)$ in \cref{eq-def-chi-bar}
is replaced by
\begin{equation}
    \bar{\chi} (\alpha_1, \alpha_2) = \int_X \alpha_1 \, \alpha_2^\vee \, \td (X) \ ,
\end{equation}
so that for objects $E_1, E_2 \in \Coh (X)$
of Chern character $\alpha_1, \alpha_2$, respectively, we have
\begin{align*}
    \bar{\chi} (\alpha_1, \alpha_2) & =
    \sum_{i=0}^{3} {} (-1)^i \dim \Ext^i (E_2, E_1)
    \\ & =
    \bigl( \dim \Ext^1 (E_1, E_2) - \dim \Ext^0 (E_1, E_2) \bigr) 
    \\ & \hspace{4em} {} -
    \bigl( \dim \Ext^1 (E_2, E_1) - \dim \Ext^0 (E_2, E_1) \bigr) \ ,
\end{align*}
where $\Ext^i (E_1, E_2) \simeq \Ext^{3-i} (E_2, E_1)^\vee$
by Serre duality, and since $X$ is Calabi--Yau.

In the situation of \cref{sect-coh-general} above,
for $X$ a Calabi--Yau threefold, we can also define
Donaldson--Thomas invariants counting
both ordinary and self-dual objects in $\calA_L$,
for the self-dual polynomial Bridgeland stability conditions
$\tau^\poly_{L, \omega}$ in \cref{para-polynomial-bridgeland},
as well as invariants counting ordinary and self-dual objects
in $\calA_{L, \omega}$, for the self-dual Bridgeland stability condition
$\tau_{L, \omega}$ in \cref{para-bridgeland-stability}.
We expect that these invariants should satisfy
wall-crossing formulae and deformation invariance,
and we hope to study these properties in future work.

\paragraph{Definition.}

Let $X$ be a Calabi--Yau threefold,
and use the notations in
\cref{para-polynomial-bridgeland,para-polynomial-bridgeland-moduli,para-polynomial-bridgeland-inv}.
Fix a choice of $L$, and an ample class $\omega \in \smash{A^1} (X)_{\bbR}$,
and we assume that the polynomial Bridgeland stability condition
$\tau^\poly_{L, \omega}$ satisfies
the openness and boundedness conditions~\tagref{Stab}.

Then, for classes $\alpha \in C (\calA_L)$ and $\theta \in C^\sd (\calA_L)$,
define \emph{Donaldson--Thomas invariants}
\begin{align}
    \DT_\alpha (\tau^\poly_{L, \omega}) & =
    -\int_{\calM} {} (\bbL - 1) \cdot \epsilon_\alpha (\tau^\poly_{L, \omega}) \cdot \nu_{\calM} \, d \chi \ ,
    \\
    \DT^\sd_\theta (\tau^\poly_{L, \omega}) & =
    \int_{\calM^\sd} {} \epsilon^\sd_\theta (\tau^\poly_{L, \omega}) \cdot \nu_{\calM^\sd} \, d \chi \ ,
\end{align}
where $\nu_{\calM}$ and $\nu_{\calM^\sd}$ are the Behrend functions
for $\calM$ and $\calM^\sd$, as in \cite{Behrend2009}.
These are well-defined by the no-pole theorems,
\cref{thm-no-pole,thm-no-pole-sd}.

Similarly, for $X$ a threefold satisfying the
Bogomolov--Gieseker inequality in \cite{BayerMacriToda2014},
choose the data $L, \omega$,
and for classes $\alpha \in C (\calA_{L, \omega})$ and $\theta \in C^\sd (\calA_{L, \omega})$,
define
\begin{align}
    \label{eq-dt-3fold-bridgeland}
    \DT_\alpha (\tau_{L, \omega}) & =
    -\int_{\calM} {} (\bbL - 1) \cdot \epsilon_\alpha (\tau_{L, \omega}) \cdot \nu_{\calM} \, d \chi \ ,
    \\
    \DT^\sd_\theta (\tau_{L, \omega}) & =
    \int_{\calM^\sd} {} \epsilon^\sd_\theta (\tau_{L, \omega}) \cdot \nu_{\calM^\sd} \, d \chi \ .
\end{align}
Note that $\calM, \calM^\sd$ now denote moduli stacks for
$\calA_{L, \omega}$ and $\calA_{L, \omega}^\sd$, instead of
$\calA_L$ and $\calA_L^\sd$.
Invariants similar to \cref{eq-dt-3fold-bridgeland}
for Bridgeland stability conditions
were already studied in the work of
Feyzbakhsh--Thomas~\cite{FeyzbakhshThomas1,FeyzbakhshThomas2}.

We formulate some expected properties of these invariants
in the following conjecture.

\begin{conjecture}
    Let $X$ be a Calabi--Yau threefold,
    and fix a choice of $L$ as above.

    \begin{enumerate}
        \item 
            For any ample class $\omega \in \smash{A^1} (X)_{\bbR}$,
            the polynomial Bridgeland stability condition
            $\tau^\poly_{L, \omega}$ on $\calA_L$ satisfies
            the openness and boundedness conditions~\tagref{Stab},
            so that the Donaldson--Thomas invariants
            $\DT_\alpha (\tau^\poly_{L, \omega})$ and $\DT^\sd_\theta (\tau^\poly_{L, \omega})$
            are defined.

        \item
            For ample classes $\omega_1, \omega_2 \in \smash{A^1} (X)_{\bbR}$
            that are sufficiently close,
            the Donaldson--Thomas invariants
            defined above satisfy wall-crossing formulae
            analogous to \cref{eq-wcf-dt,eq-wcf-dt-sd},
            with $\tau^\poly_{\smash{L, \omega_1}}, \tau^\poly_{\smash{L, \omega_2}}$
            in place of $\tau, \tilde{\tau}$.
            If, moreover, $X$ satisfies the Bogomolov--Gieseker inequality,
            then the same holds for the Bridgeland stability conditions
            $\tau_{L, \omega_1}, \tau_{L, \omega_2}$.

        \item
            If $H^1 (X, \calO_X) = 0$,
            then the Donaldson--Thomas invariants
            defined above are unchanged under deformations of~$X$,
            which also induce unique deformations of $L$.
    \end{enumerate}
\end{conjecture}

%% file: motive.tex
\label{sect-motive}

\subsection{Motives of algebraic stacks}

\paragraph{}

We briefly summarize the formalism of motives
(in the na\"ive sense of the word) of algebraic stacks,
following Joyce~\cite[\S4]{Joyce2007Stack}.

In the following, an \emph{algebraic stack} over $\bbK$
will refer to a stack on the site of $\bbK$-schemes
with the \'etale topology, with a smooth atlas.

\begin{definition}
    \label{def-motive-varieties}
    Let $\mathrm{Var}_{\bbK}$ be the set of isomorphism classes of $\bbK$-varieties.
    Define an abelian group $K_0 (\mathrm{Var}_{\bbK})$
    as the quotient of the free abelian group
    generated by symbols $\mu (X)$ for
    $X \in \mathrm{Var}_{\bbK}$, by the relations
    \begin{equation}
        \mu (X) \sim \mu (Z) + \mu (X \setminus Z)
    \end{equation}
    for closed subvarieties $Z \subset X$.

    Define multiplication on $K_0 (\mathrm{Var}_{\bbK})$ by
    \begin{equation}
        \mu (X) \cdot \mu (Y) = \mu (X \times Y) \ .
    \end{equation}
    This gives $K_0 (\mathrm{Var}_{\bbK})$ the structure of a commutative ring.

    Denote $\bbL = \mu (\bbA_{\bbK}^1) \in K_0 (\mathrm{Var}_{\bbK})$.
    We also define the following versions of rings of motives. Define
    \begin{align}
        \label{eq-def-mot-ring}
        \upM_{\bbK} & =
        K_0 (\mathrm{Var}_{\bbK}) \,
        [ \bbL^{-1} ] \otimes \bbQ \ ,
        \\
        \label{eq-def-mot-ring-circ}
        \Mhat_{\bbK}^\circ & =
        K_0 (\mathrm{Var}_{\bbK}) \,
        [ \bbL^{-1}, \, (\bbL^n + \bbL^{n-1} + \cdots + 1)^{-1} ]
        \otimes \bbQ \ ,
        \\
        \label{eq-def-mot-ring-hat}
        \Mhat_{\bbK} & =
        K_0 (\mathrm{Var}_{\bbK}) \,
        [ \bbL^{-1}, \, (\bbL - 1)^{-1}, \, (\bbL^n + \bbL^{n-1} + \cdots + 1)^{-1} ]
        \otimes \bbQ \ ,
        \\
        \label{eq-def-mot-ring-tilde}
        \Mtilde_\bbK & =
        \Mhat_\bbK^\circ / (\bbL - 1) \ ,
    \end{align}
    with $\bbL^n + \bbL^{n-1} + \cdots + 1$ inverted for all $n \geq 0$
    in \crefrange{eq-def-mot-ring-circ}{eq-def-mot-ring-hat}.
\end{definition}

\paragraph{}

For example, one can deduce that
\begin{align}
    \mu (\bbA^n)
    & = \bbL^n \ , \\
    \mu (\bbP^n)
    & = \bbL^n + \bbL^{n-1} + \cdots + 1 \ , \\
    \mu (\GL (n))
    & = \prod_{i=0}^{n-1} {} (\bbL^n - \bbL^i) \ ,
\end{align}
etc., in $K_0 (\mathrm{Var}_{\bbK})$.

\begin{definition}
    \label{def-aff-stab}
    Let $\calX$ be an algebraic $\bbK$-stack, locally of finite type.
    We say that $\calX$ \emph{has affine stabilizers},
    if the stabilizer groups of all $\bbK$-points of $\calX$
    are affine algebraic groups over $\bbK$.
\end{definition}

\begin{theorem}
    \label{def-motive-stacks}
    Let $\mathrm{St}_{\bbK}$ denote the set of
    isomorphism classes of algebraic $\bbK$-stacks
    of finite type with affine stabilizers.
    There is a unique map
    \[
        \mu \colon \mathrm{St}_{\bbK} \longrightarrow \Mhat_{\bbK}
    \]
    extending the map $\mu \colon \mathrm{Var}_{\bbK} \to K_0 (\mathrm{Var}_{\bbK})$
    in \cref{def-motive-varieties},
    satisfying the following properties:
    \begin{enumerate}
        \item 
            If $\calX \in \mathrm{St}_{\bbK}$ and
            $\calZ \subset \calX$ is a closed substack, then
            \begin{equation}
                \mu (\calX) = \mu (\calZ) + \mu (\calX \setminus \calZ) \ .
            \end{equation}
        \item
            For any two $\bbK$-stacks $\calX, \calY \in \mathrm{St}_{\bbK}$\,,
            we have
            \begin{equation}
                \mu (\calX \times \calY) =
                \mu (\calX) \cdot \mu (\calY) \ .
            \end{equation}
        \item
            For any special $\bbK$-group $G$ acting on a $\bbK$-variety $X$,
            we have
            \begin{equation}
                \mu (X) =
                \mu ([X / G]) \cdot \mu (G) \ ,
            \end{equation}
            where $[X/G]$ is the quotient stack.
    \end{enumerate}
\end{theorem}

See Joyce~\cite[Theorem~4.10]{Joyce2007Stack}.
Here, a \emph{special group} is an algebraic group $G$
such that every principal $G$-bundle on a scheme is Zariski locally trivial,
as in Serre~\cite[\S4.1]{Serre1958}.
For example, the groups $\Ga$, $\GL (n)$, and $\Sp (2n)$ are special;
finite products of special groups are special.
The group $\upO (n)$ is not special when $n > 0$.

\begin{example}
    \label{eg-motive-bg}
    Here are some examples which will be useful later.
    For $n \geq 0$, we have
    \begin{align}
        \label{eq-motive-bgl}
        \mu ([*/\GL(n)]) & = \prod_{i=0}^{n-1}
        \frac{1}{\bbL^n - \bbL^i} \ , \\
        \label{eq-motive-bo-even}
        \mu ([*/\upO(2n)]) & = \bbL^n \cdot \prod_{i=0}^{n-1}
        \frac{1}{\bbL^{2n} - \bbL^{2i}} \ , \\[1ex]
        \label{eq-motive-bo-odd}
        \mathllap{\mu ([*/\upO(2n+1)]) = {}}
        \mu ([*/\Sp(2n)])
        & = \bbL^{-n} \cdot \prod_{i=0}^{n-1}
        \frac{1}{\bbL^{2n} - \bbL^{2i}} \ ,
    \end{align}
    where the cases for orthogonal groups are due to
    Dhillon and Young~\cite[Theorem~3.7]{DhillonYoung2016}.
    Notice that when $r > 0$, the group $\upO (r)$ is not special, and we have
    \[
        \mu ([*/\upO(r)]) \neq \mu (\upO(r))^{-1} \ .
    \]
    In particular, when $r = 1$, we have
    $\mu ([*/\bbZ_2]) = 1 \neq 1/2$.
\end{example}

Next, we prove some properties of the motive of a stack.

\begin{theorem}
    \label{thm-motive-bundle}
    Let $\calX$ be an algebraic stack of finite type
    with affine stabilizers.
    \begin{enumerate}
        \item 
            \label{itm-motive-bundle-vb}
            Let $\calY \to \calX$ be a vector bundle of rank $n$. Then
            \begin{equation}
                \mu (\calY) = \mu (\calX) \cdot \bbL^n \ .
            \end{equation}
        \item 
            \label{itm-motive-bundle-pb}
            Let $\calY \to \calX$ be a principal $G$-bundle,
            where $G$ is a special $\bbK$-group. Then
            \begin{equation}
                \mu (\calY) = \mu (\calX) \cdot \mu (G) \ .
            \end{equation}
    \end{enumerate}
\end{theorem}

\begin{proof}
    By a result of Kresch \cite[Proposition~3.5.9]{Kresch1999},
    we may decompose $\calX$ into a finite disjoint union of
    quotient stacks of the form $[X / \GL (m)]$,
    where $X$ is a $\bbK$-variety.
    So we may assume that $\calX = [X / \GL (m)]$.
    
    For \cref{itm-motive-bundle-vb},
    let $Y = X \times_{\calX} \calY$, so that $\calY \simeq [Y / \GL (m)]$,
    and $Y \to X$ is a vector bundle of rank $n$.
    Since all vector bundles are Zariski trivial,
    we can deduce that $\mu (Y) = \mu (X) \cdot \bbL^n$.
    Dividing both sides by $\mu (\GL (m))$,
    we obtain that $\mu (\calY) = \mu (\calX) \cdot \bbL^n$.
    
    For \cref{itm-motive-bundle-pb},
    an analogous argument will work.
\end{proof}

\begin{theorem}
    \label{thm-motive-linear-action-quotient}
    Let $G$ be an affine algebraic group over $\bbK$,
    acting linearly on a $\bbK$-vector space $\bbA^n$.
    Then the motive of the quotient $[\bbA^n / G]$ can be expressed as
    \begin{equation}
        \mu ([\bbA^n / G]) = \bbL^n \cdot \mu([* / G]) \ .
    \end{equation}
\end{theorem}

\begin{proof}
    We can regard $[\bbA^n / G]$ as a vector bundle over $[* / G]$
    of rank $n$.
    The theorem then follows from
    \cref{thm-motive-bundle}~\cref{itm-motive-bundle-vb}.
\end{proof}

Note that, as shown in \cref{eg-motive-bg},
the motive $\mu ([*/G])$ is not always equal to $\mu (G)^{-1}$,
unless $G$ is special.

\begin{theorem}
    \label{thm-almost-bijection}
    Let $\calX, \calY$ be algebraic $\bbK$-stacks of finite type,
    and let $f \colon \calX \to \calY$ be a morphism.
    Suppose that $f$ induces a bijection of $\bbK$-points
    and induces isomorphisms of stabilizer groups at all $\bbK$-points.
    
    Then, there are decompositions into disjoint reduced locally closed substacks
    \[
        \calX = \bigcup_{i=1}^n \calX_i \ , \qquad
        \calY = \bigcup_{i=1}^n \calY_i \ ,
    \]
    such that for each $i = 1, \dotsc, n,$
    the restriction $f|_{\calX_i} \colon \calX_i \to \calY_i$
    is an isomorphism.
    
    In particular,
    if $\calX$ and $\calY$ have affine stabilizers, then
    \[
        \mu (\calX) = \mu (\calY) \ .
    \]
\end{theorem}

\begin{proof}
    Taking reductions,
    we may assume that $\calX$ and $\calY$ are reduced.

    Let $|\calX|$ be the underlying topological space of $\calX$,
    as in \cite[\S5]{LaumonMoret2000}.
    By \cite[Theorem~5.9.4]{LaumonMoret2000},
    the image $f (|\calX|) \subset |\calY|$ is constructible.
    But since it contains all $\bbK$-points,
    it must be the whole of $|\calY|$,
    so $f$ is surjective.
    
    Therefore, by noetherian induction, it is enough to show that
    there is an open substack $\calU \subset \calX$,
    such that $f|_{\calU}$ is an isomorphism onto an open substack of $\calY$.

    Choose an open morphism $U \to \calX$
    from an irreducible $\bbK$-variety $U$, with dense image.
    For example, this could be a smooth covering of a dense, open substack of $\calX$.
    Shrinking $\calX$ if necessary, we assume that
    $U \to \calX$ is surjective.
    
    Let us show that $f$ induces isomorphisms of all stabilizer groups,
    including at points other than $\bbK$-points.
    Let $\calI_{\calX} \to \calX$ and $\calI_{\calY} \to \calY$
    be the inertia stacks.
    These morphisms are representable and of finite type.
    Let $V = U \times_{\calX} \calI_{\calX}$ and
    $W = U \times_{\calY} \calI_{\calY}$\,.
    Then the induced morphism of finite type algebraic spaces
    $p \colon V \to W$ satisfies the property that, for any $u \in U (\bbK)$,
    the morphism of fibres $p_u \colon V_u \to W_u$ is an isomorphism.
    By generic flatness, shrinking $U$, and hence $\calX$, if necessary,
    we can assume that $V$ and $W$ are flat over $U$.
    By \citestacks{05X0}, $p$ is flat, and hence \'etale.
    Since it is bijective on $\bbK$-points, it is an isomorphism,
    which implies the claim.
    
    By \citestacks{04YY}, $f$ is representable by algebraic spaces.
    Shrinking $U$ again, we may assume that $f$ is flat,
    and that $f$ is a reduced morphism, as in \cite[Definition~6.8.1]{EGAIV2},
    since morphisms whose source is reduced
    are generically reduced \citestacks{054Z}.
    Now, fibres of $f$ over $\bbK$-points
    have to be either empty or $\operatorname{Spec} \bbK$, so $f$ is étale.
    On the other hand, by \citestacks{0DUC}, $f$ is universally injective,
    so it is an open immersion,
    and hence an isomorphism.
\end{proof}

\subsection{Stack functions}

\paragraph{}

We briefly summarize the theory of stack functions,
developed in \cite{Joyce2007Stack}.
Stack functions on an algebraic stack $\calX$
are a version of motives relative to $\calX$.
There are two versions of stack functions,
the representable version $\SF (\calX)$ which includes only
representable morphisms to $\calX$,
and the completed version $\SFhat (\calX)$ which includes all morphisms.

\begin{definition}
    Let $\calX$ be an algebraic $\bbK$-stack with affine stabilizers.
    Consider pairs $(\calR, \rho)$,
    where $\calR$ is an algebraic $\bbK$-stack of finite type,
    and $\rho \colon \calR \to \calX$ is a representable morphism.
    The space of \emph{stack functions} on $\calX$ is the $\bbQ$-vector space
    \begin{equation}
        \SF (\calX) = \Bigl(
            \bigoplus_{(\calR, \rho)} \bbQ \cdot [(\calR, \rho)]
        \Bigr) \Big/ {\sim} \ ,
    \end{equation}
    where we take the direct sum over
    all isomorphism classes of such pairs $(\calR, \rho)$,
    and $\sim$ is generated by the relations
    \begin{equation}
        [(\calR, \rho)] \sim
        [(\calS, \rho |_{\calS})] +
        [(\calR \setminus \calS, \rho |_{\calR \setminus \calS})]
    \end{equation}
    for closed substacks $\calS \subset \calR$.

    When the map $\rho$ is clear from the context,
    the generator $[(\calR, \rho)]$ will be denoted by $[\calR]$.
    
    One can define multiplication on $\SF (\calX)$ by
    \begin{equation}
        \label{eq-sf-mult}
        [\calR] \cdot [\calS] = [\calR \times_{\calX} \calS] \ .
    \end{equation}
    If $\calX$ is of finite type, then this gives
    $\SF (\calX)$ the structure of a $\bbQ$-algebra.
    Otherwise, the multiplication may not have a unit,
    as $[\calX]$ will not be a valid element of $\SF (\calX)$.
\end{definition}

\paragraph{}
\label[definition]{def-sfhat}

We also define a completed version $\SFhat (\calX)$,
as in~\cite[Definition~4.11]{Joyce2007Stack},
where this was denoted by $\SF (\calX, \Upsilon, \Lambda)$.

\begin{definition*}
    Define $\SFhat (\calX)$
    to be the quotient of $\SF (\calX) \otimes_{\bbQ} \Mhat_{\bbK}$, 
    as an $\Mhat_{\bbK}$-module,
    by the following relations:
    \begin{itemize}
        \item 
            For an algebraic stack $\calR$ of finite type
            with affine stabilizers, a morphism $\rho \colon \calR \to \calX$,
            and a $\bbK$-variety $U$, 
            consider the composition $\calR \times U \to \calR \to \calX$,
            where the first map is the projection.
            Then we have
            \begin{equation}
                [\calR \times U] \sim \mu (U) \cdot [\calR] \ .
            \end{equation}
        \item
            For a special algebraic $\bbK$-group $G$,
            a $\bbK$-variety $U$ acted on by $G$,
            and a morphism $\rho \colon [U/G] \to \calX$, 
            consider the composition $U \to [U/G] \to \calX$,
            where the first map is the quotient map.
            Then we have
            \begin{equation}
                [U/G] \sim \mu (G)^{-1} \cdot [U] \ .
            \end{equation}
    \end{itemize}
    As in \cite[Theorem~4.13]{Joyce2007Stack},
    multiplication on $\SF (\calX)$
    induces a multiplication map on $\SFhat (\calX)$.
    We have a natural $\bbQ$-linear map
    \begin{equation}
        \label{eq-sf-sfhat}
        \SF (\calX) \longrightarrow \SFhat (\calX)
    \end{equation}
    compatible with the multiplication maps.
\end{definition*}

As in \cite[Definition~4.11]{Joyce2007Stack},
one could also use non-representable morphisms to $\calX$
to define $\SFhat (\calX)$, and the result would be the same.

By definition and \cite[Proposition~4.15]{Joyce2007Stack},
\begin{align}
    \SF (\Spec \bbK) & \simeq K_0 (\mathrm{Var}_{\bbK}) \otimes \bbQ \ , \\
    \SFhat (\Spec \bbK) & \simeq \Mhat_{\bbK} \ .
\end{align}

\begin{theorem}
    \label{thm-sf-func}
    Let $f \colon \calX \to \calY$ and $g \colon \calY \to \calZ$
    be morphisms of algebraic $\bbK$-stacks with affine stabilizers.
    \begin{enumerate}
        \item 
            There is a pushforward map
            \[
                f_! \colon \SFhat (\calX) \longrightarrow \SFhat (\calY) \ ,
            \]
            given on generators by
            \[
                f_! \, [(\calR, \rho)] = [(\calR, f \circ \rho)] \ .
            \]
            Moreover, we have $g_! \circ f_! = (g \circ f)_!$\,.

            If\/ $f$ is representable, then there is also a pushforward map
            \[
                f_! \colon \SF (\calX) \longrightarrow \SF (\calY) \ .
            \]
            Moreover, we have $g_! \circ f_! = (g \circ f)_!$
            if\/ $f$ and $g$ are both representable.
        \item 
            If\/ $f$ is of finite type, then there are pullback maps
            \begin{align*}
                f^* \colon \SF (\calY) & \longrightarrow \SF (\calX) \ , \\
                f^* \colon \SFhat (\calY) & \longrightarrow \SFhat (\calX) \ ,
            \end{align*}
            given on generators by
            \[
                f^* [(\calR, \rho)] = [(\calR \times_{\calY} \calX, \rho')] \ ,
            \]
            where $\rho'$ is the induced morphism.
            Moreover, we have $f^* \circ g^* = (g \circ f)^*$
            if\/ $f$ and\/ $g$ are both of finite type.
        \item
            \label{itm-thm-sf-func-bct}
            Given a $2$-pullback diagram
            \[ \begin{tikzcd}
                \calX' \ar[r, "f'"] \ar[d, "h'"'] \ar[dr, phantom, pos=.2, "\ulcorner"] &
                \calY' \ar[d, "h"] \\
                \calX \ar[r, "f"] & \calY \rlap{ $,$}
            \end{tikzcd} \]
            where $\calX, \calY, \calX', \calY'$
            are algebraic $\bbK$-stacks with affine stabilizers,
            if\/ $h$ is of finite type,
            then there is a commutative diagram
            \[ \begin{tikzcd}
                \SFhat (\calX') \ar[r, "(f')_!"] &
                \SFhat (\calY') \\
                \SFhat (\calX) \ar[r, "f_!"] \ar[u, "(h')^*"] &
                \SFhat (\calY) \rlap{ $.$} \ar[u, "h^*"']
            \end{tikzcd} \]
            If, moreover, $f$ is representable,
            then there is a similar commutative diagram for $\SF (-)$.
    \end{enumerate}
\end{theorem}

This follows from \cite[Theorem~3.5 and \S4.3]{Joyce2007Stack}.
The pushforward map $f_!$ was denoted by $f_*$ there.

\paragraph{Relation to the category of spans.}
\label{para-sf-six-functor}

In the modern language of six-functor formalisms,
as in, for example, Scholze~\cite{Scholze2022},
\cref{thm-sf-func} can be reformulated as follows:

Let $\cat{Span} (\cat{St}_{\bbK}; \mathrm{ft}, \mathrm{rep})$
be the category of spans in the category
$\cat{St}_{\bbK}$ of algebraic $\bbK$-stacks with affine stabilizers,
where morphisms from $\calX$ to $\calY$ are
spans $\calX \leftarrow \calZ \to \calY$,
where the left arrow is of finite type
and the right arrow is representable.
Composition of spans are given by pullbacks.
Then $\SF$ defines a functor
\begin{equation}
    \label{eq-sf-six-functor}
    \SF \colon
    \cat{Span} (\cat{St}_{\bbK}; \mathrm{ft}, \mathrm{rep})
    \longrightarrow \bbQ \mathhyphen \cat{Mod} \ ,
\end{equation}
where $\bbQ \mathhyphen \cat{Mod}$ is the category of $\bbQ$-vector spaces,
sending a span $\calX \xleftarrow{f} \calZ \xrightarrow{g} \calY$
to the map $g_! \circ f^* \colon \SF (\calX) \to \SF (\calY)$.

Similarly, let $\cat{Span} (\cat{St}_{\bbK}; \mathrm{ft})$
be the category of spans in $\cat{St}_{\bbK}$
whose left leg is of finite type and right leg is arbitrary.
Then $\SFhat$ defines a functor
\begin{equation}
    \label{eq-sfhat-six-functor}
    \SFhat \colon
    \cat{Span} (\cat{St}_{\bbK}; \mathrm{ft})
    \longrightarrow \bbQ \mathhyphen \cat{Mod} \ .
\end{equation}

\paragraph{Motivic integration.}
\label{para-mot-int}

For an algebraic $\bbK$-stack $\calX$ of finite type with affine stabilizers,
pushing forward along the structure morphism $\calX \to \Spec \bbK$
defines a map
\begin{equation}
    \label{eq-mot-int}
    \int_{\calX} \colon \SFhat (\calX) \longrightarrow \Mhat_{\bbK} \ ,
\end{equation}
called \emph{motivic integration}.
Composing with the map~\cref{eq-sf-sfhat} gives a map
\begin{equation}
    \label{eq-mot-int-sf}
    \int_{\calX} \colon \SF (\calX) \longrightarrow \Mhat_{\bbK} \ .
\end{equation}

\paragraph{Local stack functions.}
\label[definition]{def-lsf}

There is another variant of stack functions,
called \emph{local stack functions},
which allows infinite linear combinations of stack functions
as long as they are locally finite.

\begin{definition*}
    Let $\calX$ be an algebraic $\bbK$-stack with affine stabilizers.
    Following~\cite[Definition~3.9]{Joyce2007Stack}, define
    \begin{align*}
        \LSF (\calX) & =
        \lim_{ \calU \subset \calX }
        \SF (\calU) \ , \\
        \LSFhat (\calX) & =
        \lim_{ \calU \subset \calX }
        \SFhat (\calU) \ ,
    \end{align*}
    where we take the limit over all finite type open substacks $\calU \subset \calX$,
    with respect to restriction along inclusions.

    For an algebraic $\bbK$-stack with affine stabilizers $\calR$,
    not necessarily of finite type,
    and for a finite type morphism $\rho \colon \calR \to \calX$, there is an element
    \[
        [(\calR, \rho)] \in \LSF (\calX) \ ,
    \]
    as in \cite[Definition~3.9]{Joyce2007Stack}.
    
    For an infinite collection of elements $(f_i \in \mathrm{LSF} (\calX))_{i \in I}$,
    and another element $f \in \mathrm{LSF} (\calX)$, we say that
    \[
        \sum_{i \in I} f_i = f \ ,
    \]
    or that the infinite sum \emph{converges} to $f$,
    if for any finite type open substack $\calU \subset \calX$,
    there are only finitely many $i \in I$ such that $f_i |_{\calU}$ is non-zero,
    and we have the finite sum $\sum_{i \in I} f_i |_{\calU} = f |_{\calU}$\,.
    
    Similarly, one can define convergence in $\LSFhat (\calX)$.
\end{definition*}

\subsection{Virtual rank projections}
\label{sect-vrp}

\paragraph{}
\label{para-vrp}

Let $\calX$ be an algebraic $\bbK$-stack with affine stabilizers.
Joyce~\cite[Definition~5.13]{Joyce2007Stack}
defined \emph{virtual rank projections}
\begin{equation}
    \label{eq-vrp}
    \vrp{n} \colon \SF (\calX) \longrightarrow \SF (\calX)
\end{equation}
for all integers $n \geq 0$, inducing a decomposition
\begin{equation}
    \label{eq-vrp-decomp}
    \SF (\calX) = \bigoplus_{n \geq 0} \SF_\pn{n} (\calX) \ ,
\end{equation}
where $\vrp{n}$ acts as the identity on $\SF_\pn{n} (\calX)$,
and as zero on $\SF_\pn{m} (\calX)$ for $m \neq n$.
We say that elements of $\SF_\pn{n} (\calX)$ have
(\emph{pure}) \emph{virtual rank $n$}.

Roughly speaking, a stack function $[\calR] \in \SF (\calX)$
has virtual rank $n$
if stabilizer groups in $\calR$ look like tori $T$ of rank $n$.
In general, such an element $[\calR]$ is decomposed into pieces
with virtual ranks between the minimum \emph{central rank}
(i.e.~rank of the centre)
and the maximal rank of stabilizer groups in $\calR$.
For example, $[* / \GL (n)]$ has mixed virtual ranks $1, \dotsc, n$.

As in \cite[Proposition~5.14]{Joyce2007Stack},
the virtual rank projections $\vrp{n}$ commute with
the pushforward maps $f_!$
for all representable morphisms $f \colon \calX \to \calY$ of algebraic $\bbK$-stacks
with affine stabilizers. In other words, we have
\begin{equation}
    f_! (\SF_\pn{n} (\calX)) \subset \SF_\pn{n} (\calY)
\end{equation}
for all $n \geq 0$.

\paragraph{}
\label[theorem]{thm-vrp-int}

By results in Joyce~\cite[\S6]{Joyce2007Stack},
the virtual rank decomposition satisfies
the following important property related to motivic integration.

\begin{theorem*}
    Let $\calX$ be an algebraic $\bbK$-stack with affine stabilizers,
    and let $f \in \SF_{(n)} (\calX)$ be a stack function
    of pure virtual rank~$n$. Then we have
    \begin{equation}
        \int_{\calX} f \in (\bbL - 1)^{-n} \cdot \Mhat_{\bbK}^\circ \ ,
    \end{equation}
    with notations as in \cref{eq-def-mot-ring-circ,eq-mot-int-sf}.
\end{theorem*}

%% file: stacks-cat.tex
\label{sect-stacks-in-cat}

In this \lcnamecref{sect-stacks-in-cat}, we study the notion of
an \emph{algebraic stack in categories}.
This should be seen as a geometric incarnation of a category,
whereas an ordinary algebraic stack
is a geometric incarnation of a groupoid.
From such a stack in categories,
one can obtain an underlying ordinary stack
by forgetting the non-invertible arrows;
one can also extract an underlying category
by forgetting the geometric structure.

The notion of stacks in categories in similar contexts
was discussed in Joyce~\cite[\S7]{Joyce2006I}
and Behrend--Ronagh~\cite{BehrendRonagh2019}.

In the following, we always use the term \emph{$2$-category}
to mean a $(2,1)$-category, where
all $2$-morphisms are invertible.

\subsection{Stacks in categories}

\begin{definition}
    \label{def-stack-cat}
    Let $\cat{Cat}$ be the $2$-category of small categories,
    whose morphisms are functors
    and $2$-morphisms are natural isomorphisms.
    Let $\cat{Gpd} \subset \cat{Cat}$ be the full subcategory of small groupoids.

    Let $\calC$ be a site.
    A \emph{stack in categories} on $\calC$
    is a $2$-functor 
    \begin{equation}
        \calX \colon \calC^\op \longrightarrow \cat{Cat}
    \end{equation}
    that satisfies the $2$-sheaf property,
    meaning that it takes coverings in $\calC$ to $2$-limits in $\cat{Cat}$.
    In particular, if $\calC = \cat{Sch}_{\bbK}$
    is the site of $\bbK$-schemes, equipped with the \'etale topology,
    then $\calX$ will be called a \emph{$\bbK$-stack in categories}.
    
    For any such $\calX$,
    composing with the limit-preserving $2$-functor
    $(-)^\simeq \colon \cat{Cat} \to \cat{Gpd}$
    forgetting the non-invertible arrows
    defines the \emph{underlying stack}
    $\calX^\simeq \colon \calC^\op \to \cat{Gpd}$,
    with a natural inclusion $\calX^\simeq \to \calX$.

    Similarly, composing with the $2$-functor
    $(-)^I = \cat{Fun} (I, -) \colon \cat{Cat} \to \cat{Cat}$,
    where $I = (\bullet \to \bullet)$,
    assigns to every stack in categories $\calX$
    another stack in categories $\calX^I$,
    which is the \emph{stack of arrows} in $\calX$.
    There are the source and target projections $s, t \colon \calX^I \to \calX$.
    If $\calX$ is a stack in groupoids, then $\calX^I \simeq \calX$.
\end{definition}

\paragraph{Relative groupoids.}
\label[lemma]{lem-pb-rel-gpd}

A morphism $f \colon \calX \to \calY$
of stacks in categories is called a \emph{relative groupoid},
if the induced morphism
\[
    \calX^\simeq \longrightarrow
    \calX \underset{\calY}{\times} \calY^\simeq
\]
is an isomorphism.

\begin{lemma*}
    Let $\calC$ be a site,
    and $f \colon \calX \to \calY$ a morphism
    of stacks in categories on $\calC$ that is a relative groupoid.
    Let $\calW$ be a stack in groupoids on $\calC$,
    and let $g \colon \calW \to \calY$ be a morphism.

    Then the natural morphism
    \begin{equation}
        \calW \underset{\calY^\simeq}{\times} \calX^\simeq
        \longrightarrow
        \calW \underset{\calY}{\times} \calX
    \end{equation}
    is an isomorphism.
    In particular, $\calW \times_{\calY} \calX$ is a stack in groupoids.
\end{lemma*}

\begin{proof}
    We have
    \[
        \calW \underset{\calY^\simeq}{\times} \calX^\simeq
        \simeq
        \calW \underset{\calY^\simeq}{\times}
        \Bigl( \calX \underset{\calY}{\times} \calY^\simeq \Bigr)
        \simeq
        \calW \underset{\calY}{\times} \calX \ .
    \]
\end{proof}

\begin{definition}
    A $\bbK$-stack in categories $\calX$ is \emph{representable},
    if $\calX \simeq \calX^\simeq$ is representable by an algebraic space.
    
    A morphism of $\bbK$-stacks in categories
    $f \colon \calX \to \calY$ is \emph{representable},
    if for any $\bbK$-scheme $U$ and any morphism $U \to \calY$,
    the $2$-pullback $U \times_{\calY} \calX$ is representable.
    Equivalently by \cref{lem-pb-rel-gpd}, $f$ is a relative groupoid
    and the morphism $\calX^\simeq \to \calY^\simeq$ is representable.
\end{definition}

\begin{definition}
    \label{def-alg-stack-cat}
    An \emph{algebraic $\bbK$-stack in categories}
    is a $\bbK$-stack in categories $\calX$, satisfying the following conditions.
    \begin{enumerate}
        \item \label{itm-def-alg-stack-cat-1}
            The $\bbK$-stack $\calX^\simeq$ is an algebraic $\bbK$-stack.

        \item \label{itm-def-alg-stack-cat-2}
            The $\bbK$-stack $(\calX^I)^\simeq$ is an algebraic $\bbK$-stack.
    \end{enumerate}
\end{definition}

\begin{lemma}
    \label{lem-alg-stack-cat}
    Let $\calX$ be an algebraic $\bbK$-stack in categories.
    Then 
    \begin{enumerate}
        \item 
            \label{itm-lem-alg-stack-cat-rep}
            The morphisms
            \begin{align*}
                \Delta \colon \calX & \longrightarrow
                \calX \underset{\bbK}{\times} \calX \ , \\
                (s, t) \colon \calX^I & \longrightarrow
                \calX \underset{\bbK}{\times} \calX
            \end{align*}
            are representable.
        \item
            \label{itm-lem-alg-stack-cat-xi-alg}
            $\calX^I$ is an algebraic $\bbK$-stack in categories.
    \end{enumerate}
\end{lemma}

\begin{proof}
    For \cref{itm-lem-alg-stack-cat-rep},
    it is clear that the morphisms are relative groupoids,
    and it remains to show that the corresponding morphisms
    of underlying stacks in groupoids are representable.
    The representability of $\Delta \colon \calX^\simeq \to \calX^\simeq \times \calX^\simeq$
    follows from the algebraicity of $\calX^\simeq$.
    To show that $(s,t) \colon (\calX^I)^\simeq \to \calX^\simeq \times \calX^\simeq$
    is representable, by \cite[Corollary~8.1.1]{LaumonMoret2000},
    it suffices to show that 
    it induces injections on stabilizer groups,
    which is true by definition.

    For \cref{itm-lem-alg-stack-cat-xi-alg},
    it remains to show that $((\calX^I)^I)^\simeq$ is an algebraic stack.
    One can express it as an iterated fibre product
    \begin{equation}
        ((\calX^I)^I)^\simeq \simeq \Bigl(
            (\calX^I)^\simeq \underset{\calX^\simeq}{\times}
            (\calX^I)^\simeq
        \Bigr) \underset{(\calX^I)^\simeq}{\times}
        \Bigl(
            (\calX^I)^\simeq \underset{\calX^\simeq}{\times}
            (\calX^I)^\simeq
        \Bigr) \ ,
    \end{equation}
    as it parametrizes commutative squares in $\calX$.
    Since $\calX^\simeq$ and $(\calX^I)^\simeq$ are algebraic,
    so is $((\calX^I)^I)^\simeq$.
\end{proof}

\begin{definition}
    \label{def-stack-cat-mapping-space}
    Let $\calX$ be an algebraic $\bbK$-stack in categories,
    and let $S$ be an algebraic space over $\bbK$.
    The category $\calX (S)$
    is called the category of \emph{$S$-points} in $\calX$.
    
    For any two objects $x, y \in \calX (S)$, we may form the fibre product
    \begin{equation} \begin{tikzcd}
        \Hom_{\calX} (x, y) \ar[d] \ar[r] \ar[dr, phantom, near start, "\ulcorner"] &
        \calX^I \ar[d, "{(s, t)}"] \\
        S \ar[r, "{(x, y)}"] & \calX \times \calX \rlap{ ,}
    \end{tikzcd} \end{equation}
    where the algebraic $S$-space
    $\Hom_{\calX} (x, y)$
    is called the \emph{mapping space} from $x$ to $y$.

    Unwinding the definitions, we see that
    $\Hom_{\calX} (x, y)$ represents the functor
    that assigns to any $S$-scheme $U$
    the set $\Hom_{\calX (U)} (\pi^* x, \pi^* y)$,
    where $\pi \colon U \to S$ is the structure map.
    In particular, we have
    $\Hom_{\calX} (x, y) (S) \simeq \Hom_{\calX (S)} (x, y)$.
    This endows $\calX (S)$ with the structure of
    a category enriched in algebraic $S$-spaces.

    By \cref{lem-pb-rel-gpd}, we also have
    \begin{equation}
        \Hom_{\calX} (x, y) \simeq
        S \underset{\calX^\simeq \times \calX^\simeq}{\times}
        (\calX^I)^\simeq \ ,
    \end{equation}
    which is now a fibre product of algebraic $\bbK$-stacks.
\end{definition}

\begin{example}
    \label{eg-classifying-stack}
    Let $A$ be a finite-dimensional associative $\bbK$-algebra.
    One can construct an algebraic $\bbK$-stack in categories $[*/A]$ as follows.
    
    Let $[*/A]^{\mathrm{pre}} \colon \cat{Sch}_{\bbK}^\op \to \cat{Cat}$
    be the $2$-functor which sends a $\bbK$-scheme $U$
    to the category of \emph{trivial $A$-bundles} on $U$,
    that is, the category with a unique object,
    whose endomorphism ring is $\Hom (U, A)$,
    the set of morphisms from $U$ to $A$,
    with a ring structure induced by that of $A$.
    Define $[*/A]$ to be the stackification of $[*/A]^{\mathrm{pre}}$.

    For a $\bbK$-scheme $U$, the category $[*/A] (U)$
    is the category of \emph{principal $A$-bundles} on $U$.
    When $U$ is a $\bbK$-variety, a principal $A$-bundle on $U$
    can be defined as a vector bundle on $U$ with a right $A$-action,
    such that the fibres over $\bbK$-points are isomorphic to $A$ as right $A$-modules.
    
    There is a unique $\bbK$-point in $[*/A]$,
    with endomorphism ring isomorphic to $A$.
    The morphism $* \to [*/A]$ is the \emph{universal principal $A$-bundle}.

    We observe that
    \[
        [*/A]^\simeq \simeq [*/A^\times] \ , \qquad
        ([*/A]^I)^{\simeq} \simeq [A/(A^\times)^2] \ ,
    \]
    where $A^\times$ is the group of invertible elements in $A$,
    regarded as an algebraic group,
    and $(A^\times)^2$ acts on $A$ by
    $(a, b) \cdot c = a^{-1} c b$,
    where $a, b \in A^\times$ and $c \in A$.
    This shows that $[*/A]$ is an \emph{algebraic} stack in categories.
    
    In particular, if we denote by $* \in [*/A] (\bbK)$ the unique object, then
    \begin{equation}
        \Hom_{[*/A]} (*, *) \simeq
        \Spec \bbK \underset{[*/A^\times]^2}{\times} [A/(A^\times)^2] \simeq A \ .
    \end{equation}

    However, notice that by \cref{lem-pb-rel-gpd}, one has
    \begin{equation}
        \Spec \bbK \underset{[*/A]}{\times} \Spec \bbK \simeq A^\times \ ,
    \end{equation}
    instead of $A$, since fibre products do not see
    non-invertible $2$-morphisms.
\end{example}

\subsection{Additive, exact, and linear stacks}

We now define additive, exact, and linear stacks,
which are stacks in categories with extra structures.

\paragraph{Additive stacks.}
\label{para-additive-stack}

Let $\cat{AddCat}$ be the $2$-category of small additive categories,
whose morphisms are additive functors,
and $2$-morphisms are natural isomorphisms.

Let $\calC$ be a site.
An \emph{additive stack} on $\calC$
is a $2$-functor $\calX \colon \calC^\op \to \cat{AddCat}$
that satisfies the $2$-sheaf property.

Since every additive category has a symmetric monoidal structure
given by the direct sum,
which is preserved by additive functors,
every additive stack $\calX$ carries two morphisms
\begin{align*}
    0 \colon * \longrightarrow \calX \ , \\
    {\oplus} \colon \calX \times \calX \longrightarrow \calX \ ,
\end{align*}
which establish $\calX$ as a commutative monoid object
in the $2$-category of stacks in categories,
or in other words, a stack that takes values in
the $2$-category of symmetric monoidal categories
and symmetric monoidal functors.

An \emph{additive $\bbK$-stack}
is an additive stack on the site $\cat{Sch}_{\bbK}$
of $\bbK$-schemes, with the étale topology.

An \emph{additive algebraic $\bbK$-stack}
is an additive $\bbK$-stack $\calX$
which is also an algebraic $\bbK$-stack in categories.

For an additive algebraic $\bbK$-stack $\calX$,
and two $\bbK$-points $x, y \in \calX (\bbK)$,
the mapping space $\Hom_{\calX} (x, y)$
in \cref{def-stack-cat-mapping-space}
carries the structure of an abelian group algebraic space over $\bbK$.
This gives the additive category $\calX (\bbK)$
an enrichment over abelian group algebraic spaces over $\bbK$.

\paragraph{Exact stacks.}
\label{para-exact-stack}

Let $\cat{ExCat}$ be the $2$-category of small exact categories,
whose morphisms are exact functors,
and $2$-morphisms are natural isomorphisms.

Let $\calC$ be a site.
An \emph{exact stack} on $\calC$
is a $2$-functor $\calX \colon \calC^\op \to \cat{ExCat}$
that satisfies the $2$-sheaf property.

Every exact stack $\calX$
is also an additive $\bbK$-stack.

The $2$-functor
\[
    (-)^\ex \colon \cat{ExCat} \longrightarrow \cat{ExCat}
\]
sending an exact category $\calD$
to the category whose objects are exact sequences in $\calD$,
and morphisms are commutative diagrams,
assigns to each exact stack $\calX$
another exact stack $\calX^\ex$.

An \emph{exact $\bbK$-stack}
is an exact stack on the site $\cat{Sch}_{\bbK}$
of $\bbK$-schemes, with the étale topology.

\paragraph{Exact algebraic stacks.}
\label{para-exact-alg-stack}
\label[lemma]{lem-exact-stack}

An \emph{exact algebraic $\bbK$-stack}
is an exact $\bbK$-stack $\calX$, satisfying the following conditions:
\begin{itemize}
    \item 
        $\calX$ is an algebraic $\bbK$-stack in categories.
    \item 
        The $\bbK$-stack $(\calX^\ex)^\simeq$ 
        is an algebraic $\bbK$-stack.
\end{itemize}

\begin{lemma*}
    Let $\calX$ be an exact algebraic $\bbK$-stack.
    Then 
    \begin{enumerate}
        \item \label{itm-lem-exact-stack-1}
            The morphisms
            \[
                d_1, \ d_2 \colon \calX^\ex \longrightarrow \calX^I \ ,
            \]
            sending an exact sequence to its constituting arrows,
            are representable.
        \item \label{itm-lem-exact-stack-2}
            $\calX^\ex$ is an exact algebraic $\bbK$-stack.
    \end{enumerate}
\end{lemma*}

\begin{proof}
    For \cref{itm-lem-exact-stack-1},
    it is enough to verify that these morphisms
    induce injections of stabilizer groups of the
    underlying stacks in groupoids,
    which is true since every exact sequence consists of
    a monomorphism and an epimorphism.

    For \cref{itm-lem-exact-stack-2},
    we need to show that $((\calX^\ex)^I)^\simeq$
    and $((\calX^\ex)^\ex)^\simeq$
    are algebraic stacks.
    The proof is similar to that of
    \cref{lem-alg-stack-cat}~\cref{itm-lem-alg-stack-cat-xi-alg},
    using different iterated fibre products,
    which we omit here.
\end{proof}

\paragraph{Internal pushforwards and pullbacks.}
\label{para-exact-stack-pf-pb}

Pushforwards along inclusions exist in an exact category,
and are preserved by exact functors,
as in, for example, \cite[Appendix~A]{Keller1990}.
Therefore, given an exact stack $\calX$,
there is a pushforward morphism
\begin{equation}
    \calX^\ex \underset{\pi_1, \, \calX, \, s}{\times} \calX^I \longrightarrow
    (\calX^\ex)^I \ ,
\end{equation}
where $\calX^\ex$ parametrizes the inclusion,
$\calX^I$ parametrizes the other map,
and $(\calX^\ex)^I$ parametrizes the result,
which is a square where two opposite edges are inclusions.
The map $\pi_1 \colon \calX^\ex \to \calX$ sends an exact sequence to its first object.

Dually, pullbacks along projections exist,
and there is a pullback morphism
\begin{equation}
    \calX^\ex \underset{\pi_2, \, \calX, \, t}{\times} \calX^I \longrightarrow
    (\calX^\ex)^I \ ,
\end{equation}
where $\pi_2$ sends an exact sequence to its last object.

\paragraph{Linear stacks.}
\label[definition]{def-linear-stack}

Finally, we define a notion of linear algebraic stacks,
which are the $\bbK$-linear category version of
stacks in categories.
See also Behrend--Ronagh~\cite[\S1.2]{BehrendRonagh2019}
for a similar concept.

\begin{definition*}
    A \emph{linear algebraic $\bbK$-stack}
    is an additive algebraic $\bbK$-stack $\calX$,
    such that
    \begin{enumerate}
        \item 
            The algebraic $\bbK$-stacks $\calX^\simeq$ and $(\calX^I)^\simeq$
            are locally of finite type.
        \item 
            There is a $[*/\bbA^1]$-action on $\calX$,
            compatible with the additive structures.
    \end{enumerate}

    Here, $[*/\bbA^1]$ is regarded as a commutative monoid object
    in the $2$-category of $\bbK$-stacks in categories.
    We require that the action is \emph{compatible}
    with the additive structures,
    which means that the action morphism $[*/\bbA^1] \times \calX \to \calX$
    is bilinear, in that for any $\bbK$-scheme $S$ and any $x, y \in \calX (S)$,
    the induced action of
    $\mathrm{End} _{\smash{[*/\bbA^1] (S)}} (*) \simeq \calO_S (S)$ on
    $\mathrm{Hom}_{\calX (S)} (x, y)$
    is bilinear, establishing $\mathrm{Hom}_{\calX (S)} (x, y)$
    as an $\calO_S (S)$-module.

    For $\bbK$-points $x, y \in \calX (\bbK)$,
    the abelian group algebraic $\bbK$-space $\Hom_{\calX} (x, y)$
    is locally of finite type and acted on by $\bbA^1$,
    so it is isomorphic to $\bbA^n$ for some $n \in \bbN$.
    Therefore, the category $\calX (\bbK)$ has the structure
    of a $\bbK$-linear category.

    In particular, stabilizer groups of $\bbK$-points in $\calX^\simeq$
    are linear algebraic groups,
    which also justifies the terminology.

    A \emph{morphism of linear algebraic $\bbK$-stacks}
    is a $[*/\bbA^1]$-equivariant morphism of the underlying additive $\bbK$-stacks.
    A \emph{$2$-morphism} between such morphisms is
    a $[*/\bbA^1]$-equivariant $2$-morphism
    of the underlying morphisms of additive $\bbK$-stacks.

    An \emph{exact linear algebraic $\bbK$-stack}
    is an exact algebraic $\bbK$-stack
    that is also a linear algebraic $\bbK$-stack.
    A \emph{morphism of exact linear algebraic $\bbK$-stacks}
    is a $[*/\bbA^1]$-equivariant morphism of the underlying exact $\bbK$-stacks.
    A \emph{$2$-morphism} between such morphisms is
    a $[*/\bbA^1]$-equivariant $2$-morphism
    of the underlying morphisms of exact $\bbK$-stacks.
\end{definition*}

\subsection{Self-dual stacks}

We also define a self-dual category version of
stacks in categories, called \emph{self-dual stacks}.

\begin{definition}
    \label{def-sd-stack}
    Let $\calC$ be a site, and let $\calX$ be a stack in categories on $\calC$.

    The \emph{opposite stack} of $\calX$ is given by the $2$-functor
    \[
        \calX^\op \colon
        \calC^\op \overset{\calX}{\longrightarrow}
        \cat{Cat} \overset{(-)^\op}{\longrightarrow}
        \cat{Cat} \ .
    \]
    This defines an automorphism
    of the $2$-category of stacks in categories on $\calC$.
    Note that the $2$-functor $(-)^\op \colon \cat{Cat} \to \cat{Cat}$
    reverses the direction of natural isomorphisms.

    A \emph{self-dual structure} on $\calX$
    consists of an isomorphism $D \colon \calX \simeq \calX^\op$,
    and a $2$-morphism $\eta \colon D^\op \circ D \simeq \id_{\calX}$\,,
    such that $D \circ \eta = \eta^\op \circ D \colon D \circ D^\op \circ D \simeq D$.
    Equivalently, a self-dual structure on $\calX$
    is given by a compatible way to choose
    a self-dual structure on the category
    $\calX (U)$ for every $U \in \calC$.

    A \emph{self-dual stack} on $\calC$
    is a stack in categories on $\calC$ together with a self-dual structure.
    We denote this data as a triple $(\calX, D, \eta)$.

    A \emph{self-dual $\bbK$-stack}, \emph{self-dual algebraic $\bbK$-stack}, etc.,
    is a $\bbK$-stack, or algebraic $\bbK$-stack, etc., together with a self-dual structure.
    When the stack has additional structure,
    such as an additive or exact structure,
    we require that the morphism $D$ respects these structures.
\end{definition}

\paragraph{The $\bbZ_2$-fixed locus.}
\label{para-sd-stack-z2-action}

Let $(\calX, D, \eta)$ be a self-dual stack on $\calC$.
There is a $\bbZ_2$-action $(\calX^\simeq, f, \zeta)$ on $\calX^\simeq$,
where $f$ is the composition
\[
    f \colon
    \calX^\simeq \overset{D}{\longrightarrow}
    (\calX^\simeq)^\op \overset{i}{\longrightarrow}
    \calX^\simeq \ ,
\]
where $i$ is the identification sending each arrow to its inverse,
and $\zeta \colon f^2 \simTo \id_{\calX^\simeq}$ is induced from $\eta$. Define
\begin{equation}
    \calX^\sd = (\calX^\simeq)^{\bbZ_2} 
\end{equation}
as the fixed locus of this $\bbZ_2$-action,
as a $2$-limit in the category of stacks in categories.

Explicitly, for any $U \in \calC$, we have
\begin{equation}
    \calX^\sd (U) \simeq \calX (U)^\sd \ ,
\end{equation}
where the right-hand side is the groupoid of self-dual objects,
as in \cref{def-sd-obj}.
In particular, for a self-dual $\bbK$-stack $\calX$,
the $\bbK$-points in $\calX^\sd$ are the self-dual $\bbK$-points in $\calX$,
i.e.~self-dual objects in $\calX (\bbK)$.

\begin{example}
    Let $n \geq 0$ be an integer,
    and let $\mathrm{Mat} (n)$ be the $\bbK$-algebra of $n \times n$ matrices over $\bbK$.
    Let $\calX = [*/\mathrm{Mat} (n)]$, as in \cref{eg-classifying-stack}.
    Then $\calX^\op = [*/\mathrm{Mat} (n)^\op]$.
    Define a self-dual structure on $\calX$
    by the isomorphism $D \colon \calX \simeq \calX^\op$ induced by the map
    $(-)^{\mathrm{t}} \colon \mathrm{Mat} (n) \simeq \mathrm{Mat} (n)^\op$
    sending a matrix to its transpose.
    A $2$-morphism $\eta \colon D^\op \circ D \simeq \id_{\calX}$
    is then given by an invertible element in the centre of $\mathrm{Mat} (n)$,
    which must be of the form $\varepsilon \cdot I_n$ for $\varepsilon \in \bbK^\times$,
    where $I_n$ is the identity matrix.
    The condition that $D \circ \eta = \eta^\op \circ D$
    means that $\varepsilon \cdot I_n = (\varepsilon \cdot I_n)^{-\mathrm{t}}$,
    or $\varepsilon = \pm 1$.
    
    We have $\calX^\simeq \simeq [*/\GL (n)]$.
    The morphism $f \colon \calX^\simeq \to \calX^\simeq$ induced by $D$
    is given by the morphism of algebraic groups
    $(-)^{-\mathrm{t}} \colon \GL (n) \to \GL (n)$,
    taking a matrix to the inverse of its transpose.
    We have $\calX^\sd \simeq [*/\upO (n)]$ when $\varepsilon = 1$,
    and when $\varepsilon = -1$, we have
    $\calX^\sd \simeq [*/\Sp (n)]$ if $n$ is even, or $\varnothing$ if $n$ is odd.
\end{example}

\begin{lemma}
    \label{lem-sd-stack-prop}
    Let $(\calX, D, \eta)$ be a self-dual linear algebraic $\bbK$-stack. Then,
    \begin{enumerate}
        \item \label{itm-sd-stack-lft}
            $\calX^\sd$ is an algebraic $\bbK$-stack locally of finite type,
            and has affine stabilizers
            \textnormal{(\cref{def-aff-stab}).}
        \item \label{itm-sd-stack-ft}
            The natural morphism $\calX^\sd \to \calX^\simeq$ is of finite type.
    \end{enumerate}
\end{lemma}

\begin{proof}
    Consider the pullback of algebraic stacks
    \begin{equation} \begin{tikzcd}
        \calY \ar[d] \ar[r] \ar[dr, phantom, pos=.2, "\ulcorner"]
        & \calX^\simeq \ar[d, "\Delta"]
        \\ \calX^\simeq \ar[r, "{(\id, \, D)}"]
        & \calX^\simeq \times \calX^\simeq \rlap{ ,}
    \end{tikzcd} \end{equation}
    where $\calY$ can be seen as the fixed locus of $D$
    as a $\bbZ$-action instead of a $\bbZ_2$-action.
    There is an induced morphism $j \colon \calX^\sd \to \calY$.
    We claim that $j$ is a closed immersion.

    Indeed, let $U$ be any affine $\bbK$-scheme,
    and let $x \colon U \to \calX^\simeq$ be a morphism.
    Write $x^\vee = D \circ x$, and write
    \[
        H = \mathrm{Isom}_{\calX^\simeq} (x, x^\vee) =
        U \underset{\calX^\simeq}{\times} \calY \ ,
    \]
    which is an algebraic $U$-space, such that for any $U$-scheme $f \colon V \to U$,
    its $V$-points parametrize isomorphisms between
    $x \circ f$ and $x^\vee \circ f$ in $\calX^\simeq (V)$.
    There is a $\bbZ_2$-action on $H$
    induced by $D$ and $\eta$. Write $H^\sd \subset H$ for the fixed locus,
    which is now a closed subspace of $H$,
    since $H$ is separated. We claim that
    \begin{equation}
        H^\sd \simeq U \underset{\calX^\simeq}{\times} \calX^\sd \ .
    \end{equation}
    This is because for any $U$-scheme $f \colon V \to U$,
    the $V$-points of both sides parametrize
    isomorphisms $\phi \colon x \circ f \simeq x^\vee \circ f$
    such that $\phi^\vee = \phi \circ \eta$, where $\phi^\vee = D (\phi)$.
    This shows that $j$ is a closed immersion.

    To prove \cref{itm-sd-stack-lft},
    since $\calY$ is locally of finite type, so is $\calX^\sd$.
    For $(x, \phi) \in \calX^\sd (\bbK)$,
    with $x \in \calX (\bbK)$ and $\phi$ as above,
    the stabilizer group $\mathrm{Stab}_{\calX^\sd} (x, \phi)$ is affine,
    since it can be identified with $\mathrm{End}_{\calX} (x)^\iso$,
    where $(-)^\iso$ is as in \cref{para-inv-alg},
    and the involution on $\mathrm{End}_{\calX} (x)$ is given by
    $e \mapsto \phi^{-1} \circ e^\vee \circ \phi$.

    For \cref{itm-sd-stack-ft},
    the morphism $\calY \to \calX^\simeq$ above is of finite type,
    and so is $\calX^\sd \to \calX^\simeq$.
\end{proof}

%% file: two-vb.tex
\label{sect-two-vb}
\addtocounter{subsection}{1}

\paragraph{}

We define a notion of \emph{$2$-vector bundles},
also known as \emph{vector bundle stacks} in the literature.
For a $\bbK$-scheme $U$, a $2$-vector bundle over $U$
is an algebraic stack that looks like a quotient stack
\[
    [E^1/E^0] \longrightarrow U \ ,
\]
where $E^1$ and $E^0$ are vector bundles over $U$,
and $E^0$ acts on $E^1$ via a morphism of vector bundles $d \colon E^0 \to E^1$.
This can be thought of as the \emph{total space} of the two-term complex
$E^\bullet [1]$, generalizing the notion of the total space of a vector bundle.

In fact, there is a general construction of total spaces of perfect complexes
in derived algebraic geometry (see To\"en \cite[p.\,46]{Toen2014}),
which produces, in general, derived stacks.
For two-term complexes,
the derived stacks obtained this way are classical $1$-stacks,
and our goal here is to describe this stack classically.

A $2$-vector bundle over a point is a \emph{$2$-vector space}.
There is a notion of a $2$-vector space that appeared in
Baez--Crans \cite[\S3]{BaezCrans2004},
which is essentially equivalent to ours,
except that we have assumed all vector spaces to be finite-dimensional.
However, one should be aware that there are
very different notions of $2$-vector spaces in the literature,
most notably that of Kapranov--Voevodsky \cite{KapranovVoevodsky1994}.

$2$-vector bundles were perhaps first studied by Deligne~\cite[\S1.4]{Deligne1973},
under the name of \emph{Picard stacks},
in a more general setting of abelian sheaves.
They were also discussed under the name of \emph{vector bundle stacks}
in Behrend--Fantechi \cite[\S\S1--2]{BehrendFantechi1997}.

\paragraph{Abelian stacks.}
\label{para-ab-stack}

A \emph{$2$-group} is a small monoidal groupoid $\calG$,
such that for any object $x \in \calG$,
left and right multiplication by $x$ are equivalences.

An \emph{abelian $2$-group} is a symmetric monoidal groupoid
that is also a $2$-group.

Let $2\mathhyphen\cat{Ab}$ denote the
$2$-category of abelian $2$-groups,
whose morphisms are symmetric monoidal functors,
and $2$-morphisms are monoidal natural isomorphisms.

Let $\calC$ be a site.
An \emph{abelian stack} over $\calC$ is a $2$-functor
$\calX \colon \calC^\op \to 2\mathhyphen\cat{Ab}$,
which is also a stack in groupoids over $\calC$.
This defines the $2$-category of abelian stacks over $\calC$.

An abelian stack $\calX$ naturally comes with morphisms
\begin{align*}
    0 \colon * & \longrightarrow \calX \ , \\
    {+} \colon \calX \times \calX & \longrightarrow \calX \ ,
\end{align*}
establishing $\calX$ as an abelian group object
in stacks in groupoids over $\calC$,
in a suitable $2$-categorical sense.

When $\calC = \mathsf{Sch}_{\calS}$\,,
the category of schemes over a $\bbK$-stack $\calS$,
with the \'etale topology,
an abelian stack over $\calC$ is called an \emph{abelian $\calS$-stack}.

Deligne defined a \emph{Picard stack}
to be an abelian stack that is \emph{strictly commutative}
\cite[Definition~1.4.5]{Deligne1973},
and showed that this condition ensures that
any Picard stack is a quotient stack of abelian sheaves
\cite[Lemma~1.4.13]{Deligne1973}.

\begin{definition}
    \label{def-module-stack}
    Let $\calS$ be an algebraic $\bbK$-stack.

    A \emph{module $\calS$-stack}
    is an abelian stack $\calX$ over $\calS$,
    together with an $\bbA^1_{\calS}$-action
    that respects the abelian stack structures.
    This means that the action morphism $\bbA^1_{\calS} \times \calX \to \calX$
    is bilinear.

    A \emph{morphism of module $\calS$-stacks}
    is an $\bbA^1_{\calS}$-equivariant morphism of abelian stacks.
    A \emph{$2$-morphism} between such morphisms
    is an $\bbA^1_{\calS}$-equivariant $2$-morphism of
    the underlying morphisms of abelian stacks.
\end{definition}

\begin{definition}
    \label{def-2vb}
    Let $U$ be a $\bbK$-scheme.

    Let $E^0, E^1 \to U$ be vector bundles of finite rank,
    and let $d \colon E^0 \to E^1$ be a morphism of vector bundles.
    We refer to this data as a \emph{two-term complex} on $U$,
    and regard it as a complex of $\calO_U$-modules,
    where $E^0$ lies in degree $-1$, and $E^1$ lies in degree $0$.
    We may form the relative quotient stack
    \[
        [E^1 / E^0] \longrightarrow U \ ,
    \]
    which is a module $U$-stack.
    See also \cite[\S1.4.11]{Deligne1973}.

    A \emph{$2$-vector bundle} over $U$
    is a module $U$-stack
    \[
        \pi \colon \calE \longrightarrow U \ ,
    \]
    such that it is locally isomorphic to module stacks of the form $[E^1 / E^0]$
    over open subsets of $U$.
    In this case, the underlying stack in groupoids of $\calE$
    is an algebraic $U$-stack.

    A \emph{morphism of $2$-vector bundles} over $U$
    is a morphism of the underlying module stacks,
    and a \emph{$2$-morphism} between such morphisms
    is a $2$-morphism between the morphisms of module stacks.

    The \emph{rank} of a $2$-vector bundle $\calE$
    is its relative dimension as a smooth morphism of algebraic stacks.
    This defines a locally constant function $\rank \calE \colon U \to \bbZ$.

    A \emph{$2$-vector space} is a $2$-vector bundle over $\Spec (\bbK)$.
    We see that a $2$-vector space can always be written as
    \[
        E^1 \times [*/E^0]
    \]
    for finite-dimensional vector spaces $E^0$, $E^1$.
    A morphism of $2$-vector spaces from $E^1 \times [*/E^0]$
    to $F^1 \times [*/F^0]$
    is of the form $\alpha \times \beta$,
    where $\alpha \colon E^1 \to F^1$ is a linear map,
    and $\beta \colon [*/E^0] \to [*/F^0]$
    is induced by a linear map $E^0 \to F^0$.

    Finally, for an algebraic $\bbK$-stack $\calS$,
    a \emph{$2$-vector bundle} over $\calS$
    is a module $\calS$-stack $\calE$,
    such that for any morphism $U \to \calS$ from a $\bbK$-scheme $U$,
    the pullback $U \times_{\calS} \calE$ is a $2$-vector bundle over $U$.
\end{definition}

\begin{theorem}
    \label{thm-2vb-morphism}
    Let $U$ be a $\bbK$-scheme.
    \begin{enumerate}
        \item \label{itm-lem-2vb-morphism-1}
            For two-term complexes $E^\bullet, F^\bullet$ of vector bundles on $U$,
            there is a canonical isomorphism
            \begin{equation}
                \Hom_{\calD (\calO_U)} (E^\bullet, F^\bullet) \longsimto
                \Hom_U ([E^1/E^0], [F^1/F^0]) \ ,
            \end{equation}
            where $\calD (\calO_U)$ is the derived category of $\calO_U$-modules,
            and the right-hand side is the set of $2$-isomorphism classes
            of morphisms of $2$-vector bundles over $U$.
            
            Moreover, this induces an equivalence of $2$-categories
            between the full subcategory of the derived $\infty$-category $\cat{D} (\calO_U)$
            spanned by such two-term complexes, which is a $2$-category,
            and the $2$-category of $2$-vector bundles over $U$.
            
        \item \label{itm-lem-2vb-morphism-2}
            Let $E^\bullet \to F^\bullet \to G^\bullet$
            be an exact triangle of two-term complexes of vector bundles on $U$. Then
            \begin{equation}
                [F^1/F^0] \underset{[G^1/G^0]}{\times} U \simeq [E^1/E^0] \ ,
            \end{equation}
            where the map $U \to [G^1/G^0]$ is the zero section. We say that the sequence
            \begin{equation}
                0 \to [E^1/E^0] \longrightarrow
                [F^1/F^0] \longrightarrow
                [G^1/G^0] \to 0
            \end{equation}
            is a \emph{short exact sequence} of $2$-vector bundles.
    \end{enumerate}
\end{theorem}

\begin{proof}
    Part \cref{itm-lem-2vb-morphism-1}
    was shown in Deligne~\cite[Proposition~1.4.15]{Deligne1973};
    see also Tatar~\cite[Theorem~6.4]{Tatar2011}.

    For \cref{itm-lem-2vb-morphism-2},
    we use the consequence of \cite[Theorem~6.4]{Tatar2011}
    that there is an equivalence of $2$-categories
    between $\cat{D}^{[-1,0]} (U)$,
    the full subcategory of the derived $\infty$-category $\cat{D} (U)$ of abelian sheaves
    spanned by two-term complexes in degrees $-1$ and $0$,
    and the $2$-category of Picard $U$-stacks,
    which is a full subcategory of the $2$-category of abelian $U$-stacks.
    In particular, this construction preserves $2$-pullbacks.
    It then remains to show that $2$-pullbacks in the $2$-category of Picard $U$-stacks
    coincide with $2$-pullbacks in the $2$-category of all stacks.
    Equivalently, $2$-pullbacks of \emph{Picard categories}
    \cite[Definition~1.4.2]{Deligne1973}, which are certain $2$-groups,
    coincide with $2$-pullbacks of groupoids,
    which can be verified by an explicit description of such pullbacks.
\end{proof}

\begin{remark}
    \label{rem-2vb-ab-stack}
    A $2$-vector bundle is determined by its underlying abelian stack.
    In other words, any abelian stack automorphism of a $2$-vector bundle
    $\calE \to U$ is a $2$-vector bundle automorphism,
    since Zariski locally, we may assume that $\calE \simeq [E^1/E^0]$
    with $E^0, E^1$ trivial,
    and that the given automorphism $[E^1/E^0] \simto [E^1/E^0]$
    is given by a morphism of two-term complexes,
    so that it is determined by its value on
    respective global bases of $E^0, E^1$.
\end{remark}

\begin{definition}
    \label{def-affine-2vb}
    Let $U$ be a $\bbK$-scheme.
    
    An \emph{affine $2$-vector bundle} over $U$
    is a morphism of algebraic stacks
    \[
        \pi \colon \calE \to U \ ,
    \]
    such that there is a $2$-vector bundle $\calE_0 \to U$,
    such that $\calE$ is an \emph{$\calE_0$-torsor}.
    This means that there is an $\calE_0$-action on $\calE$ as $U$-stacks,
    such that there is an
    open cover $(U_i \subset U)_{i \in I}$\,,
    and $\smash{\calE_0 |_{U_i}}$-equivariant isomorphisms of $U_i$-stacks
    $\phi_i \colon \smash{\calE |_{U_i}} \simto \smash{\calE_0 |_{U_i}}$\,.
    In particular, each transition morphism
    $\phi_j \circ \phi_i^{-1} \colon \smash{\calE_0 |_{U_i \cap U_j}} \to \smash{\calE_0 |_{U_i \cap U_j}}$
    is a translation by a section $U_i \cap U_j \to \calE_j$\,.

    A \emph{morphism} of affine $2$-vector bundles $\calE, \calF$ over $U$,
    or an \emph{affine linear morphism},
    is a morphism of stacks $f \colon \calE \to \calF$ over $U$,
    such that there is a morphism of $2$-vector bundles
    $f_0 \colon \calE_0 \to \calF_0$ over $U$,
    an open cover $(U_i \subset U)_{i \in I}$\,,
    and trivializations of torsors
    $\phi_i \colon \smash{\calE|_{U_i}} \simto \smash{\calE_0|_{U_i}}$ and
    $\psi_i \colon \smash{\calF|_{U_i}} \simto \smash{\calF_0|_{U_i}}$\,,
    such that $\psi_i \circ f \circ \phi_i^{-1} = \smash{f_0|_{U_i}}$ for all $i$.

    A \emph{$2$-morphism} of such morphisms is
    a $2$-morphism between the underlying morphisms of stacks,
    such that there exists a $2$-morphism between the corresponding
    morphisms of $2$-vector bundles,
    which can be identified with the given $2$-morphism
    using some local trivialization.

    Finally, for an algebraic $\bbK$-stack $\calS$,
    we define \emph{affine $2$-vector bundles} over $\calS$
    similarly as above, with $U$ replaced by $\calS$,
    and open covers of $U$ replaced by smooth open covers of $\calS$ by $\bbK$-schemes.
    Morphisms and $2$-morphisms of affine $2$-vector bundles over $\calS$
    are defined likewise.
\end{definition}

\begin{lemma}
    \label{lem-2vb-morphism-affine}
    In the situation of
    \cref{thm-2vb-morphism}~\cref{itm-lem-2vb-morphism-2},
    let $\alpha \colon U \to [G^1/G^0]$ be any section. Then
    \[
        [F^1/F^0] \underset{[G^1/G^0]}{\times} U
    \]
    is an affine $2$-vector bundle over $U$,
    and is an $[E^1/E^0]$-torsor.
\end{lemma}

\begin{proof}
    We claim that $\alpha$ locally lifts to a section of $[F^1/F^0]$.
    Indeed, locally, we may assume that $U$ is affine,
    that all vector bundles in question are trivial,
    and that $F^1 \to G^1$ is surjective.
    Then $\alpha$ factors through $G^1$, 
    since $\alpha$ can be represented as a $G^0$-torsor $T \to U$
    together with a $G^0$-equivariant morphism $T \to G^1$.
    However, such a $G^0$-torsor must be trivial,
    since if $m$ is the rank of $G^0$, then
    \[
        H^1_\et (U, (\Ga)^m) \simeq H^1 (U, \calO_U^{\oplus m}) \simeq 0 \ ,
    \]
    where the first two terms are \'etale and Zariski cohomology, respectively,
    and they are the same because $(\Ga)^m$ is a special group.
    Therefore, $T$ has a section, through which
    $\alpha$ lifts to $G^1$, and upon shrinking $U$, lifts to $F^1$.

    Therefore, locally, we can use the result of
    \cref{thm-2vb-morphism}~\cref{itm-lem-2vb-morphism-2}
    by pretending that this lift is the zero section of $[F^1/F^0]$.
    This shows that the fibre product locally looks like $[E^1/E^0]$,
    and transition morphisms are translations by sections.
\end{proof}

\begin{remark}
    \Cref{thm-2vb-morphism,lem-2vb-morphism-affine}
    can be generalized to the case when $U$ is replaced by an algebraic stack $\calS$,
    without much difficulty.
    For example, a fibre product of $\calS$-stacks can be seen as
    a fibre product of stacks over the site $\cat{Sch}_{\calS}$\,,
    which can be given object-wise by those lemmas.
\end{remark}

\begin{theorem}
    \label{thm-motive-2vb}
    Let $\calX$ be an algebraic $\bbK$-stack of finite type with affine stabilizers,
    and let $\calE \to \calX$ be a $2$-vector bundle of rank $n$,
    or an affine $2$-vector bundle of rank $n$.
    Then their motives satisfy
    \begin{equation}
        \mu (\calE) = \mu (\calX) \cdot \bbL^n \ .
    \end{equation}
\end{theorem}

\begin{proof}
    As in the proof of \cref{thm-motive-bundle},
    we may assume that $\calX = [X / \GL (m)]$
    for some $\bbK$-variety $X$.
    Let $\tilde{\calE}$ be the pullback of $\calE$ to $X$.
    Then $\tilde{\calE}$ is a principal $\GL (m)$-bundle over $\calE$.
    By \cref{thm-motive-bundle}, we have
    \begin{align*}
        \mu (\calX) & =
        \mu (X) \cdot \mu (\GL (m))^{-1} \ , \\
        \mu (\calE) & =
        \mu (\tilde{\calE}) \cdot \mu (\GL (m))^{-1} \ .
    \end{align*}
    Therefore, it is enough to prove that
    \begin{equation}
        \mu (\tilde{\calE}) = \mu (X) \cdot \bbL^n \ .
    \end{equation}
    
    Shrinking $X$, we may assume that $\tilde{\calE}$
    has the form $[E^1 / E^0]$,
    where $E^0, E^1 \to X$ are trivial vector bundles.
    Let $r, s$ be the ranks of $E^0, E^1$, respectively. 
    Then $E^1$ is a principal $(\Ga)^r$-bundle over $\tilde{\calE}$.
    Since $(\Ga)^r$ is special, we have
    \begin{align*}
        \mu (\tilde{\calE}) & =
        \mu (E^1) \cdot \mu ((\Ga)^r)^{-1} \\
        & = (\mu (X) \cdot \bbL^s) \cdot \bbL^{-r} \\
        & = \mu (X) \cdot \bbL^{s - r} \ ,
        \numberthis
    \end{align*}
    as required.
\end{proof}

%% file: proof-comb.tex
\label{sect-proof-comb}
\newcommand{\sI}{{\smash{I}}}

This appendix is dedicated to the proof of \cref{thm-comb},
through a complicated combinatorial argument.
We also believe that with some slight modifications,
our argument can also provide a purely combinatorial proof of
\cite[Theorem~5.4]{Joyce2008IV},
which was claimed there, but an alternative proof was given there instead.

\subsection{The setup}

\paragraph{}

Throughout this section, let $I$ be a finite set,
and let 
\begin{equation}
    C_I = \{ e_i \, , e_i^\vee \mid i \in I \}
\end{equation}
be a set of symbols.
We define a map $(-)^\vee \colon C_I \to C_I$
by sending $e_i$ to $e_i^\vee$
and $e_i^\vee$ to $e_i$ for all $i \in I$.

Let $A_I$ be the free associative algebra over $\bbQ$
generated by elements of $C_I$\,,
where we denote the multiplication by $*$.
There is an involution $(-)^\vee \colon A_I \to A_\sI^\op$,
given by
\begin{equation}
    (x_1 * \cdots * x_k)^\vee =
    x_{\smash{k}}^\vee * \cdots * x_1^\vee
\end{equation}
for $x_1, \dotsc, x_k \in C_I$\,.
This equips $A_I$ with the structure of an involutive algebra,
as in \cref{para-inv-alg}.

Let $L_I$ be the free Lie algebra over $\bbQ$
generated by elements of $C_I$\,.
Let $L_\sI^\op$ be the opposite Lie algebra of $L_I$,
i.e.\ the Lie algebra with the same underlying vector space
and with Lie bracket $[x, y]_{L_\sI^\op} = [y, x]_{L_I}$\,.
There is an involution $(-)^\vee \colon L_I \to L_\sI^\op$,
defined inductively by
\begin{alignat}{2}
    x & \mapsto x^\vee
    && \quad \text{for } x \in C_I \ , \\
    {} [x, y] & \mapsto [y^\vee, x^\vee]
    && \quad \text{for } x, y \in L_I \ .
\end{alignat}
This equips $L_I$ with the structure of an involutive Lie algebra,
as in \cref{para-inv-lie-alg}.

There is a natural inclusion of vector spaces $L_I \hookrightarrow A_I$,
which identifies $A_I$ with the universal enveloping algebra of $L_I$\,.

\paragraph{}

Define linear subspaces
\begin{align}
    L_\sI^+ & = \{ x \in L_I \mid x = -x^\vee \} \ , \\
    L_\sI^- & = \{ x \in L_I \mid x = x^\vee \} \ .
\end{align}
Then $L_I = L_\sI^+ \oplus L_\sI^-$,
and this makes $L_I$ into a $\bbZ_2$-graded Lie algebra.
In other words, we have
$[L_\sI^+, L_\sI^\pm] \subset L_\sI^\pm$ and
$[L_\sI^-, L_\sI^\pm] \subset L_\sI^\mp$.
In particular, $L_\sI^+ \subset L_I$ is a Lie subalgebra,
and there is a natural embedding
\begin{equation}
    U (L_\sI^+) \hookrightarrow A_I
\end{equation}
of associative algebras,
where the left-hand side is the universal enveloping algebra of $L_\sI^+$.
This is important,
as our main goal is to show that the wall-crossing formula in \cref{thm-comb},
which was originally expressed in terms of the
$A_I$-module structure given by the operation $\diamond$,
using the coefficients $U^\sd ({\cdots})$,
can actually be expressed solely in terms of the
$L_\sI^+$-module structure given by the operation $\heart$,
using the coefficients $\tilde{U}^\sd ({\cdots})$.

Define linear maps
$(-)^+ \colon L_I \to L_\sI^+$, $(-)^- \colon L_I \to L_\sI^-$ by
\begin{align}
    x^+ & = x - x^\vee, \\
    x^- & = x + x^\vee.
\end{align}
We have the relations
\begin{align}
    (x^+)^- & = (x^-)^+ = 0 \ , \\
    {} [x, y]^+ & = \frac{1}{2} \bigl( [x^+, y^+] + [x^-, y^-] \bigr) \ , \\
    {} [x, y]^- & = \frac{1}{2} \bigl( [x^+, y^-] + [x^-, y^+] \bigr)
\end{align}
for $x, y \in L_I$.

\paragraph{}

Write $n = |I|$.
Define a set
\begin{equation}
    P_I =
    \coprod_{\nicesubstack{
        \sigma \colon \{ 1, \dotsc, n \} \to I \\[-1.5ex]
        \text{bijective}
    }} {} \bigl\{
        x = (x_1, \dotsc, x_n) \bigm|
        x_i \in \{ e_{\sigma (i)} \, ,
        e_{\smash{\sigma (i)}}^\vee \}
        \text{ for all } i
    \bigr\} \ ,
\end{equation}
as a subset of $C_\sI^n$.

Let $K_I = \bbZ^{C_I}$ be the free abelian group generated by elements of $C_I$\,,
and let $K_\sI^+ = \bbN^{C_I} \setminus \{ 0 \} \subset K_I$\,.

\begin{definition}
    \label{def-stability-comb}
    A \emph{self-dual weak stability condition} on $I$
    is a map $\tau \colon K_\sI^+ \to T$,
    where $T$ is a totally ordered set,
    equipped with a distinguished element $0 \in T$,
    and an order-reversing involution $t \mapsto -t$,
    fixing the element $0$, such that
    \begin{enumerate}
        \item
            For any $\alpha, \beta, \gamma \in K_\sI^+$,
            such that $\beta = \alpha + \gamma$, either
            \[
                \tau (\alpha) \leq \tau (\beta) \leq \tau (\gamma) \ ,
                \quad \text{or} \quad
                \tau (\alpha) \geq \tau (\beta) \geq \tau (\gamma) \ .
            \]
        \item
            For any $\alpha \in K_\sI^+$,
            \[
                \tau (\alpha^\vee) = -\tau (\alpha) \ .
            \]
    \end{enumerate}
\end{definition}

\subsection{Wall-crossing to trivial stability}

The goal of this subsection is to prove \cref{thm-comb-to-triv} below.

\paragraph{}

Let $\tau$ be a self-dual weak stability condition on $I$,
as in \cref{def-stability-comb}. Define elements
\begin{align}
    \label{eq-comb-def-t}
    T (I; \tau) & =
    \sum_{\sigma \in \frS_n}
    U (e_{\sigma (1)}, \dotsc, e_{\sigma (n)}; \tau, 0) \cdot
    e_{\sigma (1)} * \cdots * e_{\sigma (n)} \ , \\
    \label{eq-comb-def-t-bar}
    \bar{T} (I; \tau) & =
    \sum_{x \in P_I}
    U (x_1, \dotsc, x_n; \tau, 0) \cdot
    x_1 * \cdots * x_n \ , \\
    \label{eq-comb-def-t-sd}
    T^\sd (I; \tau) & =
    \sum_{x \in P_I}
    U^\sd (x_1, \dotsc, x_n; \tau, 0) \cdot
    x_1 * \cdots * x_n
\end{align}
in the algebra $A_I$, where the coefficients
$U ({\cdots})$ and $U^\sd ({\cdots})$ are defined as in
\cref{eq-def-u} and~\cref{eq-def-usd}.

\begin{theorem}
    \label{thm-comb-to-triv}
    We have
    \begin{align}
        \label{eq-thm-comb-to-triv}
        T (I; \tau) & \in L_I \ , \\
        \label{eq-thm-comb-to-triv-bar}
        \bar{T} (I; \tau) & \in L_\sI^- \ , \\
        \label{eq-thm-comb-to-triv-sd}
        \smash{T^\sd} (I; \tau) & \in U (L_\sI^+) \ .
    \end{align}
\end{theorem}

The proof will be given at the end of this subsection.

From now on, we take $I = \{ 1, \dotsc, n \}$.
For a subset $J = \{ i_1, \dotsc, i_k \} \subset I$,
where $k \geq 1$ and $i_1 > \cdots > i_k$\,, define elements
$F (J), \bar{F} (J) \in L_I$ and $G (J) \in L_\sI^+$ by
\begin{align}
    \label{eq-def-fj}
    F (J) & =
    \frac{(-1)^{k-1}}{(k-1)!} \, B_{k-1} \cdot
    \sum_{\nicesubstack{
        \sigma \in \frS_k : \\[-.5ex]
        \sigma (1) = 1
    }} {}
    [ [ \dotsc [e_{i_1} \, , e_{i_{\sigma (2)}}], \dotsc ], e_{i_{\sigma (k)}}] \ , \\
    \label{eq-def-fj-bar}
    \bar{F} (J) & =
    \frac{(-1)^{k-1}}{(k-1)!} \, B_{k-1} \cdot
    \sum_{\nicesubstack{
        \sigma \in \frS_k : \\[-.5ex]
        \sigma (1) = 1
    }} {}
    [ [ \dotsc [e_{i_1} \, , e_{i_{\sigma (2)}}^-], \dotsc ], e_{i_{\sigma (k)}}^-] \ , \\
    \label{eq-def-gj}
    G (J) & =
    \frac{(-1)^k}{k!} \, \Bigl( B_k - B_k \Bigl( \frac{1}{2} \Bigr) \Bigr) \cdot
    \sum_{\nicesubstack{
        \sigma \in \frS_k : \\[-.5ex]
        \sigma (1) = 1
    }} {}
    [ [ \dotsc [e_{i_1}^\mp \, , e_{i_{\sigma (2)}}^-], \dotsc ], e_{i_{\sigma (k)}}^-] \ ,
\end{align}
where $B_k$ denotes the $k$-th Bernoulli number,
and $B_k (-)$ denotes the $k$-th Bernoulli polynomial.
The sign `$\mp$' is `$+$' if and only if $k$ is odd.

Note that $F (J) = \bar{F} (J) = 0$ whenever $k > 2$ is even,
and $G (J) = 0$ whenever $k > 1$ is odd.

For $x_1, \dotsc, x_k \in A_I$\,, we denote
\begin{align}
    s_k (x_1, \dotsc, x_k) & =
    \frac{1}{k!} \sum_{\sigma \in \frS_k}
    x_{\sigma (1)} * \cdots * x_{\sigma (k)} \ , \\
    \bar{s}_k (x_1, \dotsc, x_k) & =
    \frac{1}{2^k \, k!} \, s_k (x_1^-, \dotsc, x_k^-) \ .
\end{align}

\begin{lemma}
    \label{lem-comb-to-identities}
    We have combinatorial identities
    \begin{align}
        \label{eq-comb-id-f}
        e_n * s_{n-1} (e_1, \dotsc, e_{n-1}) & =
        \sum_{\nicesubstack{
            J \subset I : \\[-.5ex]
            n \in J
        }}
        s_{n-|J|+1} \bigl( F (J), \ e_i : i \in I \setminus J \bigr) \ , \\
        e_n * \bar{s}_{n-1} (e_1, \dotsc, e_{n-1}) & = \\*[1ex]
        \label{eq-comb-id-fbar-g}
        & \hspace{-6em}
        \sum_{\nicesubstack{
            J \subset I : \\[-.5ex]
            n \in J
        }}
        \bar{s}_{n-|J|+1} \bigl( \bar{F} (J), \ e_i : i \in I \setminus J \bigr)
        +
        \sum_{\nicesubstack{
            J \subset I : \\[-.5ex]
            n \in J
        }}
        \bar{s}_{n-|J|} \bigl( e_i : i \in I \setminus J \bigr) * G (J) \ .
        \hspace*{2em}
    \end{align}
\end{lemma}

\begin{proof}
    \allowdisplaybreaks
    For \cref{eq-comb-id-f}, for $1 \leq i \leq k \leq n$, write
    \begin{equation}
        E_{i,k} = \frac{1}{(n-1)!} \sum_{\nicesubstack{
            \sigma \in \frS_n : \\[-1.5ex]
            \sigma (i) = n
        }}
        e_{\sigma (1)} * \cdots * e_{\sigma (k)} \ .
    \end{equation}
    For $k = 1, \dotsc, n$, using the invariance under the
    $\frS_{n-1}$-action permuting the elements $e_1, \dotsc, e_{n-1}$\,,
    we find that
    \begin{equation}
        \frac{1}{\binom{n-1}{k-1}}
        \sum_{\nicesubstack{
            J \subset I : \\[-1.5ex]
            n \in J, \, |J| = k
        }} F (J) =
        \sum_{i=1}^k {}
        (-1)^{i-1} \binom{k-1}{i-1} \cdot
        \frac{(-1)^{k-1}}{(k-1)!} \, B_{k-1} \cdot
        (k-1)! \, E_{i,k} \ .
    \end{equation}
    Simplifying this, we obtain
    \begin{equation}
        \sum_{\nicesubstack{
            J \subset I : \\[-1.5ex]
            n \in J, \, |J| = k
        }} F (J) =
        \sum_{i=1}^k {}
        \frac{(-1)^{k-i} \, (n-1)!}{(n-k)! \, (k-i)! \, (i-1)!} \,
        B_{k-1} \cdot E_{i,k} \ .
    \end{equation}
    Therefore,
    \begin{align*}
        & \sum_{\nicesubstack{
            J \subset I : \\[-.5ex]
            n \in J, \, |J| = k
        }}
        s_{n-k+1} \bigl( F (J), \ e_i : i \in I \setminus J \bigr) \\*
        = {} &
        \frac{1}{n-k+1} \cdot \sum_{j=0}^{n-k} {}
        \sum_{i=1}^k {}
        \frac{(-1)^{k-i} \, (n-1)!}{(n-k)! \, (k-i)! \, (i-1)!} \,
        B_{k-1} \cdot E_{i+j,n} \\*
        = {} &
        \sum_{i=1}^{n} {}
        \sum_{j=i-n+k}^{i} {}
        \frac{(-1)^{k-j} \, (n-1)!}{(n-k+1)! \, (k-j)! \, (j-1)!} \,
        B_{k-1} \cdot E_{i,n} \ .
        \numberthis
    \end{align*}
    The left-hand side of \cref{eq-comb-id-f} is just $E_{1,n}$\,,
    so it suffices to prove that for any $i = 1, \dotsc, n$,
    \begin{equation}
        \sum_{k=1}^n {}
        \sum_{j=i-n+k}^{i} {}
        \frac{(-1)^{k-j} \, (n-1)!}{(n-k+1)! \, (k-j)! \, (j-1)!} \,
        B_{k-1} 
        = \begin{cases}
            1, & i = 1 \ , \\
            0, & i > 1 \ .
        \end{cases}
    \end{equation}
    Setting $p = j-1$ and $q = k-j$, the above reduces to
    \begin{equation}
        \sum_{p=0}^{i-1} {}
        \sum_{q=0}^{n-i} {}
        \frac{(-1)^{q} \, (n-1)!}{(n-p-q)! \, p! \, q!} \,
        B_{p+q} 
        = \begin{cases}
            1, & i = 1 \ , \\
            0, & i > 1 \ .
        \end{cases}
    \end{equation}
    This follows from taking $x = 0$ in \cref{lem-comb-bernoulli} below.

    For \cref{eq-comb-id-fbar-g}, for $1 \leq i \leq k \leq n$, write
    \begin{equation}
        E_{i,k}^\pm = \frac{1}{2^{k-1} \, (n-1)!} \sum_{\nicesubstack{
            \sigma \in \frS_n : \\[-1.5ex]
            \sigma (i) = n
        }}
        e_{\sigma (1)}^- * \cdots * e_{\sigma (i)}^\pm
        * \cdots * e_{\sigma (k)}^- \ .
    \end{equation}
    Similarly to the previous case, we find that
    \begin{align}
        \sum_{\nicesubstack{
            J \subset I : \\[-.5ex]
            n \in J, \, |J| = k
        }} \bar{F} (J)^- & =
        \sum_{i=1}^k {}
        \frac{(-1)^{k-i+1} \, 2^{k-2} \, (n-1)!}{(n-k)! \, (k-i)! \, (i-1)!} \,
        B_{k-1} \cdot E_{i,k}^{(-1)^k} \ , \\
        \sum_{\nicesubstack{
            J \subset I : \\[-.5ex]
            n \in J, \, |J| = k
        }} G (J) & =
        \sum_{i=1}^k {}
        \frac{(-1)^{k-i} \, 2^{k-1} \, (n-1)!}{k \, (n-k)! \, (k-i)! \, (i-1)!} \,
        B'_k \cdot E_{i,k}^{(-1)^{k-1}} \ ,
    \end{align}
    where $B'_k = B_k - B_k (1/2)$.
    Proceeding as before, we have
    \begin{align}
        & \sum_{\nicesubstack{
            J \subset I : \\[-.5ex]
            n \in J, \, |J| = k
        }}
        \bar{s}_{n-k+1} \bigl( \bar{F} (J), \ e_i : i \in I \setminus J \bigr) \\*[-3ex]
        & \hspace{6em} =
        \sum_{i=1}^{n} {}
        \sum_{j=i-n+k}^{i} {}
        \frac{(-1)^{k-j} \, 2^{k-2} \, (n-1)!}{(n-k+1)! \, (k-j)! \, (j-1)!} \,
        B_{k-1} \cdot E_{i,n}^{(-1)^k} \ , \\[2ex]
        & \sum_{\nicesubstack{
            J \subset I : \\[-.5ex]
            n \in J, \, |J| = k
        }}
        \bar{s}_{n-k} \bigl( e_i : i \in I \setminus J \bigr) * G (J) \\*[-3ex]
        & \hspace{6em} =
        \sum_{i=n-k+1}^n {}
        \frac{(-1)^{n-i+1} \, 2^{k-1} \, (n-1)!}{k \, (n-k)! \, (n-i)! \, (i+k-n-1)!} \,
        B'_k \cdot E_{i,n}^{(-1)^{k-1}} \ .
    \end{align}
    The left-hand side of~\cref{eq-comb-id-fbar-g} is just
    $(1/2) \, (E_{1,n}^+ + E_{1,n}^-)$.
    Collecting the coefficients of each $E_{i,n}^\pm$\,,
    we see that to prove~\cref{eq-comb-id-fbar-g}, it is enough to prove that
    for any $i = 1, \dotsc, n$ and $\varepsilon = \pm1$,
    \begin{multline}
        \sum_{k=1}^{n} {}
        \varepsilon^{k-1} \cdot
        \sum_{j=i-n+k}^{i} {}
        \frac{(-1)^{k-j} \, 2^{k-2} \, (n-1)!}{(n-k+1)! \, (k-j)! \, (j-1)!} \,
        B_{k-1} \\*
        {} +
        \sum_{k=n-i+1}^{n} {} \varepsilon^k \cdot 
        \frac{(-1)^{n-i+1} \, 2^{k-1} \, (n-1)!}{k \, (n-k)! \, (n-i)! \, (i+k-n-1)!} \,
        B'_k
        = \begin{cases}
            1, & i = 1 \text{ and } \varepsilon = 1 \ , \\
            0, & i > 1 \text{ or } \varepsilon = -1 \ .
        \end{cases}
    \end{multline}
    The case $\varepsilon = 1$
    follows from taking $x = 0$ in \cref{lem-comb-bernoulli-2} below.
    The case $\varepsilon = -1$
    follows from taking $x = 0$ in \cref{lem-comb-bernoulli-3} below.
\end{proof}

\paragraph{}

For $J \subset I$, we denote
\[
    I \setquot J = (I \setminus J) \sqcup \{ J \} \ .
\]
For any element $x \in A_J$\,, there is a homomorphism
\begin{align*}
    A_{I \setquot J} & \longrightarrow A_I \ , \\
    y & \longmapsto y \, |_{e_J \mapsto x} \ ,
\end{align*}
defined by mapping $e_i$ to $e_i$\,,
and $e_i^\vee$ to $e_i^\vee$\,, for $i \in I \setminus J$,
and mapping $e_J$ to the image of $x$ in $A_I$\,,
and $e_J^\vee$ to the image of $x^\vee$ in $A_I$\,.

If, moreover, $x \in L_J$\,, then this map sends the subspace $L_{I \setquot J}$
to $L_I$\,, preserving the $\bbZ_2$-grading.
In particular, this map sends the subalgebra $U (L_{\smash{I \setquot J}}^+)$
into $U (L_\sI^+)$.

If $\tau$ is a self-dual weak stability condition on $I$
such that the restriction of $\tau$ to $K^+ (J)$ is a constant map,
then $\tau$ induces a self-dual weak stability condition on $I \setquot J$.

To avoid nested subscripts,
for $I, J, \tau, x$ as above, we denote
\begin{align}
    T (I, J; \tau; x) & = T (I \setquot J; \tau) \, |_{e_J \mapsto x} \ , \\
    T^\sd (I, J; \tau; x) & = T^\sd (I \setquot J; \tau) \, |_{e_J \mapsto x} \ .
\end{align}

We also define auxiliary coefficients
\begin{multline}
    \label{eq-def-u-prime}
    U' (\alpha_1, \dotsc, \alpha_n; \tau, \tilde{\tau}) =
    \sum_{ \leftsubstack[6em]{
        & 0 = a_0 < \cdots < a_m = n. \\[-.5ex]
        & \text{Define } \beta_1, \dotsc, \beta_m \text{ by }
        \beta_i = \alpha_{a_{i-1}+1} + \cdots + \alpha_{a_i} \, . \\[-.5ex]
        & \text{We require } \tau (\beta_i) = \tau (\alpha_j)
        \text{ for all } a_{i-1} < j \leq a_i
    } } {}
    S (\beta_{1}, \dotsc, \beta_{m}; \tau, \tilde{\tau}) \cdot
    \biggl(
        \prod_{i=1}^m \frac{1}{(a_i - a_{i-1})!} 
    \biggr) \ , \hspace{-2em} \\[-4ex]
    \mathstrut
\end{multline}
\begin{multline}
    \label{eq-def-usd-prime}
    U'^\sd (\alpha_1, \dotsc, \alpha_n; \tau, \tilde{\tau}) = \\
    \shoveright{
    \sum_{ \leftsubstack[6em]{
        & 0 = a_0 < \cdots < a_m \leq n. \\[-.5ex]
        & \text{Define } \beta_1, \dotsc, \beta_m \text{ by }
        \beta_i = \alpha_{a_{i-1}+1} + \cdots + \alpha_{a_i} \, . \\[-.5ex]
        & \text{We require } \tau (\beta_i) = \tau (\alpha_j)
        \text{ for all } a_{i-1} < j \leq a_i \, , \\[-.5ex]
        & \text{and } \tau (\alpha_j) = 0
        \text{ for all } j > a_m
    } } {}
    S^\sd (\beta_{1}, \dotsc, \beta_{m}; \tau, \tilde{\tau})
    \cdot \biggl(
        \prod_{i=1}^m \frac{1}{(a_i - a_{i-1})!} 
    \biggr) \cdot \frac{1}{2^{n-a_m} \, (n - a_m)!} \ . \hspace{-2em} } \\[-4ex]
    \mathstrut
\end{multline}
In other words, in~\cref{eq-def-u-prime},
we take the sum of all terms in \cref{eq-def-u} with $l = 1$;
in~\cref{eq-def-usd-prime},
we take the sum of all terms in \cref{eq-def-usd} with $l = 0$.

Define
\begin{align}
    T' (I; \tau) & =
    \sum_{\sigma \in \frS_n}
    U' (e_{\sigma (1)}, \dotsc, e_{\sigma (n)}; \tau, 0) \cdot
    e_{\sigma (1)} * \cdots * e_{\sigma (n)} \ ,\\
    \bar{T}' (I; \tau) & =
    \sum_{x \in P_I}
    U' (x_1, \dotsc, x_n; \tau, 0) \cdot
    x_1 * \cdots * x_n \ , \\
    T'^\sd (I; \tau) & =
    \sum_{x \in P_I}
    U'^\sd (x_1, \dotsc, x_n; \tau, 0) \cdot
    x_1 * \cdots * x_n \ ,
\end{align}
as elements of $A_I$\,.
By definition, we have the relations
\begin{align}
    \label{eq-t-is-sum-t-prime}
    T (I; \tau) & =
    \sum_{\nicesubstack{
        I = I_1 \sqcup \cdots \sqcup I_l : \\[-.5ex]
        I_i \neq \varnothing \text{ for any } i
    }}
    \frac{(-1)^{l-1}}{l} \cdot
    T' (I_1; \tau) * \cdots * T' (I_l; \tau) \ , \\
    \label{eq-t-bar-is-sum-t-prime}
    \bar{T} (I; \tau) & =
    \sum_{\nicesubstack{
        I = I_1 \sqcup \cdots \sqcup I_l : \\[-.5ex]
        I_i \neq \varnothing \text{ for any } i
    }}
    \frac{(-1)^{l-1}}{l} \cdot
    \bar{T}' (I_1; \tau) * \cdots * \bar{T}' (I_l; \tau) \ , \\
    \label{eq-t-sd-is-sum-t-prime}
    T^\sd (I; \tau) & =
    \sum_{\nicesubstack{
        I = I_0 \sqcup I_1 \sqcup \cdots \sqcup I_l : \\[-.5ex]
        I_i \neq \varnothing \text{ for } i = 1, \dotsc, l
    }}
    \binom{-1/2}{l} \cdot
    \bar{T}' (I_1; \tau) * \cdots * \bar{T}' (I_l; \tau) * T'^\sd (I_0; \tau) \ .
    \hspace*{1em}
\end{align}

For $J \subset I$ as above, a self-dual weak stability condition $\tau$
on $I$ that is constant on $J$,
and an element $x \in A_I$\,, we denote
\begin{align}
    T' (I, J; \tau; x) & = T' (I \setquot J; \tau) \, |_{e_J \mapsto x} \ , \\
    \bar{T}' (I, J; \tau; x) & = \bar{T}' (I \setquot J; \tau) \, |_{e_J \mapsto x} \ , \\
    T'^\sd (I, J; \tau; x) & = T'^\sd (I \setquot J; \tau) \, |_{e_J \mapsto x} \ ,
\end{align}
as is similar to the above.

\begin{lemma}
    \label{lem-comb-to-f}
    Let $0 \leq l < m \leq n$ with $m \geq l + 2$.
    Let $\tau_1, \tau_2$ be two self-dual weak stability conditions on $I$,
    satisfying
    \[
        \scalemath{1}{ \begin{aligned}
            & \tau_1 (e_1) \leq \cdots \leq \tau_1 (e_l) < \tau_1 (e_{l+1}) =
            \cdots = \tau_1 (e_{m-1}) < \tau_1 (e_m) <
            \tau_1 (e_{m+1}) \leq \cdots \leq \tau_1 (e_n) \ , \\
            & \tau_2 (e_1) \leq \cdots \leq \tau_2 (e_l) < \tau_2 (e_{l+1}) =
            \cdots = \tau_2 (e_{m-1}) = \tau_2 (e_m) <
            \tau_2 (e_{m+1}) \leq \cdots \leq \tau_2 (e_n) \ ,
        \end{aligned} }
    \]
    where each \textnormal{`$\leq$'} sign is \textnormal{`$=$'} for $\tau_1$
    if and only if the corresponding \textnormal{`$\leq$'} sign is \textnormal{`$=$'} for $\tau_2$\,.
    Then
    \begin{align}
        \label{eq-comb-to-f-prime}
        T' (I; \tau_1) & =
        \sum_{\nicesubstack{
            J \subset \{ l+1, \dotsc, m \} : \\[-.5ex]
            m \in J
        }}
        T' (I, J; \tau_2; F (J)) \ , \\
        \label{eq-comb-to-f}
        T (I; \tau_1) & =
        \sum_{\nicesubstack{
            J \subset \{ l+1, \dotsc, m \} : \\[-.5ex]
            m \in J
        }}
        T (I, J; \tau_2; F (J)) \ , 
    \end{align}
    where $F (J)$ is given by \cref{eq-def-fj}.
\end{lemma}

\begin{proof}
    First, let us prove \cref{eq-comb-to-f-prime}.
    By the definitions,
    both sides of \cref{eq-comb-to-f-prime}
    lie in the subspace of $A_I$ spanned by the elements
    $e_{\sigma (1)} * \cdots * e_{\sigma (n)}$\,,
    where $\sigma \in \frS_n$\,,
    such that $\tau_2 (e_{\sigma (1)}) \geq \cdots \geq \tau_2 (e_{\sigma (n)})$.
    Therefore, it suffices to prove that for each of these monomials,
    its coefficients on both sides are equal.
    
    Let $I_0 = \{ l+1, \dotsc, m \}$.
    Then \cref{eq-comb-to-f-prime}
    can be rewritten as
    \begin{equation}
        \label{eq-comb-to-f-prime-red-1}
        T' (I, I_0 ; \tau_2; T' (I_0; \tau_1)) =
        \sum_{\nicesubstack{
            J \subset I_0 : \\[-1.5ex]
            m \in J
        }}
        T' \bigl( I, I_0; \tau_2; T' (I_0\,, J; \tau_2; F (J)) \bigr) \ ,
    \end{equation}
    by an elementary combinatorial argument.
    Therefore, it is enough to show that
    \begin{equation}
        \label{eq-comb-to-f-prime-red-2}
        T' (I_0; \tau_1) =
        \sum_{\nicesubstack{
            J \subset I_0 : \\[-1.5ex]
            m \in J
        }}
        T' (I_0\,, J; \tau_2; F (J)) \ ,
    \end{equation}
    which is precisely \cref{eq-comb-to-f-prime}
    with $I = I_0$\,.
    Thus, we may ease the notation by setting $l = 0$, $m = n$,
    and $I_0 = I$.
    Expanding both sides of~\cref{eq-comb-to-f-prime-red-2},
    we see that it is equivalent to
    \begin{multline}
        \sum_{\nicesubstack{
            \sigma \in \frS_n : \\[-.5ex]
            \sigma (1) = n
        }}
        \frac{1}{(n-1)!} \cdot e_{\sigma (1)} * \cdots * e_{\sigma (n)}
        = \\[-1ex]
        \sum_{\nicesubstack{
            J \subset I_0 : \\[-.5ex]
            n \in J
        }} {} \mspace{9mu}
        \sum_{\nicesubstack{
            \sigma \colon \{ 1, \dotsc, n-|J|+1 \} \to I \setquot J \\[-.5ex]
            \text{bijective}
        }}
        \frac{1}{(n-|J|+1)!} \cdot
        e_{\sigma (1)} * \cdots * e_{\sigma (n-|J|+1)} \,
        \bigg|_{e_J \mapsto F (J)} \, ,
        \raisetag{2ex}
    \end{multline}
    which is precisely~\cref{eq-comb-id-f}.
    Therefore, we have proved~\cref{eq-comb-to-f-prime}.
    
    For \cref{eq-comb-to-f},
    using \cref{eq-t-is-sum-t-prime},
    we see that the right-hand side of \cref{eq-comb-to-f} equals
    \begin{align*}
        & \sum_{\nicesubstack{
            J \subset \{ l+1, \dotsc, m \} : \\[-.5ex]
            m \in J
        }} {} \mspace{9mu}
        \sum_{\nicesubstack{
            I = I_1 \sqcup \cdots \sqcup I_k : \\[-.5ex]
            I_i \neq \varnothing \text{ for any } i, \\[-.5ex]
            J \subset I_j \text{ for some } j
        }} {}
        \frac{(-1)^{k-1}}{k} \cdot
        T' (I_1; \tau_2) * \cdots *
        T' (I_j \, , J; \tau_2; F (J)) * {} \\*[-6ex]
        && \mathllap{ \cdots * T' (I_k; \tau_2) } \\[3ex]
        = {} &
        \sum_{\nicesubstack{
            I = I_1 \sqcup \cdots \sqcup I_k : \\[-.5ex]
            I_i \neq \varnothing \text{ for any } i. \\[-.5ex]
            \mathclap{ \text{Let } j \text{ satisfy } m \in I_j }
        }} {} \mspace{9mu}
        \sum_{\nicesubstack{
            J \subset \{ l+1, \dotsc, m \} : \\[-.5ex]
            m \in J
        }} {} 
        \frac{(-1)^{k-1}}{k} \cdot
        T' (I_1; \tau_2) * \cdots *
        T' (I_j \, , J; \tau_2; F (J)) * {} \\*[-6ex]
        && \mathllap{ \cdots * T' (I_k; \tau_2) } \\[3ex]
        = {} &
        \sum_{\nicesubstack{
            I = I_1 \sqcup \cdots \sqcup I_k : \\[-.5ex]
            I_i \neq \varnothing \text{ for any } i. \\[-.5ex]
            \mathclap{ \text{Let } j \text{ satisfy } m \in I_j }
        }} 
        T' (I_1; \tau_1) * \cdots *
        T' (I_k; \tau_1) \\
        = {} &
        T (I; \tau_1) \ ,
        \numberthis
    \end{align*}
    where the second equal sign uses the fact that
    $T' (I_i; \tau_1) = T' (I_i; \tau_2)$ if $i \neq j$, by the definitions.
\end{proof}

\begin{lemma}
    \label{lem-comb-to-f-bar}
    Let $0 \leq l < m \leq n$ and $\tau_1, \tau_2$
    satisfy the assumptions of \cref{lem-comb-to-f}.
    Then
    \begin{align}
        \label{eq-comb-to-f-bar-prime}
        \bar{T}' (I; \tau_1) & =
        \sum_{\nicesubstack{
            J \subset \{ l+1, \dotsc, m \} : \\[-.5ex]
            m \in J
        }}
        \bar{T}' (I, J; \tau_2; F (J)) \ , \\
        \label{eq-comb-to-f-bar}
        \bar{T} (I; \tau_1) & =
        \sum_{\nicesubstack{
            J \subset \{ l+1, \dotsc, m \} : \\[-.5ex]
            m \in J
        }}
        \bar{T} (I, J; \tau_2; F (J)) \ ,
    \end{align}
    where $F (J)$ is given by \cref{eq-def-fj}.
\end{lemma}

\begin{proof}
    We observe that for any $x = (x_1, \dotsc, x_n) \in P_I$\,, we have
    \begin{align}
        U' (x_1, \dotsc, x_n; \tau, 0) & =
        U' (x_n^\vee, \dotsc, x_1^\vee; \tau, 0) \ , \\
        U (x_1, \dotsc, x_n; \tau, 0) & =
        U (x_n^\vee, \dotsc, x_1^\vee; \tau, 0) \ ,
    \end{align}
    which follow from the definition of these coefficients.
    Therefore, if we write
    \begin{align}
        \label{eq-def-t-prime-x}
        T' (x; \tau) & =
        \sum_{\sigma \in \frS_n}
        U' (x_{\sigma (1)}, \dotsc, x_{\sigma (n)}; \tau, 0) \cdot
        x_{\sigma (1)} * \cdots * x_{\sigma (n)} \ , \\
        \label{eq-def-t-x}
        T (x; \tau) & =
        \sum_{\sigma \in \frS_n}
        U (x_{\sigma (1)}, \dotsc, x_{\sigma (n)}; \tau, 0) \cdot
        x_{\sigma (1)} * \cdots * x_{\sigma (n)} \ ,
    \end{align}
    and write $x^\vee = (x_n^\vee, \dotsc, x_1^\vee)$, then
    \begin{align}
        T' (x^\vee; \tau) & =
        T' (x, \tau)^\vee \ , \\
        T (x^\vee; \tau) & =
        T (x, \tau)^\vee \ .
    \end{align}
    
    To prove~\cref{eq-comb-to-f-bar-prime},
    note that both sides are self-dual by the above observation,
    so it is enough to prove that the coefficients of monomials
    $x_1 * \cdots * x_n$ that involve $e_m$ (rather than $e_m^\vee$)
    are equal on both sides.
    We divide such monomials into $2^{n-1}$ classes,
    according to whether they involve $e_i$ or $e_i^\vee$
    for $i \in \{ 1, \dotsc, n \} \setminus \{ m \}$.
    For each of these classes, let $I' \subset I$ be the set of $i \in I$
    such that $e_i^\vee$ is involved in that class.
    Let $\xi \colon K_I \to K_I$ be the automorphism exchanging $e_i$ and $e_i^\vee$
    for all $i \in I'$.
    Applying \cref{lem-comb-to-f}
    to the weak stability condition $\alpha \mapsto \tau (\xi (\alpha))$,
    we see that the coefficients of these monomials
    on both sides of~\cref{eq-comb-to-f-bar-prime}
    are equal to the coefficients of the corresponding monomials
    in~\cref{eq-comb-to-f-prime}
    using the modified weak stability condition.
    This proves~\cref{eq-comb-to-f-bar-prime}.
    
    Finally, \cref{eq-comb-to-f-bar}
    follows from an analogous argument using~\cref{eq-comb-to-f}.
\end{proof}

\begin{lemma}
    \label{lem-comb-to-f-sd}
    Let $0 \leq l < m \leq n$ with $m \geq l + 2$.
    Let $\tau_1, \tau_2$ be two self-dual weak stability conditions on $I$,
    satisfying
    \[
        \scalemath{0.95}{ \begin{aligned}
            0 & \leq \tau_1 (e_1) \leq \cdots \leq \tau_1 (e_l) < \tau_1 (e_{l+1}) =
            \cdots = \tau_1 (e_{m-1}) < \tau_1 (e_m) <
            \tau_1 (e_{m+1}) \leq \cdots \leq \tau_1 (e_n) \ , \\
            0 & \leq \tau_2 (e_1) \leq \cdots \leq \tau_2 (e_l) < \tau_2 (e_{l+1}) =
            \cdots = \tau_2 (e_{m-1}) = \tau_2 (e_m) <
            \tau_2 (e_{m+1}) \leq \cdots \leq \tau_2 (e_n) \ ,
        \end{aligned} }
    \]
    where each \textnormal{`$\leq$'} sign is \textnormal{`$=$'} for $\tau_1$
    if and only if the corresponding \textnormal{`$\leq$'} sign is \textnormal{`$=$'} for $\tau_2$\,.
    Then
    \begin{align}
        \label{eq-comb-to-f-sd-prime}
        T'^\sd (I; \tau_1) & =
        \sum_{\nicesubstack{
            J \subset \{ l+1, \dotsc, m \} : \\[-.5ex]
            m \in J
        }}
        T'^\sd (I, J; \tau_2; F (J)) \ , \\
        \label{eq-comb-to-f-sd}
        T^\sd (I; \tau_1) & =
        \sum_{\nicesubstack{
            J \subset \{ l+1, \dotsc, m \} : \\[-.5ex]
            m \in J
        }}
        T^\sd (I, J; \tau_2; F (J)) \ ,
    \end{align}
    where $F (J)$ is given by \cref{eq-def-fj}.
\end{lemma}

\begin{proof}
    \allowdisplaybreaks
    The proof is similar to that of \cref{lem-comb-to-f}.
    
    First, let us prove \cref{eq-comb-to-f-sd-prime}.
    By the definitions,
    both sides of \cref{eq-comb-to-f-prime}
    lie in the subspace of $A_I$ spanned by the elements
    $e_{\sigma (1)} * \cdots * e_{\sigma (n)}$\,,
    where $\sigma \in \frS_n$\,,
    such that $\tau_2 (e_{\smash{\sigma (1)}}) \geq \cdots \geq \tau_2 (e_{\smash{\sigma (n)}})$.
    Note that the $e_i^\vee$ cannot appear.
    Therefore, it suffices to prove that for each of these monomials,
    its coefficients on both sides are equal.

    Let $I_0 = \{ l+1, \dotsc, m \}$.
    We rewrite \cref{eq-comb-to-f-sd-prime} as
    \begin{equation}
        \label{eq-comb-to-f-sd-prime-red-1}
        T'^\sd (I, I_0 ; \tau_2; T' (I_0; \tau_1)) =
        \sum_{\nicesubstack{
            J \subset I_0 : \\[-1.5ex]
            m \in J
        }}
        T'^\sd \bigl( I, I_0; \tau_2; T' (I_0\,, J; \tau_2; F (J)) \bigr) \ ,
    \end{equation}
    which follows from~\cref{eq-comb-to-f-prime-red-2}.
    This proves~\cref{eq-comb-to-f-sd-prime}.
    
    For \cref{eq-comb-to-f-sd}, 
    using \cref{eq-t-sd-is-sum-t-prime},
    we see that the right-hand side of \cref{eq-comb-to-f-sd} equals
    \begin{align*}
        & \sum_{\nicesubstack{
            J \subset \{ l+1, \dotsc, m \} : \\[-.5ex]
            m \in J
        }} {} \mspace{9mu}
        \sum_{\nicesubstack{
            I = I_0 \sqcup I_1 \sqcup \cdots \sqcup I_k : \\[-.5ex]
            I_i \neq \varnothing \text{ for } i = 1, \dotsc, k, \\[-.5ex]
            J \subset I_j \text{ for some } j > 0
        }} {} 
        \binom{-1/2}{k} \cdot
        \bar{T}' (I_1; \tau_2) * \cdots *
        \bar{T}' (I_j \, , J; \tau_2; F (J)) * {} \\*[-6ex]
        && \mathllap{ \cdots *
        \bar{T}' (I_k; \tau_2) * T'^\sd (I_0; \tau_2) } \\*[1ex]
        & \hspace{2em} {} + \sum_{\nicesubstack{
            J \subset \{ l+1, \dotsc, m \} : \\[-.5ex]
            m \in J
        }} {} \mspace{9mu}
        \sum_{\nicesubstack{
            I = I_0 \sqcup I_1 \sqcup \cdots \sqcup I_k : \\[-.5ex]
            I_i \neq \varnothing \text{ for } i = 1, \dotsc, k, \\[-.5ex]
            J \subset I_0
        }} {} 
        \binom{-1/2}{k} \cdot
        \bar{T}' (I_1; \tau_2) * \cdots *
        \bar{T}' (I_k; \tau_2) * {} \\*[-6ex]
        && \mathllap{ T'^\sd (I_j \, , J; \tau_2; F (J)) } \\[3ex]
        = {} &
        \sum_{\nicesubstack{
            I = I_0 \sqcup I_1 \sqcup \cdots \sqcup I_k : \\[-.5ex]
            I_i \neq \varnothing \text{ for } i = 1, \dotsc, k, \\[-.5ex]
            \mathclap{ m \in I_j \text{ for some } j > 0 }
        }} {} \mspace{9mu}
        \sum_{\nicesubstack{
            J \subset \{ l+1, \dotsc, m \} \cap I_j  : \\[-.5ex]
            m \in J
        }} {} 
        \binom{-1/2}{k} \cdot
        \bar{T}' (I_1; \tau_2) * \cdots *
        \bar{T}' (I_j \, , J; \tau_2; F (J)) * {} \hspace{-.6em} \\*[-6ex]
        && \mathllap{ \cdots *
        \bar{T}' (I_k; \tau_2) * T'^\sd (I_0; \tau_2) } \\*[1ex]
        & \hspace{2em} {} + \sum_{\nicesubstack{
            I = I_0 \sqcup I_1 \sqcup \cdots \sqcup I_k : \\[-.5ex]
            I_i \neq \varnothing \text{ for } i = 1, \dotsc, k, \\[-.5ex]
            m \in I_0
        }} {} \mspace{9mu}
        \sum_{\nicesubstack{
            J \subset \{ l+1, \dotsc, m \} \cap I_0 : \\[-.5ex]
            m \in J
        }} {} 
        \binom{-1/2}{k} \cdot
        \bar{T}' (I_1; \tau_2) * \cdots *
        \bar{T}' (I_k; \tau_2) * {} \\*[-6ex]
        && \mathllap{ T'^\sd (I_j \, , J; \tau_2; F (J)) } \\[3ex]
        = {} &
        \sum_{\nicesubstack{
            I = I_0 \sqcup I_1 \sqcup \cdots \sqcup I_k : \\[-.5ex]
            I_i \neq \varnothing \text{ for } i = 1, \dotsc, k
        }} {}
        \binom{-1/2}{k} \cdot
        \bar{T}' (I_1; \tau_1) * \cdots *
        \bar{T}' (I_k; \tau_1) * T'^\sd (I_0; \tau_1) \\
        = {} &
        T^\sd (I; \tau_1) \ ,
        \numberthis
    \end{align*}
    where the second equal sign is by
    \cref{eq-comb-to-f-bar-prime} and~\cref{eq-comb-to-f-sd-prime}.
    This proves \cref{eq-comb-to-f-sd}.
\end{proof}

\begin{lemma}
    \label{lem-comb-to-fg}
    Let $1 \leq m \leq n$.
    Let $\tau_1, \tau_2$ be two self-dual weak stability conditions on $I$,
    satisfying
    \begin{align*}
        0 & = \tau_1 (e_1) = \cdots = \tau_1 (e_{m-1}) < \tau_1 (e_m) <
        \tau_1 (e_{m+1}) < \cdots < \tau_1 (e_n) \ , \\
        0 & = \tau_2 (e_1) = \cdots = \tau_2 (e_{m-1}) = \tau_2 (e_m) <
        \tau_2 (e_{m+1}) < \cdots < \tau_2 (e_n) \ .
    \end{align*}
    Then
    \begin{align}
        T'^\sd (I; \tau_1) & =
        \sum_{\nicesubstack{
            J \subset \{ 1, \dotsc, m \} : \\[-.5ex]
            m \in J
        }}
        T'^\sd (I, J; \tau_2; \bar{F} (J))
        +
        \sum_{\nicesubstack{
            J \subset \{ 1, \dotsc, m \} : \\[-.5ex]
            m \in J
        }}
        T'^\sd (I \setminus J; \tau_2) * G (J) \ ,
        \raisetag{3ex}
        \label{eq-comb-to-fg-prime}
        {} \\[2ex]
        T^\sd (I; \tau_1) & =
        \sum_{\nicesubstack{
            J \subset \{ 1, \dotsc, m \} : \\[-.5ex]
            m \in J
        }}
        T^\sd (I, J; \tau_2; \bar{F} (J))
        +
        \sum_{\nicesubstack{
            J \subset \{ 1, \dotsc, m \} : \\[-.5ex]
            m \in J
        }}
        T^\sd (I \setminus J; \tau_2) * G (J) \ ,
        \raisetag{3ex}
        \label{eq-comb-to-fg}
    \end{align}
    where $\bar{F} (J)$ and $G (J)$ are given by
    \cref{eq-def-fj-bar} and~\cref{eq-def-gj}.
\end{lemma}

\begin{proof}
    \allowdisplaybreaks
    The proof is similar to that of \cref{lem-comb-to-f}.
    
    First, let us prove \cref{eq-comb-to-fg-prime}.
    By the definitions,
    both sides of \cref{eq-comb-to-f-prime}
    lie in the subspace of $A_I$ spanned by the elements
    $e_{\sigma (1)} * \cdots * e_{\sigma (n)}$\,,
    where $\sigma \in \frS_n$\,,
    such that $\tau_2 (e_{\smash{\sigma (1)}}) \geq \cdots \geq \tau_2 (e_{\smash{\sigma (n)}})$.
    Therefore, it suffices to prove that for each of these monomials,
    its coefficients on both sides are equal.
    
    Let $I_0 = \{ 1, \dotsc, m \}$.
    We rewrite \cref{eq-comb-to-fg-prime} as
    \begin{multline}
        T'^\sd (I, I_0 ; \tau_2; T'^\sd (I_0; \tau_1)) =
        \sum_{\nicesubstack{
            J \subset I_0 : \\[-.5ex]
            m \in J
        }} {} \Bigl(
            T'^\sd \bigl( I, I_0; \tau_2;
                T'^\sd (I_0\,, J; \tau_2; \bar{F} (J))
            \bigr) \\[-4ex]
            {} + 
            T'^\sd \bigl( I, I_0; \tau_2;
                T'^\sd (I_0 \setminus J; \tau_2)
            \bigr) * G (J)
        \Bigr) \ .
    \end{multline}
    Therefore, as before,
    it suffices to prove~\cref{eq-comb-to-fg-prime} in the case when $m = n$.
    But this is precisely~\cref{eq-comb-id-fbar-g}.
    This proves~\cref{eq-comb-to-fg-prime}.
    
    For \cref{eq-comb-to-fg}, 
    using \cref{eq-t-sd-is-sum-t-prime},
    as is similar to the proof of \cref{lem-comb-to-f-sd},
    we see that the right-hand side of \cref{eq-comb-to-fg} equals
    \begin{align*}
        & \sum_{\nicesubstack{
            I = I_0 \sqcup I_1 \sqcup \cdots \sqcup I_k : \\[-.5ex]
            I_i \neq \varnothing \text{ for } i = 1, \dotsc, k, \\[-.5ex]
            \mathclap{ m \in I_j \text{ for some } j > 0 }
        }} {} \mspace{9mu}
        \sum_{\nicesubstack{
            J \subset \{ 1, \dotsc, m \} \cap I_j  : \\[-.5ex]
            m \in J
        }} {} 
        \binom{-1/2}{k} \cdot
        \bar{T}' (I_1; \tau_2) * \cdots *
        \bar{T}' (I_j \, , J; \tau_2; F (J)) * {} \hspace{-.6em} \\*[-6ex]
        && \mathllap{ \cdots *
        \bar{T}' (I_k; \tau_2) * T'^\sd (I_0; \tau_2) } \\*[1ex]
        & \hspace{2em} {} + \sum_{\nicesubstack{
            I = I_0 \sqcup I_1 \sqcup \cdots \sqcup I_k : \\[-.5ex]
            I_i \neq \varnothing \text{ for } i = 1, \dotsc, k, \\[-.5ex]
            m \in I_0
        }} {} \mspace{9mu}
        \sum_{\nicesubstack{
            J \subset \{ 1, \dotsc, m \} \cap I_0 : \\[-.5ex]
            m \in J
        }} {} 
        \binom{-1/2}{k} \cdot
        \bar{T}' (I_1; \tau_2) * \cdots *
        \bar{T}' (I_k; \tau_2) * {} \\*[-4ex]
        && \mathllap{ \Bigl(
            T'^\sd (I_0 \, , J; \tau_2; \bar{F} (J)) +
            T'^\sd (I_0 \setminus J; \tau_2) * G (J)
        \Bigr)} \\
        = {} &
        \sum_{\nicesubstack{
            I = I_0 \sqcup I_1 \sqcup \cdots \sqcup I_k : \\[-.5ex]
            I_i \neq \varnothing \text{ for } i = 1, \dotsc, k
        }} {}
        \binom{-1/2}{k} \cdot
        \bar{T}' (I_1; \tau_1) * \cdots *
        \bar{T}' (I_k; \tau_1) * T'^\sd (I_0; \tau_1) \\
        = {} &
        T^\sd (I; \tau_1) \ ,
        \numberthis
    \end{align*}
    where the first equal sign is by
    \cref{eq-comb-to-f-bar-prime} and~\cref{eq-comb-to-fg-prime}.
    This proves \cref{eq-comb-to-fg}.
\end{proof}

Now, we are ready to prove \cref{thm-comb-to-triv}.

\paragraph{Proof of \cref{thm-comb-to-triv}.}

We only write down a proof of the more difficult part \cref{eq-thm-comb-to-triv-sd},
as the proof of~\cref{eq-thm-comb-to-triv}
is analogous and easier,
and \cref{eq-thm-comb-to-triv-bar} follows from~\cref{eq-thm-comb-to-triv}
together with the fact that
$U (x_1, \dotsc, x_m; \tau, 0) = U (x_m^\vee, \dotsc, x_1^\vee; \tau, 0)$
for all $x \in P_I$\,.

Let $S$ be the set of self-dual weak stability conditions on $I$.
For $\tau \in S$, let $T_{\tau}$ be its codomain, which is a totally ordered set.
For $t, t' \in T_{\tau}$, write
\[
    \operatorname{sgn} (t, t') = \begin{cases}
        1, & t > t' \ , \\
        0, & t = t' \ , \\
        -1, & t < t' \ .
    \end{cases}
\]
Write $\operatorname{sgn} (t) = \operatorname{sgn} (t, 0)$.

Define an equivalence relation $\sim$ on $S$
to be generated by the following relations:

\begin{enumerate}
    \item \label{itm-comb-eqrel-iso}
        $\tau_1 \sim \tau_2$ if for any $i, j \in \{ 1, \dotsc, n \}$,
        $\operatorname{sgn} \tau_1 (e_i) =
        \operatorname{sgn} \tau_2 (e_i)$,
        $\operatorname{sgn} \bigl( \tau_1 (e_i), \tau_1 (e_j) \bigr) =
        \operatorname{sgn} \bigl( \tau_2 (e_i), \tau_2 (e_j) \bigr)$,
        and $\operatorname{sgn} \bigl( \tau_1 (e_i), \tau_1 (e_j^\vee) \bigr) =
        \operatorname{sgn} \bigl( \tau_2 (e_i), \tau_2 (e_j^\vee) \bigr)$.
    \item \label{itm-comb-eqrel-perm}
        $\tau_1 \sim \tau_2$ if there exists $\sigma \in \frS_n$
        with $\tau_1 (e_i) = \pm\tau_2 (e_{\sigma(i)})$ for all $i$,
        where the `$\pm$' signs are arbitrary.
    \item \label{itm-comb-eqrel-fsd}
        $\tau_1 \sim \tau_2$ if they satisfy the assumption of
        \cref{lem-comb-to-f-sd}.
    \item \label{itm-comb-eqrel-fg}
        $\tau_1 \sim \tau_2$ if they satisfy the assumption of
        \cref{lem-comb-to-fg}.
\end{enumerate}

We claim that $\sim$ is trivial,
that is, all elements of $S$ are equivalent under $\sim$.

Indeed, every element $\tau \in S$ is equivalent to
the trivial stability condition $0 \in S$.
To see this, using \cref{itm-comb-eqrel-perm},
we may assume that 
\[
    0 \leq \tau (e_1) \leq \cdots \leq \tau (e_n) \ .
\]
If all the inequality signs are equalities, then $\tau = 0$ and we are done.
If not, suppose that
\[
    0 = \tau (e_1) = \cdots = \tau (e_l)
    < \tau (e_{l+1}) = \cdots = \tau (e_m)
    < \tau (e_{m+1}) \leq \cdots \leq \tau (e_n) \ ,
\]
where $0 \leq l < m \leq n$.
Using \cref{itm-comb-eqrel-fsd},
we may increase the values of $\tau (e_m)$ by a small amount,
as long as it stays below $\tau (e_{m+1})$.
We then do the same thing to $e_{m-1} \, , \dotsc, e_{l+2}$\,,
so that we can now assume that
\[
    0 = \tau (e_1) = \cdots = \tau (e_l)
    < \tau (e_{l+1}) < \tau (e_{l+2}) \leq \cdots \leq \tau (e_n) \ ,
\]
where $0 \leq l < n$.
We can then use \cref{itm-comb-eqrel-fg} to modify $\tau (e_{l+1})$,
so that now we have $\tau (e_{l+1}) = 0$.
Repeating this process,
we see that we eventually reach a point where $\tau = 0$.

Therefore, what is left to prove is that
if $\tau_1 \sim \tau_2$\,, and $\tau_1$ satisfies \cref{eq-thm-comb-to-triv-sd},
then so does $\tau_2$\,.
To see this, we only need to check the cases
\cref{itm-comb-eqrel-iso}--\cref{itm-comb-eqrel-fg} individually.
By induction on $n$,
we can assume that this is already true for all smaller values of $n$,
as the case when $n = 0$ is trivial.

For \cref{itm-comb-eqrel-iso},
we see that $U^\sd (x_1, \dotsc, x_n; \tau_1, 0) =
U^\sd (x_1, \dotsc, x_n; \tau_2, 0)$
for all $(x_1, \dotsc, x_n) \in P_I$\,,
so that $T^\sd (I; \tau_1) = T^\sd (I; \tau_2)$.
For \cref{itm-comb-eqrel-perm},
permutations does not affect $T^\sd (I; \tau)$ either,
due to the permutation symmetry of $P_I$\,.
Switching the sign of $\tau (e_i)$ amounts to exchanging the roles of
$e_i$ and $e_i^\vee$, i.e., its effect on $T^\sd (I; \tau)$
is swapping $e_i$ with $e_i^\vee$ in the expression.
However, since the subspace $L_\sI^+ \subset L_I$
is fixed under this operation, \cref{eq-thm-comb-to-triv-sd} is preserved.
For \cref{itm-comb-eqrel-fsd},
we use~\cref{eq-comb-to-f-sd},
whose right-hand side contains $T^\sd (I; \tau_2)$
as the term with $J = \{ m \}$.
All other terms involve index sets of size $< n$,
and hence, after replacing $\smash{e_J} \mapsto F (J)$,
lie in $U (L_\sI^+)$, by our induction hypothesis.
For \cref{itm-comb-eqrel-fg},
similarly, we see from~\cref{eq-comb-to-fg}
that the difference between $T^\sd (I; \tau_1)$ and $T^\sd (I; \tau_2)$
lies in $U (L_\sI^+)$.
\QED

\subsection{General wall-crossing}

\paragraph{}

Let notation be as in the previous subsection.
For $x = (x_1, \dotsc, x_n) \in P_I$\,, define
\begin{align}
    \label{eq-comb-def-q}
    Q (x) & = \biggl\{ y = (y_1, \dotsc, y_m) \biggm| \begin{array}{l}
         m \geq 1, \ 0 = a_0 < \cdots < a_m = n , \\
         y_i = x_{a_{i-1}+1} + \cdots + x_{a_i} \text{ for all } i
    \end{array} \biggr\} \ , \\
    \label{eq-comb-def-q-prime}
    Q' (x) & = \biggl\{ y = (y_1, \dotsc, y_m) \biggm| \begin{array}{l}
         m \geq 0, \ 0 = a_0 < \cdots < a_m \leq n , \\
         y_i = x_{a_{i-1}+1} + \cdots + x_{a_i} \text{ for all } i
    \end{array} \biggr\} \ , 
\end{align}
where each element $y_i$ is regarded as an element of
$K_I \simeq \bigoplus_{i} {} \bigl( \bbZ e_i \oplus \bbZ e_i^\vee \bigr)$.

\begin{lemma}
    \label{lem-comb-coeff-comp}
    \allowdisplaybreaks
    For self-dual weak stability conditions $\tau, \hat{\tau}, \tilde{\tau}$ on $I$,
    we have combinatorial identities
    \begin{align}
        \label{eq-comb-s-id}
        S (x_1, \dotsc, x_n; \tau, \tau) & = \begin{cases}
            1, & n = 1 \ , \\
            0, & n > 1 \ ,
        \end{cases} \\
        \label{eq-comb-s-comp}
        S (x_1, \dotsc, x_n; \tau, \tilde{\tau}) & =
        \sum_{y \in Q (x)} S(y_1, \dotsc, y_m; \hat{\tau}, \tilde{\tau}) \cdot
        \prod_{i=1}^m S (x_{a_{i-1}+1} \, , \dotsc, x_{a_i}; \tau, \hat{\tau}) \ , \\
        \label{eq-comb-ssd-id}
        S^\sd (x_1, \dotsc, x_n; \tau, \tau) & = \begin{cases}
            1, & n = 0 \ , \\
            0, & n > 0 \ ,
        \end{cases} \\
        S^\sd (x_1, \dotsc, x_n; \tau, \tilde{\tau}) & =
        \sum_{y \in Q' (x)} S^\sd(y_1, \dotsc, y_m; \hat{\tau}, \tilde{\tau}) \cdot {}
        \notag \\*
        \label{eq-comb-ssd-comp}
        & \hspace{1em} \biggl(
            \prod_{i=1}^m S (x_{a_{i-1}+1} \, , \dotsc, x_{a_i}; \tau, \hat{\tau})
        \biggr) \cdot S^\sd (x_{a_m+1} \, , \dotsc, x_n; \tau, \hat{\tau}) \ , \\
        \label{eq-comb-u-id}
        U (x_1, \dotsc, x_n; \tau, \tau) & = \begin{cases}
            1, & n = 1 \ , \\
            0, & n > 1 \ ,
        \end{cases} \\
        \label{eq-comb-u-comp}
        U (x_1, \dotsc, x_n; \tau, \tilde{\tau}) & =
        \sum_{y \in Q (x)} U(y_1, \dotsc, y_m; \hat{\tau}, \tilde{\tau}) \cdot
        \prod_{i=1}^m U (x_{a_{i-1}+1} \, , \dotsc, x_{a_i}; \tau, \hat{\tau}) \ , \\
        \label{eq-comb-usd-id}
        U^\sd (x_1, \dotsc, x_n; \tau, \tau) & = \begin{cases}
            1, & n = 0 \ , \\
            0, & n > 0 \ ,
        \end{cases} \\
        U^\sd (x_1, \dotsc, x_n; \tau, \tilde{\tau}) & =
        \sum_{y \in Q' (x)} U^\sd(y_1, \dotsc, y_m; \hat{\tau}, \tilde{\tau}) \cdot {}
        \notag \\*
        \label{eq-comb-usd-comp}
        & \hspace{1em} \biggl(
            \prod_{i=1}^m U (x_{a_{i-1}+1} \, , \dotsc, x_{a_i}; \tau, \hat{\tau})
        \biggr) \cdot U^\sd (x_{a_m+1} \, , \dotsc, x_n; \tau, \hat{\tau}) \ ,
    \end{align}
    where $a_i$ is as in \cref{eq-comb-def-q} and~\cref{eq-comb-def-q-prime}.
\end{lemma}

\begin{proof}
    \newcommand{\ulalpha}{\smash[b]{\mspace{1mu}\underline{\mspace{-1mu}\smash{\alpha}\mspace{-2mu}}\mspace{2mu}}}
    \newcommand{\ulalphabar}{\smash[b]{\mspace{1mu}\underline{\mspace{-1mu}\smash{\bar{\alpha}}\mspace{-2mu}}\mspace{2mu}}}
    The identities \cref{eq-comb-s-id}, \cref{eq-comb-s-comp},
    \cref{eq-comb-u-id}, and \cref{eq-comb-u-comp}
    were proved in \cite[Theorems 4.5 and~4.8]{Joyce2008IV},
    using purely combinatorial methods.
    The identities \cref{eq-comb-ssd-id} and~\cref{eq-comb-usd-id}
    follow from the definitions easily.

    One could also prove the other two identities,
    \cref{eq-comb-ssd-comp} and~\cref{eq-comb-usd-comp},
    using combinatorics. However, we take a more intuitive approach
    and deduce them from the results in \cref{sect-invariants}.
    
    Consider a self-dual quiver $(Q, \sigma, u, v)$,
    as in \cref{para-sd-quiver},
    defined as follows.
    The set of vertices of $Q$ is
    $Q_0 = I \sqcup I^\vee = \{ 1, 1^\vee, \dotsc, n, n^\vee \}$.
    There is a unique arrow $i \to j$ for any $i, j \in Q_0$,
    making a total of $4n^2$ arrows.
    Define the involution $\sigma \colon Q \simeq Q^\op$
    by exchanging the vertices $i$ and $i^\vee$ for all $i \in \{ 1, \dotsc, n \}$.
    The action on the arrows is determined accordingly.
    Let $u, v$ assign the sign $+1$ to all vertices and arrows.
    
    Let $\calA$ denote the self-dual abelian category $\Mod (\bbK Q)$.
    As in \cref{para-moduli-quiver},
    there is a self-dual moduli stack for $\calA$
    to which we can apply the results in \cref{sect-invariants}.
    For convenience, for vertices $i^\vee \in Q_0$ with $i \in I$,
    we write $e_{i^\vee} = e_i^\vee \in C_I$
    and $e_{\smash{i^\vee}}^\vee = e_i \in C_I$\,.

    Let $C'_\sI$ be the set of all $\alpha \in C (\calA)$
    that is a non-zero sum of distinct elements of $C_I$\,.
    For such $\alpha$, let $C_\alpha \subset C_I$ be the set of terms appearing in $\alpha$.
    Define an object $(E^\alpha, e^\alpha) \in \calA$ by
    \[
        E^\alpha_i = \begin{cases}
            \bbK, & i \in C_\alpha \ , \\
            0, & \text{otherwise} \ ,
        \end{cases}
        \qquad
        e^\alpha_{i \to j} = \begin{cases}
            1, & i, j \in C_\alpha \ , \\
            0, & \text{otherwise} \ .
        \end{cases}
    \]
    One can see that $(E^\alpha, e^\alpha)$ is a simple object,
    and hence is semistable under any weak stability condition.
    
    Let $\Sigma$ be the set of all $\ulalpha = (\alpha_1, \dotsc, \alpha_m)$
    with $m \geq 0$ and $\alpha_s \in C'_\sI$ for all $s$,
    such that $\alpha_1 + \cdots + \alpha_m \in C'_\sI$\,.
    For each $\ulalpha \in \Sigma$,
    define an object $(E^{\ulalpha}, e^{\ulalpha}) \in \calA$ by
    \begin{align*}
        (E^{\ulalpha})_i & = \begin{cases}
            \bbK, & i \in C_{\alpha_s} \text{ for some } s \ , \\
            0, & \text{otherwise} \ ,
        \end{cases} \\
        (e^{\ulalpha})_{i \to j} & = \begin{cases}
            1, & i \in C_{\alpha_s} \text{ and } j \in C_{\alpha_t} \text{ for some } s \geq t \ , \\
            0, & \text{otherwise} \ .
        \end{cases}
    \end{align*}
    Then $(E^{\ulalpha}, e^{\ulalpha})$ has a unique Jordan--H\"older filtration
    whose quotients are $(E^{\alpha_1}, e^{\alpha_1})$,~$\dotsc$,
    $(E^{\alpha_m}, e^{\alpha_m})$, where the uniqueness follows from an elementary argument.

    Define a partial order $\preceq$ on $\Sigma$ such that
    $\ulalpha \preceq \ulalpha'$ if and only if there exists
    $0 = s_0 < \cdots < s_{m'} = m$ such that
    $\alpha'_t = \alpha_{\smash{s_{t-1}+1}} + \cdots + \alpha_{\smash{s_t}}$ for all $t$,
    where $m$ and $m'$ are the lengths of $\ulalpha$ and $\ulalpha'$.
    
    For a fixed weak stability condition $\tau$, and for $\ulalpha \in \Sigma$, write
    $\delta_{\ulalpha} (\tau) =
    \delta_{\alpha_1} (\tau) * \cdots * \delta_{\alpha_m} (\tau) \in \SF (\calM)$.
    Then $\delta_{\smash{\ulalpha'}} (\tau)$
    being non-zero at $(E^{\ulalpha}, e^{\ulalpha})$
    implies $\ulalpha \preceq \ulalpha'$,
    since any filtration is refined by the Jordan--H\"older filtration.
    In particular, $\delta_{\ulalpha} (\tau)$
    is not in the linear span of $\delta_{\smash{\ulalpha'}} (\tau)$
    with $\ulalpha \not\preceq \ulalpha'$.
    Since $\preceq$ can be extended to a total order on the finite set $\Sigma$,
    we conclude that the stack functions
    $\delta_{\ulalpha} (\tau) \in \SF (\calM)$
    for all $\ulalpha \in \Sigma$
    are linearly independent.
    
    As a result, the stack functions
    $\epsilon_{\alpha_1} (\tau) * \cdots * \epsilon_{\alpha_m} (\tau) \in \SF (\calM)$
    are also linearly independent,
    essentially because upper triangular matrices with $1$'s on the diagonal are invertible.

    At this point, as a side note, if we apply \cref{thm-wcf-main}
    to express $\delta_\alpha (\tilde{\tau})$ in terms of $\tau$ invariants,
    where $\alpha = e_1 + \cdots + e_n$\,,
    and compare the result with first converting
    $\tilde{\tau}$ invariants to $\hat{\tau}$ invariants,
    and then converting to $\tau$ invariants,
    we have reproved \cref{eq-comb-s-comp}.
    Applying this to $\epsilon_\alpha (\tau)$ reproves \cref{eq-comb-u-comp}.

    Now, let $\Sigma^\sd$ be the set of $\ulalpha = (\alpha_1, \dotsc, \alpha_m, \rho)$
    such that each $\alpha_s$ is in $C'_\sI$\,,
    \,$\rho \in C'_\sI \cup \{ 0 \}$, \,$\rho = \rho^\vee$, 
    and $\bar{\alpha}_1 + \cdots + \bar{\alpha}_m + \rho \in C'_\sI$\,.
    For a fixed self-dual weak stability condition $\tau$, and for $\ulalpha \in \Sigma^\sd$,
    write $\delta^\sd_{\ulalpha} (\tau) =
    \delta_{\smash{\alpha_1}} (\tau) \diamond \cdots \diamond \delta_{\smash{\alpha_m}} (\tau)
    \diamond \delta^\sd_{\smash{\rho}} (\tau) \in \SF (\calM^\sd)$.
    Similarly, we claim that the stack functions
    $\delta^\sd_{\ulalpha} (\tau)$ for all $\ulalpha \in \Sigma^\sd$
    are linearly independent.
    
    Indeed, for $\ulalpha \in \Sigma^\sd$, define
    $\ulalphabar = (\alpha_1, \dotsc, \alpha_m, (\rho{,}) \ \alpha_m^\vee, \dotsc, \alpha_1^\vee) \in \Sigma$, where $\rho$ appears only when it is non-zero.
    The object
    $(E^{\ulalphabar}, e^{\ulalphabar})$ has a natural self-dual structure.
    Using its unique Jordan--H\"older filtration,
    we can show that $\delta^\sd_{\ulalpha} (\tau)$
    is not in the linear span of $\delta^\sd_{\smash{\ulalpha'}} (\tau)$
    with $\ulalphabar \not\preceq \ulalphabar'$.
    Since the map $\ulalpha \mapsto \ulalphabar$ is injective,
    it follows that the $\delta^\sd_{\ulalpha} (\tau)$ are linearly independent.

    Similarly, it follows that the stack functions
    $\epsilon_{\smash{\alpha_1}} (\tau) \diamond \cdots
    \diamond \epsilon_{\smash{\alpha_m}} (\tau)
    \diamond \epsilon^\sd_{\smash{\rho}}$
    for $\ulalpha \in \Sigma^\sd$ are linearly independent.
    
    Applying \cref{thm-wcf-main}
    to express $\delta^\sd_{\smash{\theta}} (\tilde{\tau})$ in terms of $\tau$ invariants,
    where $\theta = \bar{e}_1 + \cdots + \bar{e}_n$\,,
    and comparing the result with first converting
    $\tilde{\tau}$ invariants to $\hat{\tau}$ invariants,
    and then converting to $\tau$ invariants,
    we have proved \cref{eq-comb-ssd-comp}.
    Applying this to $\epsilon^\sd_{\smash{\theta}} (\tau)$ proves \cref{eq-comb-usd-comp}.
\end{proof}

\paragraph{}

Now, define elements
\begin{align}
    V (I; \tau) & =
    \sum_{\sigma \in \frS_n}
    U (e_{\sigma (1)}, \dotsc, e_{\sigma (n)}; 0, \tau) \cdot
    e_{\sigma (1)} * \cdots * e_{\sigma (n)} \ , \\
    V^\sd (I; \tau) & =
    \sum_{x \in P_I}
    U^\sd (x_1, \dotsc, x_n; 0, \tau) \cdot
    x_1 * \cdots * x_n
\end{align}
in the algebra $A_I$.

\begin{theorem}
    \label{thm-comb-from-triv}
    We have
    \begin{align}
        \label{eq-thm-comb-from-triv}
        V (I; \tau) & \in L_I \ , \\
        \label{eq-thm-comb-from-triv-sd}
        V^\sd (I; \tau) & \in U (L_\sI^+) \ .
    \end{align}
\end{theorem}

\begin{proof}
    The proof is essentially by formally inverting
    the results in \cref{thm-comb-to-triv}.

    We use induction on $n$, and assume that the theorem is true
    for all smaller values of $n$.
    The cases when $n = 1$ in \cref{eq-thm-comb-from-triv}
    and when $n = 0$ in \cref{eq-thm-comb-from-triv-sd}
    are trivial, since there is nothing to prove.
    
    To prove \cref{eq-thm-comb-from-triv},
    we may assume that $n > 1$.
    By \cref{eq-comb-u-id} and~\cref{eq-comb-u-comp},
    for any $x \in P_I$, we have
    \begin{equation}
        \sum_{y \in Q (x)} U(y_1, \dotsc, y_m; 0, \tau) \cdot
        \prod_{i=1}^m U (x_{a_{i-1}+1} \, , \dotsc, x_{a_i}; \tau, 0)
        = 0 \ .
    \end{equation}
    Summing over all possibilities of
    $x = (e_{\sigma (1)} , \dotsc, e_{\sigma (n)})$
    for $\sigma \in \frS_n$\,, we obtain that
    \begin{equation}
        \label{eq-comb-pf-from-triv}
        \sum_{\leftsubstack{ \\[-3ex]
            & m \geq 1, \ I = J_1 \sqcup \cdots \sqcup J_m \, \colon \\[-1.5ex]
            & J_i \neq \varnothing \text{ for all } i. \\[-1.5ex]
            & \textstyle \text{Write } y_i = \sum_{j \in J_i} e_j
        }} U (y_1, \dotsc, y_m; 0, \tau) \cdot
        T (J_1; \tau) * \cdots * T (J_m; \tau)
        = 0 \ ,
    \end{equation}
    where the $T (J_i; \tau)$ are as in \cref{eq-comb-def-t}.
    By \cref{thm-comb-to-triv},
    $T (J_i; \tau) \in L_{\smash{J_i}}$\,.
    Therefore, by the induction hypothesis,
    that is, by~\cref{eq-thm-comb-from-triv} applied to $m$ elements,
    if $m < n$, then for a fixed choice of $J_1, \dotsc, J_m$\,,
    the sum of all the $m!$ terms in~\cref{eq-comb-pf-from-triv}
    involving a permutation of $J_1, \dotsc, J_m$
    is in $L_I$\,.
    Since~\cref{eq-comb-pf-from-triv} equals zero,
    the sum of the terms that were not involved above must lie in $L_I$ as well.
    These are precisely the terms with $m = n$. 
    This gives that
    \begin{equation}
        \sum_{\sigma \in \frS_n}
        U (e_{\sigma (1)}, \dotsc, e_{\sigma (n)}; 0, \tau) \cdot
        e_{\sigma (1)} * \cdots * e_{\sigma (n)} \in L_I \ ,
    \end{equation}
    which is a restatement of~\cref{eq-thm-comb-from-triv}.
    
    To prove \cref{eq-thm-comb-from-triv-sd},
    we assume that $n > 0$, and proceed as before.
    By \cref{eq-comb-usd-id}--\cref{eq-comb-usd-comp},
    for any $x \in P_I$, we have
    \begin{multline}
        \sum_{y \in Q' (x)} U^\sd (y_1, \dotsc, y_m; 0, \tau) \cdot\biggl(
            \prod_{i=1}^m U (x_{a_{i-1}+1} \, , \dotsc, x_{a_i}; \tau, 0)
        \biggr) \cdot {} \\[-1ex]
        U^\sd (x_{a_m+1} \, , \dotsc, x_n; \tau, 0) \ = \ 0 \ .
    \end{multline}
    Summing over all possibilities of
    $x \in P_I$\,, we obtain that
    \begin{equation}
        \label{eq-comb-pf-from-triv-sd}
        \sum_{\leftsubstack{ \\[-3ex]
            & m \geq 0, \ I = J_1 \sqcup \cdots \sqcup J_m \sqcup J', \ 
            x^i \in P_{\smash{J_i}} \colon \\[-1.5ex]
            & J_i \neq \varnothing \text{ for all } i. \\[-1.5ex]
            & \textstyle \text{Write } y_i = \sum_{j \in J_i} x_j
        }} U^\sd (y_1, \dotsc, y_m; 0, \tau) \cdot
        T (x^1; \tau) * \cdots * T (x^m; \tau) * T^\sd (J'; \tau)
        = 0 \ ,
    \end{equation}
    where $T^\sd (J'; \tau)$
    is as in~\cref{eq-comb-def-t-sd},
    and $T (x^i; \tau)$ is as in~\cref{eq-def-t-x}.
    By \cref{thm-comb-to-triv},
    $T (x^i; \tau) \in L_{\smash{J_i}}$\,,
    and $T^\sd (J_i; \tau) \in U (L_{\smash{J'}}^+)$.
    For $m < n$, fix a choice of $J_1, \dotsc, J_m, J'$\,,
    and a choice of the $x^i$.
    Let $\Sigma$ be the sum of all the $2^m m!$ terms in~\cref{eq-comb-pf-from-triv-sd}
    involving a permutation of $J_1, \dotsc, J_m$\,, and for each $i$,
    either $x^i$ or $(x^i)^\vee$, where if $x^i = (x_1, \dotsc, x_k)$,
    then $(x^i)^\vee = (x_{\smash{k}}^\vee, \dotsc, x_1^\vee)$\,.
    Note that as in the proof of~\cref{eq-thm-comb-to-triv-bar},
    we have $T ((x^i)^\vee; \tau) = T (x^i; \tau)^\vee$,
    where the latter $(-)^\vee$ is the involution on $A_I$\,.
    Now, we can apply the induction hypothesis,
    or~\cref{eq-thm-comb-from-triv-sd} applied to $m$ elements,
    to see that $\Sigma$ is a linear combination of products
    of elements either in $L_{\smash{J'}}^+$,
    or of the form $T (x^i; \tau) - T (x^i; \tau)^\vee$,
    which lie in $L_{\smash{J_i}}^+$.
    Therefore, $\Sigma \in U (L_\sI^+)$.

    It then follows that the sum of the terms in~\cref{eq-comb-pf-from-triv-sd}
    with $m = n$ lies in $U (L_\sI^+)$ as well,
    which is, again, a restatement of~\cref{eq-thm-comb-from-triv-sd}.
\end{proof}

\paragraph{}

Next, for two self-dual weak stability conditions $\tau, \tilde{\tau}$ on $I$,
define elements
\begin{align}
    W (I; \tau, \tilde{\tau}) & =
    \sum_{\sigma \in \frS_n}
    U (e_{\sigma (1)}, \dotsc, e_{\sigma (n)}; \tau, \tilde{\tau}) \cdot
    e_{\sigma (1)} * \cdots * e_{\sigma (n)} \ , \\
    W^\sd (I; \tau, \tilde{\tau}) & =
    \sum_{x \in P_I}
    U^\sd (x_1, \dotsc, x_n; \tau, \tilde{\tau}) \cdot
    x_1 * \cdots * x_n
\end{align}
in the algebra $A_I$.

\begin{theorem}
    \label{thm-comb-non-triv}
    We have
    \begin{align}
        \label{eq-thm-comb-non-triv}
        W (I; \tau, \tilde{\tau}) & \in L_I \ , \\
        \label{eq-thm-comb-non-triv-sd}
        W^\sd (I; \tau, \tilde{\tau}) & \in U (L_\sI^+) \ .
    \end{align}
\end{theorem}

\begin{proof}
    Applying \cref{eq-comb-u-comp} and~\cref{eq-comb-usd-comp} with $\hat{\tau} = 0$,
    we may rewrite
    \begin{align}
        W (I; \tau, \tilde{\tau})
        & =
        \sum_{\leftsubstack[5em]{ \\[-2ex]
            & m \geq 1, \ I = J_1 \sqcup \cdots \sqcup J_m \, \colon \\[-.5ex]
            & J_i \neq \varnothing \text{ for all } i. \\[-.5ex]
            & \textstyle \text{Write } y_i = \sum_{j \in J_i} e_j
        }} U (y_1, \dotsc, y_m; 0, \tilde{\tau}) \cdot
        T (J_1; \tau) * \cdots * T (J_m; \tau) \ , \\
        W^\sd (I; \tau, \tilde{\tau})
        & =
        \sum_{\leftsubstack[5em]{ \\[-2ex]
            & m \geq 0, \ I = J_1 \sqcup \cdots \sqcup J_m \sqcup J', \ 
            x^i \in P_{\smash{J_i}} \colon \\[-.5ex]
            & J_i \neq \varnothing \text{ for all } i. \\[-.5ex]
            & \textstyle \text{Write } y_i = \sum_{j \in J_i} x_j
        }} U^\sd (y_1, \dotsc, y_m; 0, \tilde{\tau}) \cdot
        T (x^1; \tau) * \cdots * T (x^m; \tau) * T^\sd (J'; \tau) \ .
        \raisetag{3ex}
    \end{align}
    Reasoning as in the proof of \cref{thm-comb-from-triv},
    we can deduce \cref{eq-thm-comb-non-triv} and~\cref{eq-thm-comb-non-triv-sd}
    from \cref{thm-comb-to-triv,thm-comb-from-triv}.
    Indeed, we no longer need to use induction,
    and instead of proving that some of the terms lie in $L_I$ or $U (L_\sI^+)$,
    the argument now shows that all the terms are in $L_I$ or $U (L_\sI^+)$.
\end{proof}

Now, the first part of \cref{thm-comb}
already follows from~\cref{eq-thm-comb-non-triv-sd}
in \cref{thm-comb-non-triv}.
Finally, we prove the second part \cref{eq-comb-main} of \cref{thm-comb}.

\begin{numproof}[Proof of \cref{thm-comb}]
    Let $L_\tau \subset \SF (\calM)$ be the involutive Lie subalgebra
    generated by $\epsilon_{\alpha} (\tau)$ for all $\alpha \in C (\calA)$,
    and let $M_\tau \subset \SF (\calM^\sd)$ be the sub-$L_\tau^+$-module generated by
    $\epsilon^\sd_{\smash{\rho}} (\tau)$ for all $\rho \in C^\sd (\calA)$.
    
    We may rewrite~\cref{eq-wcf-epsilon-sd} as
    \begin{multline}
        \label{eq-wcf-epsilon-sd-rewr}
        \epsilon^\sd_\theta (\tilde{\tau})
        =
        \sum_{ \nicesubstack{
            n \geq 0 \\[-.5ex]
            \text{Write } I = \{ 1, \dotsc, n \}
        } }
        \frac{1}{2^n n!} \cdot
        \sum_{ \nicesubstack{
            \kappa \colon I \to C (\calA), \ \rho \in C^\sd (\calA) \colon \\[-.5ex]
            \textstyle \theta = \sum_{i \in I} \overline{\kappa (i)} + \rho 
        } } \\
        \Biggl[ \ 
            \sum_{ \nicesubstack{
                x \in P_I \\[-.5ex]
                \text{Write } \alpha_i = \kappa (j) \text{ if } x_i = \smash{e_j} \, , \\[-.5ex]
                \text{or write } \alpha_i = \kappa (j)^\vee \text{ if } x_i = \smash{e_j^\vee}
            } } {} \hspace{-1em}
            U^\sd (\alpha_1, \dotsc, \alpha_n; \tau, \tilde{\tau}) \cdot {}
            \epsilon_{\alpha_1} (\tau) \diamond \cdots \diamond
            \epsilon_{\alpha_n} (\tau) \diamond
            \epsilon^\sd_{\rho} (\tau)
        \Biggr] \ ,
        \raisetag{3ex}
    \end{multline}
    since every term in~\cref{eq-wcf-epsilon-sd}
    appears $2^n n!$ times in~\cref{eq-wcf-epsilon-sd-rewr}.

    Now, every sum in the square brackets in~\cref{eq-wcf-epsilon-sd-rewr}
    lies in $M_\tau^+$.
    This is because we can define an involutive algebra homomorphism
    $\varphi \colon A_I \to \SF (\calM)$ by sending $e_i$ to
    $\epsilon_{\smash{\kappa (i)}} (\tau)$ and
    $e_i^\vee$ to $\epsilon_{\smash{\kappa (i)^\vee}} (\tau)$.
    We have $\varphi (L_I) \subset L_\tau$\,,
    so $\varphi (L_\sI^+) \subset L_\tau^+$\,.
    The sum in the square brackets is
    $\varphi (W^\sd (I; \tau, \tilde{\tau})) \diamond \epsilon^\sd_{\smash{\rho}} (\tau)$,
    so by \cref{thm-comb-non-triv},
    it lies in $U (L_\tau^+) \diamond \epsilon^\sd_{\smash{\rho}} (\tau) \subset M_\tau$\,.

    Moreover, the above also shows that every sum in the square brackets
    can be written as a sum of terms of the form in~\cref{eq-comb-main},
    with the $\tilde{U} ({\cdots})$ coefficients.
    Each of these new terms appears $2^n n!$ times in~\cref{eq-wcf-epsilon-sd-rewr}.
    This proves~\cref{eq-comb-main}.
\end{numproof}

\subsection{Some combinatorial identities}

\paragraph{}

We prove some combinatorial identities that were used in the arguments above.
It is interesting that by working with wall-crossing structures,
we are able to write down several
combinatorial identities involving the Bernoulli numbers,
namely, \cref{lem-comb-bernoulli,lem-comb-bernoulli-2,lem-comb-bernoulli-3},
and it is unclear whether there are deeper reasons why these identities are true.

\begin{lemma}
    For any integers $i, k, n$ such that
    $1 \leq i \leq n$ and $1 \leq k \leq n-1$, we have
    \begin{equation}
        \label{eq-comb-alt-binom}
        \sum_{\nicesubstack{
            q \colon 0 \leq q \leq n-i, \\[-1.5ex]
            0 \leq k-q \leq i-1
        }}
        (-1)^{q} \, \binom{k}{\, q \,}
        = 
        (-1)^{i+k-1} \, \binom{k-1}{i-1} +
        (-1)^{n-i} \, \binom{k-1}{n-i} \ .
    \end{equation}
\end{lemma}

\begin{proof}
    We have
    \begin{align*}
        \text{l.h.s.} & =
        \sum_{q=k-i+1}^{n-i}
        (-1)^{q} \, \biggl[ \binom{k-1}{q-1} + \binom{k-1}{q} \biggr] \\
        & =
        (-1)^{k-i+1} \, \binom{k-1}{k-i} +
        (-1)^{n-i} \, \binom{k-1}{n-i} \\
        & = \text{r.h.s.}
    \end{align*}
\end{proof}

\begin{lemma}
    \label{lem-comb-bernoulli}
    For any integers $1 \leq i \leq n$, we have
    \begin{equation}
        \label{eq-comb-bernoulli}
        \sum_{p=0}^{i-1} {}
        \sum_{q=0}^{n-i} {}
        \frac{(-1)^{q} \, (n-1)!}{(n-p-q)! \, p! \, q!} \,
        B_{p+q} (x)
        =
        \binom{n-1}{i-1} \,
        x^{i-1} \, (1-x)^{n-i} \ ,
    \end{equation}
    where $B_k (x)$ denotes the $k$-th
    Bernoulli polynomial.
\end{lemma}

\begin{proof}
    Let $l (x)$ and $r (x)$ denote the left and right sides
    of~\cref{eq-comb-bernoulli}, respectively.
    By \cref{eq-comb-alt-binom}, we have
    \begin{align*}
        l (x) & =
        \frac{1}{n} \, B_0 (x) + 
        \sum_{k=1}^{n-1}
        \frac{(n-1)!}{(n-k)! \, k!} \, \biggl[
            (-1)^{i+k-1} \, \binom{k-1}{i-1} +
            (-1)^{n-i} \, \binom{k-1}{n-i}
        \biggr] \, B_k (x) \\
        & =
        \frac{1}{n} + 
        \binom{n-1}{i-1} \cdot \biggl[
            \sum_{k=i}^{n-1} {}
            \frac{(-1)^{i+k-1}}{k} \, \binom{n-i}{n-k} \,
            B_k (x) +
            \sum_{k=n-i+1}^{n-1} {}
            \frac{(-1)^{n-i}}{k} \, \binom{i-1}{n-k} \,
            B_k (x)
        \biggr] \ .
    \end{align*}
    Since $B_k (x + 1) - B_k (x) = k \, x^{k-1}$, we have
    \begin{equation*}
        l (x + 1) - l (x) =
        \binom{n-1}{i-1} \cdot \biggl[
            \sum_{k=i}^{n-1} {}
            (-1)^{i} \, \binom{n-i}{n-k} \,
            (-x)^{k-1} +
            \sum_{k=n-i+1}^{n-1} {}
            (-1)^{n-i} \, \binom{i-1}{n-k} \,
            x^{k-1}
        \biggr] \ .
    \end{equation*}
    On the other hand,
    \begin{align*}
        r (x + 1) - r (x) & =
        \binom{n-1}{i-1} \cdot \bigl[
            (x+1)^{i-1} \, (-x)^{n-i} -
            x^{i-1} \, (1-x)^{n-i}
        \bigr] \\
        & =
        \binom{n-1}{i-1} \cdot \biggl[
            \sum_{k=i}^{n-1} {}
            (-1)^{i} \, \binom{n-i}{n-k} \,
            (-x)^{k-1} +
            \sum_{k=n-i+1}^{n-1} {}
            (-1)^{n-i} \, \binom{i-1}{n-k} \,
            x^{k-1}
        \biggr] \ .
    \end{align*}
    Therefore, $l (x+1) - l (x) = r (x+1) - r (x)$,
    which means that
    \[
        l (x) - r (x) = c
    \]
    for some constant $c$.
    To show that $c = 0$, we use the fact that
    \[
        \int_0^1 B_k (x) \, d x = \begin{cases}
            1, & k = 0 \ , \\
            0, & k > 0 \ ,
        \end{cases}
    \]
    so
    \[
        \int_0^1 l (x) \, d x = \frac{1}{n} \ .
    \]
    On the other hand,
    \begin{align*}
        \int_0^1 r (x) \, d x & =
        \binom{n-1}{i-1} \cdot \mathrm{B} (i, n-i+1) \\
        & = \binom{n-1}{i-1} \cdot \frac{(i-1)! \, (n-i)!}{n!} \\
        & = \frac{1}{n} \ ,
    \end{align*}
    where $\mathrm{B}$ denotes the beta function.
    This shows that $c = 0$.
\end{proof}

\begin{lemma}
    \label{lem-comb-bernoulli-2}
    For any integers $1 \leq i \leq n$,
    \begin{multline}
        \label{eq-comb-bernoulli-2}
        \sum_{p=0}^{i-1} {}
        \sum_{q=0}^{n-i} {}
        \frac{(-1)^{q} \, 2^{p+q-1} \, (n-1)!}{(n-p-q)! \, p! \, q!} \,
        B_{p+q} \Bigl( \frac{x}{2} \Bigr) \\
        \shoveleft{ \hspace{1em} {} +
        \sum_{k=n-i+1}^{n} {}
        \frac{(-1)^{n-i+1} \, 2^{k-1} \, (n-1)!}{k \, (n-k)! \, (n-i)! \, (i+k-n-1)!} \,
        \Bigl[ B_k \Bigl( \frac{x}{2} \Bigr) 
        - B_k \Bigl( \frac{x+1}{2} \Bigr) \Bigr] } \\
        = 
        \binom{n-1}{i-1} \,
        x^{i-1} \, (1-x)^{n-i} \ .
    \end{multline}
\end{lemma}

\begin{proof}
    \allowdisplaybreaks
    Let $l_1 (x)$, $l_2 (x)$ and $r (x)$ denote
    the first and second terms on the left-hand side
    of~\cref{eq-comb-bernoulli-2},
    and the right-hand side, respectively.
    
    By \cref{eq-comb-alt-binom}, we have
    \begin{align*}
        l_1 (x) & =
        \frac{1}{2n} \, B_0 \Bigl( \frac{x}{2} \Bigr) + 
        \sum_{k=1}^{n-1}
        \frac{2^{k-1} \, (n-1)!}{(n-k)! \, k!} \, \biggl[
            (-1)^{i+k-1} \, \binom{k-1}{i-1} +
            (-1)^{n-i} \, \binom{k-1}{n-i}
        \biggr] \, B_k \Bigl( \frac{x}{2} \Bigr) \\
        & =
        \frac{1}{2n} + 
        \binom{n-1}{i-1} \cdot \biggl[
            \sum_{k=i}^{n-1} {}
            \frac{(-1)^{i+k-1} \, 2^{k-1}}{k} \, \binom{n-i}{n-k} \,
            B_k \Bigl( \frac{x}{2} \Bigr) + {} \\*
            && \mathllap{ 
            \sum_{k=n-i+1}^{n-1} {}
            \frac{(-1)^{n-i} \, 2^{k-1}}{k} \, \binom{i-1}{n-k} \,
            B_k \Bigl( \frac{x}{2} \Bigr)
        \biggr] \ . }
    \end{align*}
    Proceeding as in the proof of \cref{lem-comb-bernoulli},
    we see that
    \begin{multline*}
        l_1 (x+2) - l_1 (x) = {} \\
        \binom{n-1}{i-1} \cdot \biggl[
            \sum_{k=i}^{n-1} {}
            (-1)^{i} \, \binom{n-i}{n-k} \,
            (-x)^{k-1} + 
            \sum_{k=n-i+1}^{n-1} {}
            (-1)^{n-i} \, \binom{i-1}{n-k} \,
            x^{k-1}
        \biggr] \ .
    \end{multline*}
    Similarly,
    \begin{align*}
        l_2 (x+2) - l_2 (x) & =
        \sum_{k=n-i+1}^{n} {}
        \frac{(-1)^{n-i+1} \, (n-1)!}{(n-k)! \, (n-i)! \, (i+k-n-1)!} \,
        \bigl[ x^{k-1} - (x+1)^{k-1} \bigr] \\
        & =
        (-1)^{n-i+1} \,
        \binom{n-1}{i-1} \cdot 
        \sum_{k=n-i+1}^{n} {}
        \binom{i-1}{n-k} \, \bigl[ x^{k-1} - (x+1)^{k-1} \bigr] \ .
    \end{align*}
    Setting $l (x) = l_1 (x) + l_2 (x)$, we see that
    \begin{align*}
        & l (x+2) - l (x) \\
        & \hspace{1em} =
        \binom{n-1}{i-1} \cdot \biggl[
            \sum_{k=i}^{n-1} {}
            (-1)^{i} \, \binom{n-i}{n-k} \,
            (-x)^{k-1} + 
            \sum_{k=n-i+1}^{n-1} {}
            (-1)^{n-i} \, \binom{i-1}{n-k} \,
            (x+1)^{k-1} \\*[-.5ex]
            && \mathllap{
            {} + (-1)^{n-i+1} \,
        \bigl[ x^{n-1} - (x+1)^{n-1} \bigr]
        \biggr] } \\[1ex]
        & \hspace{1em} =
        \binom{n-1}{i-1} \cdot \biggl[
            \sum_{k=i}^{n} {}
            (-1)^{i} \, \binom{n-i}{n-k} \,
            (-x)^{k-1} + 
            \sum_{k=n-i+1}^{n} {}
            (-1)^{n-i} \, \binom{i-1}{n-k} \,
            (x+1)^{k-1}
        \biggr] \ .
    \end{align*}
    On the other hand,
    \begin{align*}
        r (x + 2) - r (x) & =
        \binom{n-1}{i-1} \cdot \bigl[
            \bigl( (x+1)+1 \bigr)^{i-1} \, 
            \bigl( -(x+1) \bigr)^{n-i} -
            x^{i-1} \, (1-x)^{n-i}
        \bigr] \\
        & = l (x+2) - l (x) \ .
    \end{align*}
    This means that
    \[
        l (x) - r (x) = c
    \]
    for some constant $c$.
    To show that $c = 0$, we use the facts that
    \begin{align*}
        \int_0^2 B_k \Bigl( \frac{x}{2} \Bigr) \, d x & =
        \begin{cases}
            2, & k = 0 \ , \\
            0, & k > 0 \ ,
        \end{cases} \\
        \int_0^2 B_k \Bigl( \frac{x+1}{2} \Bigr) \, d x & =
        \frac{1}{2^{k-1}} \ ,
    \end{align*}
    so
    \begin{align*}
        \int_0^2 l (x) \, d x & =
        \frac{1}{n} +
        \sum_{k=n-i+1}^{n} {}
        \frac{(-1)^{n-i} \, (n-1)!}{k \, (n-k)! \, (n-i)! \, (i+k-n-1)!} \ .
    \end{align*}
    On the other hand, we have seen
    in the proof of \cref{lem-comb-bernoulli} that
    \[
        \int_0^1 r (x) \, d x = \frac{1}{n} \ .
    \]
    We then calculate
    \begin{align*}
        \int_1^2 r (x) \, d x & =
        \binom{n-1}{i-1} \cdot \int_0^1 (-x)^{n-i} (1+x)^{i-1} \, d x \\
        & = \binom{n-1}{i-1} \cdot
        \int_0^1 {} \biggl[
        \sum_{k=n-i+1}^{n} {} (-1)^{n-i} \,
        \binom{i-1}{n-k} \, x^{k-1} \biggr] \, d x \\
        & =
        \sum_{k=n-i+1}^{n} {}
        \frac{(-1)^{n-i} \, (n-1)!}{k \, (n-k)! \, (n-i)! \, (i+k-n-1)!} \ .
    \end{align*}
    This shows that $c = 0$ and we are done.
\end{proof}

\begin{lemma}
    \label{lem-comb-bernoulli-3}
    For any integers $1 \leq i \leq n$,
    \begin{multline}
        \label{eq-comb-bernoulli-3}
        \sum_{p=0}^{i-1} {}
        \sum_{q=0}^{n-i} {}
        \frac{(-1)^{p} \, 2^{p+q-1} \, (n-1)!}{(n-p-q)! \, p! \, q!} \,
        B_{p+q} \Bigl( \frac{x}{2} \Bigr) \\
        {} +
        \sum_{k=n-i+1}^{n} {}
        \frac{(-1)^{i+k-n-1} \, 2^{k-1} \, (n-1)!}{k \, (n-k)! \, (n-i)! \, (i+k-n-1)!} \,
        \Bigl[ B_k \Bigl( \frac{x}{2} \Bigr) 
        - B_k \Bigl( \frac{x+1}{2} \Bigr) \Bigr] = 0 \ .
    \end{multline}
\end{lemma}

\begin{proof}
    Let $l (x)$ denote the left-hand side.
    From a similar calculation as in the previous lemma, we have
    \begin{align*}
        & l (x+2) - l (x) \\
        & \hspace{1em} = \binom{n-1}{i-1} \cdot \biggl[
            \sum_{k=i}^{n} {}
            (-1)^{i-1} \, \binom{n-i}{n-k} \,
            x^{k-1} + 
            \sum_{k=n-i+1}^{n} {}
            (-1)^{n-i+1} \, \binom{i-1}{n-k} \,
            (-x-1)^{k-1}
        \biggr]. \\
        & \hspace{1em} = \binom{n-1}{i-1} \cdot \bigl[
            (-x)^{i-1} \, (1 + x)^{n-i} -
            \bigl( (-x - 1) + 1 \bigr)^{i-1} \, (x+1)^{n-i}
        \bigr] \\
        & \hspace{1em} = 0 \ .
    \end{align*}
    Therefore, $l (x)$ is constant.
    Again, calculating as before,
    \begin{align*}
        \int_0^2 l (x) \, d x
        & = \frac{1}{n} +
        \sum_{k=n-i+1}^{n} {}
        \frac{(-1)^{i+k-n} \, (n-1)!}{k \, (n-k)! \, (n-i)! \, (i+k-n-1)!} \\
        & = \frac{1}{n} - \binom{n-1}{n-i} \cdot
        \int_0^1 x^{n-i} \, (1-x)^{i-1} \, d x \\
        & = \frac{1}{n} - \frac{1}{n} = 0 \ ,
    \end{align*}
    where the second integral was evaluated as in
    the proof of \cref{lem-comb-bernoulli}.
    This shows that $l (x) \equiv 0$.
\end{proof}

%% file: no-pole.tex
\label{sect-proof-no-pole}

In this \lcnamecref{sect-proof-no-pole},
we prove one of the main results of this paper,
\cref{thm-no-pole-sd},
which implies that the numerical enumerative invariants
$\chiJ^\sd_{\theta} (\tau)$, $\DT^\sdnai_\theta (\tau)$,
and the Donaldson--Thomas invariants
$\DT^\sd_\theta (\tau)$ in \cref{sect-quiv-dt,sect-threefolds},
are well-defined.

\subsection{Preparatory results}

\begin{lemma}
    \label{lem-max-torus}
    Let $A$ be a finite-dimensional $\bbK$-algebra,
    and $T \simeq (\Gm)^n \subset A^\times$ a maximal torus.
    Then the closure $S \subset A$ of\/~$T$ in~$A$
    is a $\bbK$-subalgebra isomorphic to $\bbK^n$, and $S^\times = T$. 
\end{lemma}

\begin{proof}
    Consider the embedding of $\bbK$-algebras
    $i \colon A \hookrightarrow \End_{\bbK} (A)$,
    which restricts to an embedding of algebraic groups
    $i \colon A^\times \hookrightarrow \GL (A)$.
    Extend $i (T)$ to a maximal torus $T' \subset \GL (A)$,
    and let $S' \subset \End_{\bbK} (A)$ be the closure of $T'$ in $\End_{\bbK} (A)$.
    Since all maximal tori are conjugate in $\GL (A)$,
    and conjugations extend to $\bbK$-linear automorphisms of $\End_{\bbK} (A)$,
    it follows that $S'$ is a $\bbK$-subalgebra of $\End_{\bbK} (A)$
    isomorphic to $\bbK^N$, where $N = \rank T' = \dim A$,
    and $S'^\times = T'$.
    Since $S = i^{-1} (S')$, we have the desired result.
\end{proof}

\paragraph{}
\label[lemma]{lem-local-decomp}

The following is an analogue of Joyce~\cite[Proposition~5.8]{Joyce2007II}
in the self-dual case,
and describes the local structure of the moduli stack of self-dual objects.

\begin{lemma*}
    Let $\calA, \calM$ be as in \tagref{SdMod},
    and let $\calS \subset \calM^\sd$ be a quasi-compact locally closed substack.
    Then there exists a finite decomposition
    $\calS = \bigcup_{i \in I} \calS_i$ into disjoint locally closed substacks
    $\calS_i$ of the form
    \begin{equation}
        \label{eq-local-decomp}
        \calS_i \simeq [U_i / A_i^\iso] \ ,
    \end{equation}
    with $A_i$ a finite-dimensional involutive $\bbK$-algebra,
    $A_i^\iso$ as in \cref{para-inv-alg},
    and $U_i$ an affine $\bbK$-scheme acted on by $A_i^\iso$,
    such that for each $u \in U_i (\bbK)$ mapping to $(E, \phi) \in \calM^\sd (\bbK)$,
    there exists an involutive subalgebra $B_u \subset A_i$
    with $B_u^\iso = \Stab_{A_i^\iso} (u)$,
    such that the induced map $\Stab_{A_i^\iso} (u) \simto \Aut (E, \phi)$
    extends to an isomorphism of involutive $\bbK$-algebras
    $B_u \simto \End (E)$.
\end{lemma*}

\begin{proof}
    \fixlineheight
    As in the proof of
    Kresch~\cite[Proposition~3.5.9]{Kresch1999},
    used in the proof of \cite[Proposition~5.8]{Joyce2007II},
    there is a decomposition $\calS = \bigcup_{i \in I} \calS_i$,
    such that each $\calS_i$ is a gerbe over a 
    connected, reduced $\bbK$-scheme $S_i$,
    with a finite flat morphism $T_i \to S_i$ of reduced $\bbK$-schemes,
    such that $\calT_i = \calS_i \times_{S_i} T_i$ as a gerbe over $T_i$
    is isomorphic to $\upB G_i$
    for a flat group scheme $G_i \to T_i$.
    Fibres of $G_i$ over $\bbK$-points of $T_i$
    are automorphism groups of self-dual objects in~$\calA$.

    \unfixlineheight
    Using the self-dual stack structure on $\+{\calM}$,
    the argument above can be strengthened to give
    an involutive algebra bundle $E_i \to T_i$,
    whose fibres over $\bbK$-points of $T_i$
    give algebras of endomorphisms of objects in $\calA$,
    with the involution given by self-dual structures on these objects.

    \fixlineheight
    We claim that $E_i$ has a self-dual faithful representation
    in a self-dual vector bundle $V_i \to T_i$,
    which can be chosen to be either orthogonal or symplectic,
    such that the involution of $E_i$
    is compatible with the self-dual structure of $V_i$.
    Indeed, one may take $V_i = E_i \oplus E_i^\vee$,
    with the self-dual structure swapping $E_i$ and $E_i^\vee$,
    with $E_i$ acting on $E_i$ by left multiplication,
    and on $E_i^\vee$ by $(e, \eta) \mapsto (e^\vee \cdot (-))^* (\eta)$,
    where $(-)^*$ denotes the linear dual.

    Now choose $V_i$ symplectic as above,
    and let $\calV_i = V_i \times_{T_i} \calT_i$,
    as a symplectic vector bundle over $\calT_i$.
    Let $\calW_i \to \calS_i$ be the symplectic vector bundle
    which is the pushforward of $\calV_i$
    along the finite flat morphism $\calT_i \to \calS_i$,
    and let $U_i \to \calS_i$ be the associated
    principal $\Sp (n_i)$-bundle of $\calW_i$,
    where $n_i = \rank \calW_i$.
    Then $U_i$ is an algebraic space acted on by $\Sp (n_i)$,
    and $\calS_i \simeq [U_i / \Sp (n_i)]$.
    As in the proof of \cite[Lemma~3.5.1]{Kresch1999},
    shrinking $U_i$ if necessary,
    we may assume that it is quasi-projective.
    (This uses the fact that $\Sp (n_i)$ is connected,
    which is why we chose $V_i$ symplectic instead of orthogonal.)

    We then choose $A_i = \GL (n_i)$,
    with an involution given by a symplectic structure on $\bbK^{n_i}$,
    so that $A_i^\iso = \Sp (n_i)$.
    Arguing as in the proof of \cite[Proposition~5.8]{Joyce2007II}
    verifies the compatibility conditions on stabilizer groups.
\end{proof}

\subsection{Virtual rank projections}

In this section, we assume the conditions
\tagref{Spl}, \tagref{SdMod}, \tagref{Fin}, and \tagref{Stab}
for a fixed self-dual weak stability condition~$\tau$.
We study the virtual rank projections,
as in \cref{sect-vrp},
of the stack functions $\delta_\alpha (\tau)$ and $\delta^\sd_\theta (\tau)$.
The main results are \cref{thm-vrp,thm-vrp-higher}.

\begin{definition}
    Define operations
    \begin{alignat}{3}
        \circledast & = \oplus_! 
        & \colon \ && \SF (\calM) \otimes \SF (\calM) & \longrightarrow \SF (\calM) \ , \\
        \circleddiamond & = \oplus^\sd_! 
        & \colon \ && \SF (\calM) \otimes \SF (\calM^\sd) & \longrightarrow \SF (\calM^\sd) \ ,
    \end{alignat}
    where $\oplus$ and $\oplus^\sd$ are maps defined in
    \cref{eq-def-oplus-moduli,eq-def-oplus-sd-moduli}.
    This makes $\SF (\calM)$ into a commutative algebra,
    and $\SF (\calM^\sd)$ into a module over this algebra.
\end{definition}

\begin{definition}
    \allowdisplaybreaks
    For classes $\alpha \in C^\circ (\calA)$ and $\theta \in C^\sd (\calA)$,
    define elements
    $\sigma_\alpha (\tau) \in \SF (\calM^\ss_\alpha (\tau))$ and
    $\sigma^\sd_\theta (\tau) \in \SF (\calM^\sdss_\theta (\tau))$ by
    \begin{align}
        \label{eq-def-sigma}
        \sigma_\alpha (\tau) & = \sum_{ \leftsubstack[6em]{
            \\[-1.5ex]
            & n > 0; \, \alpha_1, \dotsc, \alpha_n \in C (\calA) \colon \\[-.5ex]
            & \alpha = \alpha_1 + \cdots + \alpha_n \, , \\[-.5ex]
            & \tau (\alpha_1) = \cdots = \tau (\alpha_n)
        } } 
        \frac{(-1)^{n-1}}{n} \cdot
        \delta_{\alpha_1} (\tau) \circledast \cdots \circledast \delta_{\alpha_n} (\tau) \ ,
        \\
        \label{eq-def-sigma-sd}
        \sigma^\sd_\theta (\tau) & = \sum_{ \leftsubstack[6em]{
            \\[-1.5ex]
            & n \geq 0; \, \alpha_1, \dotsc, \alpha_n \in C (\calA), \,
            \rho \in C^\sd (\calA) \colon \\[-.5ex]
            & \theta = \bar{\alpha}_1 + \cdots + \bar{\alpha}_n + \rho \\[-.5ex]
            & \tau (\alpha_1) = \cdots = \tau (\alpha_n) = 0
        } } {}
        \binom{-1/2}{n} \cdot
        \delta_{\alpha_1} (\tau) \circleddiamond \cdots \circleddiamond \delta_{\alpha_n} (\tau) \circleddiamond \delta^\sd_\rho (\tau) \ .
    \end{align}
    Compare with~\cref{eq-def-epsilon,eq-def-epsilon-sd}.
    The element $\sigma_\alpha (\tau)$ was denoted by $\bar{\delta}_{\mathrm{si}}^\alpha (\tau)$
    in~\cite[Definition~8.1]{Joyce2007III}.
    More concisely, \crefrange{eq-def-sigma}{eq-def-sigma-sd}
    can be written as
    \begin{align}
        \sigma (\tau; t) & = \log_\circledast \delta (\tau; t) \ , \\
        \sigma^\sd (\tau) & = \delta (\tau; 0)^{\circledast (-1/2)} \circleddiamond \delta^\sd (\tau) \ ,
    \end{align}
    with notations analogous to those
    in~\cref{eq-epsilon-delta-exp-1,eq-def-epsilon-sd-compact}.
\end{definition}

\begin{theorem}
    \label{thm-vrp}
    For any $\alpha \in C^\circ (\calA)$ and $\theta \in C^\sd (\calA)$,
    we have
    \begin{align}
        \label{eq-thm-vrp}
        \vrp{1} (\delta_\alpha (\tau)) & = \sigma_\alpha (\tau) \ , \\
        \label{eq-thm-vrp-sd}
        \vrp{0} (\delta^\sd_\theta (\tau)) & = \sigma^\sd_\theta (\tau) \ .
    \end{align}
\end{theorem}

\begin{proof}
    Equation~\cref{eq-thm-vrp} was proved in~\cite[Theorem~8.6]{Joyce2007III}.
    Although the cited result is stated for abelian categories,
    the argument works for $\bbK$-linear exact categories~$\calA$
    satisfying the condition in \cref{lem-qa-splitting}.

    To prove~\cref{eq-thm-vrp-sd},
    we use the definition of $\vrp{0}$ from
    \cite[\S5.2]{Joyce2007Stack}.

    Choose a decomposition $\calM^\sdss_\theta (\tau) = \bigcup_{i \in I} \calS_i$
    into locally closed substacks
    $\calS_i \simeq [U_i / A_i^\iso]$ as in~\cref{lem-local-decomp},
    so that $\delta^\sd_\theta (\tau) = \sum_{i \in I} {} [\calS_i]$.
    Write $G_i = A_i^\iso$,
    and choose a maximal torus $T_i \subset G_i$.
    Extending $T_i$ to a maximal torus in $A_i^\times$,
    and taking its closure in $A_i$,
    then intersecting with its dual,
    we obtain an involutive subalgebra $S_i \subset A_i$
    isomorphic to $\bbK^I$ for a finite set $I$,
    with involution given by an involution of $I$,
    and such that $S_i^\iso = T_i$,
    as can be deduced from \cref{lem-max-torus}.

    Consider the sets
    \begin{align}
        \label{eq-def-pqr-p}
        \calP (U_i, T_i) & = \{
            \mathrm{Stab}_{T_i} (Z) \mid
            Z \subset U_i \text{ subscheme}
        \} \ , \\
        \label{eq-def-pqr-q}
        \calQ (G_i, T_i) & = \{
            T_i \cap \upC (\upC_{G_i} (S)) \mid
            S \subset T_i \text{ closed subgroup}
        \} \ , \\
        \label{eq-def-pqr-r}
        \calR (U_i, G_i, T_i) & = \{
            P \cap Q \mid
            P \in \calP (U_i, T_i), \,
            Q \in \calQ (G_i, T_i)
        \}
    \end{align}
    from \cite[Definitions~5.3, 5.5, and~5.8]{Joyce2007Stack}.
    They are finite sets of closed subgroups of $T$,
    closed under intersection.
    The sets \cref{eq-def-pqr-p,eq-def-pqr-q,eq-def-pqr-r}
    are contained in the set
    \begin{equation}
        \calT_i = \{
            B^\iso \mid B \subset S_i \text{ involutive subalgebra}
        \} \ ,
    \end{equation}
    by the condition on stabilizer groups in~\cref{lem-local-decomp},
    and the fact that $\smash{\upC_{G_i} (S)} = 
    \smash{\bigl( \upC_{A_i} (S)} \cap \smash{\upC_{A_i} (S)^\vee \bigr)^\iso}$
    in the case of \cref{eq-def-pqr-q}.
    Using the fact that we can replace the sets
    \cref{eq-def-pqr-p,eq-def-pqr-q,eq-def-pqr-r} by
    larger sets of subgroups of $T_i$
    when defining the projections $\vrp{k}$,
    as in the proof of \cite[Theorem~8.6]{Joyce2007III},
    we now replace them by the set $\calT_i$,
    giving, by \cite[Definition~5.10]{Joyce2007Stack},
    \begin{align}
        \label{eq-vrp0-1}
        \vrp{0} ([\calS_i]) & = \sum_{\leftsubstack[4em]{
            & P, Q \in \calT_i \colon \\[-.5ex]
            & M_0 (P, Q) \neq 0
        } }
        M_0 (P, Q) \cdot [ U_i^P / \upC_{G_i} (Q) ] \ ,
    \end{align}
    where $M_0 (P, Q)$ is a rational coefficient,
    $U_i^P$ denotes the fixed locus of $P$ in $U_i$,
    and we write stacks themselves for the stack functions they represent.
    The coefficient $M_0 (P, Q)$ is given by
    \begin{align}
        \label{eq-vrp0-2}
        M_0 (P, Q) & =
        \frac{1}{|W_Q|} \cdot
        \sum_{\leftsubstack[8em]{
            & P = P_0 \supsetneq P_1 \supsetneq \cdots \supsetneq P_m \text{ in } \calT_i \> , \\[-.5ex]
            & Q = Q_0 \supsetneq Q_1 \supsetneq \cdots \supsetneq Q_n \text{ in } \calT_i \colon \\[-.5ex]
            & \dim (P_m \cap Q_n) = 0
        } } {}
        (-1)^{m+n} \ ,
    \end{align}
    where $W_Q = \upN_{G_i} (T_i) \big/ \bigl( \upC_{G_i} (Q) \cap \upN_{G_i} (T_i) \bigr)$.
    As in the proof of \cite[Theorem~8.6]{Joyce2007III},
    we have $M_0 (P, Q) = 0$ unless $P = Q$,
    so we may assume that $P = Q = B^\iso$
    for an involutive subalgebra $B \subset S_i$,
    isomorphic to $\bbK^J$ for a finite set $J$,
    with involution given by a $\bbZ_2$ action on $J$,
    given by $(-)^\vee \colon J \simto J$.
    Writing $s = |J^{\smash{\bbZ_2}}|$ and $r = (|J| - s)/2$, 
    the sum in \cref{eq-vrp0-2} can be shown to be equal to
    \begin{align}
        \label{eq-vrp0-3}
        \sum_{\leftsubstack[8em]{
            & P = P_0 \supsetneq P_1 \supsetneq \cdots \supsetneq P_m \text{ in } \calT_i \colon \\[-.5ex]
            & \dim (P_m) = 0
        } } {}
        (-1)^{m} =
        \begin{cases}
            (-1)^r \, (2r - 1)!! & \text{if } s = 0 \text{ or } 1 \, , \\
            0 & \text{otherwise} \, .
        \end{cases}
    \end{align}
    Indeed, a standard argument shows that
    all terms in~\cref{eq-vrp0-2} with $n > 0$ cancel out,
    giving the left-hand side of~\cref{eq-vrp0-3}.
    The detailed proof of~\cref{eq-vrp0-3} is postponed
    to~\cref{sect-vrp-comb}.

    Let $\bbK^{r,s}$ denote the involutive $\bbK$-algebra $\bbK^J$,
    with $|J| = 2r + s$ and a $\bbZ_2$-action on~$J$
    fixing $s$~elements, as above.
    We can now rewrite~\cref{eq-vrp0-1} as
    \begin{align}
        \label{eq-vrp0-4}
        \vrp{0} (\calS_i) & = \sum_{\leftsubstack[6em]{
            & r \geq 0, \, s \in \{0, 1\}; \\[-.5ex]
            & q \colon \bbK^{r, s} \hookrightarrow S_i \text{ involutive subalgebra} . \\[-.5ex]
            & \text{Write } Q = q (\bbK^{r, s})^\iso
        } }
        \frac{(-1)^r \, (2r - 1)!!}{2^r r!} \cdot
        \frac{1}{|W_Q|} \cdot
        [ U_i^Q / \upC_{G_i} (Q) ] \ ,
    \end{align}
    where the factor $2^r r!$ is the number of maps~$q$
    giving the same~$Q$.
    The Weyl group $W_i$ of $G_i$ acts on $S_i$ by conjugation,
    and the orbit of $q$ is isomorphic to $W_Q$.
    
    Each term in the sum~\cref{eq-vrp0-4}, as a stack function,
    parametrizes splittings of the form
    $(E, \phi) \simeq \bar{E}_1 \oplus \cdots \oplus \bar{E}_r \oplus (E_0, \phi_0)$,
    with $E_i \in \calA$ non-zero for $i = 1, \dotsc, r$,
    and $E_0 = 0$ if and only if $s = 0$,
    for an object $(E, \phi) \in \calA^\sd$
    corresponding to a point $u \in U_i (\bbK)$ fixed by $Q$.
    This uses the property in \cref{lem-qa-splitting}
    which was assumed in \tagref{Spl}.
    Conversely, each such splitting corresponds to
    a map $q \colon \bbK^{r, s} \hookrightarrow S_i$ up to conjugation by $W_i$,
    giving $|W_Q|$ such maps.
    Also, such $(E, \phi)$ above is semistable if and only if
    $\bar{E}_1, \dotsc, \bar{E}_r$ are semistable of slope $0$,
    and $(E_0, \phi_0)$ is semistable.
    Summing over $i$ in~\cref{eq-vrp0-4} gives
    \begin{align}
        \label{eq-vrp0-5}
        \vrp{0} (\calM^\sdss_\theta (\tau)) & = \sum_{ \leftsubstack[4em]{
            \\[-1.5ex]
            & \alpha_1, \dotsc, \alpha_r \in C (\calA), \,
            \rho \in C^\sd (\calA) \colon \\[-.5ex]
            & \theta = \bar{\alpha}_1 + \cdots + \bar{\alpha}_r + \rho \\[-.5ex]
            & \tau (\alpha_1) = \cdots = \tau (\alpha_r) = 0
        } } {}
        \frac{(-1)^r \, (2r - 1)!!}{2^r r!} \cdot
        [ \calM^\ss_{\alpha_1} (\tau) \times \cdots \times \calM^\ss_{\alpha_r} (\tau) \times \calM^\sdss_\rho (\tau) ] \ .
        \raisetag{\baselineskip}
    \end{align}
    Comparing this with~\cref{eq-def-sigma-sd}
    proves~\cref{eq-thm-vrp-sd}.
\end{proof}

\paragraph{}
\label[theorem]{thm-vrp-higher}

As a corollary, it is possible to write down
all virtual rank projections of
$\delta_\alpha (\tau)$ and $\delta^\sd_\theta (\tau)$.

\begin{theorem*}
    Let $\alpha \in C (\calA)$ and $\theta \in C^\sd (\calA)$, and let $n \geq 0$.
    Then
    \begin{alignat}{2}
        \label{eq-thm-vrp-higher}
        \vrp{n} (\delta_\alpha (\tau)) & = {} &
        \frac{1}{n!} \cdot {} &
        \sum_{ \leftsubstack[6em]{
            \\[-2ex]
            & \alpha_1, \dotsc, \alpha_n \in C (\calA) \colon \\[-.5ex]
            & \alpha = \alpha_1 + \cdots + \alpha_n \, , \\[-.5ex]
            & \tau (\alpha_1) = \cdots = \tau (\alpha_n)
        } }
        \sigma_{\alpha_1} (\tau) \circledast \cdots \circledast \sigma_{\alpha_n} (\tau) \ , 
        \\
        \label{eq-thm-vrp-higher-sd}
        \vrp{n} (\delta^\sd_\theta (\tau)) & = {} &
        \frac{1}{2^n n!} \cdot {} &
        \sum_{ \leftsubstack[6em]{
            \\[-2ex]
            & \alpha_1, \dotsc, \alpha_n \in C (\calA), \,
            \rho \in C^\sd (\calA) \colon \\[-.5ex]
            & \theta = \bar{\alpha}_1 + \cdots + \bar{\alpha}_n + \rho \\[-.5ex]
            & \tau (\alpha_1) = \cdots = \tau (\alpha_n) = 0
        } } {}
        \sigma_{\alpha_1} (\tau) \circleddiamond \cdots \circleddiamond \sigma_{\alpha_n} (\tau) \circleddiamond \sigma^\sd_\rho (\tau) \ .
    \end{alignat}
\end{theorem*}

\begin{proof}
    Formally inverting \crefrange{eq-def-sigma}{eq-def-sigma-sd},
    similarly to~\cref{eq-def-epsilon-inverse,eq-def-epsilon-sd-inverse},
    we see that the sum of the right-hand sides of
    \crefrange{eq-thm-vrp-higher}{eq-thm-vrp-higher-sd}
    over all $n \geq 0$ is equal to
    $\delta_\alpha (\tau)$ and $\delta^\sd_\theta (\tau)$,
    respectively.
    Also, the right-hand sides of
    \crefrange{eq-thm-vrp-higher}{eq-thm-vrp-higher-sd}
    have virtual rank~$n$, 
    which follows from \cref{thm-vrp} and
    \cite[Proposition~5.14]{Joyce2007Stack}.
    This proves the theorem.
\end{proof}

\subsection{A combinatorial identity}
\label{sect-vrp-comb}

We prove the identity~\cref{eq-vrp0-3} used in the proof of~\cref{thm-vrp}.

\paragraph{}
\label{para-z2-fin}

Let $\bbZ_2 \mathhyphen \cat{Fin}$ be the category 
consisting of finite sets acted on by $\bbZ_2$,
and $\bbZ_2$-equivariant maps between them.
For integers $r, s \geq 0$, let $(r, s) \in \bbZ_2 \mathhyphen \cat{Fin}$
denote the set with $2r + s$ elements,
with a $\bbZ_2$-action fixing $s$ elements.

Let $\stirlingmat{r}{s}{r'}{s'}$
denote the number of $\bbZ_2$-equivariant surjections $(r, s) \twoheadrightarrow (r', s')$
up to automorphisms of $(r', s')$.
A standard argument shows that
\begin{align}
    \label{eq-z2-surj-0}
    \stirlingMat{r}{s}{r'}{0} & =
    \begin{cases}
        \displaystyle \stirling{r}{\, r' \,} \, 2^{r - r'}
        & \text{if } s = 0 \ , \\
        0 & \text{if } s > 0 \ ,
    \end{cases}
    \\
    \label{eq-z2-surj-1}
    \stirlingMat{r}{s}{r'}{1} & =
    \sum_{k = \max \{ 0, 1-s \}}^{r-r'} {}
    \binom{r}{\, k \,} \,
    \stirling{r - k}{r'} \,
    2^{r - r' - k} \ ,
\end{align}
where $\stirling{r}{r'}$ denotes the
\emph{Stirling number of the second kind},
defined as the number of partitions of $\{ 1, \dotsc, r \}$
into $r'$ non-empty subsets.

\paragraph{}
\label[lemma]{lem-vrp-comb}

We reformulate the identity \cref{eq-vrp0-3}
as the following result.

\begin{lemma*}
    Let $J = (r, s) \in \bbZ_2 \mathhyphen \cat{Fin}$.
    Then
    \begin{align}
        \label{eq-vrp-comb}
        \sum_{\leftsubstack[8em]{
            & J = J_0 \twoheadrightarrow \cdots \twoheadrightarrow J_m \text{ in } \bbZ_2 \mathhyphen \cat{Fin} \colon \\[-.5ex]
            & (J_m)^{\smash[b]{\bbZ_2}} = J_m
        } } {}
        (-1)^{m} =
        \begin{cases}
            (-1)^r \, (2r - 1)!! & \text{if } s \leq 1 \ , \\
            0 & \text{otherwise} \ ,
        \end{cases}
    \end{align}
    where we sum over chains of non-bijective surjections
    $J = J_0 \twoheadrightarrow \cdots \twoheadrightarrow J_m$,
    up to isomorphisms that are the identity on $J$.
\end{lemma*}

\begin{proof}
    We prove the lemma by induction on $r$, then on $s$.
    If $r = 0$, the result is standard,
    since when $s > 1$,
    chains with $J_m = (0, 1)$ and $J_m \neq (0, 1)$
    cancel out in pairs.

    Let $A (r, s)$ denote the left-hand side of~\cref{eq-vrp-comb}.
    We have the recursive formula
    \begin{equation}
        \label{eq-vrp-comb-1}
        A (r, s) =
        - \sum_{ \substack{ 0 \leq r' \leq r \\ 0 \leq s' \leq r + s \\ (r, s) \neq (r', s') } } {}
        \stirlingMat{r}{s}{r'}{s'} \cdot
        A (r', s') \ ,
    \end{equation}
    with notations as in~\cref{para-z2-fin}.
    However, by the induction hypothesis,
    only terms with $s' \leq 1$ contribute to the sum.
    Using~\crefrange{eq-z2-surj-0}{eq-z2-surj-1}, we obtain
    \begin{align}
        \label{eq-vrp-comb-2}
        A (r, 0) = A (r, 1) & =
        - \sum_{r' = 0}^{r-1} {}
        (-1)^{r'} \, (2r' - 1)!! \cdot
        \sum_{k=0}^{r-r'} {}
        \binom{r}{\, k \,} \,
        \stirling{r - k}{r'} \,
        2^{r - r' - k} \ ,
    \end{align}
    where we changed the $k$ in~\cref{eq-z2-surj-1} to $r - r' - k$.
    Therefore, to prove~\cref{eq-vrp-comb} for $s = 0$ and $1$,
    it is enough to show that
    \begin{equation}
        \label{eq-vrp-comb-3}
        \sum_{r' = 0}^{r} {}
        (-1)^{r'} \, (2r' - 1)!! \cdot
        \sum_{k=0}^{r-r'} {}
        \binom{r}{\, k \,} \,
        \stirling{r - k}{r'} \,
        2^{r - r' - k}
        = 0 \ .
    \end{equation}
    We postpone the proof of this to~\cref{para-vrp-comb-red}.

    If $s > 1$, then compared to the case when $s = 1$,
    the right-hand side of~\cref{eq-vrp-comb-1}
    has an extra term $-A (r, 1)$.
    Since all other terms add up to $A (r, 1)$,
    we obtain $A (r, s) = 0$,
    which finishes the proof.
\end{proof}

\paragraph{Proof of~\cref{eq-vrp-comb-3}.}
\label{para-vrp-comb-red}

We translate the identity~\cref{eq-vrp-comb-3}
back to chains of surjections in $\bbZ_2 \mathhyphen \cat{Fin}$,
giving
\begin{align}
    \label{eq-vrp-comb-4}
    \sum_{ \leftsubstack[6em]{
        & r = r_0 > \cdots > r_m = 0; \\[-.5ex]
        & (r_0, 1) \twoheadrightarrow \cdots \twoheadrightarrow (r_m, 1) \text{ in } \bbZ_2 \mathhyphen \cat{Fin}
    } } {}
    (-1)^{m} =
    (-1)^r \, (2r - 1)!! \ ,
\end{align}
where we sum over chains of surjections
up to isomorphisms that are the identity on $(r_0, 1)$.
Note that we have reduced~\cref{eq-vrp-comb}
to the case when $s = 1$,
and we are now restricted to contributions
from chains in which every term is of the form $(r', 1)$.

We now prove~\cref{eq-vrp-comb-4} by induction on $r$.
Write the $\bbZ_2$-action on $(r, 1)$ as an involution
$(-)^\vee \colon (r, 1) \simto (r, 1)$.
Let $0 \in (r, 1)$ be the unique element with $0^\vee = 0$,
and fix an element $1 \in (r, 1)$ with $1^\vee \neq 1$.

For a chain
$(r_\bullet, 1) = \bigl( (r_0, 1) \twoheadrightarrow \cdots \twoheadrightarrow (r_m, 1) \bigr)$,
write $p_k \colon (r_0, 1) \to (r_k, 1)$
and $p_{\smash{k}}^{\smash{k'}} \colon (r_{k'}, 1) \to (r_k, 1)$
for the surjections, where $0 \leq k' \leq k \leq m$.
Define $K_k = (p_k)^{-1} (p_k (1)) \subset (r_0, 1)$,
giving chain of subsets
$\{ 1 \} = K_0 \subset \cdots \subset K_m = (r_0, 1)$.

We say that an index $k \in \{ 0, \dotsc, m - 1 \}$
is \emph{special} for the chain $(r_\bullet, 1)$
if $0 \notin K_k \neq K_{k+1}$,
and \emph{critical} if moreover $0 \in K_{k+1}$.
Then every chain has a unique critical index.
A special index $k$ is \emph{good} if
$p_{\smash{k+1}}^{\smash{k}}$
does not factor as a composition of two non-bijective surjections
through an object of the form $(r', 1)$,
or \emph{bad} otherwise.
We will show that only chains with all indices good
up to the critical one contribute to the sum,
and that other chains cancel out in pairs.

Fix a chain $(r_\bullet, 1)$ of length~$m$,
such that not all indices up to the critical index $k_0$ are good.
Let $k \leq k_0$ be the smallest special index 
that is bad or has a preceding non-special index.
We construct a new chain $(r'_\bullet, 1)$ of length $m \pm 1$, as follows.
If $k$ is good, then $k > 0$ and $k - 1$ is non-special,
and we set $(r'_\bullet, 1)$ to be the chain
\[
    (r_0, 1) \twoheadrightarrow \cdots \twoheadrightarrow (r_{k-1}, 1) \twoheadrightarrow (r_{k+1}, 1) \twoheadrightarrow \cdots \twoheadrightarrow (r_m, 1) \ ,
\]
where $k - 1$ becomes bad.
On the other hand, if $k$ is bad in $(r_\bullet, 1)$,
we set $(r'_\bullet, 1)$ to be the chain
\[
    (r_0, 1) \twoheadrightarrow \cdots \twoheadrightarrow (r_k, 1) \twoheadrightarrow
    (r_{k+1} + 1, 1) \twoheadrightarrow (r_{k+1}, 1) \twoheadrightarrow \cdots \twoheadrightarrow (r_m, 1) \ ,
\]
where the map $(r_k, 1) \twoheadrightarrow (r_{k+1} + 1, 1)$
sends $p_k (1)$ and $p_k (1)^\vee$ to the extra pair of elements,
and all other elements to $(r_{k+1}, 1)$ via $p_{\smash{k+1}}^{\smash{k}}$;
the map $(r_{k+1} + 1, 1) \twoheadrightarrow (r_{k+1}, 1)$
sends the extra pair of elements to
$p_{k+1} (1)$ and $p_{k+1} (1)^\vee$,
and all other elements to $(r_{k+1}, 1)$ canonically.
The former map is non-bijective since $k$ is bad.
In the new chain, $k$ becomes non-special and $k + 1$ becomes good.

Note that applying the above construction twice to a chain
gives back the original chain.
This implies that contributions of these chains
to~\cref{eq-vrp-comb-4} cancel out in pairs,
and we may consider only those chains
with all indices good up to the critical one.
We say that these chains are \emph{good}.

We divide good chains into equivalence classes,
where two chains are equivalent if they are the same
before and including the critical index.
Fixing a critical index $k_0$,
the number of such equivalence classes with critical index $k_0$
is equal to $(2 r - 2) (2 r - 4) \cdots (2 r - 2 k_0)$,
and the contribution of each class to~\cref{eq-vrp-comb-4}
is $(-1)^r \, (2 r - 2 k_0 - 3)!!$,
by the induction hypothesis.
Therefore, the total contribution of good chains is
\[
    \sum_{k_0 = 0}^{r-1} {}
    (-1)^r \,
    (2 r - 2) (2 r - 4) \cdots (2 r - 2 k_0)
    \cdot (2 r - 2 k_0 - 3)!! 
    =
    (-1)^r \, (2 r - 1)!! \ ,
\]
which completes the proof of~\cref{eq-vrp-comb-4}.
\QED

\paragraph{}

For the reader's convenience,
we also state the analogous result of
\cref{lem-vrp-comb}
in the non-self-dual case,
which can be proved by a similar, but simpler argument.

\begin{lemma*}
    Let $\cat{Fin}$ be the category of finite sets,
    and let $J \in \cat{Fin}$ be non-empty.
    Then
    \begin{align}
        \sum_{\leftsubstack[8em]{
            & J = J_0 \twoheadrightarrow \cdots \twoheadrightarrow J_m = \{ * \} \textnormal{ in } \cat{Fin}
        } } {}
        (-1)^{m} =
        (-1)^{|J| - 1} \, (|J| - 1)! \ ,
    \end{align}
    where we sum over chains of non-bijective surjections
    $J = J_0 \twoheadrightarrow \cdots \twoheadrightarrow J_m = \{ * \}$,
    up to isomorphisms that are the identity on $J$. \QED
\end{lemma*}

\subsection{Proof of the theorem}

\paragraph{}
\label{para-filt-sigma}

A key idea in our proof is to apply \cref{thm-vrp,thm-vrp-higher},
not only to $\calA$, but also to
the category $\calA^{\pn{n}}$ of filtrations of $\calA$,
discussed in \cref{sect-sd-filt,sect-mod-filt-sd}.

For a self-dual weak stability condition $\tau$ on $\calA$
satisfying \tagref{Stab}, and for a slope~$t$,
let $\calA (\tau; t) \subset \calA$ be the full subcategory of
$\tau$-semistable objects $E \in \calA$ with either $E \simeq 0$ or $\tau (E) = t$.
Then $\calA (\tau; t)$ is closed under direct summands,
so it satisfies \tagref{Spl}, \tagref{Mod}, and \tagref{Fin},
and the trivial stability condition~$0$ on~$\calA (\tau; t)$ satisfies~\tagref{Stab}.
In particular, when $t = 0$, the category $\calA (\tau; 0)$
is self-dual, and also satisfies \tagref{SdMod}.

Now, for an integer $n \geq 0$,
consider the category $\calA (\tau; t)^{\pn{n}}$
of filtrations in $\calA (\tau; t)$, as in \cref{sect-sd-filt}.
It again satisfies \tagref{Spl}, \tagref{Mod}, and \tagref{Fin},
and the trivial stability condition~$0$
on~$\calA (\tau; t)^{\pn{n}}$ satisfies~\tagref{Stab}.
When $t = 0$, it also satisfies \tagref{SdMod}
by \cref{sect-mod-filt-sd}.
We also denote this by $\calA (\tau; t)^{\pn{I}}$,
where $I$ is a totally ordered set with $n$ elements.

For $\alpha_1, \dotsc, \alpha_n \in C^\circ (\calA (\tau; t))$,
write $\delta^{\pn{n}}_{\alpha_1, \dotsc, \alpha_n} (\tau) \in \SF (\calM^\pn{n})$
for the stack function
$\delta_{(\alpha_1, \dotsc, \alpha_n)} (0)$
for the category $\calA (\tau; t)^{\pn{n}}$,
the class $(\alpha_1, \dotsc, \alpha_n) \in C^\circ (\calA (\tau; t)^{\pn{n}})$
in the sense of \cref{para-filt-class-monoid},
and the trivial stability condition~$0$.
Define $\sigma^{\pn{n}}_{\alpha_1, \dotsc, \alpha_n} (\tau) \in \SF (\calM^\pn{n})$
similarly, as the stack function $\sigma_{(\alpha_1, \dotsc, \alpha_n)} (0)$
for $\calA (\tau; t)^{\pn{n}}$ and the trivial stability condition~$0$.

Similarly, for $\alpha_1, \dotsc, \alpha_n \in C^\circ (\calA (\tau; 0))$
and $\rho \in C^\sd (\calA (\tau; 0))$,
define the stack functions
$\delta^{\pn{2n+1}, \> \sd}_{\alpha_1, \dotsc, \alpha_n, \> \rho} (\tau)$,
$\sigma^{\pn{2n+1}, \> \sd}_{\alpha_1, \dotsc, \alpha_n, \> \rho} (\tau) \in \SF (\calM^{\pn{2n+1}, \> \sd})$
as the $\delta^\sd$ and $\sigma^\sd$ functions
for the self-dual category $\calA (\tau; 0)^{\pn{2n+1}}$,
the class $(\alpha_1, \dotsc, \alpha_n, \rho) \in C^\sd (\calA (\tau; 0)^{\pn{2n+1}})$
in the sense of \cref{para-filt-sd-class-monoid},
and the trivial stability condition~$0$.

\begin{theorem}
    \label{thm-no-pole-sd-app}
    The element
    \[
        \epsilon^\sd_\theta (\tau) \in \SF (\calM^\sd)
    \]
    has pure virtual rank $0$.
    In particular, we have
    \begin{equation}
        \label{eq-no-pole-sd}
        \int_{\calM^\sd} \epsilon^\sd_\theta (\tau)
        \in \Mhat^\circ_\bbK \ .
    \end{equation}
\end{theorem}

\begin{proof}
    For $\alpha_1, \dotsc, \alpha_k \in C (\calA)$
    and $\rho \in C^\sd (\calA)$, we have
    \begin{equation}
        (\pi^{(2k + 1), \> \sd})_! \,
        ( \delta^{\pn{2k+1}, \> \sd}_{\alpha_1, \> \dotsc, \> \alpha_k, \> \rho} (\tau) )
        =
        \delta_{\alpha_1} (\tau) \diamond \cdots \diamond
        \delta_{\alpha_k} (\tau) \diamond
        \delta^\sd_{\rho} (\tau) \ ,
    \end{equation}
    where $\delta^{\pn{2k+1}, \> \sd}_{\alpha_1, \> \dotsc, \> \alpha_k, \> \rho} (\tau)$
    is defined in~\cref{para-filt-sigma}.
    Recall from~\cref{eq-def-epsilon-sd}
    the definition of $\epsilon^\sd_\theta (\tau)$,
    which can now be written as
    \begin{align}
        \label{eq-def-epsilon-sd-app}
        \epsilon^\sd_\theta (\tau) & =
        \sum_{ \leftsubstack[6em]{
            \\[-1.5ex]
            & k \geq 0; \, \alpha_1, \dotsc, \alpha_k \in C (\calA), \,
            \rho \in C^\sd (\calA) \colon \\[-.5ex]
            & \theta = \bar{\alpha}_1 + \cdots + \bar{\alpha}_k + \rho \\[-.5ex]
            & \tau (\alpha_1) = \cdots = \tau (\alpha_k) = 0
        } } {}
        \binom{-1/2}{k} \cdot
        (\pi^{(2k + 1), \> \sd})_! \,
        ( \delta^{\pn{2k+1}, \> \sd}_{\alpha_1, \> \dotsc, \> \alpha_k, \> \rho} (\tau) ) \ .
    \end{align}
    To prove the theorem, it is enough to prove that
    \begin{equation}
        \label{eq-vrp-epsilon-1}
        \vrp{n} (\epsilon^\sd_\theta (\tau)) = 0
    \end{equation}
    for all integers $n > 0$.

    By \cite[Proposition~5.14]{Joyce2007Stack},
    the virtual rank projection $\vrp{n}$ commutes with
    the pushforward $(\pi^{(2k + 1), \> \sd})_!$.
    We therefore study the virtual rank projections of
    $\delta^{\pn{2k+1}, \> \sd}_{\alpha_1, \> \dotsc, \> \alpha_k, \> \rho} (\tau)$.

    Write $I = \{ 1, \dotsc, k, 0, k^\vee, \dotsc, 1^\vee \}$
    and $J = \{ 1, \dotsc, n, 0, n^\vee, \dotsc, 1^\vee \}$,
    with total orders given by the written order,
    and the obvious involutions.
    Applying \cref{thm-vrp-higher}
    to the category $\calA (\tau; 0)^{\pn{I}}$ in \cref{para-filt-sigma},
    we obtain
    \begin{multline}
        \label{eq-vrp-epsilon-2}
        \vrp{n} (\delta^{\pn{2k+1}, \> \sd}_{\alpha_1, \> \dotsc, \> \alpha_k, \> \rho} (\tau)) =
        \frac{1}{2^n n!} \cdot {} 
        \\[1ex]
        \sum_{ \leftsubstack[6em]{
            \\[-2ex]
            & \alpha_{i, j} \in C^\circ (\calA) \text{ for } (i, j) \in (I \times J) \setminus (0, 0); \,
            \rho_{0, 0} \in C^\sd (\calA) \colon \\[-.5ex] 
            & \tau (\alpha_{i, j}) = 0 \text{ for all } (i, j) \text{ with } \alpha_{i, j} \neq 0, \\[-.5ex]
            & \alpha_{\smash{i^\vee, \> j^\vee}} = \alpha_{i, j}^\vee \text{ for all } (i, j), \\[-.5ex]
            & \textstyle \alpha_i = \sum_{j \in J} \alpha_{i, j} \text{ for all } i \in I , \\[-.5ex]
            & \rho = \bar{\alpha}_{0, 1} + \cdots + \bar{\alpha}_{0, n} + \rho_{0, 0} \, , \\[-.5ex]
            & \textstyle \sum_{i \in I} \alpha_{i, j} \neq 0 \text{ for all } j \in J \setminus 0
        } } {}
        \sigma^{\pn{I}}_{(\alpha_{i, 1})_{i \in I}} (\tau) \circleddiamond \cdots \circleddiamond
        \sigma^{\pn{I}}_{(\alpha_{i, n})_{i \in I}} (\tau) \circleddiamond
        \sigma^{\pn{I}, \> \sd}_{\alpha_{1, 0}, \, \dotsc, \> \alpha_{k, 0}, \> \rho_{0, 0}} (\tau) \ .
    \end{multline}

    We abbreviate each term in the sum~\cref{eq-vrp-epsilon-2}
    as $\sigma^{(I), \> \sd}_{\underline{\alpha}} (\tau)$,
    where $\underline{\alpha} = (\alpha_{i, j})_{i \in I, \, j \in J}$
    is a matrix with $\alpha_{0, 0} = \rho_{0, 0}$, and write
    \[
        \sigma^\sd_{\underline{\alpha}} (\tau) =
        (\pi^{(I), \> \sd})_! \, \sigma^{(I), \> \sd}_{\underline{\alpha}} (\tau) \ .
    \]
    Let $A_{n, k}$ be the set of matrices $\underline{\alpha}$
    that appear in~\cref{eq-vrp-epsilon-2}
    for some choice of $\alpha_1, \dotsc, \alpha_k, \rho$.
    Then \cref{eq-def-epsilon-sd-app,eq-vrp-epsilon-2} imply that
    \begin{equation}
        \label{eq-vrp-epsilon-3}
        \vrp{n} (\epsilon^\sd_\theta (\tau)) =
        \frac{1}{2^n n!} \cdot {} 
        \sum_{k \geq 0}
        \sum_{\underline{\alpha} \in A_{n, k}} {}
        \binom{-1/2}{k} \cdot
        \sigma^\sd_{\underline{\alpha}} (\tau) \ .
    \end{equation}
    Note that the element $\sigma^\sd_{\underline{\alpha}} (\tau)$
    only depends on the equivalent class of $\underline{\alpha}$,
    where two matrices $\underline{\alpha} \in A_{n, k}$ and $\underline{\alpha}' \in A_{k'}$
    are equivalent if for all $j \in J$,
    the subsequence of $(\alpha_{i, j})_{i \in I}$ with $\alpha_{i, j} \neq 0$
    is the same as
    the subsequence of $(\alpha'_{i, j})_{i \in I'}$ with $\alpha'_{i, j} \neq 0$,
    where $I' = \{ 1, \dotsc, k', 0, k'^\vee, \dotsc, 1^\vee \}$.
    It is then enough to prove that for a fixed $\underline{\alpha}$, we have
    \begin{align}
        \label{eq-vrp-epsilon-4}
        \sum_{k \geq 0} {} \binom{-1/2}{k} \cdot
        \sum_{ \leftsubstack[2.5em]{
            & \underline{\alpha}' \in A_{n, k} \colon \\[-.5ex]
            & \underline{\alpha}' \sim \underline{\alpha}
        } } {} 1 & = 0 \ .
    \end{align}

    To prove this, we first observe that
    the number $\underline{\alpha}' \in A_{n, k}$ with $\underline{\alpha}' \sim \underline{\alpha}$
    is equal to the number of subsets of $I \times J$,
    invariant under the involution $(i, j) \mapsto (i^\vee, j^\vee)$,
    such that the number of elements in each row $\{ i \} \times J$
    is non-zero unless $i = 0$,
    and the number of elements in each column $I \times \{ j \}$
    is equal to $a_j$, where $a_j$ is the number of non-zero entries
    $\alpha_{i, j}$ in $\underline{\alpha}$ in that column.
    Note that $a_j = a_{j^\vee}$ for all $j \in J$.

    Consider the generating series of this counting problem,
    with a formal variable $x_j$ assigned to each $a_j$
    for $j \in \{ 1, \dotsc, n, 0 \}$.
    For convenience, we write $x_{j^\vee} = x_j$ for $j \in J$.
    Summing over $k \geq 0$, we obtain the generating series
    \begin{align*}
        F (x_1, \dotsc, x_n, x_0)
        & =
        \sum_{k \geq 0} {} \binom{-1/2}{k} \cdot
        \biggl(
            \ \sum_{ \varnothing \neq J' \subset J }
            \ \prod_{j \in J'} x_j
        \biggr)^k \cdot
        \sum_{J' \subset \{ 1, \dotsc, n, 0 \}}
        \prod_{j \in J'} x_j
        \\
        & =
        \biggl(
            \ \prod_{j \in J} {} (1 + x_j)
        \biggr)^{-1/2} \cdot
        \prod_{j = 0}^n {} (1 + x_j)
        \\
        & =
        (1 + x_0)^{1/2} \ .
    \end{align*}
    Therefore, the left-hand side of~\cref{eq-vrp-epsilon-4}
    is zero unless $a_j = 0$ for all $j \neq 0$.
    But this cannot happen, as we assumed that $n > 0$.
    This proves~\cref{eq-vrp-epsilon-4}, and hence~\cref{eq-vrp-epsilon-1}.
\end{proof}